\documentclass[11pt, twoside]{article}
\usepackage{amssymb}
\usepackage{amsfonts}
\usepackage{bbm}
\usepackage{mathrsfs}
\usepackage{amssymb,amsmath,amsfonts,amsthm}

\allowdisplaybreaks

\begin{document}

\baselineskip=15.5pt
\renewcommand{\arraystretch}{2}
\arraycolsep=1pt

\theoremstyle{plain}
\newtheorem{theorem}{Theorem}[section]
\newtheorem{prop}[theorem]{Proposition}
\newtheorem{lemma}[theorem]{Lemma}
\newtheorem{cor}[theorem]{Corollary}
\newtheorem{example}[theorem]{Example}
\newtheorem{remark}[theorem]{Remark}
\newcommand{\ra}{\rightarrow}
\renewcommand{\theequation}
{\thesection.\arabic{equation}}

\newtheorem*{TheoremA}{Theorem A}
\newtheorem*{TheoremB}{Theorem B}
\newtheorem*{TheoremC}{Theorem C}

\theoremstyle{definition}
\newtheorem{definition}[theorem]{Definition}

\newcommand{\XX}{\mathcal {X}}
\newcommand{\ZZ}{\mathbb{Z}}
\newcommand{\G}{\mathop G\limits^{\circ}}
\newcommand{\GG}{{\mathop G\limits^{\circ}}_{\vartheta}}
\newcommand{\GGp}{{\mathop G\limits^{\circ}}_{\vartheta_1,\vartheta_2}}

\title
{\bf\Large $T1$ theorem on product Carnot--Carath\'eodory spaces
\footnotetext{ The second author was
supported by NNSF of China (Grant No. 11001275), China Postdoctoral Science Foundation funded project (Grant No. 201104383) and the
Fundamental Research Funds for the Central Universities (No. 11lgpy56). The third author
was supported by NSC of Taiwan under Grant \#NSC 100-2115-M-008-002-MY3.}
\footnotetext{$\dag$ Corresponding author}
}

\author{Yongsheng Han, Ji Li$^{\dag}$ and Chin-Cheng Lin
}

\date{ }

\maketitle

\begin{center}
\begin{minipage}{14.5cm}
{\small{\bf Abstract:} Nagel and Stein established
$L^p$-boundedness for a class of singular
integrals of NIS type, that is, non-isotropic smoothing operators of
order 0, on spaces $\widetilde{M}=M_1\times\cdots\times M_n,$ where each
factor space $M_i, 1\leq i\leq n,$ is a smooth manifold on which the
basic geometry is given by a control, or Carnot--Carath\'eodory,
metric induced by a collection of vector fields of finite type. In this paper we
prove the product $T1$ theorem on $L^2,$ the Hardy space $H^p(\widetilde{M})$ and the space $CMO^p(\widetilde{M}),$
the dual of $H^p(\widetilde{M}),$ for a
class of product singular integral operators which covers Journ\'e's class and operators studied by Nagel and Stein.}
\end{minipage}
\end{center}

\bigskip\bigskip

\tableofcontents

\bigskip\bigskip

\section{Introduction}
\setcounter{equation}{0}
 In their remarkable theory, Calder\'{o}n and Zygmund generalized the Hilbert transform on $\mathbb{R}$ to certain convolution operators on $\mathbb{R}^n.$ These operators are of the form $T(f)=K\ast f$ and $K(x),$ the convolution kernel, is defined on $\mathbb{R}^n$ and satisfies the analogous conditions that $\frac{1}{x}$ satisfies on $\mathbb R,$ namely the regularity and cancellation conditions. This convolution operator theory was generalized in two directions. In the first extension, these convolution operators were extended to non-convolution operators associated with a kernel. To be precise, let $K(x,y)$ be a locally integrable function defined on $x\in \mathbb{R}^n$ and $y\in \mathbb{R}^n$ with $x\not= y.$ Let $T: C^\infty_0(\mathbb{R}^n)\rightarrow \big(C^\infty_0(\mathbb{R}^n)\big)^\prime$ be a linear operator associated with the kernel $K$ in the following sense: If for $f, g\in C^\infty_0(\mathbb{R}^n)$ with disjoint supports, $\langle Tf,g\rangle $ is given by $\iint g(x)K(x,y)f(y)dxdy.$ Suppose that $K$ satisfies some size and smoothness conditions analogous to those enjoined by the kernels of the Riesz transforms on $\mathbb{R}^n$. The $L^2$ boundedness of $T,$ in general, cannot conclude by using Plancherel's theorem if $T$ is not a convolution operator. Note that if $T$ is bounded on $L^2,$ then the program of Calder\'on--Zygmund can be carried out and the $L^p, 1<p<\infty,$ boundedness of $T$ follows. The $L^2$ boundedness for non-convolution operators was an open problem until David and Journ\'{e} [DJ] proved the remarkable $T1$ theorem. This theorem asserts that under some regularity conditions, $T$ is bounded on $L^2$ if and only if both $T1$ and $T^*1$, defined appropriately, lie on $BMO(\mathbb{R}^n).$

The second extension is due to R. Fefferman and Stein \cite{FS}.
They extended this theory to the multiparameter product convolution
operators. More precisely, Fefferman and Stein took the space
$\mathbb{R}^n\times \mathbb{R}^m$ along with the two parameter dilations instead of
the classical one-parameter dilations and consider convolution
operators $Tf=K\ast f$ where $K$ is defined on $\mathbb{R}^n\times \mathbb{R}^m$ and
satisfies all analogous conditions to those satisfied by
$\frac{1}{xy},$ the double Hilbert transform on $\mathbb{R}\times \mathbb{R}.$ Using
Plancherel's theorem, under some regularity and cancellation
conditions, Fefferman and Stein obtained the $L^2$ boundedness of
$T.$ However, the program of Calder\'on--Zygmund for one parameter
case doesn't work for multiparameter case. To prove the $L^p,
1<p<\infty,$ boundedness of $T,$ Fefferman and Stein developed the
multiparameter Littlewood--Paley theory on $L^p, 1<p<\infty.$
Finally, the $L^p, 1<p<\infty,$ boundedness of $T$ follows from such
a theory and the almost orthogonality argument. See [FS] for more
details.

Journ\'{e} [J] unified up these two extensions to multiparameter singular integral operators on a product of $n$ Euclidean spaces. Precisely, Journ\'{e} introduced a class of singular integral operators which coincides with one parameter non-convolution operators and coincides with the convolution case for the Fefferman and Stein class. To be more precise, let $T_1$ and $T_2$ be two classical singular integral operators on $\mathbb{R}$ and let $T=T_1\otimes T_2.$ This operator can be defined from $C^\infty_0(\mathbb{R})\otimes C^\infty_0(\mathbb{R})$ to its dual $[C^\infty_0(\mathbb{R})\otimes C^\infty_0(\mathbb{R})]^\prime$ by
$$\langle Tf_1\otimes f_2, g_1\otimes g_2\rangle=\langle T_1f_1, g_1\rangle\langle T_2f_2, g_2\rangle.$$
Let $K_1$ and $K_2$ be the kernels of $T_1$ and $T_2,$ respectively. If $f_1, g_1\in C^\infty_0(\mathbb{R})$ with disjoint supports, then
\begin{eqnarray*}
\langle Tf_1\otimes f_2, g_1\otimes g_2\rangle&=&\iint g_1(x)K_1(x,y)f_1(y)\langle T_2f_2, g_2\rangle dxdy\\
&=&\iint g_1(x)\langle{\widetilde K}_1(x,y)f_2, g_2\rangle f_1(y)dxdy,
\end{eqnarray*}
where ${\widetilde K}_1(x,y)=K_1(x,y)T_2.$ Similarly, If $f_2, g_2\in C^\infty_0(\mathbb{R})$ with disjoint supports, one can define ${\widetilde K}_2(x,y)=K_2(x,y)T_1$ and write
$$\langle Tf_1\otimes f_2, g_1\otimes g_2\rangle=\iint g_2(x)\langle{\widetilde K}_2(x,y)f_1, g_1\rangle f_2(y)dxdy.$$
The class of singular integral operators introduced by Journ\'{e} is the collection of operators $T$ which is a continuous linear mapping from  $C^\infty_0(\mathbb{R})\otimes C^\infty_0(\mathbb{R})$ to its dual $[C^\infty_0(\mathbb{R})\otimes C^\infty_0(\mathbb{R})]^\prime.$ Moreover, there exists a pair $(K_1, K_2)$ of classical kernels such that for all $f, g, h, k\in C^\infty_0(\mathbb{R}),$ with supp $f\cap$ supp $g=\emptyset,$
$$\langle Tf\otimes h, g\otimes k\rangle=\iint g(x)\langle K_1(x,y)h, k\rangle f(y)dxdy,$$
$$\langle Th\otimes f, k\otimes g\rangle=\iint g(x)\langle K_2(x,y)h, k\rangle f(y)dxdy.$$
Journ\'{e} found that the classical $T1$ theorem doesn't work for such a class of operators. Indeed, by constructing an operator, he shows that $\widetilde T1$ and  ${\widetilde T}^*1$ have to be taken into account in order to obtain the $L^2$ boundedness of $T,$ where $\widetilde T$ is called the partial adjoint operator of $T$ defined by
$$\langle Tf\otimes h, g\otimes k\rangle=\langle{\widetilde T}g\otimes h, f\otimes k\rangle.$$
Note that, in general, the $L^2$ boundedness of $T$ cannot imply the $L^2$ boundedness of $\widetilde T.$
Finally, Journ\'{e} proved the product $T1$ theorem which asserts that under some regularity conditions, the operator $T$ belonging to Journ\'{e} class and its partial adjoint $\widetilde T$ both are bounded on $L^2(\mathbb{R}\times \mathbb{R})$ if and only if $T1, T^*1, \widetilde T1, {\widetilde T}^*1$ lie on the product $BMO(\mathbb{R}\times \mathbb{R}),$ where $BMO(\mathbb{R}\times \mathbb{R})$ was introduced in [CF] in terms of the Carleson measure on $\mathbb{R}\times \mathbb{R}.$

To study fundamental solutions of $\Box_b$ on certain model domains
in several complex variables, Nagel and Stein established
$L^p$-boundedness for a class of product singular integral operators
on spaces $\widetilde{M}=M_1\times\cdots\times M_n,$ where each
factor space $M_i, 1\leq i\leq n,$ is a smooth manifold on which the
basic geometry is given by a control, or Carnot--Carath\'eodory,
metric induced by a collection of vector fields of finite type. It
was pointed out in [NS04] that any analysis of product singular
integrals on a product space $\widetilde{M}=M_1\times\cdots\times
M_n$ must be based on a formulation of standard singular integrals
on each factor $M_i, 1\leq i\leq n.$ There are two paths to do that.
One is to generalize the class of operators on each factor $M_i,
1\leq i\leq n,$ to the extended class of the $T1$ theorem of David
and Journ\'{e} [DJ] and then pass from this to a corresponding
product theory. This, as mentioned above, was carried out in [J] for
the setting where each factor is an Euclidean space. However, because
of the inherent complications, Nagel and Stein chose a simpler
approach. More precisely, they considered the class of singular
integrals of NIS type, that is, non-isotropic smoothing operators of
order 0. These operators may be viewed as Calder\'on--Zygmund
operators whose kernels are $C^\infty$ away from the diagonal and
its cancellation conditions are given by their action on smooth bump
functions. These cancellation conditions make the operators on each
$M_i, 1\leq i\leq n,$ easy to handle and then this carried out to
the product-type operators on $\widetilde{M}.$ The key to the proof
of the $L^p$ boundedness for these operators is the existence of a
Littlewood--Paley theory on $\widetilde{M},$ which itself is  a
consequence of the corresponding theory on each factor. We would
like to remark that the cancellation conditions used in [NS04] are
simple but less the generality in scope. More precisely, these
cancellation conditions imply $T1, T^*1\in BMO(M_i)$ on each $M_i,
1\leq i\leq n,$ and $T1, T^*1, {\widetilde T}1, {\widetilde T}^*1\in
BMO(\widetilde M)$ on $\widetilde M$, respectively. To see this,
recently in [HLL2] the Hardy space theory was established in the
setting of product spaces of homogeneous type in sense of Coifman
and Weiss [CW] which covers the product Carnot--Carath\'eodory
spaces. This theory includes the $H^p$ boundedness for operators
studied in [NS04] and the product $CMO^p(\widetilde M)$ space, which
is the dual space of $H^p(\widetilde M),$ particularly,
$CMO^1(\widetilde M)=BMO(\widetilde M)$ is the dual of
$H^1(\widetilde M).$ We point out that the $H^p(\widetilde M)$ boundedness of operators studied by Nagel and Stein
was proved in [HLL2] in terms of the cancellation conditions used in [NS04].
Moreover, a very general result proved in [HLL2] states that  both the
$L^2(\widetilde M)$ and $H^p(\widetilde M)$ boundedness imply the $H^p(\widetilde M)\rightarrow L^p(\widetilde M)$ boundedness without
using atomic decomposition and Journ\'{e}'s covering lemma. Thus, if
$T$ is the operator studied by Nagel and Stein then $T$ is
bounded on both $L^2(\widetilde M)$ and $H^p(\widetilde M),$ and hence $T$ is also bounded from
$H^1(\widetilde M)$ to $L^1(\widetilde M).$ From this together with the duality, $T$ is bounded
from $L^\infty(\widetilde M)$ to $BMO(\widetilde M).$

As mentioned, since the Hardy space $H^p(\widetilde M)$ and its dual space $CMO^p(\widetilde M)$ have been developed in [HLL2], particularly, the dual of $H^1(\widetilde M)$ is the space $CMO^1(\widetilde M)=BMO(\widetilde M),$ it is natural to consider the $T1$ theorem on the product Carnot--Carath\'eodory spaces $\widetilde M.$ The purpose of this paper is to prove such a product $T1$ theorem for a class of product singular integral operators whose kernels satisfy the weaker regularity properties. This class covers Journ\'{e}'s class when each factor is an Euclidean space and operators studied in [NS04]. The product $T1$ theorem proved in this paper asserts that an operator $T$ and its partial adjoint operator $\widetilde T$ are both bounded on $L^2$ if and only if $T1, T^*1, {\widetilde T}1, {\widetilde T}^*1$ lie on the product $BMO(\widetilde M),$ where $BMO(\widetilde M),$ as mentioned, was introduced in [HLL2].

To show the necessary conditions that the $L^2$ boundedness of $T$ implies that $T1$ and $T^*1$ lie on the product $BMO(\widetilde M),$ we will employ an approach which is different from one given by Journ\'{e} [J]. Journ\'e obtained this implication by showing that the $L^2(\widetilde M)$ boundedness implies the $L^\infty(\widetilde M)\rightarrow BMO(\widetilde M)$ boundedness. For this purpose, he established a fundamental geometric covering lemma. As a consequence of this implication, together with an interpolation theorem and the duality argument, Journ\'{e} proved that the $L^2(\widetilde M)$ boundedness implies the $L^p(\widetilde M), 1<p<\infty,$ boundedness. In this paper, we will prove this implication by use of the Hardy space theory developed in [HLL2]. More precisely, we will show that the $L^2(\widetilde M)$ boundedness implies the $H^1(\widetilde M)\rightarrow L^1(\widetilde M)$ boundedness. We would like to point out that under the cancellation conditions used by Nagel and Stein, the $H^1(\widetilde M)\rightarrow L^1(\widetilde M)$ boundedness was obtained in [HLL2]. However, the method used in [HLL2] does not work for the present situation. Indeed, to get the $H^1(\widetilde M)\rightarrow L^1(\widetilde M)$ boundedness in [HLL2], they show the $H^1(\widetilde M)$ boundedness first. This is why the cancellation conditions of Nagel and Stein were needed in [HLL2]. In this paper, to show that the $L^2(\widetilde M)$ boundedness implies the $H^1(\widetilde M)\rightarrow L^1(\widetilde M)$ boundedness without assuming any cancellation conditions, we will apply an atomic decomposition for $H^p(\widetilde M).$ For this purpose, we first establish Journ\'{e}-type covering lemma in our setting. Applying an atomic decomposition and a similar idea as in [F], we conclude that $L^2(\widetilde M)$ boundedness implies the $H^p(\widetilde M)\rightarrow L^p(\widetilde M)$ boundedness. And, particularly, $H^1(\widetilde M)\rightarrow L^1(\widetilde M)$ boundedness follows. From this together with the duality between $H^1(\widetilde M)$ and $BMO(\widetilde M)$ we obtain the $L^\infty(\widetilde M)\rightarrow BMO(\widetilde M)$ boundedness and hence the desired necessary conditions follow.
By an interpolation theorem proved in [HLL2], we also conclude that the $L^2(\widetilde M)$ boundedness implies the $L^p, 1<p<\infty,$ boundedness.

In [J] the proof of the sufficient conditions for the classical product $T1$ theorem was decomposed in three steps. In the first step, Journ\'e claimed that if $T$ satisfies $T_1(1)=T^*_1(1)=0,$ see definition for $T_1(1)=0$ and $T_2(1)=0$ in Subsection 3.1, and has the weak boundedness property, then it can be viewed as a classical vector valued singular integral operator, $\widetilde T$ acting on $C^\infty_0(\mathbb{R})\times H,$ where $H=L^2(\mathbb{R}, dx_2),$ and for which $\widetilde T(1)={\widetilde T}^*(1)=0.$ The proof of the $L^2$-boundedness of such an operator follows from the classical case.

The second step is the decomposition of an operator $T$ having the weak boundedness property, such that $T(1)= T^*(1)= {\widetilde T}(1)= {\widetilde T}^*(1)=0$ as the sum of two operators $S$ and $T-S$ having the weak boundedness property and such that $S_2(1)=S^*_2(1)=0$ and $(T-S)_1(1)=(T-S)^*_1(1)=0.$ The $L^2$ boundedness of $T$ is then a consequence of the first step. To construct the operator $S,$ let $\beta \in BMO(\mathbb{R})$ and let $U_\beta$ be defined by $\langle g, U_\beta f\rangle=\int\limits_0^\infty \langle (Q_tg), (Q_t\beta)(P_tf)\rangle\frac{dt}{t}.$ It is classical that this integral is absolutely convergent and that $U_\beta$ is a Carder\'on-Zygmund operator. Moreover, $U_\beta(1)=\beta$ and $U^*_\beta(1)=0.$ Now let $T(1)= T^*(1)= {\widetilde T}(1)= {\widetilde T}^*(1)=0.$ Journ\'e defined the operator $N$ as follows. For all $f_1,f_2,g_1,g_2\in C^\infty_0(\mathbb{R})$
$$\langle g_1\otimes g_2, Nf_1\otimes f_2\rangle=\langle g_1, U_{\lbrace \langle g_2, T_2f_2\rangle(1)\rbrace}f_1\rangle.$$
The operator $M,$ similar to $N,$ is defined by
$$\langle g_1\otimes g_2, Mf_1\otimes f_2\rangle=\langle g_1, U^*_{\lbrace \langle g_2, T_2f_2\rangle^*(1)\rbrace}f_1\rangle.$$
Now set $S=M+N$ so that $S_2(1)=S^*_2(1)=0$ and $(T-S)_1(1)=(T-S)^*_1(1)=0.$

The last step is, as in the classical case, to construct the para-product operators. To see this step, let $b\in BMO(\mathbb{R}\times \mathbb{R})$ and let the para-product operator $W_b: C^\infty_0(\mathbb{R})\otimes C^\infty_0(\mathbb{R})\rightarrow [C^\infty_0(\mathbb{R})\otimes C^\infty_0(\mathbb{R})]^\prime$ be defined by
$$\langle f_1\otimes f_2, W_bg_1\otimes g_2\rangle=\int_0^\infty\int_0^\infty \langle Q_{t_1}f_1\otimes Q_{t_2}f_2, (Q_{t_1}Q_{t_2}b)P_{t_1}g_1\otimes P_{t_2}g_2\rangle\frac{dt_1}{t_1}\frac{dt_2}{t_2}.$$
Then we have that $W_b1=b, W^*_b1={\widetilde W}_b1={\widetilde W}^*_b1=0.$ If set $S=T-W_{T1}-W^*_{T^*1}-{\widetilde W}_{{\widetilde T}1}-{\widetilde W}^*_{{\widetilde T}^*1},$ then $S(1)=S^*(1)={\widetilde S}(1)={\widetilde S}^*(1)=0.$ Moreover, all para-product operators $W_b, W^*_b, {\widetilde W}_b$ and ${\widetilde W}^*_b$ are in Journ\'e's class and bounded on $L^2(\mathbb{R}\times \mathbb{R}).$

We would like to point out that it seems that in the second step above, the construction and the proof of the $L^2(\mathbb{R}\times \mathbb{R})$ boundedness of $S$ both only work for functions having the form $f(x,y)=f_1(x)f_2(y),$ where $f_1,f_2\in C^\infty_0(\mathbb{R})$. See the details on the page 76-78 in [J]. Unfortunately, such a collection of functions with the form $f(x,y)=f_1(x)f_2(y)$ is not dense in $L^2(\mathbb{R}\times \mathbb{R}).$

In this paper, we will develop a new approach to prove the sufficient conditions for the $T1$ theorem on the product space $\widetilde M=M_1\times M_2.$ To describe the novelty of this approach more carefully, we first outline a new proof for the classical $T1$ theorem on $M_1.$ In the classical one parameter case, the $T1$ theorem was proved by two steps in [DJ]. In the first step, one observes that if $T$ satisfies $T(1)=T^*(1)=0$ and has the weak boundedness property, then the almost orthogonality argument together with the Littlewood--Paley estimate on $L^2$ gives the $L^2$ boundedness of $T.$ We emphasize that the conditions $T(1)=T^*(1)=0$ play a crucial role for applying the almost orthogonality argument. In the second step, one can write $T=[T-\Pi_{T1}-{\Pi}^*_{T^*1}]+\Pi_{T1}+{\Pi}^*_{T^*1},$ where for a BMO function $b,$ $\Pi_b$ is the para-product operator defined in [DJ]. It was known that the para-product is a Calder\'on--Zygmund singular integral operator and bounded on $L^2,$ and the operator $T-\Pi_{T1}-{\Pi}^*_{T^*1}$ is of the type studied in the first step. So $T$ is bounded on $L^2.$

Now we give a new approach for the $T1$ theorem on $M_1.$ Roughly speaking, we put these two steps together. More precisely, by the following Calder\'{o}n's identity on $M_1$
\begin{eqnarray*}
f(x)&=&\sum_{k=-\infty}^\infty D_{k}{\widetilde{\widetilde{D}}}_{k}(f)(x),
\end{eqnarray*}
where $D_{k}$ and ${\widetilde{\widetilde{D}}}_{k}$ were given in [HLL2, Theorem 2.7] on $M_1,$ for test functions $f, g\in\break \GG(\beta_{1},\gamma_{1})(M_1)$ with compact supports we consider the following bilinear form
\begin{eqnarray*}
\langle g, Tf\rangle&=&\langle\sum_{j=-\infty}^\infty D_{j}{\widetilde{\widetilde{D}}}_{j}(g), T\sum_{k=-\infty}^\infty D_{k}{\widetilde{\widetilde{D}}}_{k}(f)\rangle\\
&=&\sum_{j,k}\langle{\widetilde{\widetilde{D}}}_{j}(g), D_{j}T D_k {\widetilde{\widetilde{D}}}_{k}(f)\rangle,
\end{eqnarray*}
where, by the construction in [HLL2], we may assume that $D^*_j=D_j.$

As mentioned above, if $T$ is a singular integral operator defined on $M_1$ having the weak boundedness property and $T(1)=T^*(1)=0,$ then $D_{j}T D_k(x,y),$ the kernel of the operator $D_{j}T D_k,$ satisfies the following almost orthogonal estimate
\begin{eqnarray*}
|D_{j}T D_k(x,y)|&=&|\iint D_j(x,u)K(u,v)D_k(v,y)dudv|\\
&\leq& C 2^{-|j-k|\epsilon}\frac{1}{V_{2^{-(j\wedge k)}}(x)+V_{2^{-(j\wedge k)}}(y)+V(x,y)}\frac{2^{-(j\wedge k)\varepsilon}}{(2^{-(j\wedge k)}+d(x,y))^{\varepsilon}}.
\end{eqnarray*}
This almost orthogonal estimate together with the Littlewood--Paley estimate on $L^2$ implies that the bilinear form $\langle g, Tf\rangle$ is bounded by some constant times $\|f\|_2\|g\|_2$ and hence the $L^2$ boundedness of $T$ is concluded. However, without assuming $T(1)=T^*(1)=0,$ if $j\leq k$ one still has the following almost orthogonal estimate
\begin{eqnarray*}
&&|\iint [D_j(x,u)-D_j(x,y)]K(u,v)D_k(v,y)dudv|\\
&&\leq C 2^{(j-k)\epsilon}\frac{1}{V_{2^{-j}}(x)+V_{2^{-j}}(y)+V(x,y)}\frac{2^{-j\varepsilon}}{(2^{-j}+d(x,y))^{\varepsilon}}.
\end{eqnarray*}
Similarly, for $k\leq j,$
\begin{eqnarray*}
&&|\iint D_j(x,u)K(u,v)[D_k(v,y)-D_k(x,y)]dudv|\\
&&\leq C 2^{(k-j)\epsilon}\frac{1}{V_{2^{-k}}(x)+V_{2^{-k}}(y)+V(x,y)}\frac{2^{-k\varepsilon}}{(2^{-k}+d(x,y))^{\varepsilon}}.
\end{eqnarray*}
This leads to the following decomposition:
\begin{eqnarray*}
\langle g, Tf\rangle&=&\sum_{j\leq k}\int {\widetilde{\widetilde{D}}}_{j}(g)(x)  \iint [D_j(x,u)-D_j(x,y)]K(u,v)D_k(v,y)dudv {\widetilde{\widetilde{D}}}_{k}(f)(y) dydx\\
&&\hskip.5cm+\sum_{k<j}\int {\widetilde{\widetilde{D}}}_{j}(g)(x)\iint D_j(x,u)K(u,v)[D_k(v,y)-D_k(x,y)]dudv {\widetilde{\widetilde{D}}}_{k}(f)(y) dydx\\
&&\hskip.5cm+\sum_{j\leq k}\int {\widetilde{\widetilde{D}}}_{j}(g)(x)  \iint D_j(x,y)K(u,v)D_k(v,y)dudv {\widetilde{\widetilde{D}}}_{k}(f)(y) dydx\\
&&\hskip.5cm+\sum_{k<j}\int {\widetilde{\widetilde{D}}}_{j}(g)(x)\iint D_j(x,u)K(u,v)D_k(x,y)dudv {\widetilde{\widetilde{D}}}_{k}(f)(y) dydx.
\end{eqnarray*}
The almost orthogonal estimates, as mentioned above, together with the Littlewood--Paley estimate on $L^2$ imply that the first two series are bounded by some constant $C$ times $\|f\|_2\|g\|_2.$ The last two series are also bounded by $C\|f\|_2\|g\|_2.$ To see this, we only consider the third series and rewrite it as
\begin{eqnarray*}
&&\sum_{j\leq k}\int {\widetilde{\widetilde{D}}}_{j}(g)(x) \iint D_j(x,y)K(u,v)D_k(v,y)dudv \widetilde{D}_{k}(f)(y) dydx\\
&&=\int \sum_{k}{\widetilde S}_k(g)(y) D_k(T^*1)(y){{\widetilde{\widetilde D}}}_k(f)(y) dy,
\end{eqnarray*}
where ${\widetilde S}_k=\sum\limits_{j\leq k}D_j{\widetilde{\widetilde D}}_j.$
The Carleson measure estimate together Littlewood--Paley estimate yields
\begin{eqnarray*}
&&\Big|\int \sum_{k}{\widetilde S}_k(g)(y) D_k(T^*1)(y){{\widetilde{\widetilde D}}}_k(f)(y) dy\Big|\\
&&\leq \Big\lbrace \int\sum_{k} |{\widetilde S}_k(g)(y)|^2 |D_k(T^*1)(y)|^2 dy\Big\rbrace^{\frac{1}{2}} \Big\lbrace \int\sum_{k}|{{\widetilde{\widetilde D}}}_k(f)(y)|^2 dy\Big\rbrace^{\frac{1}{2}}\\
&&\leq C\|f\|_2\|g\|_2.
\end{eqnarray*}
This new approach can be carried out to the product case. Indeed, the following discrete Calder\'on's identity on the product $\widetilde M$ was proved in [HLL2, Theorem 2.9].
$$
f(x_1,x_2)=\sum_{k_1=-\infty}^\infty\sum_{k_2=-\infty}^\infty\sum_{I_1}\sum_{I_2 }\mu_1(I_1)\mu_2(I_2) D_{k_1}(x_1,x_{I_1})D_{k_2}(x_2,x_{I_2}) {\widetilde{\widetilde{D}}}_{k_1}{\widetilde{\widetilde{D}}}_{k_2}(f)(x_{I_1},x_{I_2}),$$
for test functions $f,g \in \GGp(\beta_{1},\beta_{2};\gamma_{1},\gamma_{2}).$

We consider the following bilinear form
\begin{eqnarray*}
\langle g,Tf\rangle &=& \sum_{k_1^{'}}  \sum_{I_1^{'}}\label{g
Tf}
\sum_{k_1}\sum_{I_1}\sum_{k_2^{'}}  \sum_{I_2^{'}} \sum_{k_2}\sum_{I_2}\mu_1(I_1^{'})\mu_1(I_1) \mu_2(I_2^{'})\mu_2(I_2)\\
&& \times {\widetilde{\widetilde{D}}}_{k_1^{'}}{\widetilde{\widetilde{D}}}_{k_2^{'}}(g)(x_{I_1^{'}},x_{I_2^{'}})
\big\langle D_{k_1^{'}}D_{k_2^{'}},
TD_{k_1}D_{k_2}\big\rangle(x_{I_1^{'}},x_{I_2^{'}},x_{I_1},x_{I_2})
{\widetilde{\widetilde{D}}}_{k_1}{\widetilde{\widetilde{D}}}_{k_2}(f)(x_{I_1},x_{I_2})\nonumber
\end{eqnarray*}
for test functions $f,g \in \GGp(\beta_{1},\beta_{2};\gamma_{1},\gamma_{2})
$ with compact supports.

Note that instead using continuous Calder\'on's identity as for the classical case we would like to use the discrete Calder\'on's identity because this will be convenient for us to deal with the $T1$ theorem on the Hardy space $H^p(\widetilde M)$ and space $CMO^p(\widetilde M).$  We would also like to point out that in this bilinear form the operator $T$ does not act on the function $f$ rather on the separate form $D_{k_1}D_{k_2}.$
Indeed, one can write
\begin{eqnarray*}
\big\langle D_{k_1^{'}}D_{k_2^{'}},
TD_{k_1}D_{k_2}\big\rangle&=&\big\langle D_{k_1^{'}}, \langle D_{k_2^{'}},
K_1(x_1,y_1)D_{k_2}\rangle D_{k_1}\big\rangle\\
&=&\big\langle D_{k_2^{'}}, \langle D_{k_1^{'}},
K_2(x_2,y_2)D_{k_1}\rangle D_{k_2}\big\rangle.
\end{eqnarray*}
This fact will be crucial for this new approach.

Similar to the decomposition as given above for one parameter case, if $k'_1>k_1$ and $k'_2>k_2,$ one can write
\begin{eqnarray*}
&& \langle D_{k'_1}D_{k'_2} T D_{k_1}D_{k_2}\rangle(x_{I_1^{'}},x_{I_2^{'}},x_{I_1},x_{I_2})\\
&&=\int  D_{k'_1}(x_{I_1^{'}},u_1)D_{k'_2}(x_{I_2^{'}},u_2)K(u_1,u_2,v_1,v_2)[D_{k_1}(v_1,x_{I_1})-D_{k_1}(x_{I_1^{'}},x_{I_1})]\\
&&\hskip1cm\times [D_{k_2}(v_2,x_{I_2})-D_{k_2}(x_{I_2^{'}},x_{I_2})] du_1du_2dv_1dv_2\\
&&\hskip.5cm+ \int  D_{k'_1}(x_{I_1^{'}},u_1)D_{k'_2}(x_{I_2^{'}},u_2)K(u_1,u_2,v_1,v_2)D_{k_1}(x_{I_1^{'}},x_{I_1}) D_{k_2}(v_2,x_{I_2}) du_1du_2dv_1dv_2\\
&&\hskip.5cm+ \int  D_{k'_1}(x_{I_1^{'}},u_1)D_{k'_2}(x_{I_2^{'}},u_2)K(u_1,u_2,v_1,v_2)D_{k_1}(v_1,x_{I_1})
D_{k_2}(x_{I_2^{'}},x_{I_2}) du_1du_2dv_1dv_2\\
&&\hskip.5cm- \int  D_{k'_1}(x_{I_1^{'}},u_1)D_{k'_2}(x_{I_2^{'}},u_2)K(u_1,u_2,v_1,v_2)D_{k_1}(x_{I_1^{'}},x_{I_1})D_{k_2}(x_{I_2^{'}},x_{I_2}) du_1du_2dv_1dv_2\\
&&=: I(x_{I_1^{'}},x_{I_2^{'}},x_{I_1},x_{I_2})+II(x_{I_1^{'}},x_{I_2^{'}},x_{I_1},x_{I_2})+III(x_{I_1^{'}},x_{I_2^{'}},x_{I_1},x_{I_2})
+IV(x_{I_1^{'}},x_{I_2^{'}},x_{I_1},x_{I_2}).
\end{eqnarray*}
Then the first term $I$ satisfies the following almost orthogonal estimate
\begin{eqnarray*}
|I(x_{I_1^{'}},x_{I_2^{'}},x_{I_1},x_{I_2})|
&\leq& C2^{(k_1-k'_1)\varepsilon}2^{(k_2-k'_2)\varepsilon}\\
&&\times\frac{1}{V_{2^{-k_1}}(x_{I_1^{'}})+V_{2^{-k_1}}(x_{I_1})+V(x_{I_1^{'}},x_{I_1})}\frac{2^{-k_1\varepsilon}}
{(2^{-k_1}+d_1(x_{I_1^{'}},x_{I_1}))^{\varepsilon}}\nonumber\\
&&\times \frac{1}{V_{2^{-k_2
}}(x_{I_2^{'}})+V_{2^{-k_2}}(x_{I_2})+V(x_{I_2^{'}},x_{I_2})}\frac{2^{-k_2
\varepsilon}}{(2^{-k_2
}+d_2(x_{I_2^{'}},x_{I_2}))^{\varepsilon}}.\nonumber
\end{eqnarray*}
To deal with term $II,$ we first rewrite it as
\begin{eqnarray*}
II&=&\int D_{k'_2}(x_{I_2^{'}},u_2),\langle D_{k'_1},K_2(u_2,v_2)(1)\rangle D_{k_2}(v_2,x_{I_2}) dv_2du_2
D_{k_1}(x_{I_1^{'}},x_{I_1})\\
&=&\int D_{k'_2}(x_{I_2^{'}},u_2),\langle D_{k'_1},K_2(u_2,v_2)(1)\rangle [D_{k_2}(v_2,x_{I_2})-D_{k_2}(x_{I_2^{'}},x_{I_2})] dv_2du_2\
D_{k_1}(x_{I_1^{'}},x_{I_1})\\[4pt]
&&\hskip.5cm-IV.
\end{eqnarray*}
Note that for each fixed $(u_2,v_2), K_2(u_2,v_2)(1)$ is a BMO function on $M_1$ since $K_2(u_2,v_2)$ is a Calder\'on--Zygmund operator on $M_1$ and thus, $|\langle D_{k'_1},K_2(u_2,v_2)(1)\rangle|^2$ is a Carleson measure on $M_1\times \lbrace k'_1\rbrace.$ Moreover, $\langle D_{k'_1},K_2(u_2,v_2)(1)\rangle$ is a singular integral kernel on $M_2.$ Therefore, applying the almost orthogonal estimate on $M_2$ yields
\begin{eqnarray*}
&&\Big\|\int D_{k'_2}(x_{I_2^{'}},u_2),\langle D_{k'_1},K_2(u_2,v_2)(1)\rangle [D_{k_2}(v_2,x_{I_2})-D_{k_2}(x_{I_2^{'}},x_{I_2})] dv_2du_2\Big\|_{CM(M_1\times \lbrace k'_1\rbrace }\\
&&\leq C 2^{(k_2-k'_2)\varepsilon}
\frac{1}{V_{2^{-k_2
}}(x_{I_2^{'}})+V_{2^{-k_2}}(x_{I_2})+V(x_{I_2^{'}},x_{I_2})}\frac{2^{-k_2
\varepsilon}}{(2^{-k_2
}+d_2(x_{I_2^{'}},x_{I_2}))^{\varepsilon}},\nonumber
\end{eqnarray*}
where, as mentioned, $\|\cdot\|_{CM(M_1\times \lbrace k'_1\rbrace}$ means the Carleson measure norm on $M_1\times \lbrace k'_1\rbrace.$

Term $III$ satisfies the same estimate with interchanging $k'_1, k'_2, x_{I_1^{'}}, x_{I_2^{'}} $ and $k_1, k_2, x_{I_1}, x_{I_2},$ respectively. It is not difficult to see that the last term $IV$ can be written as
\begin{eqnarray*}
IV=D_{k^{'}_1}D_{k^{'}_2}T(1)(x_{I_1^{'}},x_{I_2^{'}})D_{k_1}(x_{I_1^{'}},x_{I_1})D_{k_2}(x_{I_2^{'}},x_{I_2}).
\end{eqnarray*}
Note that $T(1)\in BMO(\widetilde M)$ and hence $\mu_1(I_1^{'})\mu_2(I_2^{'})|D_{k'_1}D_{k'_2}T(1)(x_{I_1^{'}},x_{I_2^{'}})|^2$ is a Carleson measure
on $\widetilde M\times \lbrace k'_1\times k'_2\rbrace.$

Inserting all these estimates for the terms $I-IV$ into the bilinear form with respect to the summation over $k'_1>k_1$ and $k'_2>k_2,$ one can show that it is bounded by $C\|f\|_2\|g\|_2.$ The bilinear forms with respect to the summations over other cases can be handled similarly. See more details in Subsection 3.3.

We remark that term $IV$ is similar to the para-product operator $W_{b}$ introduced by Journ\'e in [J], as mentioned above. However, the property that for a BMO function $b, W_{b}(1)=b$ in the last step and the operator $S$ constructed in the second step in Journ\'e's proof are not required in our approach.

Furthermore, in this paper, we will also show the $T1$ theorem on
$H^p(\widetilde M)$ and $CMO^p(\widetilde M),$ respectively. More
precisely, if $T$ is bounded on $L^2$ then $T$ is bounded on
$H^p(\widetilde M)$ and $CMO^p(\widetilde M)$ for $p\leq 1$ but $p$
is close to 1, if and only if $T^*_1(1)=T^*_2(1)=0$ and
$T_1(1)=T_2(1)=0,$ respectively. Note that in [J] Journ\'e proved
that if $T$ is a convolution operator and bounded on $L^2,$ then $T$
admits a bounded extension from $BMO(\mathbb{R}\times \mathbb{R})$ to itself. He
mentioned without the proof that if $T$ is a Calder\'on--Zygmund
operator and $T_1(1)=T_2(1)=0,$ then $TH_1, TH_2$ and $TH_1H_2$ are
Cadelr\'on-Zygmund operators, where $H_1, H_2$ and $H_1H_2$ are the
Hilbert transforms and double Hilbert transform. From this together
with the characterization of the product $BMO(\mathbb{R}\times \mathbb{R})$ in terms
of the bi-Hilbert transform, the boundedness of $T$ on $BMO(\mathbb{R}\times \mathbb{R})$ is obtained. In our setting, however, his method is not available.
Roughly speaking, the $L^2(\widetilde M)$ theory and the duality argument between $H^p(\widetilde M)$ and $CMO^p(\widetilde M)$
will play a crucial role in the present proofs. To be More
precise, it is known that $L^2(\widetilde M)\cap H^p(\widetilde M)$ is dense in $H^p(\widetilde M).$ Therefore, to show that $Tf$ is bounded on $H^p(\widetilde M)$ it suffices to consider $f\in L^2(\widetilde M)\cap H^p(\widetilde M).$ However, this argument for space $CMO^p(\widetilde M)$ is no long true. In this paper, we will show that $L^2(\widetilde M)\cap CMO^p(\widetilde M)$ is dense in the weak topology
$(H^p, CMO^p).$ Applying this result together with the duality argument implies that the boundedness of $T$ on $CMO^p(\widetilde M)$ will follow from the boundedness of $T$ on $H^p(\widetilde M).$ To see this, assume that the $T1$ theorem on $H^p$ holds and $T_1(1)=T_2(1)=0.$ Suppose that $f\in L^2\cap CMO^p$ and $g\in L^2\cap H^p.$ Then, by the duality argument, $|\langle Tf, g\rangle|=|\langle f, T^*g \rangle|\leq C \|f\|_{CMO^p}\|g\|_{H^p}$ since $(T^*)^*_1(1)=T_1(1)=0=T_2(1)=(T^*)^*_2(1)$ and thus $T^*$ is bounded on $H^p(\widetilde M)$ by the $T1$ theorem on $H^p(\widetilde M).$ This implies that $\langle Tf, g\rangle$ is a linear functional on the subspace $L^2(\widetilde M)\cap H^p(\widetilde M)$ with the norm less than $C\|f\|_{CMO^p}$ and hence, it can be extended to a linear functional on $H^p(\widetilde M)$ since $L^2(\widetilde M)\cap H^p(\widetilde M)$ is dense in $H^p(\widetilde M).$ Therefore, by the duality argument, $Tf\in CMO^p(\widetilde M).$ In order to estimate $\|Tf\|_{CMO^p},$ by the duality argument again, one can write
$\langle Tf, g\rangle=\langle h, g\rangle$ for all test functions $g$ and some $h\in CMO^p(\widetilde M)$ with $\|h\|_{CMO^p(\widetilde M)}\leq C\|f\|_{CMO^p(\widetilde M)}.$ See the details of the duality argument in [HLL2]. Choosing test functions $g$ as the functions in the definition of $CMO^p(\widetilde M),$ one can conclude that $\|Tf\|_{CMO^p(\widetilde M)}=\|h\|_{CMO^p(\widetilde M)}$ and thus $\|Tf\|_{CMO^p(\widetilde M)}\leq C\|f\|_{CMO^p(\widetilde M)}.$

In this paper, we prove the $T1$ theorem for $H^p(\widetilde M)$ and $CMO^p(\widetilde M)$ as follows. We first show that if $T$ is bounded on $L^2(\widetilde M)$ and $T^*_1(1)=T^*_2(1)=0$ then $T$ is bounded on $H^p(\widetilde M).$ This will be achieved by applying the almost orthogonal argument and atomic decomposition established in Subsection 3.2. Applying this result together with the duality argument as mentioned above, we prove that if $T$ is bounded on $L^2(\widetilde M)$ and $T_1(1)=T_2(1)=0$ then $T$ is bounded on $CMO^p(\widetilde M).$ To show the converse, by choosing special functions, we first prove that if $T$ is bounded on $CMO^p(\widetilde M)$ then $T_1(1)=T_2(1)=0.$ This result together with the duality argument will imply that if $T$ is bounded on $H^p(\widetilde M)$ then $T^*_1(1)=T^*_2(1)=0.$

The paper is organized as follows. In Section 2, we recall notation
and some preliminaries used in [NS04]. Particularly, we describe the
basic geometry of Carnot--Carath\'eodory space, singular integrals
studied by Nagel and Stein and the Littlewood--Paley theory and the
$L^p$ boundedness of singular integrals developed in [NS04]. We also
mention, in this section, the Hardy space theory on the product
Carnot--Carath\'eodory space established in [HLL2], which includes
the $H^p$ boundedness for operators studied by Nagel and Stein and
the duality between $H^p$ and $CMO^p,$ particularly, $CMO^1=BMO,$
the dual of $H^1.$ The product $T1$ and its proof are given in
Section 3. We first introduced singular integrals on the product
Carnot--Carath\'eodory space and state the $T1$ theorem in
Subsection 3.1. In Subsection 3.2, we prove the necessary
conditions. Journ\'e-type covering lemma and atomic decomposition
are provided in Subsections 3.2.1 and 3.2.2. We prove that if $T$ is
bounded on $L^2$ then $T$ extends to a bounded operator from $H^p$
to $L^p$, $L^\infty$ to $BMO,$ and from $L^p$ to itself in
Subsections 3.2.3, 3.2.4 and Subsection 3.2.5, respectively. The
sufficient conditions of the product $T1$ theorem are proved in the
Subsection 3.3. In
Section 4, we give the $T1$-type theorems for $H^p$ and $CMO^p.$ The
statements and the proofs are given in Subsection 4.1 and 4.2,
respectively. In the last section, we will point out that all
results and proofs in this paper can be carried out in arbitrarily
many parameters. We will only state these results and omit the
details of the proofs.

\section{Notation and preliminaries}
\setcounter{equation}{0}
In this section, we recall the basic geometry of the product Carnot--Carath\'eodory space and state the $L^p, 1<p<\infty,$ boundedness of product singular integral operators studied in [NS04]. The product Hardy space theory on the Carnot--Carath\'eodory space developed in [HLL2] will be described in the last subsection

\subsection{Basic geometry of Carnot--Carath\'eodory space}

In recent years, the optimal estimates were established for
solutions of the Kohn-Laplacian for decoupled boundaries in
$\mathbb{C}^{n+1}$ (See the series of papers \cite{NS01a},
\cite{NS01b}, \cite{NS04}, \cite{NS06}). They considered the
Kohn-Laplacian on $q-forms$,
$\Box_b^{(q)}=\Box_b=\bar{\partial}_b\bar{\partial}_b^*+\bar{\partial}_b^*\bar{\partial}_b$,
defined on the boundary $M=\partial \Omega$ of a smooth
pseudo-convex domain $\Omega\subset\mathbb{C}^{n+1}$. They studied
the relative inverse operator $\mathcal{K}$ and the corresponding
Szeg\"o projection $\mathcal{S}$, which satisfy
$\Box_b\mathcal{K}=\mathcal{K}\Box_b=I-\mathcal{S}$. By definition,
$\mathcal{S}$ is the orthogonal projection on the $L^2$ null-space
of $\Box_b$.


The model domains we recall here are the decoupled domain
$\Omega\subset \mathbb{C}^{n+1}$ and its boundary $M$, the related
product domain $\widetilde{\Omega}$ and the Shilov boundary
$\widetilde{M}$ in $\mathbb{C}^{2n}$, and the pseudoconvex domain in
$\mathbb{C}^2$, where $n\geq2$. Now we state them as follows.

A domain $\Omega\subset\mathbb{C}^{n+1}$ and its boundary $M$ are
said to be decoupled if there are sub-harmonic and non-harmonic
polynomials $P_j$ such that
\begin{eqnarray}
\Omega&=& \big\{(z_1,...,z_n,z_{n+1})\in\mathbb{C}^{n+1}:\ \Im [z_{n+1}]>\sum_{j=1}^n P_j(z_j) \big\}; \label{decoupled domain}\\
M&=& \big\{(z_1,...,z_n,z_{n+1})\in\mathbb{C}^{n+1}:\ \Im
[z_{n+1}]=\sum_{j=1}^n P_j(z_j) \big\}.\label{decoupled boundary}
\end{eqnarray}
For each $j$, the pseudoconvex domain in $\mathbb{C}^2$ we consider
is as follows.
\begin{eqnarray}
\Omega_j&=& \big\{(z_j,w_j)\in\mathbb{C}^{2}:\ \Im [w_{j}]> P_j(z_j) \big\}; \label{pseudoconvex domain}\\
M_j&=& \big\{(z_j,w_j)\in\mathbb{C}^{2}:\ \Im [w_{j}]= P_j(z_j)
\big\}.\label{pseudoconvex boundary}
\end{eqnarray}
The Cartesian products of these domains and boundaries are
\begin{eqnarray}
\widetilde{\Omega}&=& \Omega_1\times\cdots\times\Omega_n; \label{product domain}\\
\widetilde{M}&=& M_1\times\cdots\times M_n.\label{Shilov boundary}
\end{eqnarray}
$\widetilde{M}$ is the Shilov boundary of $\widetilde{\Omega}$.

One of the typical examples of $\Omega$ and $M$ is the Szeg\"o upper
half space $\mathcal{U}^n$ and its boundary Heisenberg group
$\mathbb{H}^n$ (to see this, we can take $P_j(z_j)=|z_j|^2$). As is
known to all, the Szeg\"o upper half space and its boundary are
biholomorphically equivalent to the unit ball $\mathbb{B}^n$ and its
boundary $\partial \mathbb{B}^n$. Hence we can see that the
decoupled domain and boundary are natural generalizations of the
basic model domains in several complex variables, on which the
properties of the inverse operator of Kohn-Laplacian and the
corresponding Szeg\"o projection have been studied by Christ,
Fefferman, Folland, Kohn, Stein and others, see for example \cite{Chr2},\ \cite{FoS},\
\cite{FK88},\ \cite{Kohn85},\ \cite{NRSW}, and the references therein.

Fix $1\leq j\leq n$, let $M_j$ be the hypersurface given in equation (\ref{pseudoconvex boundary}). And let $\widetilde{M}=M_1\times\cdots\times M_n$
be the Shilov boundary, i.e., the Cartesian product as in (\ref{Shilov boundary}).

We first recall the \emph{control metric on $M_j$}. Note that we write the complex (0,1) vector field $\overline{Z}_j=X_j+iX_{n+j}$, where $\{X_j, X_{n+j}\}$ are real vector fields on $M_j$. Define the metric $d_j$ on $M_j$ as follows. If $p,q\in M_j$ and $\delta>0$, let $AC(p,q,\delta)$ denote the set of absolutely continuous mapping $\gamma: [0,1]\rightarrow M_j$ such that $\gamma(0)=p$ and $\gamma(1)=q$, and such that for almost all $t\in[0,1]$ we have $\gamma'(t)=\alpha_j(t)X_j(\gamma(t))+ \alpha_{n+j}(t)X_{n+j}(\gamma(t))$ with $|\alpha_j(t)|^2+|\alpha_{n+j}(t)|^2<\delta^2$. Then we define
$$
  d_j(p,q)=\inf\{\delta>0:\ AC(p,q,\delta)\not=\emptyset\}.
$$
The corresponding nonisotropic ball is
$$
  B_j(p,\delta)=\{q\in M_j:\ d_j(p,q)<\delta\},
$$
and $|B_j(p,\delta)|$ denotes its volume. Set $$ V_j(p,q)=|B_j\big(p,d_j(p,q)\big)|. $$

The volume of the ball $B(p,\delta)$ is essentially a polynomial in $\delta$ with coefficients that depend on $p$. Let $T=\partial\slash \partial_t$ so that at each point of $M_j$ the tangent space is spanned by vectors $\{ X_j,X_{n+j},T \}$. Write the commutator
\begin{eqnarray}\label{commutator}
   [X_j,X_{n+j}]=\lambda_j T+ a_jX_j+ a_{n+j}X_{n+j},
\end{eqnarray}
where $\lambda_j, a_j, a_{n+j}\in C^\infty(M_j)$. If $\alpha=(\alpha_1,\ldots,\alpha_k)$ is a $k$-tuple with each $\alpha_j$ equal to $j$ or $n+j$, let $|\alpha|=k$ and let $X^\alpha=X_{\alpha_1}\cdots X_{\alpha_j}$ denote the corresponding $k^{\rm th}$ order differential operator. For $k\geq2$ set
$$
    \Lambda_j^k(p)=\sum_{|\alpha|\leq k-2} |X^\alpha\lambda_j(p)|,
$$
where $\lambda_j$ is defined as in (\ref{commutator}), and set
$$
    \Lambda_j(p,\delta)=\sum_{k=2}^{m_j} \Lambda_{j}^k(p)|\delta|^k.
$$
\begin{prop}[\cite{NS06}]
There are constants $C_1,C_2$ depending only on $m_j$ so that for $p\in M_j$ and $\delta>0$,
$$
    C_1\delta^2 \Lambda_j(p,\delta) \leq |B_j(p,\delta)|\leq     C_2\delta^2 \Lambda_j(p,\delta).
$$
Also, $V_j(p,q)\thickapprox V_j(q,p) \thickapprox d_j(p,q)^2 \Lambda_j(p,d_j(p,q))$, where $A\thickapprox B$ means that the ratio $A\slash B$ is bounded above and bounded away from zero.

\end{prop}

There is an alternate description of the balls $\{B_j(p,\delta)\}$ and metric $d_j$ given in terms of explicit inequalities. For $z,w\in\mathbb{C}$ let
$$
    T_j(w,z)=2 \Im\big[ \sum_{k=1}^{m_j} {\partial^k P_j\over \partial z^k}(w) {(z-w)^k\over k!} \big].
$$
Then, with $p=(w,s)\in M_j$, set
$$
   \widetilde{B}_j(p,\delta)=\{ (z,t)\in M_j \mid |z-w|<\delta \ \textup{and}\ |t-s+ T_j(w,z)|<\Lambda_j(w,\delta) \}.
$$
Note that there is a unique inverse function $\mu_j(p,\delta)$ such that for $\delta\geq0$ we have $\Lambda_j(p,\mu_j(p,\delta))=\mu_j(p,\Lambda_j(p,\delta))=\delta$. We have
$$
    \mu_j(p,\delta)^{-1}\thickapprox \sum_{k=2}^{m_j} \Lambda_j^k(p)^{1\over k} |\delta|^{-{1\over k}}.
$$
\begin{prop}[\cite{NS06}]
There are constants $C_1,C_2$ depending only on $m_j$ so that for $p\in M_j$ and $\delta>0$,
$$
     \widetilde{B}_j(p,C_1\delta)\subset B_j(p,\delta)\subset \widetilde{B}_j(p,C_2\delta).
$$
Moreover, if $(z,t), (w,s)\in M_j$,
$$
    d_j((z,t),(w,s))\thickapprox |z-w|+\mu_j(w,|t-s- T_j(w,z)|)
$$

\end{prop}

Now we turn to $\widetilde{M}=M_1\times\cdots\times M_n$. Each of the nonisotropic distance $d_j$ on $M_j$ can be regarded as a function on $\widetilde{M}$ which depends only on the variables $(z_j,t_j)$. In addition, there is a nonisotropic metric $d_{\sum}$ on $\widetilde{M}$ induced by all real vector fields $\{X_1,\ldots,X_{2n}\}$. If $p,q\in M_j$ and $\delta>0$, let $AC(p,q,\delta)$ denote the set of absolutely continuous mappings $\gamma:[0,1]\rightarrow \widetilde{M}$ such that $\gamma(0)=p$ and $\gamma(1)=q$, and such that for almost every $t\in [0,1]$ we have $\gamma'(t)=\sum_{j=1}^{2n}\alpha_j(t)X_j(\gamma(t))$ with $\sum_{j=1}^{2n}|\alpha_j(t)|^2<\delta^2$. Then
$$
    d_{\sum}(p,q)=\inf\{\delta>0 \mid AC(p,q,\delta)\not=\emptyset\}.
$$
This metric is appropriate for describing the fundamental solution of the operator $\mathcal{L}=\sum_{j=1}^{2n}X_j^2$, and it can be explicitly described as follows. Let $p=(z_1,t_1,\ldots,z_n,t_n)\in \widetilde{M}$. We can assume without loss of generality that each manifold $M_j$ is normalized at the origin. We denote the origin of $\widetilde{M}$ by $\overline{0}$. Then
$$
   d_{\sum}(\overline{0},p)\thickapprox \sum_{j=1}^n[|z_j|+\mu_j(0,|t_j|)].
$$
The ball centered at $\overline{0}$ of radius $\delta$ is, up to constants, given by
$$
   B_{\sum}(\overline{0},\delta)=\big\{ (z,t)\in \widetilde{M} \mid |z_j|<\delta\ \textup{and}\ |t_j|<\Lambda_j(0,\delta) \ \textup{for}\ 1\leq j\leq n \big\}.
$$
We have
$$
    |B_{\sum}(\overline{0},\delta)|\thickapprox \delta^{2n} \prod_{j=1}^n \Lambda_j(0,\delta),
$$
and
$$
    |B_{\sum}(\overline{0},d_{\sum}(z,t))|\thickapprox \big[ \sum_{j=1}^n|z_j|+\mu_j(0,|t_j|) \big]^{2n} \prod_{j=1}^n \Lambda_j(0,\big[ \sum_{j=1}^n|z_j|+\mu_j(0,|t_j|) \big]).
$$

When $M$ is compact then one can take any fixed smooth measure on $M$ with
strictly positive density. In the unbounded case one takes Lebesgue measure and denote the measure of a set $E$ by $|E|$.
The ball is defined by $B(x, \delta) = \lbrace y\in M , d(x, y)<\delta \rbrace$, with $0 <\delta\le 1$ in the compact
case, and $0 <\delta <\infty$ in the unbounded case and the volume
function is defined by $V(x,y)= |B(x,d(x,y))|$. The key geometric facts used in [NS04] is that the volumes of the balls $B(x,\delta)$ are essentially polynomials in $\delta$ with coefficients that depend on $x$  and satisfy the doubling property(see [49] for the details)
\begin{eqnarray}\label{doubling condition}
 |B(x,2\delta)|\leq C |B(x,\delta)| \ \ \ \ {\rm for\ all\ \ } \delta>0 {\rm \ and\ some\ constant\ } C
\end{eqnarray}
and, moreover, in the unbounded case, for $s\ge 1,$
\begin{eqnarray}\label{upper dim}
 |B(x,s\delta)|\approx s^{m+2} |B(x,\delta)|
\end{eqnarray}
and
\begin{eqnarray}\label{lower dim}
 |B(x,s\delta)|\ge s^{4} |B(x,\delta)|.
\end{eqnarray}
We point out that the doubling condition (\ref{doubling condition})
implies that there exist positive constants  $C$ and $Q$ such that
for all $x\in M$ and $\lambda\geq 1$,
\begin{eqnarray}\label{homogeneous dim}
|B(x, \lambda r)|\leq C\lambda^{Q} |B(x,r)|.
\end{eqnarray}

\subsection{Singular integrals on Carnot--Carath\'eodory space}

To state the singular integral operators on $M$ studied in [NS04], we first recall
that $\varphi$ is a bump function associated to a ball $B(x_0,r)$
if $\varphi$ is supported in this ball and satisfies the differential
inequalities $|\partial_X^a\varphi|\lesssim r^{-a}$ for all
monomials $\partial_X$ in $X_1,\cdots,X_k$ of degree $a$ and all
$a\geq 0$.

Singular integral operators $T$ considered in [NS04] are initially given as
mappings from $C_0^{\infty}(M)$ to $C^{\infty}(M)$ with a
distribution kernel $K(x,y)$ which is $C^{\infty}$ away from the
diagonal of $M\times M$, and the following
properties are satisfied:
\begin{itemize}
\item[(I-1)] If $\varphi,\psi\in C_0^{\infty}(M)$ have disjoint supports, then
$$\langle T\varphi,\psi\rangle=\int_{M\times M}K(x,y)\varphi(y)\psi(x)dydx.$$
\item[(I-2)] If $\varphi$ is a normalized bump function associated to a
ball of radius $r$, then $|\partial_X^aT\varphi|\lesssim r^{-a}$ for
each integer $a\geq 0$.
\item[(I-3)] If $x\neq y$, then for every integer $a\geq 0$,
$$|\partial_{X,Y}^aK(x,y)|\lesssim d(x,y)^{-a}V(x,y)^{-1}.$$
\item[(I-4)] Properties (I-1) through (I-3) also hold with $x$ and $y$
interchanged. That is, these properties also hold for the adjoint
operator $T^t$ defined by
$$\langle T^t\varphi,\psi\rangle=\langle T\psi,\varphi\rangle.$$
\end{itemize}

Now we turn to the product case with two factors.
 Here the operator $T$ is initially
defined from $C_0^{\infty}(\widetilde{M})$ to
$C^{\infty}(\widetilde{M}),$ where $\widetilde{M}=M_1\times M_2.$ $K(x_1,y_1,x_2,y_2),$ the distribution
kernel of $T,$ is an $C^{\infty}$ function away from the
``cross''$=\{(x,y): x_1=y_1\ {\rm and}\ x_2=y_2;\ x=(x_1,x_2),
y=(y_1,y_2)\}$ and satisfies the following additional properties:
\begin{itemize}
\item[(II-1)]
$\big<T(\varphi_1\otimes\varphi_2),\psi_1\otimes\psi_2\big>=\int
K(x_1,y_1,x_2,y_2)\varphi_1(y_1)\varphi_2(y_2)\psi_1(x_1)\psi_2(x_2)dydx$
\[{\rm whenever} \left\{\begin{array} { r@{\quad \quad}l }
\varphi_1,\psi_1\in C_0^{\infty}(M_1)\ \ {\rm and\ have\ disjoint\
support,}
\\ \varphi_2,\psi_2\in C_0^{\infty}(M_2)\ \ {\rm and\
have\ disjoint\ support.}
\end{array}\right.
\]
\item[(II-2)] For each bump function $\varphi_2$ on $M_2$ and each $x_2\in
M_2$, there exists a singular integral operator $T^{\varphi_2,x_2}$ (of one parameter) on $M_1$, so that
$$ \big<T(\varphi_1\otimes\varphi_2),\psi_1\otimes\psi_2\big>=
\int_{M_2}\big<T^{\varphi_2,x_2}\varphi_1,\psi_1\big>\psi_2(x_2)dx_2.
$$ Moreover, $x_2\mapsto T^{\varphi_2,x_2}$ is smooth and uniform in
the sense that $T^{\varphi_2,x_2}$, as well as
$\rho_2^L\partial_{X_2}^L(T^{\varphi_2,x_2})$ for each $L\geq 0$,
satisfy the conditions (I-1) to (I-4) uniformly.
\item[(II-3)] If $\varphi_i$ is a bump function on a ball $B^i(r_i)$ in
$M_i$, then for all integers $a_1, a_2\geq 0,$
$$ \big|\partial_{X_1}^{a_1}\partial_{X_2}^{a_2}T(\varphi_1\otimes\varphi_2)\big|\lesssim r_1^{-a_1}r_2^{-a_2}. $$
In (II-2) and (II-3), both inequalities are taken in the sense of
(I-2) whenever $\varphi_2$ is a bump function for $B^2(r_2)$ in $M_2$.
\item[(II-4)]
$\big|\partial_{X_1,Y_1}^{a_1}\partial_{X_2,Y_2}^{a_2}K(x_1,y_1;x_2,y_2)\big|\lesssim
\frac{\displaystyle
d_1(x_1,y_1)^{-a_1}d_2(x_2,y_2)^{-a_2}}{\displaystyle
V_1(x_1,y_1)V_2(x_2,y_2)}$ for all integers $a_1, a_2\geq 0.$
\item[(II-5)] The same conditions hold when the index 1 and 2 are
interchanged, that is, whenever the roles of $M_1$ and $M_2$ are
interchanged.
\item[(II-6)] The same properties are assumed to hold for the 3
``transposes'' of $T$, i.e. those operators which arise by
interchanging $x_1$ and $y_1$, or interchanging $x_2$ and $y_2$, or
doing  both interchanges.
\end{itemize}

As mentioned in Section 1, we would like to point out that in the
cancellation conditions (I-2) and (II-2), one can take $0\leq a,
a_1, a_2\leq 1.$ However, even for such choices, these cancellation
conditions are still little bit strong. See the remark after Theorem
2.18 in Subsection 2.4. To show the $L^p$ boundedness for such
operators, the key idea is to use the Littlewood--Paley theory
developed in [NS04].

\subsection{Littlewood--Paley theory and the $L^p$ boundedness of singular integrals}

To construct the Littlewood--Paley square function, in [NS04] the authors
considered the sub-Laplacian $\mathcal {L}$ on $M$ in self-adjoint
form, given by
\begin{eqnarray*}
\mathcal {L}=\sum_{j=1}^k \mathbb{X}_j^{*}\mathbb{X}_j.
\end{eqnarray*}
Here $(\mathbb{X}_j^{*}\varphi, \psi)=(\varphi, \mathbb{X}_j\psi)$,
where $(\varphi,\psi)=\int\limits_{M} \varphi(x)\bar{\psi}(x)d\mu(x)
$, and $\varphi, \psi\in C_0^{\infty}(M)$, the space of $C^{\infty}$
functions on $M$ with compact support. In general,
$\mathbb{X}_j^{*}= -\mathbb{X}_j+a_j$, where $a_j\in C^{\infty}(M)$.
The solution of the following initial value problem for the heat
equation,
\begin{eqnarray*}
{{\partial u} \over {\partial s}}(x,s)+      \mathcal {L}_x u(x,s)=0
\end{eqnarray*}
with $u(x,0)=f(x)$, is given by $u(x,s)=H_s(f)(x)$, where $H_s$ is
the operator given via the spectral theorem by $H_s=e^{-s\mathcal
{L}}$, and an appropriate self-adjoint extension of the non-negative
operator $\mathcal {L}$ initially defined on $C_0^{\infty}(M)$. And
they proved that for $f\in L^2(X)$,
\begin{eqnarray*}
H_s(f)(x)=\int_M H(s,x,y)f(y)d\mu(y).
\end{eqnarray*}
Moreover, $H(s,x,y)$ has some nice properties (see Proposition 2.3.1
in \cite{NS04} and Theorem 2.3.1 in \cite{NS01a}). We restate them
as follows:
\begin{itemize}
\item[(1)] $H(s,x,y)\in C^{\infty}\big([0,\infty)\times M\times M
\backslash \{s=0\ {\rm and}\ x=y \}\big).$
\item[(2)] For every integer $N\geq 0$,
\begin{eqnarray*}
\lefteqn{|\partial_s^j\partial_X^L\partial_Y^K H(s,x,y)|}\\
& \lesssim &\frac{\displaystyle 1 }{\displaystyle
(d(x,y)+\sqrt{s})^{2j+K+L} } \frac{\displaystyle 1 }{\displaystyle
V(x,y)+V_{\sqrt{s}}(x) +V_{\sqrt{s}}(y) } \bigg(\frac{\displaystyle
\sqrt{s} }{\displaystyle d(x,y)+\sqrt{s} } \bigg)^{N\over 2}
\end{eqnarray*}
\item[(3)] For each integer $L\geq 0$ there exist an integer $N_L$ and a
constant $C_L$ so that if $\varphi\in C_0^{\infty}(B(x_0,\delta))$,
then for all $s\in(0,\infty),$
$$ |\partial_X^L H_s[\varphi](x_0)|\leq C_L\delta^{-L}\sup_x\sum_{|J|\leq N_L}\delta^{|J|}|\partial_X^J\varphi(x)|. $$
\item[(4)] For all $(s,x,y)\in (0,\infty)\times M \times M$,
\begin{eqnarray*}
H(s,x,y)&=&H(s,y,x);\\
H(s,x,y)&\geq& 0.
\end{eqnarray*}
\item[(5)] For all $(s,x)\in (0,\infty)\times M$, $\int H(s,x,y)dy=1.$
\item[(6)]  For $1\leq p\leq \infty$, $\|H_s[f]\|_{L^p(M)}\leq
\|f\|_{L^p(M)}$.
\item[(7)]  For every $\varphi\in C_0^{\infty}(M)$ and every $t\geq 0$,
$\lim\limits_{s\rightarrow 0}\|H_s[\varphi]-\varphi\|_t=0$, where
$\|\cdot\|_t$ denotes the Sobolev norm.
\end{itemize}

To introduce the reproducing identity and the Littlewood--Paley
square function, they define a bounded operator
$Q_s=2s{\displaystyle{\partial H_s}\over\displaystyle \partial s}$,
$s>0$, on $L^2(M)$. Denote by $q_s(x,y)$ the kernel of $Q_s$. Then
from the estimates of $H(s,x,y)$, we have
\begin{itemize}
\item[(a)] $q_s(x,y)\in C^{\infty}\big( M\times M \backslash \{x=y \}\big).$
\item[(b)] For every integer $N\geq 0$,
\begin{eqnarray*}
|\partial_X^L\partial_Y^K q_s(x,y)|  \lesssim
\frac{\displaystyle 1 }{\displaystyle (d(x,y)+\sqrt{s})^{K+L} }
\frac{\displaystyle 1 }{\displaystyle V(x,y)+V_{\sqrt{s}}(x)
+V_{\sqrt{s}}(y) } \bigg(\frac{\displaystyle \sqrt{s}
}{\displaystyle d(x,y)+\sqrt{s} } \bigg)^{N\over 2}.
\end{eqnarray*}
\item[(c)] $\int q_s(x,y)dy=\int q_s(x,y)dx=0.$
\end{itemize}

The reproducing identity was established via the operators
$\{Q_s\}_{s>0}$, which plays an important role in Littlewood--Paley
theory and boundedness of singular integral operators. We state it
as follows.
\begin{prop}[\cite{NS04}]\label{prop-of-reproducing formula by nagel and stein}
Let $Q_s^2=Q_s\cdot Q_s$. For $f\in L^2(M)$,
\begin{eqnarray}\label{reproducing formula by nagel and stein}
 \int_0^\infty Q_s^2[f]{ds\over s} =f,
\end{eqnarray}
where the integral on the left is defined as
$\lim\limits_{\epsilon\rightarrow
0}\int_{\epsilon}^{1/\epsilon}Q_s^2[f]{ds\over s}$, with the limit
taken in the $L^2$ norm.
\end{prop}

\noindent The Littlewood--Paley square function $S(f)$ is defined by
$$\big(S[f](x)\big)^2=\int_0^\infty |Q_s[f](x)|^2{ds\over s},$$
and we have
\begin{prop}[\cite{NS04}]\label{prop-of-Nagel-Stein square function}
For $1<p<\infty$,
$\|S[f]\|_{L^p(M)}\thickapprox \|f\|_{L^p(M)}.$
\end{prop}

We now consider that $\widetilde{M}=M_1\times M_2$, where each $M_i$
is as in Subsection 2.1. For each $M_i$, we have a heat operator
$H^i_{s_i}$, and a corresponding $Q_{s_i}^i$. If $f$ is a function
on $\widetilde{M}$ we define $Q_{s_1}^1\cdot
Q_{s_2}^2(f)=Q_{s_1}^1\otimes Q_{s_2}^2(f)$, with $Q^1$ acting on the
$M_1$ variable and $Q^2$ acting on the $M_2$ variable, respectively. The product
square function $\widetilde{S}$ is then given by
$$ \big(\widetilde{S}(f)(x,y)\big)^2=\int_0^\infty\int_0^\infty |Q_{s_1}^1\cdot Q_{s_2}^2(f)(x,y)|^2{ds_1ds_2\over s_1s_2}, $$
and, as showed in \cite{NS04}, we have
\begin{prop}[\cite{NS04}]\label{Nagel-Stein product square function}
For $1<p<\infty$,
$\|\widetilde{S}(f)\|_{L^p(\widetilde{M})}\thickapprox\|f\|_{L^p(\widetilde{M})}$.
\end{prop}

The following $L^p, 1<p<\infty,$ boundedness for the product singular integral operator was obtained in \cite{NS04}.
\begin{theorem}[\cite{NS04}]\label{theorem-of-product T-bd-on-Lp}
For $1<p<\infty$, each product singular integral satisfying
conditions (II-1) to (II-6) extends to be a bounded operator on
$L^p(\widetilde{M})$.
\end{theorem}
We would like to point again that the cancellation conditions in (II-2) plays a key role in the proof of the above theorem.

\subsection{Hardy space theory on product Carnot--Carath\'eodory spaces}

In this subsection, we describe the product Hardy space theory on $\widetilde M,$ where $\widetilde M=M_1\times M_2$ is a product homogeneous type spaces in the sense of Coifman and Weiss [CW]. See [HLL2]  for more details.
This theory includes the $H^p$ boundedness for operators studied in [NS04] and the space $CMO^p(\widetilde M),$ the dual of $H^p(\widetilde M),$ in particular, $CMO^1(\widetilde M)=BMO(\widetilde M),$ the dual of $H^1(\widetilde M).$

We begin with recalling some necessary results on one-parameter setting. Here we denote by $M$ a homogeneous type spaces in the sense of Coifman and Weiss [CW].
We first recall the definition of an approximation to the
identity, which plays the same role as the heat kernel $H(s,x,y)$
does in [NS04].

\begin{definition}[\cite{HMY1}]\label{def-of-ATI compact}
Let $\vartheta$ be the regularity exponent of $M.$ A sequence
$\{S_k\}_{k\in\mathbb{Z}}$ of operators is said to be an
approximation to the identity if there exists constant $C_0>0$ such
that for all $k\in\mathbb{Z}$ and all $x,x',y$ and $y'\in M$,
$S_k(x,y)$, the kernel of $S_k$ satisfy the following conditions:
\begin{eqnarray}
&{(i)}&\ S_k(x,y)=0\ {\rm if}\ d(x,y)\geq C_02^{-k}\ {\rm
and}\ |S_k(x,y)| \leq C_0 \frac{1}{V_{2^{-k}}(x)+V_{2^{-k}}(y)};\label{size of Sk}\\[4pt]
&{(ii)}&\ |S_k(x,y)-S_k(x',y)| \leq
C_02^{k\vartheta}d(x,x')^{\vartheta}
\frac{1}{V_{2^{-k}}(x)+V_{2^{-k}}(y)};\label{smoothness of Sk}\\[4pt]
&{(iii)}&\ {\rm Property}\ (ii)\ {\rm also\ holds\ with\ } x\
{\rm and}\ y\ {\rm interchanged};\\[4pt]
&{(iv)}&\ |[S_k(x,y)-S_k(x,y')]-[S_k(x',y)-S_k(x',y')]|\\[4pt]
&&\qquad\leq
C_02^{2k\vartheta}d(x,x')^{\vartheta}d(y,y')^{\vartheta}
\frac{1}{V_{2^{-k}}(x)+V_{2^{-k}}(y)};\nonumber\\[4pt]
&{(v)}&\ \int\limits_M S_k(x,y) d\mu(y) =\int\limits_M S_k(x,y)
d\mu(x) = 1.
\end{eqnarray}
\end{definition}

We remark that the existence of such an approximation to the
identity follows from Coifman's construction which was first
appeared in \cite{DJS} on space of homogeneous type.
See also
\cite{HMY2} for more details on $M.$

To define the Littlewood--Paley square function, we also need to
recall the spaces of test functions and distributions on $M$.

\begin{definition}[\cite{HMY1}]\label{def-of-test-func-space}
Let $\vartheta$ be the regularity exponent of $M$ and let $0<\gamma, \beta\leq \vartheta$, $x_0\in M$ and $r>0.$
A function $f$ defined on $M$ is said to be a test function of type
$(x_0,r,\beta,\gamma)$ centered at $x_0$ if $f$ satisfies the
following conditions
\begin{itemize}
\item[(i)] $|f(x)|\leq C \frac{\displaystyle 1}{\displaystyle
V_r(x_0)+V(x,x_0)} \Big(\frac{\displaystyle r}{\displaystyle
r+d(x,x_0)}\Big)^{\gamma}$;
\item[(ii)] $|f(x)-f(y)|\leq C \Big(\frac{\displaystyle
d(x,y)}{\displaystyle r+d(x,x_0)}\Big)^{\beta} \frac{\displaystyle
1}{\displaystyle V_r(x_0)+V(x,x_0)} \Big(\frac{\displaystyle
r}{\displaystyle r+d(x,x_0)}\Big)^{\gamma}$
\item[] for all $x,y\in M$ with
$d(x,y)<{\frac{1}{2A}}(r+d(x,x_0)).$
\end{itemize}
\end{definition}

If $f$ is a test function of type $(x_0,r,\beta,\gamma)$, we write
$f\in G(x_0,r,\beta,\gamma)$ and the norm of $f\in
G(x_0,r,\beta,\gamma)$ is defined by
$$\|f\|_{G(x_0,r,\beta,\gamma)}=\inf\{C>0:\ (i)\ {\rm and }\ (ii)\ {\rm hold} \}.$$
Now fix $x_0\in M$ we denote $G(\beta,\gamma)=G(x_0,1,\beta,\gamma)$
and by $G_0(\beta,\gamma)$ the collection of all test functions in
$G(\beta,\gamma)$ with $\int_M f(x) dx=0.$ It is easy to
check that $G(x_1,r,\beta,\gamma)=G(\beta,\gamma)$ with equivalent
norms for all $x_1\in M$ and $r>0$. Furthermore, it is also easy to
see that $G(\beta,\gamma)$ is a Banach space with respect to the
norm in $G(\beta,\gamma)$.

Let $\GG(\beta,\gamma)$ be the completion of the space
$G_0(\vartheta,\vartheta)$ in the norm of $G(\beta,\gamma)$ when
$0<\beta,\gamma<\vartheta$. If $f\in \GG(\beta,\gamma)$, we then
define $\|f\|_{\GG(\beta,\gamma)}=\|f\|_{G(\beta,\gamma)}$.
$(\GG(\beta,\gamma))'$, the distribution space, is defined by the
set of all linear functionals $L$ from $\GG(\beta,\gamma)$ to
$\mathbb{C}$ with the property that there exists $C\geq0$ such that
for all $f\in \GG(\beta,\gamma)$,
$$|L(f)|\leq C\|f\|_{\GG(\beta,\gamma)}.$$

Let $D_k=S_k-S_{k-1},$ where $S_k$ is an approximation to the
identity on $M$ with the regularity exponent $\vartheta.$ The
Littlewood--Paley square function is defined as follows.

\begin{definition}[\cite{HMY1}]\label{def-of-square-func}
For each $f\in (\GG(\beta,\gamma))'$ with $0<\beta,
\gamma<\vartheta,$ $S(f),$ the Littlewood--Paley square function of
$f,$ is defined by
$$S(f)(x)=\big\lbrace \sum\limits_k |D_k(f)(x)|^2\big\rbrace^{\frac{1}{2}}.$$
\end{definition}

We pass the above one parameter case to the product case. We first introduce the space of test functions and distributions on
$\widetilde{M}=M_1\times M_2.$
\begin{definition}[\cite{HLL2}]\label{def-of-test-func-on-M times M}
Let $\vartheta_1$ and $\vartheta_2$ be the regularity exponents of $M_1$ and $M_2$, respectively. Let $(x_1^0,x_2^0)\in\widetilde{M}$,
$0<\gamma_1,\beta_1\leq \vartheta_1$, $0<\gamma_2,\beta_2\leq \vartheta_2$ and $r_1,
r_2>0.$ A function $f(x_1,x_2)$ defined on $\widetilde{M}$ is said to be
a test function of type
$(x_1^0,x_2^0;r_1,r_2;\beta_1,\beta_2;\gamma_1,\gamma_2)$ if for any
fixed $x_2\in M_2,$ $f(x_1,x_2),$ as a function of the variable
$x_1,$ is a test function in $G(x_1^0,r_1,\beta_1,\gamma_1)$ on $M_1.$
Similarly, for any fixed $x_1\in M_1,$ $f(x_1,x_2),$ as a function of
the variable of $x_2,$ is a test function in
$G(x_2^0,r_2,\beta_2,\gamma_2)$ on $M_2.$ Moreover, the following
conditions are satisfied:
\begin{itemize}
\item[(i)] $\Vert f(\cdot,x_2)\Vert_{G(x_1^0,r_1,\beta_1,\gamma_1)}\leq C
\frac{\displaystyle 1}{\displaystyle
V_{r_2}(x_2^0)+V(x_2^0,x_2)}\Big(\frac{\displaystyle r_2}{\displaystyle
r_2+d_2(x_2^0,x_2)}\Big)^{\gamma_2}$
\item[(ii)] $\Vert f(\cdot,x_2)-f(\cdot,x'_2)\Vert_{G(x_1^0,r_1,\beta_1,\gamma_1)}$
\item[] $\leq C \Big(\frac{\displaystyle d(x_2,x'_2)}{\displaystyle
r_2+d_2(x_2^0,x_2)}\Big)^{\beta_2} \frac{\displaystyle 1}{\displaystyle
V_{r_2}(x_2^0)+V(x_2^0,x_2)}\Big(\frac{\displaystyle r_2}{\displaystyle
r_2+d_2(x_2,x_2^0)}\Big)^{\gamma_2}$
\item[] for all $x_2,x'_2\in M_2$
with $d_2(x_2,x'_2)\leq (r_2+d(x_2,x_2^0))/2A$;
\item[(iii)] Properties $(i)-(ii)$ also hold with $x_1$ and $x_2$ interchanged.
\end{itemize}
\end{definition}

If $f$ is a test function of type
$(x_1^0,x_2^0;r_1,r_2;\beta_1,\beta_2;\gamma_1,\gamma_2)$, we write
$f\in G(x_1^0,x_2^0;r_{1},r_{2};\beta_{1},\beta_{2};$
$\gamma_{1},\gamma_{2})$ and the norm of $f$ is defined by
$$\|f\|_{G(x_1^0,x_2^0;r_{1},r_{2};\beta_{1},\beta_{2};\gamma_{1},\gamma_{2})}=\inf\{C:\ (i),(ii)\
{\rm and}\ (iii)\ \ {\rm hold}\}.$$ Similarly, we denote by
$G(\beta_{1},\beta_{2};\gamma_{1},\gamma_{2})$ the class of
$G(x_1^0,x_2^0;1,1;\beta_{1},\beta_{2};\gamma_{1},\gamma_{2})$ for
any fixed $(x_1^0,x_2^0)\in \widetilde{M}.$ We can
check that
$G(x_{0},y_{0};r_{1},r_{2};\beta_{1},\beta_{2};\gamma_{1},\gamma_{2})=
G(\beta_{1},\beta_{2};\gamma_{1},\gamma_{2})$ with equivalent norms
for all $(x_{0},y_{0})\in \widetilde{M}$ and $r_1,r_2>0$.
Furthermore, it is easy to see that
$G(\beta_{1},\beta_{2};\gamma_{1},\gamma_{2})$ is a Banach space
with respect to the norm in
$G(\beta_{1},\beta_{2};\gamma_{1},\gamma_{2})$.

Next we denote by $G_0(\beta_{1},\beta_{2};\gamma_{1},\gamma_{2})$ the set of all test functions in $G(\beta_{1},\beta_{2};\gamma_{1},\gamma_{2})$ satisfying the cancellation conditions on both variables $x$ and $y$, i.e.,  if $f(x,y)\in
G_0(\beta_{1},\beta_{2};\gamma_{1},\gamma_{2})$, then
$\int_{M_1}f(x,y)dx=\int_{M_2}f(x,y)dy=0.$ Let $\GGp(\beta_1,\beta_2;\gamma_1,\gamma_2)$ be the completion of
the space $G_0(\vartheta_1,\vartheta_2;\vartheta_1,\vartheta_2)$ in
$G(\beta_1,\beta_2;\gamma_1,\gamma_2)$ with
$0<\beta_i,\gamma_i<\vartheta_i,$ for $ i=1,2.$ If
$f\in\GGp(\beta_{1},\beta_{2};\gamma_{1},\gamma_{2}) $, we then
define
$\|f\|_{\GGp(\beta_{1},\beta_{2};\gamma_{1},\gamma_{2})}=\|f\|_{G(\beta_{1},\beta_{2};\gamma_{1},\gamma_{2})}$.

We define the distribution space
$\big(\GGp(\beta_{1},\beta_{2};\gamma_{1},\gamma_{2})\big)^{'}$ by
all linear functionals $L$ from
$\GGp(\beta_{1},\beta_{2};\gamma_{1},\gamma_{2})$ to $\mathbb{C}$
with the property that there exists $C\geq 0$ such that for all
$f\in \GGp(\beta_{1},\beta_{2};\gamma_{1},\gamma_{2})$,
$$|L(f)|\leq C
\|f\|_{\GGp(\beta_{1},\beta_{2};\gamma_{1},\gamma_{2})}.$$

Now the Littlewood--Paley square function on $\widetilde M$ is
defined by
\begin{definition}[\cite{HLL2}]\label{def-of-squa-func}
Let $\{S_{k_i}\}_{k_i\in\mathbb{Z}}$ be approximations to the
identity on $M_i$ and $D_{k_i}=S_{k_i}-S_{k_i-1}, i=1,2.$ For $f\in
\big(\GGp(\beta_1,\beta_2;\gamma_1,\gamma_2)\big)'$ with
$0<\beta_i,\gamma_i<\vartheta_i, i=1,2$, $\widetilde S(f),$ the
Littlewood--Paley square function of $f,$ is defined by
$$ \widetilde S(f)(x_1,x_2)=\Big\{ \sum_{k_1=-\infty}^\infty\sum_{k_2=-\infty}^\infty\big| D_{k_1}D_{k_2}(f)(x_1,x_2) \big|^2 \Big\}^{1/2}.$$
\end{definition}

By the results on each $M_i, i=1,2,$ and iteration as given in
\cite{FS}, we immediately obtain
\begin{theorem}[\cite{HLL2}]\label{theorem Lp bd of suqa func}
If $f\in L^p(\widetilde M), 1<p<\infty,$ then $\Vert \widetilde
S(f)\Vert_p\approx \Vert f\Vert_p.$
\end{theorem}

We would like to point out that the following discrete
Littlewood--Paley square function is more convenient for the study
of the Hardy space $H^p$ when $p\le 1.$ See [HLL2] for more details.

\begin{definition}\label{def-of-squa-func}
Let $\{S_{k_i}\}_{k_i\in\mathbb{Z}}$ be approximations to the
identity on $M_i$ and $D_{k_i}=S_{k_i}-S_{k_i-1}, i=1,2.$ For $f\in
\big(\GGp(\beta_1,\beta_2;\gamma_1,\gamma_2)\big)'$ with
$0<\beta_i,\gamma_i<\vartheta_i, i=1,2$, $\widetilde S_d(f),$ the
discrete Littlewood--Paley square function of $f,$ is defined by
$$
\widetilde S_d(f)(x_1,x_2)=\Big\{\sum_{k_1=-\infty}^\infty\sum_{k_2=-\infty}^\infty\sum_{I_1}\sum_{I_2 }
|D_{k_1}D_{k_2}(f)(x_1,x_2)|^2\chi_{I_{1}}(x_1)\chi_{I_{2}}(x_2) \Big\}^{1/2},
$$
where for each $k_1$ and $k_2$, $I_{1}$ and $I_{2}$ range over all
the dyadic cubes in $M_1$ and $M_2$ with length
$\ell(I_1)=2^{-k_1-N_1}$ and $\ell(I_2)=2^{-k_2-N_2}$, respectively and $N_1$ and $N_2$ are fixed positive large integers.
\end{definition}

By the Plancherel--P\^olya inequalities in [HLL2], it was shown that
the $L^p$ norm of these two kinds of square functions are
equivalent. More precisely, we have
\begin{prop}[\cite{HLL2}]\label{prop discrete Littlewood--Paley}
For all $f\in \big(\GGp(\beta_1,\beta_2;\gamma_1,\gamma_2)\big)'$
with
$0<\beta_i,\gamma_i<\vartheta_i$, and for $\max\big(\frac{{{Q}_1}}{{{Q}_1}+\vartheta_1},\frac{{{Q}_2}}{{{Q}_2}+\vartheta_2}\big)<p<\infty,
i=1,2,$ we have $\Vert \widetilde S(f)\Vert_p\approx \Vert \widetilde
S_d(f)\Vert_p$, where the implicit constants are independent of $f$.
\end{prop}

We are ready to introduce the Hardy spaces on $\widetilde M.$
\begin{definition}[\cite{HLL2}]\label{def-of-product-Hardy-space}
Let $\max\big(\frac{ Q_1}{ Q_1+\vartheta_1},\frac{ Q_2}{
Q_2+\vartheta_2} \big) <p\leq1$ and
$0<\beta_i,\gamma_i<\vartheta_i$ for $i=1,2$.
$$H^p(\widetilde M):=\big\lbrace f \in
\big(\GGp(\beta_1,\beta_2;\gamma_1,\gamma_2)\big)':\
\widetilde S_d(f)\in L^p(\widetilde M)\big\rbrace$$
and if $f\in H^p(\widetilde M),$ the norm of $f$ is defined by $\Vert f\Vert_{H^p(\widetilde M)}=\Vert \widetilde S_d(f)\Vert_p.$
\end{definition}

The space $CMO^p(\widetilde M)$ is defined as follows.

\begin{definition}[\cite{HLL2}]\label{def-of-product-CMO-space}
Let $\max\big(\frac{ 2Q_1}{ 2Q_1+\vartheta_1},\frac{ 2Q_2}{
2Q_2+\vartheta_2} \big) <p\leq1$ and
$0<\beta_i,\gamma_i<\vartheta_i$ for $i=1,2$. Let
$\{S_{k_i}\}_{k_i\in\mathbb{Z}}$ be approximations to the identity
on $M_i$ and for $k_i\in\mathbb{Z}$, set
$D_{k_i}=S_{k_i}-S_{k_i-1}$, $i=1,2$. The generalized Carleson
measure space $CMO^p(\widetilde{M})$ is defined, for
$f\in\big(\GGp(\beta_1,\beta_2;\gamma_1,\gamma_2)\big)',$ by
\begin{eqnarray}\label{product-CMOp-norm}
&&\|f\|_{CMO^p(\widetilde{M})}\\
&&\hskip.5cm=\sup_{\Omega} \bigg\{
\frac{\displaystyle 1}{\displaystyle \mu(\Omega)^{{2\over
p}-1}}\int_{\Omega} \sum_{k_1,k_2} \sum_{I_1\times I_2\subseteq \Omega}
\big|D_{k_1}D_{k_2}(f)(x_1,x_2)\big|^2 \chi_{I_1}(x_1)\chi_{I_2}(x_2)dx_1dx_2
\bigg\}^{1\over 2}<\infty,\nonumber
\end{eqnarray}
where $\Omega$ are taken over all open sets in $\widetilde{M}$ with
finite measures and for each $k_1$ and $k_2$, $I_1,I_2$ range over all
the dyadic cubes in $M_1$ and $M_2$ with length
$\ell(I_1)=2^{-k_1-N_1}$ and $\ell(I_2)=2^{-k_2-N_2}$, respectively.
\end{definition}

The main results in [HLL2] are the following

\begin{theorem}[\cite{HLL2}]\label{theorem T bd on Hp}
Each singular integral $T$ satisfying (II-1) through (II-6) extends
to a bounded operator on $H^p(\widetilde{M})$, and
from $H^p(\widetilde{M})$ to $L^p(\widetilde{M})$ for
$\max\big(\frac{ Q_1}{ Q_1+ \vartheta_1 },\frac{ Q_2}{ Q_2+
\vartheta_2 }\big)<p<\infty.$ Moreover, $T$ extends to a bounded
operator on $BMO(\widetilde{M})$.
\end{theorem}

\begin{theorem}[\cite{HLL2}]\label{theorem-of-duality-on-product-case} For
 $\max\big(\frac{ 2Q_1}{ 2Q_1+
\vartheta_1 },\frac{ 2Q_2}{ 2Q_2+ \vartheta_2 }\big)<p\leq1$,
$ \big(H^p(\widetilde{M})\big)'=CMO^p(\widetilde{M}).$
More precisely, for $g\in CMO^p(\widetilde M)$ then $\ell_g(f)=\langle f,g\rangle,$ initially defined on
$\GGp(\beta_1,\beta_2;\gamma_1,\gamma_2)$ for $0<\beta_i,\gamma_i<\vartheta_i$ for $i=1,2,$ is a continuous linear functional on $H^p(\widetilde M)$ with the norm $\|\ell_g\|\leq C \|g\|_{CMO^p}.$ Conversely, if $\ell$ is a continuous linear functional on $H^p(\widetilde M)$ then there exists a $g\in CMO^p(\widetilde M),$ such that $\ell(f)=\langle f, g\rangle$ for $f\in \GGp(\beta_1,\beta_2;\gamma_1,\gamma_2)$ with $\|g\|_{CMO^p}\leq C\|\ell\|.$
In particular,
$ \big(H^1(\widetilde{M})\big)'=CMO^1(\widetilde{M})=BMO(\widetilde{M}).$
\end{theorem}

We remark that the spaces $H^p(\widetilde M)$ and $CMO^p(\widetilde M)$ defined in Definitions \ref{def-of-product-Hardy-space} and \ref{def-of-product-CMO-space}, respectively, are independent of the choices of the approximations to the identity. Moreover, the cancellation conditions in (II-2) are crucial in the proof of Theorems \ref{theorem T bd on Hp} and \ref{theorem-of-duality-on-product-case}.

\section{$T1$ theorem  on product Carnot--Carath\'eodory spaces}
\setcounter{equation}{0}

In this section, we first introduce a class of singular integral operators on product Carnot--Carath\'eodory spaces. As mentioned, this class includes Journ\'{e}'s class on product Euclidean spaces and operators studied in [NS04]. We then prove the product $T1$ theorem on product Carnot--Carath\'eodory spaces, the main result of this paper.

\subsection{Singular integrals on product Carnot--Carath\'eodory spaces}

Suppose that $M_1$ and $M_2$ are Carnot--Carath\'eodory spaces and $\widetilde M=M_1\times M_2$ is the product Carnot--Carath\'eodory space. Let $C^\eta_0(M_1)$ denote the space of continuous functions
$f$ with compact support such that
$$\|f\|_{\eta(M_1)}:= \sup\limits_{x,y\in M_1, x\ne y} \frac{|f(x)-f(y)|}{d_1(x,y)^\eta}<\infty$$
and $C^\eta_0(M_2)$ is defined similarly.

Now let $C^\eta_0(\widetilde{M}), \eta>0,$ denote the space of continuous functions $f$ with compact support such that
$$\| f\|_{\eta} : = \sup_{x_1\ne y_1, x_2\ne y_2}
  \frac{|f(x_1,x_2)-f(y_1,x_2)-f(x_1,y_2)+f(y_1,y_2)|}{d_1(x_1,y_1)^\eta d_2(x_2,y_2)^\eta} <\infty.$$

We first consider one factor case. A continuous function $K(x_1,y_1)$
defined on  $M_1\backslash \{ (x_1, y_1): x_1=y_1\}$ is called a {\it
Calder\'on--Zygmund kernel} if there exist constant $C>0$ and a
regularity exponent $\varepsilon\in (0,1]$ such that
\begin{eqnarray*}
&(a)&\  |K(x_1,y_1)|\leq C V(x_1,y_1)^{-1};\\[4pt]
&(b)&\ \vert K(x_1,y_1)-K(x_1,y_1')\vert\leq C
\big(\frac{d_1(y_1,y_1')}{d_1(x_1,y_1)}\big)^{\varepsilon}V(x_1,y_1)^{-1}
             \qquad {\rm if}\ d_1(y_1,y_1')\le d_1(x_1,y_1)/2A;\\[4pt]
&(c)&\ \vert K(x_1,y_1)-K(x_1',y_1)\vert\leq C
\big(\frac{d_1(x_1,x_1')}{d_1(x_1,y_1)}\big)^{\varepsilon}V(x_1,y_1)^{-1}
             \qquad {\rm if}\ d_1(x_1,x_1')\le d_1(x_1,y_1)/2A.
\end{eqnarray*}

The smallest such constant $C$ is denoted by $|K|_{CZ}.$ We say
that an operator $T$ is a {\it Calder\'on--Zygmund singular
integral operator} associated with a Calder\'on--Zygmund kernel
$K$ if the operator $T$ is a continuous linear operator from
$C^\eta_0(M_1)$ into its dual such that
$$\langle Tf, g \rangle=\iint g(x_1)K(x_1,y_1)f(y_1)dy_1dx_1$$
for all functions $f, g\in C^\eta_0(M_1)$ with disjoint supports.
$T$ is said to be a {\it Calder\'on--Zygmund operator} if it
extends to be a bounded operator on $L^2(M_1).$ If $T$ is a
Calder\'on--Zygmund operator associated with a kernel $K$, its
operator norm is defined by $\|T\|_{CZ}=\|T\|_{L^2\rightarrow
L^2}+ | K|_{CZ}$.

Similarly, we can define the {\it Calder\'on--Zygmund operator}
$T$ on $M_2$ associated with a Calder\'on--Zygmund kernel
$K(x_2,y_2)$, whose operator norm is defined by
$\|T\|_{CZ}=\|T\|_{L^2\rightarrow L^2}+ | K|_{CZ}$.

Now we introduce a class of the {\it product Calder\'on--Zygmund
singular integral operators} on $\widetilde{M}$. Let $T:
C^\eta_0({\widetilde M})\rightarrow[C^\infty_0({\widetilde M})]'$ be
a linear operator defined in the weakest possible sense. $T$ is said
to be a Calder\'on--Zygmund singular integral operator if there
exists a pair $(K_1, K_2)$ of Calder\'on--Zygmund valued operators
on $M_2$ and $M_1,$ respectively, such that

$$\langle g\otimes k, Tf\otimes h\rangle
   =\iint g(x_1)\langle k, K_1(x_1,y_1)h\rangle f(y_1) dx_1dy_1$$
for all $f, g\in C^\eta_0(M_1)$ and $h, k\in C^\eta_0(M_2),$ with supp $f\ \cap$ supp
$g=\emptyset$
and
$$\langle k\otimes g, Th\otimes f\rangle =\iint g(x_2)\langle k, K_2(x_2,y_2)h\rangle f(y_2) dx_2dy_2$$
for all $f, g\in C^\eta_0(M_2)$ and $h, k\in C^\eta_0(M_1),$ with supp $f\ \cap$ supp
$g=\emptyset.$
Moreover, $\|K_i(x_i,y_i)\|_{CZ}$, $i=1,2$, as functions of $x_i, y_i\in M_i,$ satisfy the following conditions:
\begin{eqnarray*}
&(i)&\  \|K_i(x_i,y_i)\|_{CZ}\leq C V(x_i,y_i)^{-1};\\[4pt]
&(ii)&\ \|K_i(x_i,y_i)-K_i(x_i,y_i^{'})\|_{CZ}\leq C
\big(\frac{d_i(y_i,y_i^{'})}{d_i(x_i,y_i)}\big)^{\varepsilon}V(x_i,y_i)^{-1}
             \qquad {\rm if}\ d_i(y_i,y_i^{'})\le d_i(x_i,y_i)/2A;\\[4pt]
&(iii)&\ \|K_i(x_i,y_i)-K_i(x_i^{'},y_i)\|_{CZ}\leq C
\big(\frac{d_i(x_i,x_i^{'})}{d_i(x_i,y_i)}\big)^{\varepsilon}V(x_i,y_i)^{-1}
             \qquad {\rm if}\ d_i(x_i,x_i^{'})\le d_i(x_i,y_i)/2A.
\end{eqnarray*}
We remark, as mentioned, that the above class of the product
Calder\'on--Zygmund singular integral operators includes the class
of operators introduced by Journ\'e on the Euclidean spaces and
studied in [NS04].

Suppose that $T$ is such a {\it product Calder\'on--Zygmund singular
integral operator} on $\widetilde{M}.$ $T$ is said to be a {\it
product Calder\'on--Zygmund operator} on $\widetilde{M}$ if $T$
extends to be a bounded operator on $L^2.$

Before stating the $T1$ theorem on $\widetilde{M}$, we first
describe, for one factor case, how a Calder\'on--Zygmund singular
integral operator $T$ acts on bounded $C^\eta(M_1)$ functions
(denote by $C^\eta_b(M_1)$). Following [J], for $f\in
C^\eta_b(M_1)$, $Tf$ will be defined by a distribution acting on
$C^\eta_{00}(M_1)$, which is a subspace of $C^\eta_0(M_1)$ of
functions $g$ such that $\int g(x) dx =0$. To do this, let $g\in
C^\eta_{00}(M_1)$ and $h\in C^\eta_0(M_1)$ be equal to $f$ on a
neighborhood of supp $g$, so that $g$ and $f-h$ have disjoint
supports.

If $f$ has compact support, then
\begin{eqnarray*}
\langle g, Tf\rangle&=&\langle g, Th\rangle + \langle g, T(f-h)\rangle,
\end{eqnarray*}
and
\begin{eqnarray*}
\langle g, T(f-h)\rangle &=& \iint g(x)K(x_1,y_1)[f(y_1)-h(y_1)]dx_1 dy_1,
\end{eqnarray*}
because $g$ and $f-h$ have disjoint supports.
Since $g$ has cancellation, the second equality above is also equal to
\begin{eqnarray*}
\iint g(x_1)[K(x_1,y_1)-K(x_0,y_1)][f(y_1)-h(y_1)]dx_1 dy_1,
\end{eqnarray*}
where $x_0$ is any point in the support of $g$. Note that this integral is, by the regularity on the kernel $K,$ absolutely convergent even if $(f-h)$ has non-compact support, and is
independent of $x_0$. This integral can therefore serve as a definition of $\langle g, T(f-h)\rangle$. Obviously $\langle g, Th\rangle+
\langle g, T(f-h)\rangle$ does not depend on the choice of $h$. Hence we can set
\begin{eqnarray*}
\langle g, Tf\rangle=\langle g, Th\rangle + \langle g, T(f-h)\rangle
\end{eqnarray*}
for $f\in C^\eta_b(M_1)$ and this gives the desired extension.

In order to state an analogue in the product setting, that is, how a
product Calder\'on--Zygmund singular integral operator $T$ acts on
bounded $C^\eta({\widetilde M})$ functions (denote by
$C^\eta_b({\widetilde M})$), we can first define the operator $T_1$
by the following
$$\langle g_1\otimes g_2, Tf_1\otimes f_2\rangle =\langle g_2, \langle g_1, T_1f_1\rangle f_2\rangle $$
for $f_1, g_1\in C^\eta_0(M_1)$ and $f_2, g_2\in C^\eta_0(M_2).$

Note that when $g_1\in C^\eta_{00}(M_1)$ and $f_1\in C^\eta_b(M_1),
\langle g_1, T_1f_1\rangle $ is well defined. Moreover, $\langle g_1, T_1f_1\rangle $ is a
Calder\'on--Zygmund singular integral operator on $M_2$ with a
Calder\'on--Zygmund kernel
 $\langle g_1, T_1f_1\rangle (x_2,y_2)=\langle g_1, K_2(x_2,y_2)f_1\rangle .$ Therefore, for $g_2\in C^\eta_{00}(M_2)$ and $f_2\in C^\eta_b(M_2), \langle g_2, \langle g_1, T_1f_1\rangle f_2\rangle $
 is well defined. One defines $\langle g_1, T_2f_1\rangle $ similarly for $g_1\in C^\eta_{00}(M_1)$ and $f_1\in C^\eta_b(M_1).$ Using these definitions,
 we can give  a meaning of the notation $T1=0.$ More precisely, $T1=0$ means $\langle g_1\otimes g_2, T1\rangle =0$ for all $g_1\in C^\eta_{00}(M_1)$
 and $g_2\in C^\eta_{00}(M_2),$ that is,
 $$\iint g(x_1)g(x_2)K(x_1,x_2,y_1,y_2)dx_1dx_2dy_1dy_2=0.$$
 Similarly, $T_1(1)=0$ is equivalent to $\langle g_1, \langle g_2, T_2f_2\rangle 1\rangle =0$ for all $g_1\in C^\eta_{00}(M_1)$
 and $f_2, g_2\in C^\eta_0(M_2),$ that is, for $g_1\in C^\eta_{00}(M_1), g_2\in C^\eta_{00}(M_2)$ and almost everywhere $y_2\in M_2,$
 $$\iint g(x_1)g(x_2)K(x_1,x_2,y_1,y_2)dx_1dx_2dy_1=0.$$
 While ${T_1}^*(1)=0$ means ${\langle g_2, T_2f_2\rangle }^*1=0$ in the same conditions. Interchanging the role of indices one obtains the meaning of $T_2(1)=0$ and ${T_2}^*(1)=0.$

We also  need to introduce the definition of weak boundedness
property (denote by WBP). We begin with the one factor case. Let $T$
be a Calder\'on--Zygmund singular integral operator on $M_1$ and let
$A_{M_1}(\delta, x_1^0, r_1), \delta\in (0, \vartheta_1], x_1^0\in M_1$ and
$r_1>0,$ be a set of all $f\in C^\delta_0(M_1)$ supported in $B(x_1^0,
r_1)$ satisfying $\|f\|_\infty\leq 1$ and $\|f\|_{\delta}\leq
r_1^{-\delta}.$  We say that $T$ has the weak boundedness property
(denote by $T\in WBP$) if there exist $0<\delta\leq \vartheta_1$ and a
constant $C>0$ such that for all $x_1^0\in M_1, r_1>0,$ and all $\phi,
\psi\in A_{M_1}(\delta, x_1^0, r_1),$
$$|\langle T\phi, \psi\rangle |\leq C V_{r_1}(x_1^0).$$

Similarly we can define the set $A_{M_2}(\delta, x_2^0, r_2), \delta\in
(0, \vartheta_2], x_2^0\in M_2$ and the  weak boundedness property for a
Calder\'on--Zygmund singular integral operator on $M_2$.

In the following, we define the weak boundedness property in the product setting.

\begin{definition}\label{def-of-product-WBP}
Let $T$ be a product Calder\'on--Zygmund singular integral operator
on $\widetilde M.$ $T$ has the WBP if
\begin{eqnarray}\label{wbp1}
\|\langle T_2\phi^1, \psi^1\rangle \|_{CZ}\leq C V_{r_1}(x_1^0)\hskip1cm {\rm
for\ all\ } \phi^1, \psi^1\in A_{M_1}(\delta, x_1^0, r_1),
\end{eqnarray}
\begin{eqnarray}\label{wbp2}
\|\langle T_1\phi^2, \psi^2\rangle \|_{CZ}\leq C V_{r_2}(x_2^0)\hskip1cm {\rm
for\ all\ } \phi^2, \psi^2\in A_{M_2}(\delta, x_2^0, r_2).
\end{eqnarray}
\end{definition}
It is easy to see that if $T$ satisfies (\ref{wbp1}) and (\ref{wbp2}), then
\begin{eqnarray}\label{wbp3}
|\langle T\phi^1\otimes\phi^2, \psi^1\otimes\psi^2\rangle |\leq C V_{r_1}(x_1^0) V_{r_2}(x_2^0)
\end{eqnarray}
for all $\phi^1, \psi^1\in A_{M_1}(\delta, x_1^0, r_1)$ and $\phi^2,
\psi^2\in A_{M_2}(\delta, x_2^0, r_2)$.

It is easy to see that if $T$ is a product Calder\'on--Zygmund
operator on $\widetilde M,$ then $T$ has the weak boundedness
property.

We are ready to state the $T1$ theorem, the main result in this paper.

\vskip.3cm \noindent{\bf Theorem A}\ \ Let $T$ be a product
Calder\'on--Zygmund singular integral operator on $\widetilde M$.
Then $T$ and $\widetilde T$ are both bounded on $L^2(\widetilde M)$
if and only if $T1$, $T^*1$ $\widetilde T1$, and $(\widetilde T)^*1$
lie on $BMO(\widetilde M)$ and $T$ has the weak boundedness
property.

The proof of Theorem A will be given in Subsection 3.2 and 3.3, respectively.
\vskip.1cm

\subsection{Necessary conditions of $T1$  Theorem}
To show the necessary conditions in Theorem A, we will employ the
Hardy space theory on $\widetilde M$ developed in [HLL2]. As
mentioned in Section 1, we first show that if $T$ is a
Calder\'on--Zygmund operator on $\widetilde M$ then $T$ extends to a
bounded operator from $H^p(\widetilde M)$ to $L^p(\widetilde M)$ for
$p\leq 1$ and is close to 1. This, particularly for $p=1$, together
with the duality $(L^1, L^\infty)$ and $(H^1, BMO)$, implies that
$T$ is bounded from $L^\infty$ to $BMO.$ To achieve this goal, the
main tool we need is an atomic decomposition for $H^p(\widetilde
M).$ To this end, as in the classical case, we shall first provide
Journ\'{e}-type covering lemma on $\widetilde M,$ for which we turn
to next subsection.

\subsubsection{Journ\'e-type covering lemma }

We first need a result of Christ.

\begin{theorem}[\cite{Chr1}] \label{Theorem-dyadic-cubes}
Let $(M,\rho,\mu)$ be a space of homogeneous type, then, there
exists a collection $\{I_\alpha^k \subset M: k\in \mathbb{Z},
\alpha \in I^k \}$ of open subsets, where $I^k$ is some index set,
and $C_1, C_2>0$, such that
\begin{itemize}
\item[\rm(i)]  $\mu(M\setminus \bigcup_\alpha I_\alpha^k)=0$ for each fixed
$k$ and $I_\alpha^k\bigcap I_\beta^k = \empty$ if $\alpha\neq\beta$;
\item[\rm(ii)]  for any $\alpha,\beta,k,l$ with $l\geq k$, either
$I_\beta^l\subset I_\alpha^k$ or $I_\beta^l\bigcap
I_\alpha^k=\emptyset$;
\item[\rm(iii)]  for each $(k,\alpha)$ and each $l<k$ there is a unique
$\beta$ such that $ I_\alpha^k \subset I_\beta^l $;
\item[\rm(iv)]  ${\rm diam}(I_\alpha^k)\leq C_1 2^{-k} $;
\item[\rm(v)]  each $I_\alpha^k$ contains some ball $B(z_\alpha^k, C_2 2^{-k}
)$, where $z_\alpha^k \in M$.
\end{itemize}
\end{theorem}
Note that Carnot--Carath\'eodory spaces are spaces of homogeneous type. Therefore, we can think of $I_\alpha^k$ as being a dyadic cube with
diameter rough $2^{-k}$ centered at $z_\alpha^k$. As a result, we
consider $CI_\alpha^k$ to be the cube with the same center as
$I_\alpha^k$ and diameter $C{\rm diam}(I_\alpha^k)$. To simplify notations, we will call $I$ dyadic cubes and denote the side length of $I$ by $\ell(I).$

Let $\{I_{\tau_i}^{k_i} \subset M_i: k_i\in \mathbb{Z},
\tau_i \in I^{k_i} \}$ be the same as in Theorem 3.2.
We call $R=I_{\tau_1}^{k_1}\times
I_{\tau_2}^{k_2}$ a dyadic rectangle in $\widetilde{M}$.   Let
$\Omega\subset \widetilde{M}$ be an open set of finite measure and
$\mathcal{M}_i(\Omega)$ denote the family of dyadic rectangles
$R\subset\Omega$ which are maximal in the $i$th ``direction'', $i=1,2$. Also we denote by $\mathcal{M}(\Omega)$ the set of all maximal
dyadic rectangles contained in $\Omega$. For the sake of
simplicity, we denote by $R=I_1\times I_2$ any dyadic rectangles on
$M_1\times M_2$. Given $R=I_1\times I_2\in \mathcal{M}_1(\Omega)$,
let $\widehat{I}_2=\widehat{I}_2(I_1)$ be the biggest dyadic cube
containing $I_2$ such that
$$\mu\big(\big(I_1\times \widehat{I}_2\big)\cap \Omega\big)>{1\over 2}\mu(I_1\times \widehat{I}_2),$$
where $\mu=\mu_1\times\mu_2$ is the measure on $\widetilde{M}$.
Similarly, Given $R=I_1\times I_2\in \mathcal{M}_2(\Omega)$, let
$\widehat{I}_1=\widehat{I}_1(I_2)$ be the biggest dyadic cube
containing $I_1$ such that
$$\mu\big(\big(\widehat{I}_1\times I_2\big)\cap \Omega\big)>{1\over 2}\mu(\widehat{I}_1\times I_2).$$

For $I_i=I_{\tau_i}^{k_i}\subset {M}_i$, we denote by
$(I_i)_k$, $k\in\mathbb{N}$, any dyadic cube $I_{\beta_i}^{k_i-k}$
containing $I_{\tau_i}^{k_i}$, and $(I_i)_0=I_i$, where $i=1,2$.
Moreover, let $w(x)$ be any increasing function such that
$\sum_{j=0}^\infty jw(C_02^{-j})<\infty$, where $C_0$ is any given
positive constant. In applications, we may take $w(x)=x^\delta$ for
any $\delta>0$.

The Journ\'e-type covering lemma on $\widetilde M$ is the following

\begin{lemma}\label{theorem-cover lemma}
Let  $\Omega$ be any open subset in $\widetilde{M}$ with finite
measure. Then there
exists a positive constant $C$ such that
\begin{eqnarray}\label{1 direction}
\sum_{R=I_1\times I_2\in
\mathcal{M}_1(\Omega)}\mu(R)w\Big({\mu_2(I_2)\over\mu_2(\widehat{I}_2)}\Big)\leq
C\mu(\Omega)
\end{eqnarray}
and
\begin{eqnarray}\label{2 direction}
\sum_{R=I_1\times I_2\in
\mathcal{M}_2(\Omega)}\mu(R)w\Big({\mu_1(I_1)\over\mu_1(\widehat{I}_1)}\Big)\leq
C\mu(\Omega).
\end{eqnarray}

\end{lemma}

\begin{proof}
It suffices to prove (\ref{2 direction}) since (\ref{1
direction}) follows similarly.

Following \cite{P}, let $R=I_1\times I_2\in \mathcal{M}_2(\Omega)$ and for $k\in\mathbb{N}$ let
$$A_{I_1,k}=\cup\big\{I_2: I_1\times I_2\in \mathcal{M}_2(\Omega)\ {\rm and}\ \widehat{I}_1=(I_1)_{k-1} \big\}$$
where we use $(I_1)_{1}$ the denote the father of $I_1$ in the setting of dyadic cubes in $M_1$. Hence, $(I_1)_{k-1}$ means
the ancestor of $I_1$ at $(k-1)$-level.
We also denote the set
$$
  A(\Omega)=\{I_1 \subset M_1:\ {\rm dyadic,\ and}\ \exists\ {\rm a\ dyadic}\ I_2\in M_2, {\rm\ s.t.}\ I_1\times I_2\in \mathcal{M}_2(\Omega) \}.
$$
We rewrite the left side in (3.5) as
\begin{eqnarray*}
\sum_{R=I_1\times I_2\in
\mathcal{M}_2(\Omega)}\mu(R)w\Big({\mu_1(I_1)\over\mu_1(\widehat{I}_1)}\Big)
=\sum_{I_1\in A(\Omega)}\mu_1(I_1)\sum_{k=1}^\infty\sum_{I_2:\ I_2\in
A_{I_1,k}}\mu_2(I_2)w\Big({\mu_1(I_1)\over\mu_1(\widehat{I}_1)}\Big).
\end{eqnarray*}
Note that for $i=1,2, x\in M_i$ and $\lambda\geq 1$, by (2.10) and (2.11),
\begin{eqnarray*}
\lambda^{\kappa_i} \mu_i(B(x,r)) \leq\mu_i(B(x, \lambda r))\leq
\lambda^{Q_i} \mu_i(B(x,r))
\end{eqnarray*}
which implies that ${\mu_i(B(x,r))\over \mu_i(B(x, \lambda r))}\leq
\lambda^{-\kappa_i}$ for $i=1,2$. Thus, for $k\in\mathbb{N}$ and $\widehat{I}_1=(I_1)_{k-1}$, we have ${\mu_1(I_1)\over\mu_1(\widehat{I}_1)}\leq C2^{-\kappa_1k}.$
This yields
\begin{eqnarray}\label{e1 covering lemma}
\sum_{R=I_1\times I_2\in
\mathcal{M}_2(\Omega)}\mu(R)w\Big({\mu_1(I_1)\over\mu_1(\widehat{I}_1)}\Big)
&\leq&\sum_{I_1\in A(\Omega)}\mu_1(I_1)\sum_{k=1}^\infty
w(C2^{-\kappa_1k})\sum_{I_2:\ I_2\in A_{I_1,k}}\mu_2(I_2)\nonumber\\
&\leq& \sum_{I_1\in A(\Omega)}\mu_1(I_1)\sum_{k=1}^\infty w(C2^{-\kappa_1k})
\mu_2(A_{I_1,k}),
\end{eqnarray}
where we use the fact that all $I_2$ in $A_{I_1,k}$ are disjoint
since $I_2$ are the maximal dyadic cubes  and $\widehat{I}_1=(I_1)_{k-1}$ for each fixed $k\in\mathbb{N}$. We now estimate $\mu_2(A_{I_1,k})$. For any $x_2\in A_{I_1,k},$ by the definition of $A_{I_1,k},$ there exists some
dyadic cube $I_2$ such that $I_1\times I_2\in \mathcal{M}_2(\Omega),$
$x_2\in I_2,$ and $\widehat{I}_1=(I_1)_{k-1}$ for some $k\in\mathbb{N}.$ Thus, by the definition of $\widehat{I}_1,$
$\mu\big((I_1)_{k-1}\times I_2\cap \Omega \big)>{1\over2}\mu\big((I_1)_{k-1}\times I_2\big)$ and
$\mu\big((I_1)_{k}\times I_2\cap \Omega \big)\leq{1\over2}\mu\big((I_1)_{k}\times I_2\big).$
Now set $E_{I_1}(\Omega)=\cup\{I_2: I_1\times I_2\subset\Omega\}$,
then from the last inequality above, we have
$$ \mu\big((I_1)_{k}\times (I_2\cap E_{(I_1)_k}) \big)\leq{1\over2}\mu\big((I_1)_{k}\times I_2\big), $$
which implies that $\mu_2(I_2\cap E_{(I_1)_k})\leq {1\over2}\mu_2(I_2)$
and hence $\mu_2(I_2\cap (E_{(I_1)_k})^c)> {1\over2}\mu_2(I_2)$, where
we denote $(E_{(I_1)_k})^c=E_{I_1}\backslash E_{(I_1)_k}$. This gives
$$ M_{HL,2}\big(\chi_{E_{I_1}\backslash E_{(I_1)_k}}\big)(x_2)>{1\over2}, $$
and hence $A_{I_1,k}\subset \big\{ x_2\in M_2:
M_{HL,2}\big(\chi_{E_{I_1}\backslash E_{(I_1)_k}}\big)(x_2)>{1\over
2} \big\}$, which implies that
\begin{eqnarray}\label{e2 covering lemma}
\mu_2(A_{I_1,k})\leq \mu_2\big(\big\{ x_2\in M_2:
M_{HL,2}\big(\chi_{E_{I_1}\backslash E_{(I_1)_k}}\big)(x_2)>{1\over
2} \big\}\big)\leq C\mu_2(E_{I_1}\backslash E_{(I_1)_k}),
\end{eqnarray}
where we use $M_{HL,2}$ to denote the Hardy--Littlewood maximal
function on $M_2$.

 Thus, combining the estimates of (\ref{e1 covering lemma}) and (\ref{e2 covering lemma}), we obtain
\begin{eqnarray*}
&&\sum_{R=I_1\times I_2\in
\mathcal{M}_2(\Omega)}\mu(R)w\Big({\mu_1(I_1)\over\mu_1(\widehat{I}_1)}\Big)
\leq C\sum_{I_1\in A(\Omega)}\mu_1(I_1)\sum_{k=1}^\infty w(C2^{-\kappa_1k})
\mu_2(E_{I_1}\backslash E_{(I_1)_k}).
\end{eqnarray*}
Next, we point out that for each $k\in\mathbb{N}$,
\begin{eqnarray*}
\mu_2(E_{I_1}\backslash E_{(I_1)_k})&\leq& \mu_2(E_{I_1}\backslash
E_{(I_1)_1})+\cdots+\mu_2(E_{(I_1)_{k-1}}\backslash
E_{(I_1)_k})\\[4pt]
&\leq& C\sum_{\widetilde{I}:\ dyadic,\ I_1\subseteq
\widetilde{I}\varsubsetneqq (I_1)_k,\
\widetilde{I}\times(E_{\widetilde{I}}\backslash
E_{(\widetilde{I})_1})\subset\Omega}
\mu_2(E_{\widetilde{I}}\backslash
E_{(\widetilde{I})_1}),
\end{eqnarray*}
where the last inequality follows from the definition of $(I_1)_k$.
As a consequence,
\begin{eqnarray*}
&&\sum_{R=I_1\times I_2\in
\mathcal{M}_2(\Omega)}\mu(R)w\Big({\mu_1(I_1)\over\mu_1(\widehat{I}_1)}\Big)\\[4pt]
&&\hskip.5cm\leq C\sum_{I_1\in
A(\Omega)}\mu_1(I_1)\sum_{k=1}^\infty
w(C2^{-\kappa_1k})\sum_{\widetilde{I}:\ dyadic,\ I_1\subseteq
\widetilde{I}\varsubsetneqq (I_1)_k,\
\widetilde{I}\times(E_{\widetilde{I}}\backslash
E_{(\widetilde{I})_1})\subset\Omega}
\mu_2(E_{\widetilde{I}}\backslash
E_{(\widetilde{I})_1}).
\end{eqnarray*}
Now interchanging the order of the sums we can obtain that the above inequality is bounded by
\begin{eqnarray*}
&& C\sum_{k=1}^\infty
w(C2^{-\kappa_1k})\sum_{\widetilde{I}:\ dyadic,\
\widetilde{I}\times(E_{\widetilde{I}}\backslash
E_{(\widetilde{I})_1})\subset\Omega}\mu_1(\widetilde{I})
\mu_2(E_{\widetilde{I}}\backslash E_{(\widetilde{I})_1}) \sum_{I_1:\
dyadic,\ I_1\subseteq
\widetilde{I}\varsubsetneqq (I_1)_k,}{\mu_1(I_1)\over \mu_1(\widetilde{I})}\\[5pt]
&&\leq C\sum_{k=1}^\infty
w(C2^{-\kappa_1k})\sum_{\widetilde{I}:\ dyadic,\
\widetilde{I}\times(E_{\widetilde{I}}\backslash
E_{(\widetilde{I})_1})\subset\Omega}\mu_1(\widetilde{I})
\mu_2(E_{\widetilde{I}}\backslash E_{(\widetilde{I})_1})
\sum_{j=1}^k\sum_{I_1:\ dyadic,\ I_1\subseteq
\widetilde{I}\varsubsetneqq (I_1)_j,}{\mu_1(I_1)\over
\mu_1(\widetilde{I})}.
\end{eqnarray*}
Note that in the last inequality above, we have ${\mu_1(I_1)\over
\mu_1(\widetilde{I})}\leq 2^{-j\kappa_1}$. Hence
\begin{eqnarray*}
&&\sum_{R=I_1\times I_2\in
\mathcal{M}_2(\Omega)}\mu(R)w\Big({\mu_1(I_1)\over\mu_1(\widehat{I}_1)}\Big)\\[4pt]
&&\hskip.5cm\leq C\sum_{k=1}^\infty
kw(C2^{-\kappa_1k})\sum_{\widetilde{I}:\ dyadic,\
\widetilde{I}\times(E_{\widetilde{I}}\backslash
E_{(\widetilde{I})_1})\subset\Omega}\mu_1(\widetilde{I})
\mu_2(E_{\widetilde{I}}\backslash E_{(\widetilde{I})_1})\\[4pt]
&&\hskip.5cm\leq C \sum_{k=1}^\infty
kw(C2^{-\kappa_1k})\mu(\Omega)\\[4pt]
&&\hskip.5cm\leq C \mu(\Omega),
\end{eqnarray*}
since $\widetilde{I}\times(E_{\widetilde{I}}\backslash
E_{(\widetilde{I})_1})$ are contained in $\{\widetilde{I}\ dyadic,\
\widetilde{I}\times(E_{\widetilde{I}}\backslash
E_{(\widetilde{I})_1})\subset\Omega\}$ and are disjoint.
\end{proof}

 The proof of Lemma \ref{theorem-cover lemma} is concluded. This covering lemma will be a key tool to obtain an atomic decomposition for $H^p(\widetilde{M}),$ which will be given in next subsection.

\subsubsection{Atomic decomposition } In this subsection, we will apply Journ\'{e}-type covering lemma to provide an atomic decomposition for $H^p(\widetilde{M}).$ We point out that the atomic decomposition provided in this subsection is different from the classical ones. More precisely, we will prove an atomic decomposition for $L^q(\widetilde{M})\cap H^p(\widetilde{M}), 1<q<\infty,$ where the decomposition converges in both $L^q(\widetilde{M})$ and $H^p(\widetilde{M})$ norms. The convergence in both $L^q(\widetilde{M})$ and $H^p(\widetilde M)$ norms will be crucial for proving the boundedness for operators from $H^p(\widetilde{M})$ to $L^p(\widetilde{M}).$

Suppose that $\max\big(\frac{ {Q}_1}{ {Q}_1+ \vartheta_1 },\frac{ {Q}_2}{
{Q}_2+ \vartheta_2 }\big)<p\leq1$ and $1<q<\infty$. We first define an $(p,q)$-atom for  the Hardy space $H^p(\widetilde{M})$ as follows.

\begin{definition}\label{def-of-p q atom}
A function $a(x_1,x_2)$ defined on $\widetilde{M}$ is called an $(p,q)$-atom of $H^p(\widetilde{M})$ if $a(x_1,x_2)$ satisfies:
\begin{itemize}
\item[(1)] supp $a\subset\Omega$, where $\Omega$ is an open set of $\widetilde{M}$ with finite measure;
\item[(2)] $\|a\|_{L^q}\leq \mu(\Omega)^{1/q-1/p}$;
\item[(3)]  $a$ can be further decomposed into rectangle $(p,q)$-atoms $a_R$ associated to dyadic rectangle $R=I_1\times I_2$, satisfying the following

\smallskip
(i) there exist two constants $C_1$ and $C_2$ such that supp $a_R\subset C_1I_1\times C_2I_2$;

\smallskip
(ii) $\int_{M_1}a_R(x_1,x_2)dx_1=0$ for a.e. $x_2\in M_2$ and $\int_{M_2}a_R(x_1,x_2)dx_2=0$ for a.e.

\hskip.8cm $x_1\in M_1$;

\smallskip
(iii-a) for $2\leq q<\infty$, $a=\sum\limits_{R\in \mathcal{M}(\Omega)}a_R$ and $ \Big(\sum\limits_{R\in\mathcal{M}(\Omega)}\|a_R\|_{L^q}^q\Big)^{1/q} \leq \mu(\Omega)^{1/q-1/p}$.

\smallskip
(iii-b) for $1< q < 2$, $a=\sum\limits_{R\in \mathcal{M}_1(\Omega)}a_R+\sum\limits_{R\in \mathcal{M}_2(\Omega)}a_R$ and for some $\delta>0$, there exists a
constant $C_{q,\delta}$ such that
$$
    \bigg(\sum_{R\in\mathcal{M}_1(\Omega)} \Big({\mu_2(I_2)\over\mu_2(\widehat{I}_2)}\Big)^\delta \|a_R\|_{L^q}^q
   + \sum_{R\in\mathcal{M}_2(\Omega)} \Big({\mu_1(I_1)\over\mu_1(\widehat{I}_1)}\Big)^\delta  \|a_R\|_{L^q}^q
    \bigg)^{1/q} \leq C_{q,\delta}\mu(\Omega)^{1/q-1/p}.
$$
\end{itemize}
\end{definition}

We remark that when $\widetilde{M}=\mathbb{R}^n\times \mathbb{R}^m$ an $(p,2)$-atom with the conditions (i), (ii) and (iii-a)($q=2$) was introduced by R. Fefferman [F]. Note that the condition in (iii-b) is new, which was appeared in the classical case if the $(p,q)$-atom is defined. See [HLZ] for more details.

The main result in this subsection is the following

\begin{theorem}\label{theorem-Hp atom decomp}
Suppose that $\max\big(\frac{ {Q}_1}{ {Q}_1+ \vartheta_1 },\frac{ {Q}_2}{
{Q}_2+ \vartheta_2 }\big)<p\leq 1<q<\infty$. Then $f\in L^q(\widetilde{M})\cap H^p(\widetilde{M})$ if
and only if $f$ has an atomic decomposition, that is,
\begin{eqnarray}\label{atom decom}
f=\sum_{i=-\infty}^\infty\lambda_ia_i,
\end{eqnarray}
where $a_i$ are $(p, q)$ atoms, $\sum_i|\lambda_i|^p<\infty,$ and  the series converges in both $H^p(\widetilde{M})$ and $L^q(\widetilde{M})$. Moreover,
\begin{eqnarray*}
\|f\|_{H^p(\widetilde{M})}\approx \inf \big\{ \lbrace\sum_i|\lambda_i |^p \rbrace^{\frac{1}{p}}, f=\sum_{i}\lambda_ia_i\big\},
\end{eqnarray*}
where the infimum is taken over all decompositions as above and the implicit constants are independent of the $L^q(\widetilde{M})$ and $H^p(\widetilde{M})$ norms of $f.$
\end{theorem}

\begin{proof}[Proof of Theorem \ref{theorem-Hp atom decomp}]
Let $f\in L^q(\widetilde{M})\cap H^p(\widetilde{M})$. We prove that $f$ has an atomic
decomposition. The key tool to do this is the following discrete Calder\'on's identity in [HLL2, Theorem 2.9].
\begin{eqnarray}\label{identity for atom}
f(x_1,x_2)&=&\sum_{k_1=-\infty}^\infty\sum_{k_2=-\infty}^\infty\sum_{I_1}\sum_{I_2 }\mu_1(I_{1})\mu_2(I_{2})\\
&&\times D_{k_1}(x_1,x_{I_1})D_{k_2}(x_2,x_{I_2}) \widetilde{\widetilde{D}}_{k_1}\widetilde{\widetilde{D}}_{k_2}(f)(x_{I_1},x_{I_2})\nonumber
\end{eqnarray}
where the series converges in the norm of $L^q(\widetilde{M})$, $1<q<\infty$ and $H^p(\widetilde M).$ See [HLL2] for more details.

Note that as a function of $x_1$, $D_{k_1}(x_1,x_{I_1})$ is supported
in $\{x_1: d_1(x_1,x_{I_1})\leq C2^{-k_1+N_1}\}$ and similarly
for $D_{k_2}(x_2,x_{I_2})$.
For each $k\in\mathbb{Z}$, let
$$\Omega_k=\{ (x_1,x_2)\in M_1\times M_2:
{{\widetilde {\widetilde S}}_d}(f)(x_1,x_2)>2^k \},$$
where ${{\widetilde {\widetilde S}}_d}(f)$ is similar to ${{\widetilde S}_d}(f)$ but with $D_{k_1}D_{k_2}$ replaced by $\widetilde{{\widetilde D}}_{k_1}\widetilde{{\widetilde D}}_{k_2}.$ More precisely,
$${{\widetilde {\widetilde S}}_d}(f)(x_1,x_2)=\Big\{\sum_{k_1=-\infty}^\infty\sum_{k_2=-\infty}^\infty\sum_{I_1}\sum_{I_2}
|{{\widetilde{\widetilde D}}}_{k_1}{{\widetilde{\widetilde
D}}}_{k_2}(f)(x_1,x_2)|^2\chi_{I_1}(x_1)\chi_{I_2}(x_2)
\Big\}^{1/2}.$$ By the Plancherel--P\^olya inequality in [HLL2], it
follows that
$$\|{{\widetilde S}_d}(f)\|_p\approx \|{{\widetilde {\widetilde S}}_d}(f)\|_p$$
for $\max\big(\frac{ {Q}_1}{ {Q}_1+ \vartheta_1 },\frac{ {Q}_2}{
{Q}_2+ \vartheta_2 }\big)<p<\infty.$ Therefore,
$$\|f\|_{H^p(\widetilde M)}\approx \|{{\widetilde {\widetilde S}}_d}(f)\|_p.$$
Set
$$ \widetilde{\Omega}_k =\{ (x_1,x_2)\in M_1\times M_2:
\mathcal{M}_s(\chi_{\Omega_k})(x_1,x_2)>\widetilde{C} \},$$
where $\mathcal{M}_s$ is the strong maximal function on $\widetilde M$ and $\widetilde{C}$ is a constant to be decided later.
Let
$$B_k=\big\{ R=I_1\times I_2: \mu(\Omega_k\cap R)>{1\over2}\mu(R),\ {\rm
and}\ \mu(\Omega_{k+1}\cap R)\leq {1\over2}\mu(R)\big\}.$$

Rewrite (\ref{identity for atom}) as
\begin{eqnarray*}
f(x_1,x_2)
&=&\sum_{k=-\infty}^\infty\sum_{R=I_1\times
I_2\in B_k}
\mu_1(I_1)\mu_2(I_2) D_{k_1}(x_1,x_{I_1})D_{k_2}(x_2,x_{I_2}) \widetilde{\widetilde{D}}_{k_1}\widetilde{\widetilde{D}}_{k_2}(f)(x_{I_1},x_{I_2})\\
&=&\sum_{k=-\infty}^\infty \lambda_k a_k(x_1,x_2),
\end{eqnarray*}
where
\begin{eqnarray}\label{atom ak}
a_k(x_1,x_2)={1\over
\lambda_k }\sum_{R=I_1\times
I_2\in B_k}
\mu_1(I_1)\mu_2(I_2) D_{k_1}(x_1,x_{I_1})D_{k_2}(x_2,x_{I_2}) \widetilde{\widetilde{D}}_{k_1}\widetilde{\widetilde{D}}_{k_2}(f)(x_{I_1},x_{I_2})\ \ \ \
\end{eqnarray}
and
\begin{eqnarray}\label{atom lambda k q big}
    \lambda_k=C\bigg\|\bigg\{ \sum_{R=I_1\times
I_2\in B_k} \Big|\widetilde{\widetilde{D}}_{k_1}\widetilde{\widetilde{D}}_{k_2}(f)(x_{I_1},x_{I_2})\chi_R(\cdot,\cdot)\Big|^2   \bigg\}^{1/2}\bigg\|_{q}\ \big|\widetilde{\Omega}_k\big|^{1/p-1/q}\hskip.7cm
\end{eqnarray}
when $2\leq q<\infty$, and for $1< q<2 $,
\begin{eqnarray}\label{atom lambda k q small}
    \lambda_k=C\bigg\|\bigg\{ \sum_{R=I_1\times
I_2\in B_k} \Big|\widetilde{\widetilde{D}}_{k_1}\widetilde{\widetilde{D}}_{k_2}(f)(x_{I_1},x_{I_2})\chi_R(\cdot,\cdot)\Big|^2   \bigg\}^{1/2}\bigg\|_{2}\ \big|\widetilde{\Omega}_k\big|^{1/p-1/2}.\hskip.7cm
\end{eqnarray}
To see that the atomic decomposition $\sum_{k=-\infty}^\infty \lambda_k a_k(x_1,x_2)$ converges to $f$ in the $L^q$ norm, we only need to show that
$\|\sum_{|k|>\ell}\lambda_k a_k(x_1,x_2)\|_q\rightarrow 0$ as $\ell\rightarrow \infty.$ This follows from the following duality argument: Let $h\in L^{q'}$ with $\|h\|_{q'}=1,$ then
$$\big\|\sum_{|k|>\ell}\lambda_k a_k(x_1,x_2)\big\|_q=\sup_{\|h\|_{q'}=1}\big|\langle\sum_{|k|>\ell}\lambda_k a_k(x_1,x_2), h\rangle\big|.$$
Note that
\begin{eqnarray*}
\big\langle\sum_{|k|>\ell}\lambda_k a_k(x_1,x_2), h\big\rangle&=&\sum_{|k|>\ell}\sum_{R=I_1\times
I_2\in B_k}
\mu_1(I_1)\mu_2(I_2)
D^*_{k_1}D^*_{k_2}(h)(x_{I_1},x_{I_2}) \widetilde{\widetilde{D}}_{k_1}\widetilde{\widetilde{D}}_{k_2}(f)(x_{I_1},x_{I_2})\\[4pt]
&=&\int\sum_{|k|>\ell}\sum_{R=I_1\times
I_2\in
B_k}D^*_{k_1}D^*_{k_2}(h)(x_{I_1},x_{I_2})\\[4pt]
&&\hskip.3cm\times\widetilde{\widetilde{D}}_{k_1}\widetilde{\widetilde{D}}_{k_2}(f)(x_{I_1},x_{I_2})\chi_R(x_1,x_2)d\mu(x_1)d\mu(x_2).
\end{eqnarray*}
Applying H\"older's inequality gives
\begin{eqnarray*}
\big|\big\langle\sum_{|k|>\ell}\lambda_k a_k(x_1,x_2), h\big\rangle\big|
&\leq &\Big\| \Big\{\sum_{|k|>\ell}\ \sum_{R=I_1\times
I_2\in B_k}
 \big|D^*_{k_1}D^*_{k_2}(h)(x_{I_1},x_{I_2})\big|^2\chi_R(\cdot,\cdot)\Big\}^{1/2} \Big\|_{q'}\\[4pt]
 &&\hskip.5cm\times \Big\| \Big\{\sum_{|k|>\ell}\ \sum_{R=I_1\times
I_2\in B_k}
 \big|\widetilde{\widetilde{D}}_{k_1}\widetilde{\widetilde{D}}_{k_2}(f)(x_{I_1},x_{I_2})\big|^2\chi_R(\cdot,\cdot)\Big\}^{1/2} \Big\|_{q}.
\end{eqnarray*}
Note again that
$$\Big\| \Big\{\sum_{|k|>\ell}\sum_{R=I_1\times
I_2\in B_k}
 \big|D^*_{k_1}D^*_{k_2}(h)(x_{I_1},x_{I_2})\big|^2\chi_R(\cdot,\cdot)\Big\}^{1/2} \Big\|_{q'}
 \leq C\|h\|_{q'}$$
and
\begin{eqnarray*}
\Big\| \Big\{\sum_{|k|>\ell}\sum_{R=I_1\times
I_2\in B_k}
 \big|\widetilde{\widetilde{D}}_{k_1}\widetilde{\widetilde{D}}_{k_2}(f)(x_{I_1},x_{I_2})\big|^2\chi_R(\cdot,\cdot)\Big\}^{1/2} \Big\|_{q}
\end{eqnarray*}
tends to zero as $\ell$ tends to infinity. This implies that $\|\sum_{|k|>\ell}\lambda_k a_k(x_1,x_2)\|_q\rightarrow 0$ as $\ell\rightarrow \infty$
and hence, the atomic decomposition $\sum_{k=-\infty}^\infty \lambda_k a_k(x_1,x_2)$ converges to $f$ in the $L^q$ norm.

To see that $a_k$ has the compact support, by choosing $\widetilde{C}$ sufficiently small, we can conclude
that $ {\rm supp}a_k\subset \widetilde{\Omega}_k $ since
$D_{k_1}(x_1,x_{I_1})$ and $D_{k_2}(x_2,x_{I_2}),$ as
functions of $x_1$ and $x_2,$ respectively, have compact supports with diameters being equivalent to $2^{-k_1}$ and $2^{-k_2},$ respectively.
This implies that $a_k$ satisfies the condition (1) of Definition \ref{def-of-p q atom}.

We now verify that $a_k$ satisfies (2) of Definition \ref{def-of-p q atom}. To this end, let
$h\in L^{q'}(\widetilde{M})$ with
$\|h\|_{L^{q'}}=1$, where $q'$ is the conjugate index of $q$. By the duality argument,
\begin{eqnarray*}
&&\Big\|   \sum_{R=I_1\times
I_2\in B_k}
\mu(R) D_{k_1}(\cdot,x_{I_1})D_{k_2}(\cdot,x_{I_2}) \widetilde{\widetilde{D}}_{k_1}\widetilde{\widetilde{D}}_{k_2}(f)(x_{I_1},x_{I_2})  \Big\|_q\\
&=& \sup_{\|h\|_{L^{q'}}=1}\Big|\Big\langle  \sum_{R=I_1\times
I_2\in B_k}
\mu(R) D_{k_1}(\cdot,x_{I_1})D_{k_2}(\cdot,x_{I_2}) \widetilde{\widetilde{D}}_{k_1}\widetilde{\widetilde{D}}_{k_2}(f)(x_{I_1},x_{I_2}),h   \Big\rangle\Big|.
\end{eqnarray*}
Applying H\"older's inequality and the discrete Littlewood--Paley square function estimates on $L^q$ for $1<q<\infty$, the last term above is dominated by
\begin{eqnarray*}
&& \sup_{\|h\|_{L^{q'}}=1}
\Big\| \Big\{\sum_{R=I_1\times
I_2\in B_k}
 \big|D_{k_1}D_{k_2}(h)(x_{I_1},x_{I_2})\big|^2\chi_R(\cdot,\cdot)\Big\}^{1/2} \Big\|_{q'}\\[4pt]
 &&\hskip1cm\times \Big\| \Big\{\sum_{R=I_1\times
I_2\in B_k}
 \big|\widetilde{\widetilde{D}}_{k_1}\widetilde{\widetilde{D}}_{k_2}(f)(x_{I_1},x_{I_2})\big|^2\chi_R(\cdot,\cdot)\Big\}^{1/2} \Big\|_{q}\\[4pt]
&&\leq C\Big\| \Big\{\sum_{R=I_1\times
I_2\in B_k}
 \big|\widetilde{\widetilde{D}}_{k_1}\widetilde{\widetilde{D}}_{k_2}(f)(x_{I_1},x_{I_2})\big|^2\chi_R(\cdot,\cdot)\Big\}^{1/2} \Big\|_{q}.
\end{eqnarray*}
This yields that when $2\leq q<\infty$,
\begin{eqnarray*}
\|a_k\|_{q}&=& \bigg(C\bigg\|\bigg\{ \sum_{R=I_1\times
I_2\in B_k} \Big|\widetilde{\widetilde{D}}_{k_1}\widetilde{\widetilde{D}}_{k_2}(f)(x_{I_1},x_{I_2})\chi_R(\cdot,\cdot)\Big|^2   \bigg\}^{1/2}\bigg\|_{q}\ \mu\big(\widetilde{\Omega}_k\big)^{1/p-1/q}\bigg)^{-1}\\[4pt]
&&\times \Big\|   \sum_{R=I_1\times
I_2\in B_k}
\mu(R) D_{k_1}(\cdot,x_{I_1})D_{k_2}(\cdot,x_{I_2}) \widetilde{\widetilde{D}}_{k_1}\widetilde{\widetilde{D}}_{k_2}(f)(x_{I_1},x_{I_2})  \Big\|_q\\[4pt]
&\leq& \mu\big(\widetilde{\Omega}_k\big)^{1/q-1/p}.
\end{eqnarray*}
For $1<q<2$, since $a_k$ is supported in $\widetilde{\Omega}_k$, applying H\"older's inequality yields
\begin{eqnarray*}
\|a_k\|_{q}
&=& \bigg(C\bigg\|\bigg\{ \sum_{R=I_1\times
I_2\in B_k} \Big|\widetilde{\widetilde{D}}_{k_1}\widetilde{\widetilde{D}}_{k_2}(f)(x_{I_1},x_{I_2})\chi_R(\cdot,\cdot)\Big|^2   \bigg\}^{1/2}\bigg\|_{2}\ \mu\big(\widetilde{\Omega}_k\big)^{1/p-1/2}\bigg)^{-1}\\[4pt]
&&\times \Big\|   \sum_{R=I_1\times
I_2\in B_k}
\mu(R) D_{k_1}(\cdot,x_{I_1})D_{k_2}(\cdot,x_{I_2}) \widetilde{\widetilde{D}}_{k_1}\widetilde{\widetilde{D}}_{k_2}(f)(x_{I_1},x_{I_2})  \Big\|_q\\[4pt]
&\leq& \bigg(C\bigg\|\bigg\{ \sum_{R=I_1\times
I_2\in B_k} \Big|\widetilde{\widetilde{D}}_{k_1}\widetilde{\widetilde{D}}_{k_2}(f)(x_{I_1},x_{I_2})\chi_R(\cdot,\cdot)\Big|^2   \bigg\}^{1/2}\bigg\|_{2}\ \mu\big(\widetilde{\Omega}_k\big)^{1/p-1/2}\bigg)^{-1}\\[4pt]
&&\times \mu\big(\widetilde{\Omega}_k\big)^{1/q-1/2} \Big\|   \sum_{R=I_1\times
I_2\in B_k}
\mu(R) D_{k_1}(\cdot,x_{I_1})D_{k_2}(\cdot,x_{I_2}) \widetilde{\widetilde{D}}_{k_1}\widetilde{\widetilde{D}}_{k_2}(f)(x_{I_1},x_{I_2})  \Big\|_2\\[4pt]
&\leq& \mu\big(\widetilde{\Omega}_k\big)^{1/q-1/p},
\end{eqnarray*}
where we use the fact that
$$\Big\|   \sum_{R=I_1\times
I_2\in B_k}
\mu(R) D_{k_1}(\cdot,x_{I_1})D_{k_2}(\cdot,x_{I_2}) \widetilde{\widetilde{D}}_{k_1}\widetilde{\widetilde{D}}_{k_2}(f)(x_{I_1},x_{I_2})  \Big\|_2$$
$$\leq C\bigg\|\bigg\{ \sum_{R=I_1\times
I_2\in B_k} \Big|\widetilde{\widetilde{D}}_{k_1}\widetilde{\widetilde{D}}_{k_2}(f)(x_{I_1},x_{I_2})\chi_R(\cdot,\cdot)\Big|^2   \bigg\}^{1/2}\bigg\|_{2}.$$
As a consequence, we get that $a_k$ satisfies (2) of Definition \ref{def-of-p q atom}.
It remains to check that $a_k$ satisfies the condition (3) of Definition \ref{def-of-p q atom}. To see this, we can further decompose $a_k$
as
$$
   a_k=\sum_{\overline{R}\in \mathcal{M}(\widetilde{\Omega}_k)} a_{k, \overline{R}},
$$
where
\begin{eqnarray*}
a_{k, \overline{R}}(x_1,x_2)&=&{1\over
\lambda_k }\ \ \sum_{R=I_1\times
I_2\in B_k,\ \ R\subset \overline{R}}
\mu_1(I_1)\mu_2(I_2)\\[5pt]
&&\times D_{k_1}(x_1,x_{I_1})D_{k_2}(x_2,x_{I_2}) \widetilde{\widetilde{D}}_{k_1}\widetilde{\widetilde{D}}_{k_2}(f)(x_{I_1},x_{I_2}).
\end{eqnarray*}

Similar to $a_k$, we can verify that
$$
    {\rm supp}a_{k,\overline{R}} \subset C\overline{R}
$$
and by the facts that $\int D_{k_1}(x_1,x_{I_1})dx_1=\int D_{k_2}(x_2,x_{I_2})dx_2=0,$ for a.e. $x_2\in M_2$,
$$ \int_{M_1} a_{k,\overline{R}}(x_1,x_2)dx_1=0 $$
and for a.e. $x_1\in M_1$,
$$ \int_{M_2} a_{k,\overline{R}}(x_1,x_2)dx_2=0, $$
which yield that the conditions (i) and (ii) of (3) in Definition \ref{def-of-p q atom} hold. Now it's left to show that $a_k$ satisfies the conditions (iii-a) and (iii-b) of (3).

For $2\leq q<\infty$, we verify  that  $a_k$ satisfies (iii-a). To do this, by the definition of $\lambda_k$, we have
\begin{eqnarray*}
\|a_{k,\overline{R}}\|_{q}&=& \bigg(C\bigg\|\bigg\{ \sum_{R=I_1\times
I_2\in B_k} \Big|\widetilde{\widetilde{D}}_{k_1}\widetilde{\widetilde{D}}_{k_2}(f)(x_{I_1},x_{I_2})\chi_R(\cdot,\cdot)\Big|^2   \bigg\}^{1/2}\bigg\|_{q}\ \mu\big(\widetilde{\Omega}_k\big)^{1/p-1/q}\bigg)^{-1}\\
&&\times \Big\|   \sum_{R=I_1\times
I_2\in B_k,\ R\subset \overline{R} }
\mu(R) D_{k_1}(\cdot,x_{I_1})D_{k_2}(\cdot,x_{I_2}) \widetilde{\widetilde{D}}_{k_1}\widetilde{\widetilde{D}}_{k_2}(f)(x_{I_1},x_{I_2})  \Big\|_q.
\end{eqnarray*}
Applying the same argument for the estimates of $\|a_k\|_{q}$ with $2\leq q<\infty $ yields
$$
\Big\{ \sum_{\overline{R}\in \mathcal{M}(\widetilde{\Omega}_k)}\big\|a_{k,\overline{R}}\big\|^q_{L^q} \Big\}^{1/q}
\leq \mu(\widetilde{\Omega}_k)^{1/q-1/p},
$$
which concludes that the condition (iii-a) holds.

For $1<q<2$, we first write
$$\sum_{\overline{R}=I_1\times I_2\in \mathcal{M}_1(\widetilde{\Omega}_k)} \Big({\mu_2(I_2)\over\mu_2(\widehat{I}_2)}\Big)^\delta \big\|a_{k,\overline{R}}\big\|^q_{L^q}
\leq {C\over \lambda_k^q}\sum_{\overline{R}=I_1\times I_2\in
\mathcal{M}_1(\widetilde{\Omega}_k)}\Big({\mu_2(I_2)\over\mu_2(\widehat{I}_2)}\Big)^\delta
\hskip3cm$$
$$\times\Big\| \Big\{\sum_{R=I_1\times
I_2\in B_k,\ R\subset \overline{R}}
 \big|\widetilde{\widetilde{D}}_{k_1}\widetilde{\widetilde{D}}_{k_2}(f)(x_{I_1},x_{I_2})\big|^2\chi_R(\cdot,\cdot)\Big\}^{1/2} \Big\|_{q}^q.$$
Applying H\"older's inequality and the definition of $\lambda_k$, the last term above then is less or equal to
\begin{eqnarray*}
&&{C\over \lambda_k^q}\sum_{\overline{R}=I_1\times I_2\in
\mathcal{M}_1(\widetilde{\Omega}_k)}\Big({\mu_2(I_2)\over\mu_2(\widehat{I}_2)}\Big)^\delta
\mu(R)^{1-q/2}\\[5pt]
&&\hskip.5cm\times \Big\{\int \sum_{R=I_1\times
I_2\in B_k,\ R\subset \overline{R}}
 \big|\widetilde{\widetilde{D}}_{k_1}\widetilde{\widetilde{D}}_{k_2}(f)(x_{I_1},x_{I_2})\big|^2\chi_R(x_1,x_2) dx_1dx_2 \Big\}^{q/2}\\[5pt]
&&\leq {C\over \lambda_k^q}  \Big\{ \sum_{\overline{R}=I_1\times
I_2\in
\mathcal{M}_1(\widetilde{\Omega}_k)}\Big({\mu_2(I_2)\over\mu_2(\widehat{I}_2)}\Big)^{\delta'}
\mu(R)\Big\}^{1-q/2} \\[5pt]
&&\hskip.5cm\times \Big\{\int
\sum_{R=I_1\times I_2\in B_k,\
R\subset \overline{R}}
 \big|\widetilde{\widetilde{D}}_{k_1}\widetilde{\widetilde{D}}_{k_2}(f)(x_{I_1},x_{I_2})\big|^2\chi_R(x_1,x_2) dx_1dx_2
 \Big\}^{q/2}\\[5pt]
&&\leq C_{q,\delta} \mu(\widetilde{\Omega}_k)^{1-q/2}
\mu(\widetilde{\Omega}_k)^{q/2-q/p}\\
&& =
C_{q,\delta}\mu(\widetilde{\Omega}_k)^{1-q/p},
\end{eqnarray*}
where the last inequality follows from Journ\'e-type covering lemma with $\delta'=\frac{2\delta}{2-q}.$

Similarly,
 \begin{eqnarray*}
\sum_{\overline{R}=I_1\times I_2\in \mathcal{M}_2(\widetilde{\Omega}_k)} \Big({\mu_1(I_1)\over\mu_1(\widehat{I}_1)}\Big)^\delta \big\|a_{k,\overline{R}}\big\|^q_{L^q}\leq C_{q,\delta}\mu(\widetilde{\Omega}_k)^{1-q/p}.
\end{eqnarray*}
 This implies that the condition (iii-b) holds and hence, we obtain
a desired atomic decomposition for $f$.


To prove the converse, it suffices to verify that
there exists a positive constant $C$ such that
\begin{eqnarray}
\|\widetilde{S}(a)\|_{L^p(\widetilde{M})}\leq C
\end{eqnarray}
for each $(p,q)$-atom $a$ of $H^p(\widetilde{M})$ with $1<q<\infty$.
This is because if $f$ has an atomic decomposition $f=\sum_i \lambda_i a_i,$ where the series converges in both $L^q$ and $H^p(\widetilde M)$ norms, then
$$\|{\widetilde S}(f)\|_p^p\leq \sum_i |\lambda_i|^p \|{\widetilde S}(a_i)\|_p^p,$$
where the fact that the series in the atomic decomposition of $f$ converges in the norm of $L^q$ is used. This together with (3.13) gives
\begin{eqnarray*}
\|f\|^p_{H^p}\leq C\|{\widetilde S}(f)\|_p^p\leq C\sum_i |\lambda_i|^p \|{\widetilde S}(a_i)\|_p^p\leq C\sum_i |\lambda_i|^p<\infty.
\end{eqnarray*}
Finally, it remains to show (3.13). Fix an $(p,q)$-atom $a$ with ${\rm supp}a\subset \Omega$
and $a=\sum_{R\in\mathcal{M}(\Omega)}a_R$. Set
\begin{eqnarray*}
\widetilde{\Omega}=\{(x_1,x_2)\in \widetilde{M}:\ \mathcal{M}_s(\chi_\Omega)(x_1,x_2)>1/2 \}
\end{eqnarray*}
and
\begin{eqnarray*}
\widetilde{\widetilde{\Omega}}=\{(x_1,x_2)\in \widetilde{M}:\ \mathcal{M}_s(\chi_{\widetilde{\Omega}})(x_1,x_2)>1/2 \}.
\end{eqnarray*}
Moreover, for any $R=I_1\times I_2 \in\mathcal{M}_1(\Omega)$, set
$\widehat{R}=\widehat{I}_1\times I_2 \subset
\mathcal{M}_1(\widetilde{\Omega}).$ Then $ \mu(\widehat{R}\cap\Omega)>{\mu(\widehat{R})\over2}.$ Similarly,
set $\widehat{\widehat{R}}=\widehat{I}_1\times \widehat{I}_2 \subset
\mathcal{M}_2(\widetilde{\widetilde{\Omega}}).$ Then $ \mu(\widehat{\widehat{R}}\cap\widetilde{\Omega})>{\mu(\widehat{\widehat{R}})\over2}.$

Now let $\overline{C}$ be a constant to be chosen later. We write
\begin{eqnarray*}
&&\|\widetilde{S}(a)\|_{L^p(\widetilde{M})}^p\\
&&=
\int_{\cup_{R\in\mathcal{M}(\Omega)}100\overline{C}\widehat{\widehat{R}}}
\widetilde{S}(a)(x_1,x_2)^pdx_1dx_2+
\int_{(\cup_{R\in\mathcal{M}(\Omega)}100\overline{C}\widehat{\widehat{R}})^c}
\widetilde{S}(a)(x_1,x_2)^pdx_1dx_2\\
&&=: A+B.
\end{eqnarray*}

For $A$, applying the H\"older inequality and Theorem \ref{theorem
Lp bd of suqa func} and using the $L^q$ boundedness of ${\widetilde S}$, we have
\begin{eqnarray*}
A&\leq&
\mu\big(\cup_{R\in\mathcal{M}(\Omega)}100\overline{C}\widehat{\widehat{R}}\big)^{1-p/q}
\left(\int_{\widetilde{M}}
|\widetilde{S}(a)(x_1,x_2)|^q dx_1dx_2\right)^{p/q}\\
&\leq& C\mu(\Omega)^{1-p/q}\|a\|_{L^q(M)}^{p}\\
&\leq& C.
\end{eqnarray*}

To estimate $B$, we write
\begin{eqnarray*}
B&\leq&
\sum_{R\in\mathcal{M}(\Omega)}\int_{(100\overline{C}\widehat{\widehat{R}})^c}
\widetilde{S}(a_R)(x_1,x_2)^p dx_1dx_2\\
&\leq&
\sum_{R\in\mathcal{M}(\Omega)}\int_{x_1\not\in100\overline{C}\widehat{I}_1}\int_{M_2}
\widetilde{S}(a_R)(x_1,x_2)^p dx_1dx_2\\
&&+\sum_{R\in\mathcal{M}(\Omega)}\int_{M_1}\int_{x_2\not\in100\overline{C}\widehat{I}_2}
\widetilde{S}(a_R)(x_1,x_2)^p dx_1dx_2\\
&=:&B_{1}+B_{2}.
\end{eqnarray*}

It suffices to estimate $B_1$ since the estimate for $B_2$ is similar. We further decompose $B_1$ as follows.
\begin{eqnarray*}
B_{1}&=&\sum_{R\in\mathcal{M}(\Omega)}\int_{x_1\not\in100\overline{C}\widehat{I}_1}\int_{x_2\in 100\overline{C}I_2}
\widetilde{S}(a_R)(x_1,x_2)^p dx_1dx_2\\[5pt]
&&+ \sum_{R\in\mathcal{M}(\Omega)}\int_{x_1\not\in100\overline{C}\widehat{I}_1}\int_{x_2\not\in 100\overline{C}I_2}
\widetilde{S}(a_R)(x_1,x_2)^p dx_1dx_2\\[5pt]
&=:&B_{11}+B_{12}.
\end{eqnarray*}

Applying H\"older's inequality for $B_{11}$ implies

$$\int_{x_1\not\in100\overline{C}\widehat{I}_1}\int_{x_2\in 100\overline{C}I_2}
\widetilde{S}(a_R)(x_1,x_2)^p dx_1dx_2\hskip3cm$$
$$\leq C\mu_2(I_2)^{1-p/q}\int_{x_1\not\in100\overline{C}\widehat{I}_1}\left[\int_{M_2}
\widetilde{S}(a_R)(x_1,x_2)^q dx_2\right]^{p/q}  dx_1.$$
To estimate the last term above, write
\begin{eqnarray*}
&&\int_{M_2}
\widetilde{S}(a_R)(x_1,x_2)^q dx_2\\[5pt]
&&=\int_{M_2}
[\sum_{k_1=-\infty}^\infty\sum_{k_2=-\infty}^\infty\big|D_{k_1}D_{k_2}(a_R)(x_1,x_2)\big|^2]^{\frac{q}{2}} dx_2.
\end{eqnarray*}
Consider the Hilbert space $H=\big\{ F_{k_1}(x_1): \|F_{k_1}(x_1)\|_{H}=\lbrace\sum\limits_{k_1}|F_{k_1}(x_1)|^2\rbrace^{\frac{1}{2}}\big\}.$
Then the last term above can be written as
$$\int_{M_2}\big[\sum_{k_2}\|D_{k_2}(D_{k_1}a_R)(x_1,\cdot)(x_2)\|^2_H\big]^{\frac{q}{2}}dx_2.$$
Applying the vector-valued Littlewood--Paley estimate, we have
\begin{eqnarray}\label{e1 S on atom}
\int_{M_2}
\widetilde{S}(a_R)(x_1,x_2)^q dx_2&\leq& C\int_{M_2}\big\{\|(D_{k_1}a_R)(x_1,x_2)\|^2_H\big\}^{\frac{q}{2}}dx_2\nonumber\\[5pt]
&=&C\int_{M_2}\Big[\sum_{k_1=-\infty}^\infty
\big|\int_{M_1}D_{k_1}(x_1,y_1)a_R(y_1,x_2)dy_1\big|^2\Big]^{\frac{q}{2}} dx_2.
\end{eqnarray}
We first consider the term $\int_{M_1}D_{k_1}(x_1,y_1)a_R(y_1,x_2)dy_1$ in (\ref{e1 S on atom}).
Using the cancellation condition of the atom $a_R$ and the smoothness conditions on $D_{k_1}$ yields
\begin{eqnarray*}
&&\Big|\int_{M_1}D_{k_1}(x_1,y_1)a_R(y_1,x_2)dy_1\big |\\[5pt]
&&=\Big|\int_{M_1}[D_{k_1}(x_1,y_1)-D_{k_1}(x_1,z_1)]a_R(y_1,x_2)dy_1\Big|\\[5pt]
&&\leq C2^{k_1\vartheta_1}\ell(I_1)^{\vartheta_1} \Big(\frac{1}{V_{2^{-k_1}}(x_1)+V_{2^{-k_1}}(z_1)+V(x_1,z_1)}\Big) \int_{M_1}|a_R(y_1,x_2)|dy_1,
\end{eqnarray*}
where we use $z_1$ to denote the center of $I_1.$

Putting the above estimate into (\ref{e1 S on atom}) implies
\begin{eqnarray*}
&&\int_{M_2}
\widetilde{S}(a_R)(x_1,x_2)^q dx_2\\[5pt]
&&\leq C\int_{M_2}\Big[\sum_{k_1=-\infty}^\infty
\big(2^{k_1\vartheta_1}\ell(I_1)^{\vartheta_1} \frac{1}{V_{2^{-k_1}}(x_1)+V_{2^{-k_1}}(z_1)+V(x_1,z_1)}\big)^2\Big]^{\frac{q}{2}} dx_2
\mu_1(I_1)^{q-1}\|a_R\|_{L^q(\widetilde{M})}^q.
\end{eqnarray*}
Note that supp$a_R\subset C_1I_1\times C_2I_2.$ So $y_1\in C_1I_1.$ Since $D_{k_1}(x_1,y_1)$ is supported in $\{y_1: d_1(x_1,y_1)<C2^{-k_1}\},$ if $x_1\not\in100\overline{C}\widehat{I}_1,$ then, by choosing $\overline{C}$ large enough, $k_1\leq \widetilde{k}_1,$ where $\widetilde{k}_1$ is chosen such that $2^{-\widetilde{k}_1}\approx 100\overline{C}\ell(\widetilde{I}_1)$. Applying the above estimate gives
\begin{eqnarray*}
&&\int_{x_1\not\in100\overline{C}\widehat{I}_1}\int_{x_2\in 100\overline{C}I_2}
\widetilde{S}(a_R)(x_1,x_2)^p dx_1dx_2\\[5pt]
&&\leq C\mu_2(I_2)^{1-p/q}\int_{x_1\not\in100\overline{C}\widehat{I}_1}\Big[\sum_{k_1=-\infty}^{\widetilde{k}_1}C2^{qk\vartheta_1}\ell(I_1)^{q\vartheta_1} \big(\frac{1}{V_{2^{-k_1}}(x_1)+V_{2^{-k_1}}(z_1)+V(x_1,z_1)}\Big)^q\\[5pt]
&&\hskip1cm\times \Big({2^{-\widehat{k}_1}\over d(x_1,z_1)}\Big)^{q\vartheta_1} \mu_1(I_1)^{q-1}\|a_R\|_{L^q(\widetilde{M})}^q\Big]^{p/q}  dx_1\\[5pt]
&&\leq C\mu_2(I_2)^{1-p/q}\mu_1(I_1)^{p-p/q}\ell(I_1)^{p\vartheta_1}\|a_R\|_{L^q(\widetilde{M})}^p
 \int_{x_1\not\in100\overline{C}\widehat{I}_1} \big(\frac{1}{V(x_1,z_1)d_1(x_1,z_1)^{\vartheta_1}}\Big)^p
  dx_1.
\end{eqnarray*}
By decomposing the set $\{x_1\not\in100\overline{C}\widehat{I}_1\}$ into annuli according to $\ell(\widehat{I}_1)$, we can verify that
\begin{eqnarray*}
 \int_{x_1\not\in100\overline{C}\widehat{I}_1} \big(\frac{1}{V(x_1,z_1)d_1(x_1,z_1)^{\vartheta_1}}\Big)^p
  dx_1\leq C{1\over \ell(\widehat{I}_1)^{p\vartheta_1}}V(z_1,\ell(\widehat{I}_1))^{1-p}.
\end{eqnarray*}
As a consequence, we obtain
\begin{eqnarray}
&&\int_{x_1\not\in100\overline{C}\widehat{I}_1}\int_{x_2\in 100\overline{C}I_2}
\widetilde{S}(a_R)(x_1,x_2)^pdx_1dx_2\nonumber\\[5pt]
&&\leq C\mu_2(I_2)^{1-p/q}\mu_1(I_1)^{p-p/q}\ell(I_1)^{p\vartheta_1}\|a_R\|_{L^q(\widetilde{M})}^p {1\over \ell(\widehat{I}_1)^{p\vartheta_1}}V(z_1,\ell(\widehat{I}_1))^{1-p}\nonumber\\
&&\leq C\mu(R)^{1-p/q}\|a_R\|_{L^q(\widetilde{M})}^p \Big({\ell(I_1)\over \ell(\widehat{I}_1)}\Big)^{p\vartheta_1}\Big({V(z_1,\ell(\widehat{I}_1))\over \mu_1(I_1)}\Big)^{1-p}. \label{estimate B11}
\end{eqnarray}
Next, since
$$ {\mu_1({\widehat{I}_1})\over \mu_1(I_1)}\leq \Big( {\ell({\widehat{I}_1})\over \ell(I_1)} \Big)^{Q_1},  $$
where $Q_1$ is the upper dimension of $M_1$, we have that
$$ {\ell(I_1)\over \ell({\widehat{I}_1})} \leq \Big( {\mu_1(I_1)\over \mu_1({\widehat{I}_1})} \Big)^{1\over Q_1},  $$
which yields that
\begin{eqnarray*}
 \Big({\ell(I_1)\over \ell(\widehat{I}_1)}\Big)^{p\vartheta_1}\Big({V(z_1,\ell(\widehat{I}_1))\over \mu_1(I_1)}\Big)^{1-p}&\leq&
 C\Big({\mu_1(I_1)\over \mu_1(\widehat{I}_1)}\Big)^{{p\vartheta_1\over Q_1}+p-1}
 =:w\Big({\mu_1(I_1)\over \mu_1(\widehat{I}_1)}\Big),
\end{eqnarray*}
where $w(x)=x^\alpha$ with $\alpha={p\vartheta_1\over Q_1}+p-1>0$ since $p>\frac{ {Q}_1}{ {Q}_1+ \vartheta_1 }.$

When $2\leq q<\infty$, applying H\" older's inequality yields
\begin{eqnarray*}
B_{11}&\leq& C\sum_{R\in\mathcal{M}(\Omega)}\|a_R\|_{L^q(\widetilde{M})}^p\mu(R)^{1-p/q} w\Big({\mu_1(I_1)\over \mu_1(\widehat{I}_1)}\Big)\\
&\leq& C\Big( \sum_{R\in\mathcal{M}(\Omega)} \|a_R\|_{L^q(\widetilde{M})}^q\ \Big)^{p/q}\Big( \sum_{R\in\mathcal{M}(\Omega)} \mu(R) \widetilde{w}\Big({\mu_1(I_1)\over \mu_1(\widehat{I}_1)}\Big) \Big)^{1-p/q}\\
&\leq& C\mu(\Omega)^{p/q-1}\mu(\Omega)^{1-p/q}\\
&\leq& C,
\end{eqnarray*}
where the last inequality follows from Journ\'e's covering Lemma 3.3 with $\widetilde{w}=w^{q\over q-p}$.

If $1<q<2$, we have
\begin{eqnarray*}
B_{11}&\leq& C\sum_{R\in\mathcal{M}(\Omega)}\|a_R\|_{L^q(\widetilde{M})}^p\mu(R)^{1-p/q} w\Big({\mu_1(I_1)\over \mu_1(\widehat{I}_1)}\Big)\\
&\leq& C\sum_{R\in\mathcal{M}(\Omega)}\|a_R\|_{L^q(\widetilde{M})}^p \overline{w}\Big({\mu_1(I_1)\over \mu_1(\widehat{I}_1)}\Big) \mu(R)^{1-p/q} \overline{w}\Big({\mu_1(I_1)\over \mu_1(\widehat{I}_1)}\Big).
\end{eqnarray*}
Applying H\"older's inequality implies that the last term above is bounded by
\begin{eqnarray*}
&& C\Big( \sum_{R\in\mathcal{M}(\Omega)} \|a_R\|_{L^q(\widetilde{M})}^q\ \widetilde{\widetilde{w}}\Big({\mu_1(I_1)\over \mu_1(\widehat{I}_1)}\Big)\Big)^{p/q}\Big( \sum_{R\in\mathcal{M}(\Omega)} \mu(R) \widetilde{w}\Big({\mu_1(I_1)\over \mu_1(\widehat{I}_1)}\Big) \Big)^{1-p/q}\\
&&\leq C\mu(\Omega)^{p/q-1}\mu(\Omega)^{1-p/q}
\leq C,
\end{eqnarray*}
where $\overline{w}=w^{1\over 2}$, $\widetilde{w}=\overline{w}^{q\over q-p}$ and $\widetilde{\widetilde{w}}=\overline{w}^{q\over p}$.

Now we consider $B_{12}$. Note that in this case, we have $x_1\not\in100\overline{C}\widehat{I}_1$ and $x_2\not\in 100\overline{C}I_2$. Thus, similar to the arguments in the case of $B_{11}$, by choosing $\overline{C}$ large enough, we have two constants $\widehat{k}_1$ and $\widehat{k}_2$ such that $2^{-\widehat{k}_1}\approx \ell(\widehat{I}_1)$, $2^{-\widehat{k}_2}\approx \ell(I_2)$ and
$k_1\leq \widehat{k}_1$ and $k_2\leq \widehat{k}_2$. Hence, we can rewrite
\begin{eqnarray*}
B_{12}&=&\sum_{R\in\mathcal{M}(\Omega)}\int_{x_1\not\in100\overline{C}\widehat{I}_1}\int_{x_2\not\in 100\overline{C}I_2}\\
&&\left| \sum_{k_1=-\infty}^{\widehat{k}_1}\sum_{k_2=-\infty}^{\widehat{k}_2}\Big|\int_{\widetilde{M}}D_{k_1}(x_1,y_1)D_{k_2}(x_2,y_2)
a_{R}(y_1,y_2)dy_1dy_2\Big|^q \right|^{p/q}dx_1dx_2\\
&=&\sum_{R\in\mathcal{M}(\Omega)}\int_{x_1\not\in100\overline{C}\widehat{I}_1}\int_{x_2\not\in 100\overline{C}I_2}\bigg| \sum_{k_1=-\infty}^{\widehat{k}_1}\sum_{k_2=-\infty}^{\widehat{k}_2}\Big|\int_{\widetilde{M}}[D_{k_1}(x_1,y_1)-D_{k_1}(x_1,z_1)]\\
&&\times [D_{k_2}(x_2,y_2)-D_{k_2}(x_2,z_2)]
a_{R}(y_1,y_2)dy_1dy_2\Big|^q \bigg|^{p/q}dx_1dx_2,
\end{eqnarray*}
where the second equality follows from the cancellation condition of the atoms $a_{R}(y_1,y_2)$.
Then, by applying smoothness properties of $D_{k_1}(x_1,y_1)$ and $D_{k_2}(x_2,y_2)$, we have
\begin{eqnarray*}
B_{12}
&\leq&\sum_{R\in\mathcal{M}(\Omega)}\int_{x_1\not\in100\overline{C}\widehat{I}_1}\int_{x_2\not\in 100\overline{C}I_2}\Bigg[ \sum_{k_1=-\infty}^{\widehat{k}_1}\sum_{k_2=-\infty}^{\widehat{k}_2}\\
&&\times\Big|\int_{\widetilde{M}}2^{qk_1\vartheta_1}\ell(I_1)^{q\vartheta_1} \big(\frac{1}{V_{2^{-k_1}}(x_1)+V_{2^{-k_1}}(z_1)+V(x_1,z_1)}\Big)^q \Big({2^{-\widehat{k}_1}\over d(x_1,z_1)}\Big)^{q\vartheta_1} \\[4pt]
&&\hskip.3cm\times 2^{qk_2\vartheta_2}\ell(I_2)^{q\vartheta_2} \big(\frac{1}{V_{2^{-k_2}}(x_2)+V_{2^{-k_2}}(z_2)+V(x_2,z_2)}\Big)^q \Big({2^{-\widehat{k}_2}\over d(x_2,z_2)}\Big)^{q\vartheta_2}\\
&&\hskip.5cm\times|a_{R}(y_1,y_2)|dy_1dy_2\Big|^q \Bigg]^{p/q}dx_1dx_2\\
&\leq&C\mu(R)^{1-p/q}\|a_R\|_{L^q(\widetilde{M})}^p \Big({\ell(I_1)\over \ell(\widehat{I}_1)}\Big)^{p\vartheta_1}\Big({V(z_1,\ell(\widehat{I}_1))\over \mu_1(I_1)}\Big)^{1-p}.
\end{eqnarray*}

Similar to estimates as those in $B_{11}$, we obtain
\begin{eqnarray*}
B_{12}&\leq& C\mu(\Omega)^{p/q-1}\mu(\Omega)^{1-p/q}\leq C.
\end{eqnarray*}
Combining the estimates of $B_{11}$ and $B_{12}$ yields $B_1\leq C$, which in turn gives $B_2\leq C$. The proof of Theorem 3.5 is concluded.
\end{proof}

\subsubsection{$H^p\rightarrow L^p$ boundedness }

In this subsection, applying the atomic decomposition provided in the previous subsection, we show the following
\begin{theorem}\label{theorem-bd from Hp to Lp}
Suppose that $T$ is a product Calder\'on--Zygmund operator defined
in Subsection 3.1.
Then  for $\max\big(\frac{ {Q}_1}{ {Q}_1+
\vartheta_1 },\frac{ {Q}_2}{ {Q}_2+ \vartheta_2 }\big)<p\leq1$, $T$
extends to a bounded operator from $H^p(\widetilde{M})$ to
$L^p(\widetilde{M})$. Moreover, there exists a constant $C$ such
that
$$\|Tf\|_{L^p(\widetilde{M})}\leq C\|f\|_{H^p(\widetilde{M})}.$$
\end{theorem}

\begin{proof}
Fix $\max\big(\frac{ {Q}_1}{ {Q}_1+ \vartheta_1 },\frac{ {Q}_2}{
{Q}_2+ \vartheta_2 }\big)<p\leq1$. Since $H^p(\widetilde{M})\cap L^2$ is dense in $H^p(\widetilde{M}),$ it suffices to prove that there exists a positive constant $C$ such that for every $f\in H^p(\widetilde{M})\cap L^2$,
\begin{eqnarray}\label{T bd from Hp to Lp}
   \|Tf\|_{L^p(\widetilde{M})}\leq    C\|f\|_{H^p(\widetilde{M})}.
\end{eqnarray}

To prove (\ref{T bd from Hp to Lp}), similar to the proof of Theorem 3.5, we only need to show that for any $(p,2)$-atom $a$ of $H^p(\widetilde{M})$, $\|Ta\|_{L^p(\widetilde{M})}$ is uniformly bounded. To do this, suppose that $a$ is an $(p,2)$-atom with ${\rm supp}a\subset \Omega$
and $a=\sum_{R\in\mathcal{M}(\Omega)}a_R$. Set $\widetilde{\Omega}, \widetilde{\widetilde{\Omega}}, R, \widehat{R}$ and $\widehat{\widehat{R}}$ as in the proof of Theorem 3.5.

To prove that $\|T(a)\|_{L^p(\widetilde{M})}^p\leq C$, where $C$ is a positive constant independent of $a$, we decompose $\|T(a)\|_{L^p(\widetilde{M})}^p$ as follows.
\begin{eqnarray*}
&&\|T(a)\|_{L^p(\widetilde{M})}^p\\
&&=
\int_{\cup_{R\in\mathcal{M}(\Omega)}100\overline{C}\widehat{\widehat{R}}}
T(a)(x_1,x_2)^pdx_1dx_2+
\int_{(\cup_{R\in\mathcal{M}(\Omega)}100\overline{C}\widehat{\widehat{R}})^c}
T(a)(x_1,x_2)^pdx_1dx_2\\[4pt]
&&=: A+B.
\end{eqnarray*}

Applying the H\"older inequality and the $L^2$ boundedness of $T$ implies
\begin{eqnarray*}
A&\leq&
\mu\big(\bigcup_{R\in\mathcal{M}(\Omega)}100\overline{C}\widehat{\widehat{R}}\big)^{1-p/2}
\left(\int_{\widetilde{M}}
|T(a)(x_1,x_2)|^2dx_1dx_2\right)^{p/2}\\
&\leq& C\mu(\Omega)^{1-p/2}\|a\|_{L^2(\widetilde{M})}^{p}\\
&\leq& C.
\end{eqnarray*}

To estimate $B$, we write
\begin{eqnarray*}
B&\leq&
\sum_{R\in\mathcal{M}(\Omega)}\int_{(100\overline{C}\widehat{\widehat{R}})^c}
T(a_R)(x_1,x_2)^pdx_1dx_2\\
&\leq&
\sum_{R\in\mathcal{M}(\Omega)}\int_{x_1\not\in100\overline{C}\widehat{I}_1}\int_{M_2}
T(a_R)(x_1,x_2)^pdx_1dx_2\\
&&+\sum_{R\in\mathcal{M}(\Omega)}\int_{M_1}\int_{x_2\not\in100\overline{C}\widehat{I}_2}
T(a_R)(x_1,x_2)^pdx_1dx_2\\[4pt]
&=:&B_{1}+B_{2}.
\end{eqnarray*}

We only need to estimate $B_1$ since the proof of estimate for $B_2$ is similar. To do this, we write
\begin{eqnarray*}
B_1&=&
\sum_{R\in\mathcal{M}(\Omega)}\Big(\int_{x_1\not\in100\overline{C}\widehat{I}_1}\int_{x_2\in 10I_2} + \int_{x_1\not\in100\overline{C}\widehat{I}_1}\int_{x_2\not\in 10I_2} \Big)
T(a_R)(x_1,x_2)^pdx_1dx_2\\[4pt]
&=:&B_{11}+B_{12}.
\end{eqnarray*}
By H\"older's inequality we obtain
$$
B_{11}\leq
C\sum_{R\in\mathcal{M}(\Omega)}\mu_2(I_2)^{1-p/2}\int_{x_1\not\in100\overline{C}\widehat{I}_1}\Big(
\int_{x_2\in 10I_2} T(a_R)(x_1,x_2)^2dx_2 \Big)^{p/2}
dx_1.$$ To estimate the inside integral above, using the
cancellation condition on $a_R$, we write
$$T(a_R)(x_1,x_2)=\iint_{3R} \big[K(x_1,x_2,y_1,y_2)-K(x_1,x_2,y_{I_1},y_2)\big] a_R(y_1,y_2) dy_1dy_2.$$
Applying the smoothness condition on $K$ yields
\begin{eqnarray}\label{CZ norm on the first variable}
&&\int_{x_2\in 10I_2} |T(a_R)(x_1,x_2)|^2dx_2\nonumber\\[4pt]
&&\hskip1cm\leq C\mu_1(I_1)
  \iint_{3I_1} \|{K_1}(x_1,y_1)-{K_1}(x_1,y_{I_1})\|_{CZ}^2\ \|a_R(y_1,\cdot)\|_{L^2(M_2)}^2
  dy_1\\[4pt]
&&\hskip1cm\leq
C\Big(\frac{d_1(y_1,y_{I_1})}{d_1(x_1,y_{I_1})}\Big)^{2\epsilon}V(x_1,y_{I_1})^{-2}\
\mu_1(I_1)\|a_R\|_{L^2(\widetilde{M})}^2.\nonumber
\end{eqnarray}
Inserting this estimate into the right side of the estimate for
$B_{11}$ implies
\begin{eqnarray*}
B_{11}&\leq&
C\sum_{R\in\mathcal{M}(\Omega)}\mu_2(I_2)^{1-p/2}\\[4pt]
&&\hskip1cm\times\int_{x_1\not\in 100\overline{C}\widehat{I}_1} \bigg(  \Big(\frac{d_1(y_1,y_{I_1})}{d_1(x_1,y_{I_1})}\Big)^{2\epsilon}V(x_1,y_{I_1})^{-2}\ \mu_1(I_1)\|a_R\|_{L^2(\widetilde{M})}^2\bigg)^{p/2}  dx_1\\[5pt]
&\leq& C\sum_{R\in\mathcal{M}(\Omega)}\mu_2(I_2)^{1-p/2}
\mu_1(I_1)^{p/2} \ell(I_1)^{p\epsilon}
\|a_R\|_{L^2(\widetilde{M})}^p\\
&&\hskip1cm\times \int_{x_1\not\in100\overline{C}\widehat{I}_1}
d_1(x_1,y_{I_1})^{-p\epsilon}V(x_1,y_{I_1})^{-p} dx_1,
\end{eqnarray*}
where $y_{I_1}$ is the center of the cube $I_1$ and the fact that $d_1(y_1,y_{I_1})\leq {\frac{1}{2A}}d_1(x_1,y_{I_1})$ is used.

We now estimate the last integral above. To this end, we
decompose the set $\{ x_1\not\in 100\overline{C}\widehat{I}_1 \}$ into annuli and then get
\begin{eqnarray}\label{estimate B11 integral}
&&\int_{x_1\not\in100\overline{C}\widehat{I}_1}  d_1(x_1,y_{I_1})^{-p\epsilon}V(x_1,y_{I_1})^{-p}  dx_1\\[4pt]
&& \leq C \sum_{k=0}^\infty \big(2^k\ell(\widehat{I}_1)\big)^{-p\epsilon} V\big(y_{I_1}, 2^k\ell(\widehat{I}_1)\big)^{1-p}\nonumber\\[4pt]
&& \leq C \sum_{k=0}^\infty 2^{-kp\epsilon} \ell(\widehat{I}_1)^{-p\epsilon} 2^{kQ_1(1-p)} V\big(y_{I_1}, \ell(\widehat{I}_1)\big)^{1-p}\nonumber\\[4pt]
&& \leq C  \ell(\widehat{I}_1)^{-p\epsilon}  V\big(y_{I_1}, \ell(\widehat{I}_1)\big)^{1-p},\nonumber
\end{eqnarray}
where the last inequality follows from the condition that $\max\big(\frac{ {Q}_1}{ {Q}_1+ \vartheta_1 },\frac{ {Q}_2}{
{Q}_2+ \vartheta_2 }\big)<p\leq1$.

Putting all estimates together implies
\begin{eqnarray*}
B_{11}
\leq_
C\sum_{R\in\mathcal{M}(\Omega)}\mu(R)^{1-p/2} \Big(\frac{ \ell(I_1)}{\ell(\widehat{I}_1)}\Big)^{p\epsilon} \Big(\frac{ V\big(y_{I_1}, \ell(\widehat{I}_1)\big)}{ \mu_1(I_1) } \Big)^{1-p} \|a_R\|_{L^2(\widetilde{M})}^p.
\end{eqnarray*}
Repeating the same argument as in (3.12) gives
$$
   B_{11}\leq C,
$$
where $C$ is a positive constant independent of the atom $a $.

We now turn to estimate $B_{12}$. To do this, again using the cancellation conditions on $a_R$ yields

\begin{eqnarray*}
&&\hskip-.5cmTa_R(x_1,x_2)\\
&&=\iint_{3R}
\big[K(x_1,x_2,y_1,y_2)-K(x_1,x_2,y_{I_1},y_2)-K(x_1,x_2,y_1,y_{I_2})
+K(x_1,x_2,y_{I_1},y_{I_2})\big]\\[4pt]
&&\hskip1cm\times a_R(y_1,y_2) dy_1dy_2.
\end{eqnarray*}
By the smoothness condition on $K$ and  we obtain
\begin{eqnarray*}
&&|Ta_R(x_1,x_2)|\\[4pt]
&&\leq C\Big(\frac{d_1(y_1,y_{I_1})}{d_1(x_1,y_{I_1})}\Big)^{\epsilon}V(x_1,y_{I_1})^{-1} \Big(\frac{d_2(y_2,y_{I_2})}{d_2(x_2,y_{I_2})}\Big)^{\epsilon}V(x_2,y_{I_2})^{-1}
\iint_{3R}|a_R(y_1,y_2)| dy_1dy_2
\end{eqnarray*}
and hence
\begin{eqnarray*}
B_{12}
&\leq&  C\sum_{R\in\mathcal{M}(\Omega)} \int_{x_1\not\in100\overline{C}\widehat{I}_1}\int_{x_2\not\in 10I_2}
\Big(\iint_{3R}     \Big(\frac{d_1(y_1,y_{I_1})}{d_1(x_1,y_{I_1})}\Big)^{\epsilon}V(x_1,y_{I_1})^{-1}
\\[4pt]
&&\times\Big(\frac{d_2(y_2,y_{I_2})}{d_2(x_2,y_{I_2})}\Big)^{\epsilon}V(x_2,y_{I_2})^{-1}|a_R(y_1,y_2)|
dy_1dy_2\Big)^pdx_1dx_2,
\end{eqnarray*}
where $y_{I_1}$ and $y_{I_2}$ are the centers of the cubes $I_1$ and $I_2$, respectively and the fact that $d_1(y_1,y_{I_1})\leq {\frac{1}{2A}}d_1(x_1,y_{I_1})$ and $d_2(y_2,y_{I_2})\leq {\frac{1}{2A}}d_2(x_2,y_{I_2})$ is used.

Applying H\"older's inequality implies
\begin{eqnarray*}
B_{12}
&\leq&  C \sum_{R\in\mathcal{M}(\Omega)} \ell(I_1)^{p\epsilon}\ell(I_2)^{p\epsilon} \mu(R)^{p/2} \|a_R\|_{L^2(\widetilde{M})}^p\\[4pt]
&&\times\int_{x_1\not\in100\overline{C}\widehat{I}_1}\int_{x_2\not\in 10I_2}  d_1(x_1,y_{I_1})^{-p\epsilon}V(x_1,y_{I_1})^{-p}
d_2(x_2,y_{I_2})^{-p\epsilon}V(x_2,y_{I_2})^{-p}
 dx_1dx_2\\[4pt]
&\leq&  C \sum_{R\in\mathcal{M}(\Omega)} \ell(I_1)^{p\epsilon}\ell(I_2)^{p\epsilon} \mu(R)^{p/2} \|a_R\|_{L^2(\widetilde{M})}^p\ \ell(\widehat{I}_1)^{-p\epsilon}  V\big(y_{I_1}, \ell(\widehat{I}_1)\big)^{1-p}
\ell(I_2)^{-p\epsilon}  V\big(y_{I_2}, \ell(I_2)\big)^{1-p}\\[4pt]
&\leq&  C \sum_{R\in\mathcal{M}(\Omega)}\Big(\frac{
\ell(I_1)}{\ell(\widehat{I}_1)}\Big)^{p\epsilon} \Big(\frac{
V\big(y_{I_1}, \ell(\widehat{I}_1)\big)}{ \mu_1(I_1) } \Big)^{1-p}
\mu(R)^{1-p/2} \|a_R\|_{L^2(\widetilde{M})}^p\\[4pt]
&\leq& C,
\end{eqnarray*}
where the last inequality follows from the same estimate for $B_{11}.$

As a consequence, we obtain that $B_{1}\leq C$ and similarly $B_2\leq C$. The proof of Theorem 3.6 is concluded.
\end{proof}

\subsubsection{$L^\infty\rightarrow {BMO}$ boundedness }

As a consequence of Theorem 3.6 with $p=1,$ together with the duality that $(H^1(\widetilde{M}))^*=BMO(\widetilde{M}),$ we obtain the following
\begin{theorem}\label{theorem-bd from Linfty to Lp}
Suppose that $T$ is a Calder\'on--Zygmund operator defined in
Subsection 3.1. Then $T$ extends to a bounded operator from
$L^\infty(\widetilde{M})$ to ${BMO}(\widetilde{M})$. Moreover, there
exists a constant $C$ such that
$$\|Tf\|_{BMO(\widetilde M)}\leq C \|f\|_\infty.$$
\end{theorem}
Theorem 3.7 gives the necessary conditions of Theorem A as follows.

\begin{cor}\label{corollary of theorem-bd from Linfty to Lp}
Suppose that $T$ and $\widetilde T$ are Calder\'on--Zygmund
operators defined in Subsection 3.1. Then $T(1), T^*(1), \widetilde
T(1)$ and $(\widetilde T)^*(1)$ lie on $BMO(\widetilde M).$
\end{cor}

\begin{proof}[Proof of Theorem 3.7]
Suppose that $T$ is a Calder\'on--Zygmund operator defined in
Subsection 3.1. We have to define $Tf$ for $f\in L^\infty(\widetilde M).$ To this
end, we first observe that if $f\in L^\infty(\widetilde M)\cap L^2(\widetilde M)$ then $Tf$ is
well defined, and moreover, for $g\in H^1(\widetilde M)\cap L^2(\widetilde M),$ we
have
$$\langle Tf, g\rangle=\langle f, T^*g\rangle,$$
which together with the fact that, by Theorem 3.6, $T^*$ is bounded from $H^1(\widetilde M)$ to $L^1(\widetilde M)$ and the duality arguments $(L^1, L^\infty)$ and $(H^1, BMO)$  gives $Tf\in BMO(\widetilde M)$ since $T^*g\in L^1(\widetilde M)$ and $H^1(\widetilde M)\cap L^2(\widetilde M)$ is dense in $H^1(\widetilde M).$ To define $Tf$ for $f\in L^\infty,$ we define functions $f_j(x,y)$ by $f_j(x,y)=f(x,y),$ when $d(x,x_0)\leq j, d(y, y_0)\leq j$ and $f_j(x,y)=0,$ otherwise, where $x_0\in M_1$ and $y_0\in M_2$ are any fixed points. Then $f_j\in L^\infty(\widetilde M)\cap L^2(\widetilde M)$ and thus for $g\in H^1(\widetilde M)\cap L^2(\widetilde M),$
$$\langle Tf_j, g\rangle =\langle f_j, T^*g\rangle \rightarrow \langle f, T^*g\rangle .$$
Indeed, $\|f_j\|_{L^\infty(\widetilde M)}\leq \|f\|_{L^\infty(\widetilde M)},$ $f_j\rightarrow f$ almost everywhere, and $T^*g\in L^1(\widetilde M),$ so that we can apply Lebesgue's dominated convergence theorem. This implies that functions $Tf_j$ form a bounded sequence in $BMO(\widetilde M)$ and this sequence converges to $Tf$ in the topology $(H^1, BMO).$ It remains to show the estimate in Theorem 3.7. To do this, we first consider $f\in L^2(\widetilde M)\cap L^\infty(\widetilde M).$ Then for $g\in H^1(\widetilde M)\cap L^2(\widetilde M),$ as mentioned,
$$|\langle Tf, g\rangle|\leq C\|f\|_{L^\infty(\widetilde M)} \|g\|_{H^1(\widetilde M)}.$$
This together with the fact that $H^1(\widetilde M)\cap L^2(\widetilde M)$ is dense in $H^1(\widetilde M)$ implies that $\langle Tf, g\rangle$ defines a continuous linear functional on $H^1(\widetilde M)$ and its norm is dominated by $C\|f\|_{L^\infty(\widetilde M)}.$ By Theorem 2.18, these exists $h\in CMO^1(\widetilde M)$ such that
$$\langle Tf, g\rangle=\langle h, g\rangle$$
for all $g\in \GGp(\beta_1,\beta_2;\gamma_1,\gamma_2)$ and
$\|h\|_{CMO^1(\widetilde M)}\leq C\|f\|_{L^\infty(\widetilde M)}.$ Now we point out that
$D_{k_2}D_{k_1}(x_1,x_2)\in \GGp(\beta_1,\beta_2;\gamma_1,\gamma_2)$
since $D_{k_1}$ and $D_{k_2}$ satisfy the size and smoothness
conditions in (\ref{size of Sk}) and (\ref{smoothness of Sk}).
Taking $g(x_1,x_2)=D_{k_2}D_{k_1}(x_1,x_2)$ in the above equality yields
that $D_{k_2}D_{k_1}(Tf)(x_1,x_2)=D_{k_2}D_{k_1}(h)(x_1,x_2)$ and hence for
$f\in L^2(\widetilde M)\cap L^\infty(\widetilde M),$
$$\|Tf\|_{CMO^1(\widetilde M)}=\|h\|_{CMO^1(\widetilde M)}\leq C\|f\|_{L^\infty(\widetilde M)}.$$
For $f\in L^\infty,$ by the definition for $Tf,$ we have $D_{k_2}D_{k_1}(Tf)(x_1,x_2)=D_{k_2}D_{k_1}(\lim\limits_jTf_j)(x_1,x_2)$ since $D_{k_2}D_{k_1}(x_1,x_2)\in \GGp(\beta_1,\beta_2;\gamma_1,\gamma_2)$ so $D_{k_2}D_{k_1}(x_1,x_2)\in H^1(\widetilde M)\cap L^2(\widetilde M).$  Thus
$$\|Tf\|_{CMO^1(\widetilde M)}=\|\lim_jTf_j\|_{CMO^1(\widetilde M)}\leq \liminf_j \|Tf_j\|_{CMO^1(\widetilde M)}$$
$$\leq C\liminf_j\|f_j\|_{L^\infty(\widetilde M)}\leq C\|f\|_{L^\infty(\widetilde M)}.$$
Note that $CMO^1(\widetilde M)=BMO(\widetilde M).$ The proof of Theorem 3.7 is concluded.
\end{proof}

\subsubsection{$L^{p}, 1<p<\infty,$ boundedness }

In this subsection we prove the $L^p, 1<p<\infty,$ boundedness, namely the following
\begin{theorem}\label{theorem-bd on Lp}
Suppose $T$ is a Calder\'on--Zygmund operator defined in Section
3.1. Then  $T$ extends to a bounded operator from $L^p, 1<p<\infty,$
to itself. Moreover, there exists a constant $C$ such that
$$\|Tf\|_{p}\leq C \|f\|_p.$$
\end{theorem}
Indeed, in [HLL2] the following Calder\'on--Zygmund decomposition
was obtained.
\begin{theorem}
Let $\max\big(\frac{ Q_1}{ Q_1+\vartheta_1},\frac{ Q_2}{
Q_2+\vartheta_2} \big)<p_2<p<p_1<\infty,$ $\alpha>0$
be given and $f\in H^p(\widetilde{M})$. Then we may write $f=g+b$ where $g\in
H^{p_1}(\widetilde{M})$ and $b\in H^{p_2}(\widetilde{M})$ such
 that $\|g\|^{p_1}_{H^{p_1}(\widetilde M)}\le C\alpha^{p_1-p}\|f\|^p_{H^p(\widetilde M)}$ and
 $\|b\|^{p_2}_{H^{p_2}(\widetilde M)}\le C\alpha^{p_2-p}\|f\|^p_{H^p(\widetilde M)}$, where $C$ is an
absolute constant.
\end{theorem}
As a consequence of Theorem 3.10, the following interpolation theorem was proved in [HLL2].
\begin{theorem}\label{theorem interpolation}
Let $\max\big(\frac{ Q_1}{ Q_1+\vartheta_1},\frac{ Q_2}{
Q_2+\vartheta_2} \big)<p_2<p_1<\infty$
and $T$ be a linear operator which is bounded from $H^{p_2}(\widetilde M)$ to $L^{p_2}(\widetilde{M})$
and from $H^{p_1}(\widetilde{M})$ to  $L^{p_1}(\widetilde{M})$, then $T$ is bounded on $H^p(\widetilde M)$ for $p_2<p<p_1.$
\end{theorem}
Note that $H^{p}(\widetilde{M})=L^p(\widetilde M)$ for $1<p<\infty.$
Now the proof of Theorem 3.9 with $1<p<2$ follows from Theorem 3.6 and 3.11 directly by taking $p_2=1$ and $p_1=2.$ The duality argument gives the proof of Theorem 3.9 for $2<p<\infty.$

\subsection{Sufficient conditions of $T1$  Theorem}

In this section, we prove the sufficient conditions of Theorem A. To show that $T$ is bounded on $L^2$ it suffices to prove that for $f, g\in \GGp(\beta_1,\beta_2;\gamma_1,\gamma_2)$ with compact supports, there exists a constant $C$ such that
$$|\langle g, Tf\rangle|\leq C\|f\|_2\|g\|_2.$$
This is because, by Calder\'on's identity established in [HLL2], the collection of functions in $\GGp(\beta_1,\beta_2;\gamma_1,\gamma_2)$ having compact supports is dense in $L^2.$

As described in Section 1, we write
\begin{eqnarray}\label{bilinear form}
\langle g,Tf\rangle &=& \sum_{k_1^{'}}  \sum_{I_1^{'}}
\sum_{k_1}\sum_{I_1}\sum_{k_2^{'}}  \sum_{I_2^{'}} \sum_{k_2}\sum_{I_2}\mu_1(I_1^{'})\mu_1(I_1) \mu_2(I_2^{'})\mu_2(I_2)\\
&&\times{\widetilde{\widetilde{D}}}_{k_1^{'}}{\widetilde{\widetilde{D}}}_{k_2^{'}}(g)(x_{I_1^{'}},x_{I_2^{'}})
\Big\langle D_{k_1^{'}}D_{k_2^{'}},
TD_{k_1}D_{k_2}\Big\rangle(x_{I_1^{'}},x_{I_2^{'}},x_{I_1},x_{I_2})
{\widetilde{\widetilde{D}}}_{k_1}{\widetilde{\widetilde{D}}}_{k_2}(f)(x_{I_1},x_{I_2}).\nonumber
\end{eqnarray}
To see the above equality, we first consider one parameter case. Let $f_1, g_1\in \GG(\beta,\gamma)(M_1)$ with compact supports and $T_1$ be a singular integral operator on $M_1.$  Then by the discrete Carlder\'on identity on $M_1,$
\begin{eqnarray}
\langle g_1,T_1f_1\rangle &=& \sum_{k_1^{'}}  \sum_{I_1^{'}}\mu_1(I_1^{'}){\widetilde{\widetilde{D}}}_{k_1^{'}}(g)(x_{I_1^{'}})
\Big\langle D_{k_1^{'}}(\cdot,x_{I_1^{'}}), T_1f_1\Big\rangle \label{one parameter bilinear form}\\
&=&\sum_{k_1^{'}}  \sum_{I_1^{'}}
\sum_{k_1}\sum_{I_1} \mu_1(I_1^{'})\mu_1(I_1){\widetilde{\widetilde{D}}}_{k_1^{'}}(g)(x_{I_1^{'}})
\Big\langle D_{k_1^{'}},
T_1 D_{k_1}\Big\rangle(x_{I_1^{'}},x_{I_1})
\widetilde{D}_{k_1}(f_1)(x_{I_1}).\nonumber
\end{eqnarray}
For the equality (\ref{one parameter bilinear form}), we use the fact that $\sum\limits_{k_1^{'}>0} \sum\limits_{I_1^{'}}\mu_1(I_1^{'}){\widetilde{\widetilde{D}}}_{k_1^{'}}(g)(x_{I_1^{'}})
D_{k_1^{'}}(x_1,x_{I_1^{'}})$ converges in the test function space $\GG(\beta,\gamma)(M_1)$ with compact support, so that
$$\big\langle \sum_{k_1^{'}>0} \sum_{I_1^{'}}\mu_1(I_1^{'}){\widetilde{\widetilde{D}}}_{k_1^{'}}(g)(x_{I_1^{'}})
D_{k_1^{'}}(\cdot,x_{I_1^{'}}), T_1f_1\big\rangle$$
$$\hskip.7cm=\sum_{k_1^{'}>0} \sum_{I_1^{'}}\mu_1(I_1^{'}){\widetilde{\widetilde{D}}}_{k_1^{'}}(g)(x_{I_1^{'}})
\langle D_{k_1^{'}}(\cdot,x_{I_1^{'}}), T_1f_1\rangle.$$
This, however, is not true for $\sum\limits_{k_1^{'}\leq 0} \sum\limits_{I_1^{'}}\mu_1(I_1^{'}){\widetilde{\widetilde{D}}}_{k_1^{'}}(g)(x_{I_1^{'}})
D_{k_1^{'}}(x_1,x_{I_1^{'}}),$ because the support of $D_{k_1^{'}}(x_1,x_{I_1^{'}})$ gets big as $k_1^{'}$ tends to $-\infty,$ even though
$\sum\limits_{k_1^{'}\leq 0} \sum_{I_1^{'}}\mu_1(I_1^{'}){\widetilde{\widetilde{D}}}_{k_1^{'}}(g)(x_{I_1^{'}})
D_{k_1^{'}}(x_1,x_{I_1^{'}})\in \GG(\beta,\gamma)(M_1)$ having compact support. Now if $\theta \in \GG(\beta,\gamma)(M_1)$ and has compact support, then $\theta(x_1)\sum\limits_{k_1^{'}\leq 0} \sum_{I_1^{'}}\mu_1(I_1^{'}){\widetilde{\widetilde{D}}}_{k_1^{'}}(g)(x_{I_1^{'}})
D_{k_1^{'}}(x_1,x_{I_1^{'}})$ converges in the topology of $C^\beta_0(M_1).$ If we choose $\theta=1$ on a large enough set which contains the support of $f_1,$ then, by the standard estimate on the kernel of $T_1,$
$$\big\langle (1-\theta)\sum\limits_{k_1^{'}\leq 0} \sum_{I_1^{'}}\mu_1(I_1^{'}){\widetilde{\widetilde{D}}}_{k_1^{'}}(g)(x_{I_1^{'}})
D_{k_1^{'}}(\cdot,x_{I_1^{'}}), T_1f_1\big\rangle$$
$$\hskip.7cm=\sum\limits_{k_1^{'}\leq 0} \sum_{I_1^{'}}\mu_1(I_1^{'}){\widetilde{\widetilde{D}}}_{k_1^{'}}(g)(x_{I_1^{'}})
\langle (1-\theta)D_{k_1^{'}}(\cdot,x_{I_1^{'}}), T_1f_1\rangle.$$
This implies the equality (\ref{one parameter bilinear form}). For fixed $k_1^{'}$ we can do the same thing to $f_1$ to obtain the second equality.
Repeating the same things above twice, first on $M_1$ and then on $M_2,$ gives (\ref{bilinear form}).

As described in Section 1, we consider the following four cases:

\smallskip
Case 1. $k_1^{'}\geq k_1 $ and $k_2^{'}\geq k_2 $;

\smallskip
Case 2. $k_1^{'}\geq k_1 $ and $k_2^{'}< k_2 $;

\smallskip
Case 3. $k_1^{'}< k_1 $ and $k_2^{'}\geq k_2 $;

\smallskip
Case 4. $k_1^{'}< k_1 $ and $k_2^{'}< k_2. $

Now we decompose the bilinear form $\langle g,Tf\rangle$ as
$$\langle g,Tf\rangle=\langle g,Tf\rangle_{\rm Case\ 1}+\langle g,Tf\rangle_{\rm Case\ 2}+\langle g,Tf\rangle_{\rm Case\ 3}+\langle g,Tf\rangle_{\rm Case\ 4},$$
where
\begin{eqnarray}\label{g Tf case1}
\langle g,Tf\rangle_{\rm Case\ 1}&=&\sum_{k_1\leq k_1^{'} }\sum_{k_2\leq k_2^{'}}\sum_{I_1^{'}}\sum_{I_2^{'}}
\sum_{I_1}\sum_{I_2}\mu_1(I_1^{'})\mu_1(I_1) \mu_2(I_2^{'})\mu_2(I_2) {\widetilde{\widetilde{D}}}_{k_1^{'}}{\widetilde{\widetilde{D}}}_{k_2^{'}}(g)(x_{I_1^{'}},x_{I_2^{'}})   \nonumber\\
&&\times {\widetilde{\widetilde{D}}}_{k_1}{\widetilde{\widetilde{D}}}_{k_2}(f)(x_{I_1},x_{I_2})
\Big\langle D_{k_1^{'}}D_{k_2^{'}}, TD_{k_1}D_{k_2}\Big\rangle(x_{I_1^{'}},x_{I_2^{'}},x_{I_1},x_{I_2})
\end{eqnarray}
and similarly for other three terms.

Since the estimates for $\langle g,Tf\rangle_{\rm Case\ 4}$ and $\langle g,Tf\rangle_{\rm Case\ 3}$ are similar to $\langle g,Tf\rangle_{\rm Case\ 1}$ and $\langle g,Tf\rangle_{\rm Case\ 2},$ respectively, so we only prove that under the sufficient conditions the first two terms are bounded by some constant times $\|f\|_2\|g\|_2.$ This will conclude the proof of the sufficient conditions of Theorem A.

To deal with the first term $\langle g,Tf\rangle_{\rm Case\ 1},$ as mentioned in Section 1, for $k_1\leq k_1^{'}$ and $k_2\leq k_2^{'}$ we first decompose
\begin{eqnarray*}
&&\Big\langle D_{k_1^{'}}D_{k_2^{'}}, TD_{k_1}D_{k_2}\Big\rangle(x_{I_1^{'}},x_{I_2^{'}},x_{I_1},x_{I_2})\\[4pt]
&&=\int  D_{k'_1}(x_{I_1^{'}},u_1)D_{k'_2}(x_{I_2^{'}},u_2)K(u_1,u_2,v_1,v_2)[D_{k_1}(v_1,x_{I_1})-D_{k_1}(x_{I_1^{'}},x_{I_1})]\\[4pt]
&&\hskip1cm\times [D_{k_2}(v_2,x_{I_2})-D_{k_2}(x_{I_2^{'}},x_{I_2})] du_1du_2dv_1dv_2\\[4pt]
&&\hskip.5cm+ \int  D_{k'_1}(x_{I_1^{'}},u_1)D_{k'_2}(x_{I_2^{'}},u_2)K(u_1,u_2,v_1,v_2)D_{k_1}(x_{I_1^{'}},x_{I_1}) D_{k_2}(v_2,x_{I_2}) du_1du_2dv_1dv_2\\[4pt]
&&\hskip.5cm+ \int  D_{k'_1}(x_{I_1^{'}},u_1)D_{k'_2}(x_{I_2^{'}},u_2)K(u_1,u_2,v_1,v_2)D_{k_1}(v_1,x_{I_1})
D_{k_2}(x_{I_2^{'}},x_{I_2}) du_1du_2dv_1dv_2\\[4pt]
&&\hskip.5cm- \int  D_{k'_1}(x_{I_1^{'}},u_1)D_{k'_2}(x_{I_2^{'}},u_2)K(u_1,u_2,v_1,v_2)D_{k_1}(x_{I_1^{'}},x_{I_1})D_{k_2}(x_{I_2^{'}},x_{I_2}) du_1du_2dv_1dv_2\\[4pt]
&&=: I(x_{I_1^{'}},x_{I_2^{'}},x_{I_1},x_{I_2})+II(x_{I_1^{'}},x_{I_2^{'}},x_{I_1},x_{I_2})+III(x_{I_1^{'}},x_{I_2^{'}},x_{I_1},x_{I_2})
+IV(x_{I_1^{'}},x_{I_2^{'}},x_{I_1},x_{I_2})
\end{eqnarray*}
and then write
$$\langle g,Tf\rangle_{\rm Case\ 1}=\langle g,Tf\rangle_{\rm Case\ 1.1}+\langle g,Tf\rangle_{\rm Case\ 1.2}+\langle g,Tf\rangle_{\rm Case\ 1.3}+\langle g,Tf\rangle_{\rm Case\ 1.4},$$
where
\begin{eqnarray*}
\langle g,Tf\rangle_{\rm Case\ 1.1}&=& \sum_{k_1\leq k_1^{'} }\sum_{k_2\leq k_2^{'}}\sum_{I_1^{'}}\sum_{I_2^{'}}
\sum_{I_1}\sum_{I_2}\mu_1(I_1^{'})\mu_1(I_1) \mu_2(I_2^{'})\mu_2(I_2) {\widetilde{\widetilde{D}}}_{k_1^{'}}{\widetilde{\widetilde{D}}}_{k_2^{'}}(g)(x_{I_1^{'}},x_{I_2^{'}})   \nonumber\\[4pt]
&&\times {\widetilde{\widetilde{D}}}_{k_1}{\widetilde{\widetilde{D}}}_{k_2}(f)(x_{I_1},x_{I_2})
I(x_{I_1^{'}},x_{I_2^{'}},x_{I_1},x_{I_2}).
\end{eqnarray*}
The other terms $\langle g,Tf\rangle_{\rm Case\ 1.i}, i=2,3,4,$ are defined similarly.

Corresponding the case 2, that is, $k_1^{'}\geq k_1 $ and $k_2^{'}< k_2,$ we give the decomposition of term $\langle g,Tf\rangle_{\rm Case\ 2}.$ Similarly, we first write
\begin{eqnarray*}
&&\Big\langle D_{k_1^{'}}D_{k_2^{'}}, TD_{k_1}D_{k_2}\Big\rangle(x_{I_1^{'}},x_{I_2^{'}},x_{I_1},x_{I_2})\\[4pt]
&&=\int  D_{k'_1}(x_{I_1^{'}},u_1)[D_{k'_2}(x_{I_2^{'}},u_2)-D_{k'_2}(x_{I_2^{'}},x_{I_2})]K(u_1,u_2,v_1,v_2)[D_{k_1}(v_1,x_{I_1})-D_{k_1}(x_{I_1^{'}},x_{I_1})]\\[4pt]
&&\hskip1cm\times D_{k_2}(v_2,x_{I_2})du_1du_2dv_1dv_2\\
&&\hskip.5cm+ \int  D_{k'_1}(x_{I_1^{'}},u_1)D_{k'_2}(x_{I_2^{'}},u_2)K(u_1,u_2,v_1,v_2)D_{k_1}(x_{I_1^{'}},x_{I_1}) D_{k_2}(v_2,x_{I_2}) du_1du_2dv_1dv_2\\[4pt]
&&\hskip.5cm+ \int  D_{k'_1}(x_{I_1^{'}},u_1)D_{k'_2}(x_{I_2^{'}},x_{I_2})K(u_1,u_2,v_1,v_2)D_{k_1}(v_1,x_{I_1})
D_{k_2}(v_2,x_{I_2}) du_1du_2dv_1dv_2\\[4pt]
&&\hskip.5cm- \int  D_{k'_1}(x_{I_1^{'}},u_1)D_{k'_2}(x_{I_2^{'}},x_{I_2})K(u_1,u_2,v_1,v_2)D_{k_1}(x_{I_1^{'}},x_{I_1})D_{k_2}(v_2,x_{I_2}) du_1du_2dv_1dv_2\\[4pt]
&&=: V(x_{I_1^{'}},x_{I_2^{'}},x_{I_1},x_{I_2})+VI(x_{I_1^{'}},x_{I_2^{'}},x_{I_1},x_{I_2})+VII(x_{I_1^{'}},x_{I_2^{'}},x_{I_1},x_{I_2})\\[4pt]
&&\hskip1cm+VIII(x_{I_1^{'}},x_{I_2^{'}},x_{I_1},x_{I_2}),
\end{eqnarray*}
and then decompose
$$\langle g,Tf\rangle_{\rm Case\ 2}=\langle g,Tf\rangle_{\rm Case\ 2.1}+\langle g,Tf\rangle_{\rm Case\ 2.2}+\langle g,Tf\rangle_{\rm Case\ 2.3}+\langle g,Tf\rangle_{\rm Case\ 2.4},$$
where
\begin{eqnarray*}
\langle g,Tf\rangle_{\rm Case\ 2.1}&=& \sum_{k_1\leq k_1^{'} }\sum_{k_2> k_2^{'}}\sum_{I_1^{'}}\sum_{I_2^{'}}
\sum_{I_1}\sum_{I_2}\mu_1(I_1^{'})\mu_1(I_1) \mu_2(I_2^{'})\mu_2(I_2) \widetilde{\widetilde{D}}_{k_1^{'}}\widetilde{\widetilde{D}}_{k_2^{'}}(g)(x_{I_1^{'}},x_{I_2^{'}})
\\[4pt]
&&\times \widetilde{\widetilde{D}}_{k_1}\widetilde{\widetilde{D}}_{k_2}(f)(x_{I_1},x_{I_2})
V(x_{I_1^{'}},x_{I_2^{'}},x_{I_1},x_{I_2}).
\end{eqnarray*}
Similarly for other terms $\langle g,Tf\rangle_{\rm Case\ 2.i}, i=2,3,4.$

Before we get into the details of estimates for $\langle g,Tf\rangle_{\rm Case\ 1}$ and $\langle g,Tf\rangle_{\rm Case\ 2},$ we would like to point out the main methods for doing this. Roughly speaking, in the classical one parameter case, the main methods are the almost orthogonality argument and Carleson measure estimate. In our setting with two parameter case, besides the almost orthogonality argument and Carleson measure estimate on $\widetilde M=M_1\times M_2,$ there are two more situations, that are, the almost orthogonality argument on one factor, say $M_1$ and Carleson measure estimate on other factor, say $M_2$, and the Littlewood--Paley estimate on one factor, say $M_1$ and Carleson measure estimate on other factor, say $M_2$.  These details will be given in next subsections.

\subsubsection{Almost orthogonality argument on $\widetilde M=M_1\times M_2$ }
In this subsection, we deal with $\langle g,Tf\rangle_{\rm Case\ 1.1}$ and $\langle g,Tf\rangle_{\rm Case\ 2.1}.$ The main method is the almost orthogonality argument on $\widetilde M=M_1\times M_2.$ Indeed, we will show the following estimate, that is, there exists a constant $C$ such that for $k'_1>k_1$ and $k'_2>k_2,$
\begin{eqnarray}
&&|I(x_{I_1^{'}},x_{I_2^{'}},x_{I_1},x_{I_2})|\nonumber\\
&&=\Big|\int  D_{k'_1}(x_{I_1^{'}},u_1)D_{k'_2}(x_{I_2^{'}},u_2)K(u_1,u_2,v_1,v_2)[D_{k_1}(v_1,x_{I_1})-D_{k_1}(x_{I_1^{'}},x_{I_1})]\nonumber\\
&&\hskip.6cm\times [D_{k_2}(v_2,x_{I_2})-D_{k_2}(x_{I_2^{'}},x_{I_2})] du_1du_2dv_1dv_2\Big|\nonumber\\
&&\leq C2^{(k_1-k'_1)\varepsilon}2^{-(k_2-k'_2)\varepsilon}
\frac{1}{V_{2^{-k_1}}(x_{I_1^{'}})+V_{2^{-k_1}}(x_{I_1})+V(x_{I_1^{'}},x_{I_1})}\frac{2^{-k_1\varepsilon}}
{(2^{-k_1}+d_1(x_{I_1^{'}},x_{I_1}))^{\varepsilon}}\nonumber\\
&&\hskip.6cm\times \frac{1}{V_{2^{-k_2}}(x_{I_2^{'}})+V_{2^{-k_2}}(x_{I_2})+V(x_{I_2^{'}},x_{I_2})}\frac{2^{-k_2\varepsilon}}
{(2^{-k_2}+d_2(x_{I_2^{'}},x_{I_2}))^{\varepsilon}}.\label{almost orth case 1.1}
\end{eqnarray}
We would like to remark that the cancellation condition on the kernel $K$ is not required in the above almost orthogonality estimate and only side, smoothness on $K$ and the weak boundedness property on $T$ are needed. To show the above estimate, we first consider the one parameter case. The estimate for two parameter case will follow from the iterative methods.
As mentioned in Section 1, let $T_1$ be a singular integral operator associated with the kernel $K_1$ defined on $M_1$ having the weak boundedness property. Then for $k_1<k'_1$
there exists a constant $C$ such that the following orthogonal estimate holds
\begin{eqnarray*}
&&\Big|\iint D_{k'_1}(x_1,u_1)K_1(u_1,v_1)[D_{k_1}(v_1,y_1)-D_{k_1}(x_1,y_1)]du_1dv_1\Big|\\
&&\leq C|K^1|_{CZ} 2^{(k_1-k'_1)\epsilon}\frac{1}{V_{2^{-k_1}}(x_1)+V_{2^{-k_1}}(y_1)+V(x_1,y_1)}\frac{2^{-k_1\varepsilon}}{(2^{-k_1}+d_1(x_1,y_1))^{\varepsilon}}.
\end{eqnarray*}
To see the above estimate, we first consider the case where $d_1(x_1,y_1)\geq C_1 2^{-k}.$
Note that if choosing $C_1$ sufficiently
large (depending on $C_0$) then $D_{k_1}(x_1,y_1)=0.$ Thus,
$$\iint D_{k'_1}(x_1,u_1)K_1(u_1,v_1)[D_{k_1}(v_1,y_1)-D_{k_1}(x_1,y_1)]du_1dv_1$$
$$=\iint D_{k'_1}(x_1,u_1)K_1(u_1,v_1)D_{k_1}(v_1,y_1)du_1dv_1.$$
Furthermore, $d_1(x_1,y_1)\geq C_1 2^{-k_1}$ implies $d_1(u_1,v_1)\geq
C_1^{'} d_1(x_1,y_1)$, where $C_1^{'}$ is a constant depending on
$C_0$ and $C_1$ since the support of $D_{k'_1}(x_1,u_1)$ is contained in $\{u_1: d_1(x_1,u_1)\leq C_02^{-k'_1}\}.$  Here $C_0$ is the constant given in  Definition
\ref{def-of-ATI compact}. Therefore, we can use the smoothness condition on the kernel $K_1(u_1,v_1).$ By the fact that
$\int D_{k'_1}(x_1,u_1)du_1=0,$ we write
\begin{eqnarray*}
&&\iint D_{k'_1}(x_1,u_1)K_1(u_1,v_1)D_{k_1}(v_1,y_1)du_1dv_1\\
&&=\int  D_{k'_1}(x_1,u_1)[K_1(u_1,v_1)-K_1(x_1,v_1)]D_{k_1}(v_1,y_1)du_1dv_1.
\end{eqnarray*}
Now applying the smoothness condition on the kernel $K_1$ yields
\begin{eqnarray*}
&&|\iint D_{k'_1}(x_1,u_1)K_1(u_1,v_1)D_{k_1}(v_1,y_1)du_1dv_1|\\
&&\leq C|K^1|_{CZ}\int\big(\frac{d_1(x_1,u_1)}{d_1(u_1,v_1)}\big)^{\varepsilon}V(u_1,v_1)^{-1}
|D_{k'_1}(x_1,u_1)||D_{k_1}(v_1,y_1)| du_1dv_1.
\end{eqnarray*}
Note that $d_1(u_1,v_1)\geq
C_1^{'} d_1(x_1,y_1)$ and $d_1(x_1,u_1)\leq C_02^{-k'_1}.$ The last integral is bounded by some constant times
\begin{eqnarray*}
 \Big(\frac{2^{-k'_1}}{d_1(x_1,y_1)}\Big)^{\varepsilon}V(x_1,y_1)^{-1}=2^{-(k'_1-k_1)\varepsilon}
\Big(\frac{2^{-k_1}}{d_1(x_1,y_1)}\Big)^{\varepsilon}V(x_1,y_1)^{-1},
\end{eqnarray*}
which gives the desired estimate when $k_1<k^{'}_1$ because $d_1(x_1,y_1)\geq C_1 2^{-k_1}$ implies
$V_{2^{-k_1}}(x_1)+V_{2^{-k_1}}(y_1)\leq CV(x_1,y_1).$

Now we consider $d_1(x_1,y_1)< C_1 2^{-k_1}.$ Note that for this case one can not apply the smoothness condition on the kernel $K_1$ to get the desired estimate as in the case $d_1(x_1,y_1)\geq C_1 2^{-k_1}$ because the variables $u_1$ and $v_1$ in the kernel $K_1(u_1,v_1)$ could be close. The weak boundedness property of $T_1$ can not be applied either since $D_{k_1}(v_1,y_1)-D_{k_1}(x_1,y_1),$ as the function of $v_1,$ has no compact support. Thus, we need to introduce a smooth cutoff function $\eta_1(x)\in C^1(\mathbb{R})$ so that
$\eta_1(x)=1$ when $|x|\leq 1$ and $\eta_1(x)=0$ when $|x|> 2$. And
set $\eta_2=1-\eta_1$. then
\begin{eqnarray*}
&&\big|\iint D_{k'_1}(x_1,u_1)K_1(u_1,v_1)[D_{k_1}(v_1,y_1)-D_{k_1}(x_1,y_1)]du_1dv_1\big|\\
&&=\int  D_{k'_1}(x_1,u_1)K_1(u_1,v_1)[D_{k_1}(v_1,y_1)-D_{k_1}(x_1,y_1)]\eta_1\Big({d_1(v_1,x_1)\over C_12^{-k'_1}}\Big)du_1dv_1\\
&&\hskip1cm+\int  D_{k'_1}(x_1,u_1)K_1(u_1,v_1)[D_{k_1}(v_1,y_1)-D_{k_1}(x_1,y_1)]\eta_2\Big({d_1(v_1,x_1)\over C_12^{-k'_1}}\Big)du_1dv_1\\
&&=: I+II.
\end{eqnarray*}

We will apply the weak boundedness property for term $I.$ For this purpose, setting
$\psi_{k_1}(v_1)=[D_{k_1}(v_1,y_1)-D_{k_1}(x_1,y_1)]\eta_1\Big({d_1(v_1,x_1)\over
C_12^{-k'_1}}\Big)$ we write term $I$ as
\begin{eqnarray*}
I=\langle D_{k'_1}(x_1,\cdot),
T_1\psi_{k_1}(\cdot)\rangle.
\end{eqnarray*}
Then the weak boundedness property of $T_1$ yields
\begin{eqnarray*}
|I|&\leq&   |\langle D_{k'_1}(x_1,\cdot), T_1\psi_{k_1}(\cdot)\rangle|\\[4pt]
&\leq&
C|K^1|_{CZ}V_{2^{-k'_1}}(x_1)2^{-k'_1\delta}\|D_{k'_1}(x_1,\cdot)\|_{\delta}\|\psi_{k_1}(\cdot)\|_{\delta}.
\end{eqnarray*}
It is easy to verify that $\|D_{k'_1}(x_1,\cdot)\|_\delta \leq
C2^{k'_1\delta}V_{2^{-k'_1}}(x_1)^{-1}$. We claim that
$\|\psi_{k_1}(\cdot)\|_{\delta}$ is bounded by $C 2^{k_1\delta}
2^{-(k'_1-k_1)\vartheta}V_{2^{-k_1}}(y_1)^{-1}$. In fact, using the
smoothness property of $D_{k_1}(v_1,y_1)$, we obtain
$$ \|\psi_{k_1}(\cdot)\|_{\infty}\leq C 2^{-(k'_1-k_1)\vartheta}V_{2^{-k_1}}(y_1)^{-1}. $$
Moreover,
\begin{eqnarray*}
&&|\psi_{k_1}(v)-\psi_{k_1}(v')|=
[D_{k_1}(v,y_1)-D_{k_1}(v',y_1)]\eta_1\Big({d_1(v,x_1)\over C_12^{-k'_1}}\Big)\\
&&\hskip2cm +          \Big[D_{k_1}(v',y_1)-D_{k_1}(x_1,y_1) \Big]
\Big[\eta_1\Big({d_1(v,x_1)\over
C_12^{-k'_1}}\Big)-\eta_1\Big({d_1(v',x_1)\over
C_12^{-k'_1}}\Big)\Big].
\end{eqnarray*}
Thus, using the smoothness property of  the kernel
$D_{k_1}(v_1,y_1)$ and smoothness property of the function $\eta_1$,
we can obtain that
$$\|\psi_{k_1}(\cdot)\|_{\delta}\leq C 2^{k_1\delta} 2^{-(k'_1-k_1)\vartheta}V_{2^{-k_1}}(y_1)^{-1}.$$
As a consequence of these estimates, we have
\begin{eqnarray*}
|I| &\leq&
C |K^1|_{CZ}V_{2^{-k'_1}}(x_1)2^{-k'_1\delta}2^{-k_1\delta}{2^{k'_1\delta}\over V_{2^{-k'_1}}(x_1)}2^{k_1\delta} 2^{-(k'_1-k_1)\vartheta}V_{2^{-k_1}}(y_1)^{-1}\\
&\leq& C |K^1|_{CZ}2^{-(k'_1-k_1)\vartheta}V_{2^{-k_1}}(y_1)^{-1},
\end{eqnarray*}
which is a desired estimate in this case since $\vartheta\geq
\varepsilon$.

We now deal with term $II$. Note that $d_1(x_1,u_1)\leq C_0 2^{-k'_1}$ and that
by the support of $\eta_2,$ $d_1(v_1,x_1)>C_12^{-k'_1}$, where $C_1$ is sufficiently large so
that $d_1(x_1,u_1)\leq C d_1(u_1,v_1)$. Therefore, we can apply the smoothness condition on the kernel $K_1.$ To this end, using the fact that  $\int
D_{k'_1}(x_1,u_1)du_1=0$, we write
\begin{eqnarray*}
II=\int  D_{k'_1}(x_1,u_1)\big[K_1(u_1,v_1)-K_1(x_1,v_1)\big][D_{k_1}(v_1,y_1)-D_{k_1}(x_1,y_1)]\eta_2\Big({d_1(v_1,x_1)\over C_12^{-k'_1}}\Big)du_1dv_1.
\end{eqnarray*}
Applying the smoothness condition on $K^1$ we obtain
\begin{eqnarray*}
|II|&\leq&
C|K^1|_{CZ}\int_{u_1: d_1(u_1,x_1)\leq C_02^{-k'_1}}\int_{v_1: d_1(v_1,x_1)> C_12^{-k'_1}}  \big(\frac{d_1(x_1,u_1)}{d_1(u_1,v_1)}\big)^{\varepsilon}\\[4pt]
&&\times V(u_1,v_1)^{-1}|D_{k'_1}(x_1,u_1)||D_{k_1}(v_1,y_1)-D_{k_1}(x_1,y_1)| du_1dv_1.
\end{eqnarray*}
Note that
$$ |D_{k_1}(v_1,y_1)-D_{k_1}(x_1,y_1)|\leq C V_{2^{-k_1}}(y_1)^{-1} $$
and
$$ |D_{k_1}(v_1,y_1)-D_{k_1}(x_1,y_1)|\leq C \Big( { d_1(x_1,v_1)\over 2^{-k_1}+d_1(x_1,y_1) } \Big)^{\varepsilon}V_{2^{-k_1}}(y_1)^{-1}$$
when $ d_1(x_1,v_1)\leq C_12^{-k_1}$.

Splitting the above last integral into
\begin{eqnarray*}
&&\int_{u_1:\ d_1(u_1,x_1)\leq C_02^{-k'_1}}\int_{v_1:\ d_1(v_1,x_1)> C_12^{-k_1}}  \big(\frac{d_1(x_1,u_1)}{d_1(u_1,v_1)}\big)^{\varepsilon}V(u_1,v_1)^{-1}\\[6pt]
&&\hskip1cm\times |D_{k'_1}(x_1,u_1)||D_{k_1}(v_1,y_1)-D_{k_1}(x_1,y_1)| du_1dv_1\\[4pt]
&&+\int_{u_1:\ d_1(u_1,x_1)\leq C_02^{-k'_1}}\int_{v_1:\ C_12^{-k_1} \geq d_1(v_1,x_1)> C_12^{-k'_1}} \big(\frac{d_1(x_1,u_1)}{d_1(u_1,v_1)}\big)^{\varepsilon}V(u_1,v_1)^{-1}\\[6pt]
&&\hskip1cm\times |D_{k'_1}(x_1,u_1)||D_{k_1}(v_1,y_1)-D_{k_1}(x_1,y_1)| du_1dv_1
\end{eqnarray*}
and applying the above two estimates for $|D_{k_1}(v_1,y_1)-D_{k_1}(x_1,y_1)|$ to above two integrals, respectively, yield
\begin{eqnarray*}
|II|&\leq&
C|K^1|_{CZ}V_{2^{-k_1}}(y_1)^{-1} 2^{-k'_1\varepsilon}2^{k_1\varepsilon}\\[4pt]
&&+C|K^1|_{CZ}V_{2^{-k_1}}(y_1)^{-1}2^{-k'_1\varepsilon}2^{k_1\varepsilon}\int_{v_1: C_12^{-k_1} \geq d_1(v_1,x_1)> C_12^{-k'_1}} V(x_1,v_1)^{-1}dv_1\\[4pt]
&\leq&C|K^1|_{CZ}V_{2^{-k_1}}(y_1)^{-1}
2^{-(k'_1-k_1)\varepsilon}\big(1+(k'_1-k_1)\big),
\end{eqnarray*}
which again is a desired estimate.

Now we turn to the present case, that is, the proof of the estimate in (\ref{almost orth case 1.1}). To see that this can be done by the iteration, we write
\begin{eqnarray*}
&&\int  D_{k'_1}(x_{I_1^{'}},u_1)D_{k'_2}(x_{I_2^{'}},u_2)K(u_1,u_2,v_1,v_2)[D_{k_1}(v_1,x_{I_1})-D_{k_1}(x_{I_1^{'}},x_{I_1})]\nonumber\\[4pt]
&&\hskip1cm\times [D_{k_2}(v_2,x_{I_2})-D_{k_2}(x_{I_2^{'}},x_{I_2})] du_1du_2dv_1dv_2\\[4pt]
&&=\big\langle D_{k'_2}(x_{I_2^{'}},u_2), \langle D_{k'_2}(x_{I_1^{'}},\cdot), K_2(u_2,v_2)[D_{k_1}(\cdot,x_{I_1})-D_{k_1}(x_{I_1^{'}},x_{I_1})]\rangle\\[4pt]
&&\hskip1cm\times[D_{k_2}(v_2,x_{I_2})-D_{k_2}(x_{I_2^{'}},x_{I_2})]\big\rangle,
\end{eqnarray*}
where, by definition of the product singular integral operator given in Subsection 3.1, for fixed points $u_2, v_2\in M_2, K_2(u_2,v_2)$ is a Calder\'on--Zygmund operator on $M_1$ with the operator norm $\|K_2(u_2,v_2)\|_{CZ(M_1)}$ which is a singular integral operator on $M_2.$ By the estimate for one parameter case provided above, for $k_1^{'}>k_1,$
\begin{eqnarray*}
&&|\langle D_{k'_1}(x_{I_1^{'}},\cdot), K_2(u_2,v_2)[D_{k_1}(\cdot,x_{I_1})-D_{k_1}(x_{I_1^{'}},x_{I_1})]\rangle|\\[4pt]
&&\leq C \|K_2(u_2,v_2)\|_{CZ(M_1)}2^{(k_1-k'_1)\epsilon}\frac{1}{V_{2^{-k_1}}(x_{I_1^{'}})+V_{2^{-k_1}}(x_{I_1})+V(x_{I_1^{'}},x_{I_1})}
\frac{2^{-k_1\varepsilon}}{(2^{-k_1}+d_1(x_{I_1^{'}},x_{I_1}))^{\varepsilon}}.
\end{eqnarray*}
Similarly,
\begin{eqnarray*}
&&|\langle D_{k'_1}(x_{I_1^{'}},\cdot), [K_2(u_2,v_2)-K_2(u_2,v_2^{'}][D_{k_1}(\cdot,x_{I_1})-D_{k_1}(x_{I_1^{'}},x_{I_1})]\rangle|\\[4pt]
&&\leq C \|K_2(u_2,v_2)-K_2(u_2,v_2^{'}]\|_{CZ(M_1)}2^{(k_1-k'_1)\epsilon}\\
&&\hskip1cm\times\frac{1}{V_{2^{-k_1}}(x_{I_1^{'}})+V_{2^{-k_1}}(x_{I_1})+V(x_{I_1^{'}},y_1)}
\frac{2^{-k_1\varepsilon}}{(2^{-k_1}+d_1(x_{I_1^{'}},x_{I_1}))^{\varepsilon}}
\end{eqnarray*}
and the same estimate holds with interchanging $u_2$ and $v_2.$

This together with the fact that $\|K_2(u_2,v_2)\|_{CZ(M_1)}$ is a singular integral operator on $M_2$ having the weak boundedness property implies that $\langle D_{k'_1}(x_{I_1^{'}},\cdot), K_2(u_2,v_2)[D_{k_1}(\cdot,x_{I_1})-D_{k_1}(x_{I_1^{'}},x_{I_1})]\rangle$ is a Calder\'on--Zygmund singular integral on $M_2$ having the weak boundedness property. Moreover,
\begin{eqnarray*}
&&|\langle D_{k'_1}(x_{I_1^{'}},\cdot), K_2(u_2,v_2)[D_{k_1}(\cdot,x_{I_1})-D_{k_1}(x_{I_1^{'}},x_{I_1})]\rangle|_{CZ}\\[4pt]
&&\leq C 2^{(k_1-k'_1)\epsilon}\frac{1}{V_{2^{-k_1}}(x_{I_1^{'}})+V_{2^{-k_1}}(x_{I_1})+V(x_{I_1^{'}},x_{I_1})}
\frac{2^{-k_1\varepsilon}}{(2^{-k_1}+d_1(x_{I_1^{'}},x_{I_1}))^{\varepsilon}}.
\end{eqnarray*}
Applying the estimate for one parameter case again yields that for $k_2^{'}>k_2,$
\begin{eqnarray*}
&&|\langle D_{k'_2}(x_{I_2^{'}},u_2), \langle D_{k'_1}(x_{I_1^{'}},\cdot), K_2(u_2,v_2)[D_{k_1}(\cdot,x_{I_1})-D_{k_1}(x_{I_1^{'}},x_{I_1})]\rangle
[D_{k_2}(v_2,x_{I_2})-D_{k_2}(x_{I_2^{'}},x_{I_2})]\rangle\\[4pt]
&&\leq C|\langle D_{k'_1}(x_{I_1^{'}},\cdot), K_2(u_2,v_2)[D_{k_1}(\cdot,x_{I_1})-D_{k_1}(x_{I_1^{'}},x_{I_1})]\rangle|_{CZ}\\[4pt]
&&\hskip1cm\times 2^{(k_2-k'_2)\varepsilon}
\frac{1}{V_{2^{-k_2}}(x_{I_2^{'}})+V_{2^{-k_2}}(y_2)+V(x_{I_2^{'}},y_2)}\frac{2^{-k_2\varepsilon}}{(2^{-k_2}+d_2(x_{I_2^{'}},y_2))^{\varepsilon}}\\[4pt]
&&\leq C2^{(k_1-k'_1)\varepsilon}2^{(k_2-k'_2)\varepsilon}
\frac{1}{V_{2^{-k_1}}(x_{I_1^{'}})+V_{2^{-k_1}}(x_{I_1})+V(x_{I_1^{'}},y_1)}\frac{2^{-k_1\varepsilon}}
{(2^{-k_1}+d_1(x_{I_1^{'}},x_{I_1}))^{\varepsilon}}\nonumber\\[4pt]
&&\hskip1cm\times \frac{1}{V_{2^{-k_2}}(x_{I_2^{'}})+V_{2^{-k_2}}(x_{I_2})+V(x_{I_2^{'}},x_{I_2})}\frac{2^{-k_2\varepsilon}}{(2^{-k_2}+d_2(x_{I_2^{'}},x_{I_2}))^{\varepsilon}},
\end{eqnarray*}
which concludes the proof of (\ref{almost orth case 1.1}).

Applying the Cauchy-Schwartz inequality implies that $|\langle g,Tf\rangle_{\rm Case\ 1.1}|$ is bounded by
\begin{eqnarray*}
&&\Big\lbrace\sum_{k_1\leq k_1^{'} }\sum_{k_2\leq k_2^{'}}\sum_{I_1^{'}}\sum_{I_2^{'}}
\sum_{I_1}\sum_{I_2}\mu_1(I_1^{'})\mu_1(I_1) \mu_2(I_2^{'})\mu_2(I_2) |{\widetilde{\widetilde{D}}}_{k_1^{'}}{\widetilde{\widetilde{D}}}_{k_2^{'}}(g)(x_{I_1^{'}},x_{I_2^{'}})|^2\\
&&\hskip1.5cm|I(x_{I_1^{'}},x_{I_2^{'}},x_{I_1},x_{I_2})|\Big\rbrace^{\frac{1}{2}}\\
&&\times \Big\lbrace\sum_{k_1\leq k_1^{'} }\sum_{k_2\leq k_2^{'}}\sum_{I_1^{'}}\sum_{I_2^{'}}
\sum_{I_1}\sum_{I_2}\mu_1(I_1^{'})\mu_1(I_1) \mu_2(I_2^{'})\mu_2(I_2) |{\widetilde{\widetilde{D}}}_{k_1}{\widetilde{\widetilde{D}}}_{k_2}(f)(x_{I_1},x_{I_2})|^2\\
&&\hskip1.5cm|I(x_{I_1^{'}},x_{I_2^{'}},x_{I_1},x_{I_2})|\Big\rbrace^{\frac{1}{2}}.
\end{eqnarray*}
Note that by the estimates for $|I(x_{I_1^{'}},x_{I_2^{'}},x_{I_1},x_{I_2})|$ in (\ref{almost orth case 1.1}) we have
$$\sum_{I_1^{'}}\sum_{I_2^{'}}\mu_1(I_1^{'})\mu_2(I_2^{'}) |I(x_{I_1^{'}},x_{I_2^{'}},x_{I_1},x_{I_2})|\leq C 2^{(k_1-k'_1)\epsilon}
2^{(k_2-k'_2)\epsilon}$$
and similarly
$$\sum_{I_1}\sum_{I_2}\mu_1(I_1)\mu_2(I_2) |I(x_{I_1^{'}},x_{I_2^{'}},x_{I_1},x_{I_2})|\leq C 2^{(k_1-k'_1)\epsilon}
2^{(k_2-k'_2)\epsilon}.$$
Therefore,
\begin{eqnarray*}
&&\sum_{k_1\leq k_1^{'} }\sum_{k_2\leq k_2^{'}}\sum_{I_1^{'}}\sum_{I_2^{'}}
\sum_{I_1}\sum_{I_2}\mu_1(I_1^{'})\mu_1(I_1) \mu_2(I_2^{'})\mu_2(I_2) |{\widetilde{\widetilde{D}}}_{k_1^{'}}{\widetilde{\widetilde{D}}}_{k_2^{'}}(g)(x_{I_1^{'}},x_{I_2^{'}})|^2
|I(x_{I_1^{'}},x_{I_2^{'}},x_{I_1},x_{I_2})|\\[4pt]
&&\leq C\sum_{k_1\leq k_1^{'} }\sum_{k_2\leq k_2^{'}}2^{(k_1-k'_1)\epsilon}
2^{(k_2-k'_2)\epsilon}\sum_{I_1^{'}}\sum_{I_2^{'}}
\mu_1(I_1^{'})\mu_2(I_2^{'})|{\widetilde{\widetilde{D}}}_{k_1^{'}}{\widetilde{\widetilde{D}}}_{k_2^{'}}(g)(x_{I_1^{'}},x_{I_2^{'}})|^2\\[4pt]
&&\leq C\sum_{k_1^{'} }\sum_{k_2^{'}}\sum_{I_1^{'}}\sum_{I_2^{'}}
\mu_1(I_1^{'})\mu_2(I_2^{'})|{\widetilde{\widetilde{D}}}_{k_1^{'}}{\widetilde{\widetilde{D}}}_{k_2^{'}}(g)(x_{I_1^{'}},x_{I_2^{'}})|^2.
\end{eqnarray*}
The last series above, by the discrete Littlewood--Paley $L^2$ estimate established in [HLL2], is dominated by the constant times $\|g\|_2^2.$ Similarly,
\begin{eqnarray*}
&&\sum_{k_1\leq k_1^{'} }\sum_{k_2\leq k_2^{'}}\sum_{I_1^{'}}\sum_{I_2^{'}}
\sum_{I_1}\sum_{I_2}\mu_1(I_1^{'})\mu_1(I_1) \mu_2(I_2^{'})\mu_2(I_2) |{\widetilde{\widetilde{D}}}_{k_1}{\widetilde{\widetilde{D}}}_{k_2}(f)(x_{I_1},x_{I_2})|^2
|I(x_{I_1^{'}},x_{I_2^{'}},x_{I_1},x_{I_2})|\\[4pt]
&&\leq C \|f\|_2^2.
\end{eqnarray*}
We thus conclude that $|\langle g,Tf\rangle_{\rm Case\ 1.1}|\leq C\|f\|_2\|g\|_2.$ The estimate for $|\langle g,Tf\rangle_{\rm Case\ 2.1}|$ is the same. Indeed, if we write
\begin{eqnarray*}
&&V(x_{I_1^{'}},x_{I_2^{'}},x_{I_1},x_{I_2})\\[4pt]
&&=\int  D_{k'_1}(x_{I_1^{'}},u_1)[D_{k'_2}(x_{I_2^{'}},u_2)-D_{k'_2}(x_{I_2^{'}},x_{I_2})]K(u_1,u_2,v_1,v_2)[D_{k_1}(v_1,x_{I_1})-D_{k_1}(x_{I_1^{'}},x_{I_1})]\\[4pt]
&&\hskip1cm\times D_{k_2}(v_2,x_{I_2})du_1du_2dv_1dv_2\\[4pt]
&&=\big\langle [D_{k'_2}(x_{I_2^{'}},u_2)-D_{k'_2}(x_{I_2^{'}},x_{I_2})],\\[4pt]
&&\hskip2cm
\langle D_{k'_1}(x_{I_1^{'}},u_1), K_2(u_2,,v_2)[D_{k_1}(v_1,x_{I_1})-D_{k_1}(x_{I_1^{'}},x_{I_1})]\rangle D_{k_2}(v_2,x_{I_2})\big\rangle
\end{eqnarray*}
and repeat the same proof, it is not difficult to see that $V(x_{I_1^{'}},x_{I_2^{'}},x_{I_1},x_{I_2})$ satisfies the same estimate in (\ref{almost orth case 1.1}) as for $I(x_{I_1^{'}},x_{I_2^{'}},x_{I_1},x_{I_2})$ with interchanging $k_2$ and $k'_2.$ As a result,
$$|\langle g,Tf\rangle_{\rm Case\ 2.1}|\leq C\|f\|_2\|g\|_2.$$

\subsubsection{Carleson measure on $\widetilde M=M_1\times M_2$ }
In this subsection, we handle bilinear form $\langle g, Tf\rangle_{Case 1.4}.$ The estimate of this term will be achieved by applying the Carleson measure estimate on $\widetilde M=M_1\times M_2.$ To see this, we first write
\begin{eqnarray*}
&&IV(x_{I_1^{'}},x_{I_2^{'}},x_{I_1},x_{I_2})\\
&&=\int  D_{k'_1}(x_{I_1^{'}},u_1)D_{k'_2}(x_{I_2^{'}},u_2)K(u_1,u_2,v_1,v_2)D_{k_1}(x_{I_1^{'}},x_{I_1})D_{k_2}(x_{I_2^{'}},x_{I_2}) du_1du_2dv_1dv_2\\
&&=D_{k'_1}D_{k'_2}(T1)(x_{I_1^{'}},x_{I_2^{'}})D_{k_1}(x_{I_1^{'}},x_{I_1})D_{k_2}(x_{I_2^{'}},x_{I_2}).
\end{eqnarray*}
And then we rewrite $\langle g,Tf\rangle_{\rm Case\ 1.4}$ by
$$ \sum_{k_1^{'} }\sum_{k_2^{'}}\sum_{I_1^{'}}\sum_{I_2^{'}}
\mu_1(I_1^{'}) \mu_2(I_2^{'}){\widetilde{\widetilde{D}}}_{k_1^{'}}{\widetilde{\widetilde{D}}}_{k_2^{'}}(g)(x_{I_1^{'}},x_{I_2^{'}})
D_{k'_1}D_{k'_2}(T1)(x_{I_1^{'}},x_{I_2^{'}}){S}_{k_1^{'}}{S}_{k_2^{'}}(f)(x_{I_1^{'}},x_{I_2^{'}}),$$
where for $x_1, y_1\in M_1,$
$${S}_{k_1^{'}}(x_1,y_1)=\sum_{k_1\leq k_1^{'} }\sum\limits_{I_1}\mu(I_{1})
D_{k_1}(x_1,x_{I_{1}}){\widetilde{\widetilde{D}}}_{k_1}(x_{I_{1}},y_1)$$
and similarly for ${S}_{k_2^{'}}(x_2,y_2)$ on $M_2.$

In order to apply the Carleson measure estimate to $\langle g,Tf\rangle_{\rm Case\ 1.4},$ we claim that ${S}_{k_1^{'}}(x_1,y_1),$ the kernel of ${S}_{k_1^{'}},$ satisfies the following estimate
\begin{eqnarray*}
|{S}_{k_1^{'}}(x_1,y_1)|\leq C {1\over V_{2^{-k_1^{'}}}(x_1)+V_{2^{-k_1^{'}}}(y_1)+V(x_1,y_1)}
\Big({2^{-k_1^{'}}\over 2^{-k_1^{'}}+d_1(x_1,y_1)}\Big)^{\vartheta'}.
\end{eqnarray*}
Similarly, ${S}_{k_2^{'}}(x_2,y_2),$ the kernel of ${S}_{k_2^{'}},$ satisfies the same estimate above with interchanging $k_1^{'}, k_2^{'}; x_1, x_2$ and $y_1, y_2,$ respectively.

Assuming the claim for the moment, then applying the Cauchy-Schwartz inequality yields
\begin{eqnarray}
&&|\langle g,Tf\rangle_{\rm Case\ 1.4}|\nonumber\\[4pt]
&&\leq \big\lbrace \sum_{k_1^{'} }\sum_{k_2^{'}}\sum_{I_1^{'}}\sum_{I_2^{'}}
\mu_1(I_1^{'}) \mu_2(I_2^{'})|{\widetilde{\widetilde{D}}}_{k_1^{'}}{\widetilde{\widetilde{D}}}_{k_2^{'}}(g)(x_{I_1^{'}},x_{I_2^{'}})|^2\big\rbrace^{\frac{1}{2}}\\[4pt]
&&\hskip1cm\times \big\lbrace \sum_{k_1^{'} }\sum_{k_2^{'}}\sum_{I_1^{'}}\sum_{I_2^{'}}
\mu_1(I_1^{'}) \mu_2(I_2^{'})|D_{k'_1}D_{k'_2}(T1)(x_{I_1^{'}},x_{I_2^{'}})|^2|{S}_{k_1^{'}}{S}_{k_2^{'}}(f)(x_{I_1^{'}},x_{I_2^{'}})|^2\big\rbrace^{\frac{1}{2}}.\nonumber
\end{eqnarray}

Thus, the first series above, by the discrete Littlewood--Paley $L^2,$ is bounded by a constant times $\|g\|_2.$ And the second series is bounded by $C\|f\|_2$ by applying the Carleson measure estimate on $\widetilde M$ since $T1\in BMO(\widetilde M)$ and hance $\mu_1(I_1^{'}) \mu_2(I_2^{'})|D_{k'_1}D_{k'_2}(T1)(x_1,x_2)|^2$ is a Carleson measure on $\widetilde M\times \lbrace \mathbb Z\times \mathbb Z\rbrace.$

We now show the claim. To do this, we first consider the case when $d_1(x_1,y_1)< 2^{-k_1^{'}}.$ Then
\begin{eqnarray}
&&\big|\sum_{k_1\leq k_1^{'},\ d_1(x_1,y_1)< 2^{-k_1^{'}} }\sum_{I_1} \mu_1(I_{1})D_{k_1}(x_1,x_{I_{1}})
 {\widetilde{\widetilde{D}}}_{k_1}(x_{I_{1}},y_1) \big|\label{key
claim e1}\\[4pt]
 &&\leq C\sum_{k_1\leq k_1^{'},\ d_1(x_1,y_1)< 2^{-k_1^{'}} } {1\over V_{2^{-k_1}}(x_1)+V_{2^{-k_1}}(y_1)+V(x_1,y_1)} \Big({2^{-k_1}\over 2^{-k_1}+d_1(x_1,y_1)}\Big)^{\vartheta'}\nonumber\\[4pt]
&&\leq C {1\over V_{2^{-k_1^{'}}}(x_1)+V_{2^{-k_1^{'}}}(y_1)+V(x_1,y_1)}
\Big({2^{-k_1^{'}}\over 2^{-k_1^{'}}+d_1(x_1,y_1)}\Big)^{\vartheta'},\nonumber
\end{eqnarray}
where $\vartheta'$ is the order of ${\widetilde{\widetilde{D}}}_{k_1}(x_1,y_1)$. Next,
we consider the case when $d_1(x_1,y_1)\geq 2^{-k_1^{'}}$. Note first that by the
discrete Calder\'on's identity in [HLL2],
\begin{eqnarray*}
\sum_{k_1\leq k_1^{'} }\sum_{I_1} \mu_1(I_{1})D_{k_1}(x_1,x_{I_{1}})
 {\widetilde{\widetilde{D}}}_{k_1}(f)(x_{I_{1}}) + \sum_{k_1> k_1^{'} }\sum_{I_1} \mu_1(I_{1})D_{k_1}(x_1,x_{I_{1}})
 {\widetilde{\widetilde{D}}}_{k_1}(f)(x_{I_{1}}) =f(x_1)
\end{eqnarray*}
for all test functions $f\in\GG(\beta,\gamma)(M_1)$ and the series converge in the norm of $\GG(\beta,\gamma).$ This implies that
\begin{eqnarray}
&&\sum_{k_1\leq k_1^{'} }\sum_{I_1} \mu_1(I_{1})D_{k_1}(x_1,x_{I_{1}})
 {\widetilde{\widetilde{D}}}_{k_1}(x_{I_{1}},y_1) \nonumber\\[4pt]
&& + \sum_{k_1> k_1^{'} }\sum_{I_1} \mu_1(I_{1})D_{k_1}(x_1,x_{I_{1}})
 {\widetilde{\widetilde{D}}}_{k_1}(x_{I_{1}},y_1) = \delta(x_1,y_1), \label{dirac
 function}
\end{eqnarray}
where we use $\delta$ to denote the Dirac function. Consequently, when $d_1(x_1,y_1)\geq 2^{-k_1^{'}}$,
\begin{eqnarray*}
&&\big|\sum_{k_1\leq k_1^{'}, d_1(x_1,y_1)\geq 2^{-k_1^{'}} }\sum_{I_1} \mu_1(I_{1})D_{k_1}(x_1,x_{I_{1}})
 {\widetilde{\widetilde{D}}}_{k_1}(x_{I_{1}},y_1)\big|\label{key claim e2} \\
&&=\big|\sum_{k_1> k_1^{'}, d_1(x_1,y_1)\geq 2^{-k_1^{'}} }\sum_{I_1}
\mu_1(I_{1})D_{k_1}(x_1,x_{I_{1}})
 {\widetilde{\widetilde{D}}}_{k_1}(x_{I_{1}},y_1)\big|\nonumber\\
 &&\leq C {1\over V_{2^{-k_1^{'}}}(x_1)+V_{2^{-k_1^{'}}}(y_1)+V(x_1,y_1)} \Big({2^{-k_1^{'}}\over
 2^{-k_1^{'}}+d_1(x_1,y_1)}\Big)^{\vartheta'},\nonumber
\end{eqnarray*}
where the last inequality follows from similar estimates in 3.24 and hence the claim is proved.

\subsubsection{Almost orthogonality argument on $M_1$ and Carleson measure estimate on $M_2$ }
In this subsection, we estimate $\langle g,Tf\rangle_{\rm Case\ 1.2}, \langle g,Tf\rangle_{\rm Case\ 1.3}, \langle g,Tf\rangle_{\rm Case\ 2.2}$ and $\langle g,Tf\rangle_{\rm Case\ 2.3}.$ Since all proofs for $\langle g,Tf\rangle_{\rm Case\ 1.3}, \langle g,Tf\rangle_{\rm Case\ 2.2}$ and $\langle g,Tf\rangle_{\rm Case\ 2.3}$ are similar to the proof of $\langle g,Tf\rangle_{\rm Case\ 1.2},$ so we only give the proof for $\langle g,Tf\rangle_{\rm Case\ 1.2}.$ We first write
\begin{eqnarray*}
&&II(x_{I_1^{'}},x_{I_2^{'}},x_{I_1},x_{I_2})\\[4pt]
&&=\int  D_{k'_1}(x_{I_1^{'}},u_1)D_{k'_2}(x_{I_2^{'}},u_2)K(u_1,u_2,v_1,v_2)\\[4pt]
&&\hskip1cm\times[D_{k_2}(v_2,x_{I_2})-D_{k_2}(x_{I_2^{'}},x_{I_2})] du_1du_2dv_1dv_2D_{k_1}(x_{I_1^{'}},x_{I_1}) \ +IV(x_{I_1^{'}},x_{I_2^{'}},x_{I_1},x_{I_2})\\[4pt]
&&=\langle D_{k'_2}(x_{I_2^{'}},u_2),\langle D_{k'_1}(x_{I_1^{'}},\cdot), K_2(u_2,v_2)(1)\rangle[D_{k_2}(v_2,x_{I_2})-D_{k_2}(x_{I_2^{'}},x_{I_2})]\rangle D_{k_1}(x_{I_1^{'}},x_{I_1})\\[4pt]
&&\hskip.5cm+IV(x_{I_1^{'}},x_{I_2^{'}},x_{I_1},x_{I_2}).
\end{eqnarray*}
Set
\begin{eqnarray*}
&&J_{k'_2,k_2}(u_2,v_2)\\
&&=\sum_{k'_1}\sum_{I_1^{'}}\mu_1(I_1^{'}){\widetilde{\widetilde{D}}}_{k_1^{'}}\big({\widetilde{\widetilde{D}}}_{k_2^{'}}(g)(\cdot,x_{I_2^{'}})\big)(x_{I_1^{'}}) \langle D_{k'_1}(x_{I_1^{'}},\cdot), K_2(u_2,v_2)(1)\rangle
S_{k'_1}\big({\widetilde{\widetilde{D}}}_{k_2}(f)(\cdot,x_{I_2})\big)(x_{I_1^{'}}),
\end{eqnarray*}
where $S_{k'_1}$ is defined as in Subsection 3.3.2.

Then, as in Subsection 3.3.2, summing up for $k'_1$ and $I_1^{'}$ and using the notation $J_{k'_2,k_2}(u_2,v_2),$ we can rewrite $\langle g,Tf\rangle_{\rm Case\ 1.2}$ as
\begin{eqnarray*}
&&\langle g,Tf\rangle_{\rm Case\ 1.2}\\[4pt]
&&= \sum_{k_2\leq k_2^{'}}\sum_{I_2^{'}}
\sum_{I_2}\mu_2(I_2^{'})\mu_2(I_2)\int {\widetilde{\widetilde{D}}}_{k_2^{'}}(x_{I_2^{'}},u_2)J_{k'_2,k_2}(u_2,v_2)[D_{k_2}(v_2,x_{I_2})-D_{k_2}(x_{I_2^{'}},x_{I_2})]
du_2dv_2 \nonumber\\[4pt]
&&\hskip.5cm+\langle g,Tf\rangle_{\rm Case\ 1.4}.
\end{eqnarray*}
Therefore, it suffices to estimate the above series since the estimate $|\langle g,Tf\rangle_{\rm Case\ 1.4}|\leq C\|f\|_2\|g\|_2$ has been proved in Subsection 3.3.1. To this end, we claim that for fixed $k'_2$ and $k_2,$ $J_{k'_2,k_2}(u_2,v_2)$ is a Calder\'on--Zygmund singular integral kernel on $M_2$ and the corresponding operator has the weak boundedness property. Moreover,
\begin{eqnarray}\label{CZ operator norm of J}
|J_{k'_2,k_2}(u_2,v_2)|_{CZ}\leq C\|{\widetilde{\widetilde{D}}}_{k_2^{'}}(g)(\cdot,x_{I_2^{'}})\|_2\|{\widetilde{\widetilde{D}}}_{k_2}(f)(\cdot,x_{I_2})\|_2.
\end{eqnarray}
Assuming the claim for the moment, by the almost orthogonality argument as in Subsection 3.3.1 we obtain

\begin{eqnarray*}
&&|\sum_{k_2\leq k_2^{'}}\sum_{I_2^{'}}
\sum_{I_2}\mu_2(I_2^{'})\mu_2(I_2)\int {\widetilde{\widetilde{D}}}_{k_2^{'}}(x_{I_2^{'}},u_2)J_{k'_2,k_2}(u_2,v_2)[D_{k_2}(v_2,x_{I_2})-D_{k_2}(x_{I_2^{'}},x_{I_2})]
du_2dv_2| \nonumber\\[4pt]
&&\leq C\sum_{k_2\leq k_2^{'}}\sum_{I_2^{'}}
\sum_{I_2}\mu_2(I_2^{'})\mu_2(I_2)|J_{k'_2,k_2}(u_2,v_2)|_{CZ}\\[4pt]
&&\hskip.5cm\times 2^{-(k_2-k'_2)\varepsilon}\frac{1}{V_{2^{-k_2}}(x_{I_2^{'}})+V_{2^{-k_2}}(x_{I_2})+V(x_{I_2^{'}},x_{I_2})}\frac{2^{-k_2\varepsilon}}
{(2^{-k_2}+d_2(x_{I_2^{'}},x_{I_2}))^{\varepsilon}}
\end{eqnarray*}
which, by a similar estimate as in Subsection 3.3.1, implies that the above series is dominated by a constant times
\begin{eqnarray*}
&&\sum_{k_2\leq k_2^{'}}\sum_{I_2^{'}}
\sum_{I_2}\mu_2(I_2^{'})\mu_2(I_2)2^{-(k_2-k'_2)\varepsilon}
\frac{1}{V_{2^{-k_2}}(x_{I_2^{'}})+V_{2^{-k_2}}(x_{I_2})+V(x_{I_2^{'}},x_{I_2})}\\[4pt]
&&\hskip.5cm\times\frac{2^{-k_2\varepsilon}}
{(2^{-k_2}+d_2(x_{I_2^{'}},x_{I_2}))^{\varepsilon}}\|{\widetilde{\widetilde{D}}}_{k_2^{'}}(g)(\cdot,x_{I_2^{'}})\|_2\|{\widetilde{\widetilde{D}}}_{k_2}(f)(\cdot,x_{I_2})\|_2\\[4pt]
&&\leq C\|f\|_2\|g\|_2.
\end{eqnarray*}

Now we prove the claim for $J_{k'_2,k_2}(u_2,v_2)$. We first denote by $ J_{k'_2,k_2} $ the operator on $M_2$ associated with the kernel $J_{k'_2,k_2}(u_2,v_2)$. We verify that $ J_{k'_2,k_2} $ satisfies the weak boundedness property. In fact, using the weak boundedness property of $T$ on $M_2$, that is, (\ref{wbp2}) and the one-parameter discrete Carleson measure estimate,
we have
\begin{eqnarray*}
|\langle J_{k_2^{'},k_2}\phi^2, \psi^2\rangle|&\leq& CV_{r_2}(x_2^0)\|{\widetilde{\widetilde{D}}}_{k_2^{'}}(g)(\cdot,x_{I_2^{'}})\|_{L^2(M_1)}\|{\widetilde{\widetilde{D}}}_{k_2^{'}}(f)(\cdot,x_{I_2^{'}})\|_{L^2(M_1)}
\end{eqnarray*}
for all $\phi^2,
\psi^2\in A_{M_2}(\delta, x_2^0, r_2)$, where the set $A_{M_2}(\delta, x_2^0, r_2)$ is defined in Subsection 3.1. Next we verify that $J_{k'_2,k_2}(u_2,v_2)$ satisfies the size and smoothness properties as defined in Subsection 3.1.
Using the one-parameter discrete Carleson measure estimate again we can obtain  that
\begin{eqnarray*}
|J_{k_2^{'},k_2}(u_2,v_2)|&\leq& C\|K_2( u_2,v_2)(1)\|_{BMO(M_1)}\|{\widetilde{\widetilde{D}}}_{k_2^{'}}(g)(\cdot,x_{I_2^{'}})\|_{L^2(M_1)}\|{\widetilde{\widetilde{D}}}_{k_2^{'}}(f)(\cdot,x_{I_2^{'}})\|_{L^2(M_1)}\\[4pt]
&\leq&C\|K_2(u_2,v_2)(1)\|_{CZ}\|{\widetilde{\widetilde{D}}}_{k_2^{'}}(g)(\cdot,x_{I_2^{'}})\|_{L^2(M_1)}\|{\widetilde{\widetilde{D}}}_{k_2^{'}}(f)(\cdot,x_{I_2^{'}})\|_{L^2(M_1)}\\[4pt]
&\leq&C{1\over V(u_2,v_2)}\|{\widetilde{\widetilde{D}}}_{k_2^{'}}(g)(\cdot,x_{I_2^{'}})\|_{L^2(M_1)}\|{\widetilde{\widetilde{D}}}_{k_2^{'}}(f)(\cdot,x_{I_2^{'}})\|_{L^2(M_1)}.
\end{eqnarray*}
Similarly,
\begin{eqnarray*}
&&|J_{k_2^{'},k_2}(u_2,v_2)-h_{k_2^{'},k_2}(u_2^{'},v_2)|\\
&\leq& C\|K_2( u_2,v_2)(1)-K_2( u_2^{'},v_2)(1)\|_{CZ)}\|{\widetilde{\widetilde{D}}}_{k_2^{'}}(g)(\cdot,x_{I_2^{'}})\|_{L^2(M_1)}\|{\widetilde{\widetilde{D}}}_{k_2^{'}}(f)(\cdot,x_{I_2^{'}})\|_{L^2(M_1)}\\
&\leq&C\Big({d_2(u_2,u_2^{'})\over d_2(u_2,v_2)}\Big)^{\varepsilon}{1\over V(u_2,v_2)}\|{\widetilde{\widetilde{D}}}_{k_2^{'}}(g)(\cdot,x_{I_2^{'}})\|_{L^2(M_1)}\|{\widetilde{\widetilde{D}}}_{k_2^{'}}(f)(\cdot,x_{I_2^{'}})\|_{L^2(M_1)}
\end{eqnarray*}
for $d_2(u_2,u_2^{'})\leq {1\over 2A} d_2(u_2,v_2)$. The same estimate holds with $u_2$ and $v_2$ interchanged. Combining the estimates above, we get that $J_{k'_2,k_2}(u_2,v_2)$ is a Calder\'on--Zygmund singular integral kernel on $M_2$ and hence (\ref{CZ operator norm of J}) holds. The claim is concluded.

\subsubsection{The Littlewood--Paley estimate on $M_1$ and Carleson measure estimate on $M_2$}

In this subsection, we deal with $\langle g,Tf\rangle_{\rm Case\ 2.4}$.
We first write
\begin{eqnarray*}
&&VIII(x_{I_1^{'}},x_{I_2^{'}},x_{I_1},x_{I_2})\\[4pt]
&&\hskip.4cm =- \int  D_{k'_1}(x_{I_1^{'}},u_1)D_{k'_2}(x_{I_2^{'}},x_{I_2})K(u_1,u_2,v_1,v_2)D_{k_1}(x_{I_1^{'}},x_{I_1})D_{k_2}(v_2,x_{I_2}) du_1du_2dv_1dv_2\\[4pt]
&&\hskip.4cm =-D_{k'_1}D_{k_2}\big((\widetilde{T})^*1\big)(x_{I_1^{'}},x_{I_2}) D_{k_1}(x_{I_1^{'}},x_{I_1})D_{k'_2}(x_{I_2^{'}},x_{I_2}).
\end{eqnarray*}
We would like to point out that the partial adjoint operator $\widetilde{T}$ appears and will play a crucial role in the estimate for $\langle g,Tf\rangle_{\rm Case\ 2.4}.$ This is why $\widetilde{T}$ and $\widetilde{T}^*$ have to be taken into account in the proof of the sufficient conditions of the product $T1$ theorem.

To estimate $\langle g,Tf\rangle_{\rm Case\ 2.4}$ we rewrite
\begin{eqnarray*}
&&\langle g,Tf\rangle_{\rm Case\ 2.4}\\[4pt]
&=& -\sum_{k_1\leq k_1^{'} }\sum_{k_2> k_2^{'}}\sum_{I_1^{'}}\sum_{I_2^{'}}
\sum_{I_1}\sum_{I_2}\mu_1(I_1^{'})\mu_1(I_1) \mu_2(I_2^{'})\mu_2(I_2) D_{k'_2}(x_{I_2^{'}},x_{I_2}){\widetilde{\widetilde{D}}}_{k_1^{'}}{\widetilde{\widetilde{D}}}_{k_2^{'}}(g)(x_{I_1^{'}},x_{I_2^{'}})
\\[4pt]
&&\times D_{k_1}(x_{I_1^{'}},x_{I_1}){\widetilde{\widetilde{D}}}_{k_1}{\widetilde{\widetilde{D}}}_{k_2}(f)(x_{I_1},x_{I_2})
D_{k'_1}D_{k_2}\big((\widetilde{T})^*1\big)(x_{I_1^{'}},x_{I_2}) \\[4pt]
&=&-\sum_{ k_1^{'} }\sum_{k_2}\sum_{I_1^{'}}
\sum_{I_2}\mu_1(I_1^{'})\mu_2(I_2){\widetilde{\widetilde{D}}}_{k'_1}S_{k_2}(g)(x_{I_1^{'}},x_{I_2})S_{k'_1}{\widetilde{\widetilde{D}}}_{k_2}(f)(x_{I'_1},x_{I_2})D_{k'_1}D_{k_2}\big((\widetilde{T})^*1\big)(x_{I_1^{'}},x_{I_2}),
\end{eqnarray*}
where the operators $S_{k'_1}$ and $S_{k_2}$ are defined as in Subsection 3.3.2.

In order to estimate the last series above, for a $BMO(\widetilde M)$ function $b$ we introduce an operator $W_{b}$ by the bilinear form $\langle g, W_{b}f\rangle$ which equals
\begin{eqnarray*}
 \sum_{ k_1^{'} }\sum_{k_2}\sum_{I_1^{'}}
\sum_{I_2}\mu_1(I_1^{'})\mu_2(I_2){\widetilde{\widetilde{D}}}_{k'_1}S_{k_2}(g)(x_{I_1^{'}},x_{I_2})S_{k'_1}
{\widetilde{\widetilde{D}}}_{k_2}(f)(x_{I'_1},x_{I_2})D_{k'_1}D_{k_2}\big(b\big)(x_{I_1^{'}},x_{I_2}).
\end{eqnarray*}
It is easy to see that when $b=(\widetilde{T})^*1 \in BMO(\widetilde{M})$, then $\langle g, W_{b}f\rangle=-\langle g,Tf\rangle_{\rm Case\ 2.4}$. Thus, we only need to show that for each $b\in BMO(\widetilde M)$ the operator $W_b$ is bounded on $L^2,$ which would imply that $|\langle g,Tf\rangle_{\rm Case\ 2.4}|\leq C\|f\|_2\|g\|_2.$ For this purpose, following an idea in [J] and interchanging the positions of functions $f$ and $b$ we define the operator $V_f(b)=W_b(f)$ and will prove that for each fixed $f\in L^\infty$ the operator $V_f$ is a Calder\'on--Zygmund singular integral operator and bounded on $L^2.$ Moreover, there exists a constant $C$ independent of $f$ such that for all $b\in L^2,$
$$\|V_f(b)\|_2\leq C\|f\|_\infty \|b\|_2.$$
Furthermore, we will show that $V_f$ satisfies the conditions in Theorem C below in Section 4 and thus, $V_f$ is also bounded on $BMO(\widetilde M)$ satisfying
$$\|V_f(b)\|_{BMO}\leq C\|f\|_\infty \|b\|_{BMO}.$$
We can rewrite the above estimate by
$$\|W_b(f)\|_{BMO}\leq C\|f\|_\infty \|b\|_{BMO}$$
for each $b\in BMO(\widetilde M)$ and all $f\in L^\infty.$

This means that for each $b\in BMO(\widetilde M)$ the operator $W_b$ is a bounded operator from $L^\infty$ to $BMO(\widetilde M).$
Similarly, the operator $W_b^*,$ the adjoint operator of $W_b,$ is a bounded operator from $L^\infty$ to $BMO(\widetilde M)$ since $W_b$ and $W_b^*$ satisfy the same conditions. Finally, by the duality argument and interpolation, $W_b$ is bounded on $L^2$ and hence, as mentioned, the bilinear form $\langle g,Tf\rangle_{\rm Case\ 2.4}$ is bounded by the constant times $\|f\|_2\|g\|_2.$

To achieve this goal, we will show that for each fixed $f\in L^\infty, V_f$ is a Calder\'on--Zygmund singular integral operator as defined in Subsection 3.1 and moreover, there exists a constant $C$ independent of $f$ and $b\in L^2$ such that
$$\|V_f(b)\|_2\leq C \|f\|_\infty\|b\|_2.$$
We first prove that $V_f$ is bounded on $L^2.$ To this end, for $g\in L^2,$ we write
\begin{eqnarray*}
&&\langle g, V_f(b)\rangle\\[4pt]
&=&\sum_{ k_1^{'} }\sum_{k_2}\sum_{I_1^{'}}
\sum_{I_2}\mu_1(I_1^{'})\mu_2(I_2){\widetilde{\widetilde{D}}}_{k'_1}S_{k_2}(g)(x_{I_1^{'}},x_{I_2})S_{k'_1}
{\widetilde{\widetilde{D}}}_{k_2}(f)(x_{I'_1},x_{I_2})D_{k'_1}D_{k_2}\big(b\big)(x_{I_1^{'}},x_{I_2}).
\end{eqnarray*}
Note that if $f\in L^\infty$ then $S_{k_1^{'}}(f)(x_{I_1^{'}},\cdot)$ is also a bounded function on $M_2$ for fixed $k_1^{'}$ and $I_1^{'}$ with
$$\|S_{k_1^{'}}(f)(x_{I_1^{'}},\cdot)\|_\infty\leq C\|f\|_\infty.$$
Thus, $\mu_2(I_2)|{\widetilde{\widetilde{D}}}_{k_2}\big(S_{k_1^{'}}(f)(x_{I_1^{'}},\cdot)\big)(x_{I_2})|^2$ is a Carleson measure on $M_2\times k_2$ uniformly for
all $k_1^{'}$ and $x_{I_1^{'}}
\in M_1.$ Therefore,
\begin{eqnarray*}
\big|\langle g,
V_f(b)\rangle\big|&=&\bigg|\sum_{k_1^{'}}\sum_{I_1^{'}}
\mu_1(I_{1}^{'})\bigg[\sum_{k_2}\sum_{\tau_2
}\mu_2(I_2)
S_{k_2}\big({\widetilde{\widetilde{D}}}_{k_1}(g)(x_{I_1^{'}},\cdot)\big)(x_{I_2})D_{k_2}\big(D_{k_1^{'}}(b)(x_{I_1^{'}},\cdot)\big)(x_{I_2})\\[4pt]
&&\times
{\widetilde{\widetilde{D}}}_{k_2}\big(S_{k_1^{'}}(f)(x_{I_1^{'}},\cdot)\big)(x_{I_2})\bigg]\bigg|\\[4pt]
&\leq& \sum_{k_1^{'}}\sum_{I_1^{'}}\mu_1(I_{1}^{'})
\|{\widetilde{\widetilde{D}}}_{k_1}(g)(x_{I_1^{'}},\cdot)\|_{L^2(M_2)}
\|D_{k_1^{'}}(b)(x_{I_1^{'}},\cdot)\|_{L^2(M_2)}\|S_{k_1^{'}}(f)(x_{I_1^{'}},\cdot)\|_{L^\infty(M_2)}\\[4pt]
&\leq& C\|f\|_{L^\infty(\widetilde{M})}\Big(\sum_{k_1^{'}}\sum_{I_1^{'}}\mu_1(I_{1}^{'})
\|{\widetilde{\widetilde{D}}}_{k_1}(g)(x_{I_1^{'}},\cdot)\|_{L^2(M_2)}^2\Big)^{1/2}\\[4pt]
&&\hskip.5cm\times
\Big(\sum_{k_1^{'}}\sum_{I_1^{'}}\mu_1(I_{1}^{'})
\|D_{k_1^{'}}(b)(x_{I_1^{'}},\cdot)\|_{L^2(M_2)}^2\Big)^{1/2}\\[4pt]
&\leq& C\|f\|_{L^\infty(\widetilde{M})}\|g\|_{L^2(\widetilde{M})}\|b\|_{L^2(\widetilde{M})},
\end{eqnarray*}
which, by taking the supremum for all $\|g\|_2\leq 1,$ implies that $V_f$ is bounded on $L^2(\widetilde{M})$ with $\|V_f\|_{L^2\rightarrow L^2}\leq C\|f\|_{L^\infty}.$

To verify that $V_f$ is a Carlder\'on-Zygmund singular integral operator as defined in Subsection 3.1, we can consider $V_f$ as a pair $\big((V_f)_1,(V_f)_2\big)$ of operators on $M_2$ and $M_1$, respectively, such that
$$\langle g_1\otimes g_2, V_f h_1\otimes h_2\rangle = \iint g_1(x_1)\langle g_2, (V_f)_1(x_1,y_1)h_2 \rangle h_1(y_1) dx_1dy_1$$
for all $g_1,h_1 \in C_0^\eta(M_1)$ and $g_2,h_2 \in C_0^\eta(M_2)$ with supp$g_1\cap$supp$h_1=\emptyset$ and
$$\langle g_1\otimes g_2, V_f h_1\otimes h_2\rangle = \iint g_2(x_2)\langle g_1, (V_f)_2(x_2,y_2)h_1 \rangle h_2(y_2) dx_2dy_2$$
for all $g_1,h_1 \in C_0^\eta(M_1)$ and $g_2,h_2 \in C_0^\eta(M_2)$ with supp$g_2\cap$supp$h_2=\emptyset$.

It suffices to show that $(V_f)_i(x_i,y_i), i=1,2,$ satisfies the properties $(i)$, $(ii)$ and $(iii)$ in Subsection 3.1.
We need only to verify $(V_f)_1(x_1,y_1)$ since the estimates for $(V_f)_2(x_2,y_2)$ are similar.

Note that for any fixed $x_1,y_1$ on $M_1$, $(V_f)_1(x_1,y_1)$ is an operator on $M_2$ associated with the kernel $(V_f)_1(x_1,y_1)(x_2,y_2) $ which is equal to $V_f(x_1,x_2,y_1,y_2)$. We recall that
$\|(V_f)_1(x_1,y_1)\|_{CZ}\break=\|(V_f)_1(x_1,y_1)\|_{L^2(M_2)\rightarrow L^2(M_2)}+ |(V_f)_1(x_1,y_1)|_{CZ(M_2)}$, where $|(V_f)_1(x_1,y_1)|_{CZ(M_2)}$ is the smallest constant that the inequalities $(a)$, $(b)$ and $(c)$ in Subsection 3.1 holds for the kernel $(V_f)_1(x_1,y_1)(x_2,y_2)$ when $x_1,y_1$ are fixed and $x_2,y_2\in M_2.$ Therefore, to verify that
$(V_f)_1(x_1,y_1)$ satisfies the properties $(i)$, $(ii)$ and $(iii)$ in Subsection 3.1, all we need to do is to show the following estimates:
\begin{itemize}
\item[(I)] $\displaystyle\|(V_f)_1(x_1,y_1)\|_{L^2\rightarrow L^2}\leq C \|f\|_{L^\infty}{1\over V(x_1,y_1)}$;
\item[(II)] $\|(V_f)_1(x_1,y_1)-(V_f)_1(x_1,y_1^{'})\|_{L^2\rightarrow L^2}$
\item[]\qquad $\displaystyle\leq C\|f\|_{L^\infty}\Big({d_1(y_1,y_1^{'})\over d_1(x_1,y_1)}\Big)^{\varepsilon} {1\over V(x_1,y_1)}$
    \qquad if\ \ $d_1(y_1,y_1')\le d_1(x_1,y_1)/2A$.
\item[] Similarly for interchanging $x_1$ and $y_1$;
\item[(III)] $\displaystyle\|(V_f)_1(x_1,y_1)(x_2,y_2)|\leq C\|f\|_{L^\infty(\widetilde{M})}{1\over V(x_1,y_1)}{1\over V(x_2,y_2)}$;
\item[(IV)] $\|(V_f)_1(x_1,y_1)(x_2,y_2)-(V_f)_1(x_1,y_1)(x_2^{'},y_2)|$
\item[] \qquad $\displaystyle\leq C\|f\|_{L^\infty(\widetilde{M})}{1\over V(x_1,y_1)} \Big(\frac{d_2(x_2,x_2^{'})}{d_2(x_2,y_2)}\Big)^{\varepsilon}  {1\over V(x_2,y_2)}$
    \qquad if\ \ $d_2(x_2,x_2')\le d_2(x_2,y_2)/2A$.
\item[] Similarly for interchanging $x_2$ and $y_2$;
\item[(V)] $\|(V_f)_1(x_1,y_1)(x_2,y_2)-(V_f)_1(x_1^{'},y_1)(x_2,y_2)|$
\item[] \qquad $\displaystyle\leq C\|f\|_{L^\infty(\widetilde{M})}\Big(\frac{d_1(x_1,x_1^{'})}{d_1(x_1,y_1)}\Big)^{\varepsilon}
    {1\over V(x_1,y_1)}{1\over V(x_2,y_2)}$
    \qquad if\ \ $d_1(y_1,y_1')\le d_1(x_1,y_1)/2A$.
\item[] Similarly for interchanging $x_1$ and $y_1$;
\item[(VI)] $\big|\big[(V_f)_1(x_1,y_1)(x_2,y_2)-(V_f)_1(x_1^{'},y_1)(x_2,y_2)\big]-
\big[(V_f)_1(x_1,y_1)(x_2^{'},y_2)-(V_f)_1(x_1^{'},y_1)(x_2^{'},y_2)\big]\big|$
\item[] \qquad $\displaystyle\leq C\|f\|_{L^\infty(\widetilde{M})}\Big(\frac{d_1(x_1,x_1^{'})}{d_1(x_1,y_1)}\Big)^{\varepsilon}
    {1\over V(x_1,y_1)} \Big(\frac{d_2(x_2,x_2^{'})}{d_2(x_2,y_2)}\Big)^{\varepsilon}{1\over V(x_2,y_2)}$.
\item[] Similarly for interchanging $x_2$ and $y_2$, or interchanging $x_1$ and $y_1$.
\end{itemize}

To see (I), for fixed $x_1,y_1\in M_1$ we have
\begin{eqnarray}
&&\|(V_f)_1(x_1,y_1)\|_{L^2\rightarrow L^2}=\sup_{g_2:\ \|g_2\|_{L^2(M_2)}\leq1}\ \ \sup_{h_2:\ \|h_2\|_{L^2(M_2)}\leq1}|\langle h_2, (V_f)_1(x_1,y_1)g_2 \rangle|\nonumber\\[4pt]
&&= \sup_{g_2:\ \|g_2\|_{L^2(M_2)}\leq1}\ \ \sup_{h_2:\ \|h_2\|_{L^2(M_2)}\leq1} \bigg|\sum_{ k_1^{'} }\sum_{I_1^{'}}\mu_1(I_1^{'}){\widetilde{\widetilde{D}}}_{k'_1}(x_1,x_{I_1^{'}})D_{k'_1}(x_{I_1^{'}},y_1)\nonumber\\
&&\hskip2cm\times \Big[
\sum_{k_2}\sum_{I_2}\mu_2(I_2)S_{k_2}(h_2)(x_{I_2}) D_{k_2}(g_2)(x_{I_2})  S_{k'_1}
{\widetilde{\widetilde{D}}}_{k_2}(f)(x_{I'_1},x_{I_2})\Big]\bigg|\nonumber\\[4pt]
&&\leq C \|f\|_{L^\infty} \sup_{g_2:\ \|g_2\|_{L^2(M_2)}\leq1}\ \ \sup_{h_2:\ \|h_2\|_{L^2(M_2)}\leq1}\|h_2\|_{L^2(M_2)}\|g_2\|_{L^2(M_2)}\nonumber\\
&&\hskip2cm\times\sum_{ k_1^{'} }\sum_{I_1^{'}}\mu_1(I_1^{'})|{\widetilde{\widetilde{D}}}_{k'_1}(x_1,x_{I_1^{'}})||D_{k'_1}(x_{I_1^{'}},y_1)|\nonumber\\
&&\leq C \|f\|_{L^\infty}{1\over V(x_1,y_1)}, \label{Vf e1}
\end{eqnarray}
where in the first inequality we first apply Schwartz's inequality and then use the Littlewood--Paley estimate on $L^2$ for $g_2$ and the fact that if $f\in L^\infty$ then $\mu_2(I_2)| D_{k_2}(S_{k'_1}f)(x_{I_1^{'}}, x_{I_2})|^2 $ is a Carleson measure on $M_2\times k_2$ uniformly for all $k'_1$ and all $x_{I_1^{'}}\in M_1.$ Moreover, The Carleson measure norm of $\mu_2(I_2)| D_{k_2}(S_{k'_1}f)(x_{I_1^{'}}, x_{I_2})|^2 $ is bounded by some constant times $\|f\|_{L^\infty}.$ The last inequality follows from the standard estimate.

To verify (II), for $d_1(y_1,y_1^{'})\leq d_1(x_1,y_1)/2$ and $\|g_2\|_{L^2(M_2)},\|h_2\|_{L^2(M_2)}\leq 1,$
\begin{eqnarray*}
&&|\langle h_2, [(V_f)_1(x_1,y_1)-(V_f)_1(x_1,y_1^{'})]g_2 \rangle|\\[4pt]
&&= \bigg|\sum_{ k_1^{'} }\sum_{I_1^{'}}\mu_1(I_1^{'}){\widetilde{\widetilde{D}}}_{k'_1}(x_1,x_{I_1^{'}})[D_{k'_1}(x_{I_1^{'}},y_1)-D_{k'_1}(x_{I_1^{'}},y_1^{'})]\\[4pt]
&&\hskip1cm\times \Big[
\sum_{k_2}\sum_{I_2}\mu_2(I_2)S_{k_2}(h_2)(x_{I_2}) D_{k_2}(g_2)(x_{I_2})  S_{k'_1}
{\widetilde{\widetilde{D}}}_{k_2}(f)(x_{I'_1},x_{I_2})\Big]\bigg|.
\end{eqnarray*}
Applying the smoothness property of $D_{k'_1}(x_{I_1^{'}},y_1)$ and the same proof above for the second series yields
\begin{eqnarray*}
|\langle h_2, [(V_f)_1(x_1,y_1)-(V_f)_1(x_1,y_1^{'})]g_2 \rangle|&&\leq C\|f\|_{L^\infty}\Big({d_1(y_1,y_1^{'})\over d_1(x_1,y_1)}\Big)^{\varepsilon} {1\over V(x_1,y_1)},
\end{eqnarray*}
which, by taking the supremum over all $\|g_2\|_{L^2(M_2)},\|h_2\|_{L^2(M_2)}\leq 1$ implies
\begin{eqnarray}
\|(V_f)_1(x_1,y_1)-(V_f)_1(x_1,y_1^{'})\|_{L^2\rightarrow L^2}\leq C\|f\|_{L^\infty}\Big({d_1(y_1,y_1^{'})\over d_1(x_1,y_1)}\Big)^{\varepsilon} {1\over V(x_1,y_1)} . \label{Vf e2}
\end{eqnarray}
Similarly, (\ref{Vf e2}) holds with interchanging $x_1$ and $y_1.$

We now turn to estimate (III). This follows directly from the following standard estimate.
\begin{eqnarray}
&&|(V_f)_1(x_1,y_1)(x_2,y_2)|\nonumber\\[4pt]
&&\leq
\sum_{ k_1^{'} }\sum_{k_2}\sum_{I_1^{'}}
\sum_{I_2}\mu_1(I_1^{'})\mu_2(I_2)
|{\widetilde{\widetilde{D}}}_{k'_1}(x_1,x_{I_1^{'}})S_{k_2}(x_2,x_{I_2})|
|S_{k'_1}{\widetilde{\widetilde{D}}}_{k_2}(f)(x_{I'_1},x_{I_2})|\nonumber\\
&&\hskip1cm\times|D_{k'_1}(x_{I_1^{'}},y_1)D_{k_2}(x_{I_2},y_2)|\nonumber\\[4pt]
&&\leq
\sum_{ k_1^{'} }\sum_{k_2}\sum_{I_1^{'}}
\sum_{I_2}\mu_1(I_1^{'})\mu_2(I_2)
|{\widetilde{\widetilde{D}}}_{k'_1}(x_1,x_{I_1^{'}})D_{k'_1}(x_{I_1^{'}},y_1)|
|S_{k_2}(x_2,x_{I_2})D_{k_2}(x_{I_2},y_2)|\nonumber\\[4pt]
&&\leq
C\|f\|_{L^\infty(\widetilde{M})}{1\over V(x_1,y_1)}{1\over V(x_2,y_2)}.\label{Vf e3}
\end{eqnarray}
To estimate (IV), for $d_2(x_2,x_2^{'})\leq d_2(x_2,y_2)/2A$ we write
\begin{eqnarray*}
&&|(V_f)_1(x_1,y_1)(x_2,y_2)-(V_f)_1(x_1,y_1)(x_2^{'},y_2)|\\[4pt]
&&\leq
\sum_{ k_1^{'} }\sum_{k_2}\sum_{I_1^{'}}
\sum_{I_2}\mu_1(I_1^{'})\mu_2(I_2)
|{\widetilde{\widetilde{D}}}_{k'_1}(x_1,x_{I_1^{'}})[S_{k_2}(x_2,x_{I_2})-S_{k_2}(x_2^{'},x_{I_2})]|\\[4pt]
&&\hskip1cm\times|S_{k'_1}{\widetilde{\widetilde{D}}}_{k_2}(f)(x_{I'_1},x_{I_2})||D_{k'_1}(x_{I_1^{'}},y_1)D_{k_2}(x_{I_2},y_2)|.
\end{eqnarray*}
We {\bf claim} that $S_{k_2}(x_2,x_{I_2})$, which is defined in Subsection 3.3.2, satisfies the following smoothness estimate.
\begin{eqnarray}\label{smooth of Sk}
&&|{S}_{k_2}(x_2,x_{I_2})-{S}_{k_2}(x_2^{'},x_{I_2})|\\[4pt]
&&\leq C \Big({d_2(x_2,x^{'}_2)\over 2^{-k_2}+d_2(x_2,x_{I_2})}\Big)^{\varepsilon}{1\over V_{2^{-k_2}}(x_2)+V(x_2,x_{I_2})} \Big({2^{-k_2}\over
 2^{-k_2}+d_2(x_2,x_{I_2})}\Big)^{\varepsilon}\nonumber
 \end{eqnarray}
for $\varepsilon<\vartheta$ and $d_2(x_2,x'_2)<(2^{-k_2}+d_2(x_2,x_{I_2}))/2$. We assume (\ref{smooth of Sk}) first and then obtain
\begin{eqnarray}
&&|(V_f)_1(x_1,y_1)(x_2,y_2)-(V_f)_1(x_1,y_1)(x_2^{'},y_2)| \label{Vf e4}\\[4pt]
&&\hskip1cm \leq C\|f\|_{L^\infty(\widetilde{M})}{1\over V(x_1,y_1)} \Big(\frac{d_2(x_2,x_2^{'})}{d_2(x_2,y_2)}\Big)^{\varepsilon}  {1\over V(x_2,y_2)}.\nonumber
\end{eqnarray}
Similarly, (\ref{Vf e4}) holds with interchanging $x_2$ and $y_2.$ The estimates in (\ref{Vf e3}) and (\ref{Vf e4}) imply
\begin{eqnarray}\label{Vf e5}
    |(V_f)_1(x_1,y_1)|_{CZ}\leq C\|f\|_{L^\infty(\widetilde{M})}{1\over V(x_1,y_1)}.
\end{eqnarray}
Next, we turn to verify the estimate in (V). For $d_1(x_1,x_1^{'})\le d_1(x_1,y_1)/2A$ We write
\begin{eqnarray*}
&&(V_f)_1(x_1,y_1)(x_2,y_2)-(V_f)_1(x_1^{'},y_1)(x_2,y_2)\\[4pt]
&&=\sum_{ k_1^{'} }\sum_{k_2}\sum_{I_1^{'}}
\sum_{I_2}\mu_1(I_1^{'})\mu_2(I_2)
[{\widetilde{\widetilde{D}}}_{k'_1}(x_1,x_{I_1^{'}})-{\widetilde{\widetilde{D}}}_{k'_1}(x_1^{'},x_{I_1^{'}})]S_{k_2}(x_2,x_{I_2})
S_{k'_1}{\widetilde{\widetilde{D}}}_{k_2}(f)(x_{I'_1},x_{I_2})\\[4pt]
&&\hskip1cm\times
D_{k'_1}(x_{I_1^{'}},y_1)D_{k_2}(x_{I_2},y_2).
\end{eqnarray*}
As in the proof of (3.32), instead of using the smoothness estimate for $S_{k_2}(x_2,x_{I_2}),$ applying the smoothness condition of $\widetilde{D}_{k'_1}$, we get
\begin{eqnarray}\label{Vf e6}
&&|(V_f)_1(x_1,y_1)(x_2,y_2)-(V_f)_1(x_1^{'},y_1)(x_2,y_2)|\\[4pt]
&&\leq
C\|f\|_{L^\infty(\widetilde{M})}\Big(\frac{d_1(x_1,x_1^{'})}{d_1(x_1,y_1)}\Big)^{\varepsilon}
{1\over V(x_1,y_1)}{1\over V(x_2,y_2)}.\nonumber
\end{eqnarray}
Similarly, (\ref{Vf e6}) holds with interchanging $x_1$ and $y_1.$
Finally, to see (VI), for $d_2(x_2,x_2^{'})\leq d_2(x_2,y_2)/2A$ we have
\begin{eqnarray}
&&\Big|\big[(V_f)_1(x_1,y_1)(x_2,y_2)-(V_f)_1(x_1^{'},y_1)(x_2,y_2)\big]-\big[(V_f)_1(x_1,y_1)(x_2^{'},y_2)-(V_f)_1(x_1^{'},y_1)(x_2^{'},y_2)\big]\Big|\nonumber\\[4pt]
&&=\Big|\sum_{ k_1^{'} }\sum_{k_2}\sum_{I_1^{'}}
\sum_{I_2}\mu_1(I_1^{'})\mu_2(I_2)
[{\widetilde{\widetilde{D}}}_{k'_1}(x_1,x_{I_1^{'}})-{\widetilde{\widetilde{D}}}_{k'_1}(x_1^{'},x_{I_1^{'}})][S_{k_2}(x_2,x_{I_2})-S_{k_2}(x_2^{'},x_{I_2})]
\nonumber\\[4pt]
&&\hskip.5cm\times
S_{k'_1}{\widetilde{\widetilde{D}}}_{k_2}(f)(x_{I'_1},x_{I_2})D_{k'_1}(x_{I_1^{'}},y_1)D_{k_2}(x_{I_2},y_2)\Big|\nonumber\\[4pt]
&&\leq
C\|f\|_{L^\infty(\widetilde{M})}\Big(\frac{d_1(x_1,x_1^{'})}{d_1(x_1,y_1)}\Big)^{\varepsilon}
{1\over V(x_1,y_1)} \Big(\frac{d_2(x_2,x_2^{'})}{d_2(x_2,y_2)}\Big)^{\varepsilon}{1\over V(x_2,y_2)},\label{Vf e7}
\end{eqnarray}
where in the last inequality we use the smoothness property of ${\widetilde{\widetilde{D}}}_{k'_1}$ and (\ref{smooth of Sk}).
Similarly, (\ref{Vf e7}) holds with interchanging $x_2$ and $y_2$ or $x_1$ and $y_1.$

All the estimates of (\ref{Vf e6}) and (\ref{Vf e7}) give
\begin{eqnarray}\label{Vf e8}
  &&  \big|\big[(V_f)_1(x_1,y_1)(x_2,y_2)-(V_f)_1(x_1^{'},y_1)(x_2,y_2)\big]\big|_{CZ}\\[4pt]
  &&\hskip.5cm\leq C\|f\|_{L^\infty(\widetilde{M})}\Big(\frac{d_1(x_1,x_1^{'})}{d_1(x_1,y_1)}\Big)^{\varepsilon}{1\over V(x_1,y_1)}.\nonumber
\end{eqnarray}
Similarly, (\ref{Vf e8}) holds interchanging $x_1$ and $y_1.$

As a consequence, (\ref{Vf e5}) and (\ref{Vf e8}) yield that  $(V_f)_1(x_1,y_1)$ satisfies the properties $(i)$, $(ii)$ and $(iii)$ in Subsection 3.1.
It remains to show the claim, that is, the estimate in (\ref{smooth of Sk}). Indeed, when $d_2(x_2,x_{I_2})< 2^{-k_2}$ and $d_1(x_2,x'_2)<(2^{-k_2}+d_2(x_2,x_{I_2}))/2$, we
have
\begin{eqnarray*}
&&|{S}_{k_2}(x_2,x_{I_2})-{S}_{k_2}(x_2^{'},x_{I_2})|\\[4pt]
&&=\Big|\sum_{k_2^{'}\leq k_2,\ d_2(x_2,x_{I_2})< 2^{-k_2} }\sum_{I_2^{'}}
\mu(I_{2}^{'})D_{k_2^{'}}(x_2,x_{I_{2}^{'}})
 {\widetilde{\widetilde{D}}}_{k_2^{'}}(x_{I_{2}^{'}},x_{I_2})\nonumber\\[4pt]
&&\hskip1cm  -\sum_{ k_2^{'}\leq k_2,\ d_2(x_2,x_{I_2})< 2^{-k_1^{'}} } \sum_{I_2^{'}}
\mu(I_{2}^{'})D_{k_2^{'}}(x'_2,x_{I_2^{'}})
 {\widetilde{\widetilde{D}}}_{k_1}(x_{I_{2}^{'}},x_{I_2}) \Big| \label{key
claim e3}\\[4pt]
 &&\leq C\sum_{k_2^{'}\leq k_2,\ d_2(x_2,x_{I_2})< 2^{-k_2} } \Big({d_2(x_2,x'_2)\over 2^{-k_2^{'}}+d_2(x_2,x_{I_2})}\Big)^{\varepsilon}
 {1\over V_{2^{-k_2^{'}}}(x_2)+V(x_2,x_{I_2})} \Big({2^{-k_2^{'}}\over 2^{-k_2^{'}}+d_2(x_2,x_{I_2})}\Big)^{\varepsilon}\nonumber\\[4pt]
&&\leq C\Big({d_2(x_2,x^{'}_2)\over 2^{-k_2}+d_2(x_2,x_{I_2})}\Big)^{\varepsilon}{1\over V_{2^{-k_2}}(x_2)+V(x_2,x_{I_2})} \Big({2^{-k_2}\over
 2^{-k_2}+d_2(x_2,x_{I_2})}\Big)^{\varepsilon}.\nonumber
\end{eqnarray*}
Next, we consider the case when $d_2(x_2,x_{I_2})\geq 2^{-k_2}$ and
$d_2(x_2,x'_2)<(2^{-k_2}+d_2(x_2,x_{I_2}))/2$. In this case, using the identity (\ref{dirac function}), we obtain
\begin{eqnarray*}
&&\Big|\sum_{k_2^{'}\leq k_2, d_2(x_2,x_{I_2})\geq 2^{-k_2} }\sum_{I_2^{'}}
\mu(I_{2}^{'})D_{k_2^{'}}(x_2,x_{I_{2}^{'}})
 {\widetilde{\widetilde{D}}}_{k_2^{'}}(x_{I_{2}^{'}},x_{I_2})\nonumber\\[4pt]
&&\hskip3cm -\sum_{k_2^{'}\leq k_2, d_2(x_2,x_{I_2})\geq 2^{-k_2} }\sum_{I_2^{'}}
\mu(I_{2}^{'})D_{k_2^{'}}(x'_2,x_{I_{2}^{'}})
 {\widetilde{\widetilde{D}}}_{k_2^{'}}(x_{I_{2}^{'}},x_{I_2})\Big| \nonumber\\[4pt]
&&\leq  \Big|\sum_{k_2^{'}> k_2, d_2(x_2,x_{I_2})\geq 2^{-k_2} }\sum_{I_2^{'}}
\mu(I_{2}^{'})D_{k_2^{'}}(x_2,x_{I_{2}^{'}})
 {\widetilde{\widetilde{D}}}_{k_2^{'}}(h)(x_{I_{2}^{'}},x_{I_2})\nonumber\\[4pt]
 &&\hskip3cm -\sum_{k_2^{'}> k_2, d_2(x_2,x_{I_2})\geq 2^{-k_2} }\sum_{I_2^{'}}
\mu(I_{2}^{'})D_{k_2^{'}}(x'_2,x_{I_{2}^{'}})
 {\widetilde{\widetilde{D}}}_{k_2^{'}}(h)(x_{I_{2}^{'}},x_{I_2})\Big|\nonumber\\[4pt]
&&\leq C\Big({d_2(x_2,x^{'}_2)\over 2^{-k_2}+d_2(x_2,x_{I_2})}\Big)^{\varepsilon}{1\over V_{2^{-k_2}}(x_2)+V(x_2,x_{I_2})} \Big({2^{-k_2}\over
 2^{-k_2}+d_2(x_2,x_{I_2})}\Big)^{\varepsilon},\nonumber
\end{eqnarray*}
which implies the claim.

Now we have proved that $V_f$ is a product Calder\'on--Zygmund operator with $\|V_f\|_{L^2\rightarrow L^2}\leq C\|f\|_{L^\infty}.$ In order to apply Theorem C given in next section to show that $V_f$ is bounded on $BMO(\widetilde M),$ we only need to verify that $(V_f)_1(1)=(V_f)_2(1)=0.$ To do this, we would like to recall the definition of
$T_1(1)=T_2(1)=0$ and $(T^*)_1(1)=(T^*)_2(1)=0$ as defined in Subsection 3.1. $T_1(1)=0$ is equivalent to $\langle g_1, \langle g_2, T_2f_2\rangle 1\rangle=0$ for all $g_1\in C^\eta_{00}(M_1)$ and $f_2, g_2\in C^\eta_0(M_2),$ that is, for $g_1\in C^\eta_{00}(M_1), g_2\in C^\eta_{00}(M_2)$ and almost everywhere $y_2\in M_2,$
$$\iint g(x_1)g(x_2)K(x_1,x_2,y_1,y_2)dx_1dx_2dy_1=0.$$
While ${T_1}^*(1)=0$ means ${\langle g_2, T_2f_2\rangle}^*1=0$ in the same conditions, that is, for $g_1\in C^\eta_{00}(M_1), g_2\in C^\eta_{00}(M_2)$ and almost everywhere $x_2\in M_2,$
$$\iint g(y_1)g(y_2)K(x_1,x_2,y_1,y_2)dx_1dy_1dy_2=0.$$
To verify $(V_f)_1(1)=0,$ for $g_1\in C^\eta_{00}(M_1), g_2\in C^\eta_{00}(M_2)$ and almost everywhere $y_2\in M_2$ we have
\begin{eqnarray*}
&&\iint g(x_1)g(x_2)V_f(x_1,x_2,y_1,y_2)dx_1dx_2dy_1\\[4pt]
&&=\iint g(x_1)g(x_2)\sum_{ k_1^{'} }\sum_{I_1^{'}}\sum_{k_2}
\sum_{I_2}\mu_1(I_1^{'})\mu_2(I_2){\widetilde{\widetilde{D}}}_{k'_1}(x_1,x_{I_1^{'}})S_{k_2}(x_2,x_{I_2})\\[4pt]
&&\hskip1cm\times S_{k'_1}{\widetilde{\widetilde{D}}}_{k_2}(f)(x_{I'_1},x_{I_2})D_{k'_1}(x_{I_1^{'}},y_1)D_{k_2}(x_{I_2}, y_2)dx_1dx_2dy_1=0,
\end{eqnarray*}
where the last equality follows from the fact that $\int D_{k'_1}(x_{I_1^{'}},y_1)dy_1=0.$ Similarly for $(V_f)_2(1)=0.$
As mentioned, we conclude that $|\langle g, Tf\rangle_{Case 2.4}|\leq C\|f\|_2\|g\|_2.$

The proof of the sufficient conditions for Theorem A is complete and hence the proof of Theorem A is concluded.

\section{$T1$-type theorems on $H^p$ and $CMO^p$}
\setcounter{equation}{0}

In this section we prove the $T1$ -type theorems on $H^p$ and $CMO^p,$ namely the following

\vskip.3cm \noindent{\bf Theorem B}\ \ Let $T$ be the $L^2$ bounded product
Calder\'on--Zygmund singular integral operator on $\widetilde M$ with a pair kernel
$(K_1, K_2)$ satisfying the conditions $(i),$ $(ii)$ and $(iii)$ in Subsection 3.1. Then $T$ extends to a bounded operator from
$H^p(\widetilde{M}), \max\big(\frac{ Q_1}{ Q_1+ \vartheta_1 },\frac{
Q_2}{ Q_2+ \vartheta_2 }\big)<p\leq1,$ to itself if and only if
$(T^*)_1(1)=(T^*)_2(1)=0.$

\vskip.3cm \noindent{\bf Theorem C}\ \ Let $T$ be the $L^2$ bounded product
Calder\'on--Zygmund operator on $\widetilde M$ with a pair kernel
$(K_1, K_2)$ satisfying the conditions $(i),$ $(ii)$ and $(ii)$ in Subsection 3.1. Then $T$ extends to a bounded operator from
$CMO^p(\widetilde{M}), \max\big(\frac{ 2Q_1}{ 2Q_1+ \vartheta_1
},\frac{ 2Q_2}{ 2Q_2+ \vartheta_2 }\big)<p\leq1,$ to itself,
particularly from $BMO(\widetilde{M})$ to itself, if and only if
$T_1(1)=T_2(1)=0.$

\vskip.3cm

We first remark that the range of $p$ in Theorems B and C could be smaller if the smoothness of a pair kernel $(K_1, K_2)$ and the cancellation conditions of $T$ both are required to be higher. We leave these details to the reader.

As mentioned in Section 1, we will first prove the ``if'' part of Theorem B. This will be achieved by applying the almost orthogonal argument, the Plancherel--P\^olya inequality and atomic decomposition of $H^p(\widetilde M)$ for the vector-valued product Calder\'on--Zygmund operators. The ``if'' part of Theorem C then follows from the ``if'' part of Theorem B by the duality argument. We emphasize that Lemma 4.1 below plays a crucial role in this proof. To show the converse, we will prove the ``only if'' part of Theorem C first and the ``only if'' part of Theorem B then follows from the duality argument directly.

\subsection{``If'' part of $T1$  theorem on $H^p$ }

To show the  ``if'' part of Theorem B, note first that $L^2\cap H^p(\widetilde M)$ is dense in $H^p(\widetilde
M),$ see [HLL2] for this result, and thus it suffices to prove that if $T$ is the $L^2$ bounded product
Calder\'on--Zygmund operator on $\widetilde M$ with a pair kernel
$(K_1, K_2)$ satisfying the conditions $(i)$ - $(iii)$ and $(T^*)_1(1)=(T^*)_2(1)=0$ then there exists a constant $C$ independent of $f$ such that
$$\|Tf\|_{H^p}\leq C\|f\|_{H^p}$$
for all $f\in L^2\cap H^p(\widetilde M).$

by Proposition 2.14 this is equivalent to show
\begin{eqnarray}\label{STf Lp bounded by f Hp}
\|\widetilde S(Tf)\|_{p}\leq C\|f\|_{H^p},
\end{eqnarray}
where, as in Definition 2.11, $\widetilde S(f)$ is the Littlewood--Paley square
function of $f$ given by
\begin{eqnarray}
\widetilde S(Tf)(x_1,x_2)=
\Big\{\sum_{k_1^{'}=-\infty}^\infty\sum_{k_2^{'}=-\infty}^\infty
|D_{k_1^{'}}D_{k_2^{'}}(Tf)(x_1,x_2)|^2
\Big\}^{1/2}.
\end{eqnarray}
To show the estimate in (\ref{STf Lp bounded by f Hp}), as in the classical case, we introduce the Hilbert space $\mathcal{H}$ by
\begin{eqnarray*}
&&\mathcal{H}=\bigg\{
\{h_{k_1^{'},k_2^{'}}\}_{k_1^{'},k_2^{'}\in\mathbb{Z}}: \|h_{k_1^{'},k_2^{'}}\|_{\mathcal{H}}:=
\Big(\sum_{k_1^{'}=-\infty}^\infty\sum_{k_2^{'}=-\infty}^\infty
|h_{k_1^{'},k_2^{'}}|^2
\Big)^{1/2} <\infty\bigg\}.
\end{eqnarray*}
Then we can write the estimate in (\ref{STf Lp bounded by f Hp}) by
\begin{eqnarray*}
\|D_{k_1^{'}}D_{k_2^{'}}(Tf)(x_1,x_2)\|_{L^p_{\mathcal H}}\leq C\|f\|_{H^p}.
\end{eqnarray*}
The crucial idea is that for $f\in L^2,$ by the discrete Calder\'on identity
\begin{eqnarray*}
f(x_1,x_2)&=&\sum_{k_1=-\infty}^\infty\sum_{k_2=-\infty}^\infty\sum_{I_1}\sum_{I_2 }\mu_1(I_1)\mu_2(I_2) D_{k_1}(x_1,x_{I_1})D_{k_2}(x_2,x_{I_2}) {\widetilde{\widetilde{D}}}_{k_1}{\widetilde{\widetilde{D}}}_{k_2}(f)(x_{I_1},x_{I_2}),\nonumber
\end{eqnarray*}
we can write
\begin{eqnarray*}
D_{k'_1}D_{k'_2}(Tf)(x_1,x_2)&=&\sum_{k_1=-\infty}^\infty\sum_{k_2=-\infty}^\infty\sum_{I_1}\sum_{I_2 }\mu_1(I_1)\mu_2(I_2)\\[4pt]
&&\hskip-1cm\times
D_{k'_1}D_{k'_2}TD_{k_1}(\cdot,x_{I_1})D_{k_2}(\cdot,x_{I_2})(x_1,x_2)
{\widetilde{\widetilde{D}}}_{k_1}{\widetilde{\widetilde{D}}}_{k_2}(f)(x_{I_1},x_{I_2}),
\end{eqnarray*}
where the fact that $T$ is bounded on $L^2$ is used.
Therefore, the estimate in (\ref{STf Lp bounded by f Hp}) is equivalent to
\begin{eqnarray}\label{L bd Hp to vectorLp}
\|\mathcal{L}_{k_1^{'},k_2^{'}}(f)\|_{L^p_{\mathcal{H}}}\leq C\|f\|_{H^p},
\end{eqnarray}
where
\begin{eqnarray}
\mathcal {L}_{k_1^{'},k_2^{'}}(f)(x_1,x_2)&=&
\sum_{k_1=-\infty}^{\infty}\sum_{k_2=-\infty}^{\infty}
\sum_{I_1}\sum_{I_2
}\mu_1(I_1)\mu_2(I_2)D_{k'_1}D_{k'_2}TD_{k_1}D_{k_2}(x_1,x_2,x_{I_1},x_{I_2})
\nonumber\\[4pt]
&&\hskip1cm\times
{\widetilde{\widetilde{D}}}_{k_1}{\widetilde{\widetilde{D}}}_{k_2}(f)(x_{I_1},x_{I_2}).
\nonumber
\end{eqnarray}
The estimate of (\ref{L bd Hp to vectorLp}), however, means that the $\mathcal {H}$-valued operator $\mathcal
{L}_{k_1^{'},k_2^{'}}$ is bounded from $H^p$ to the $L^p$ and hence, as in the proof of Theorem 3.6, we can apply atomic decomposition. For this purpose, we need to show that $\mathcal {L}_{k_1^{'},k_2^{'}}$ is a $L^2$ bounded $\mathcal {H}$-valued
product Calder\'on--Zygmund singular integral operator whose pair kernel $\big((\mathcal
{L}_{k_1^{'},k_2^{'}})_1, (\mathcal{L}_{k_1^{'},k_2^{'}})_2\big)$ satisfies the condition $(i)$ - $(iii)$ in Subsection 3.1 with the absolute value replaced by $\mathcal {H}$ valued. The estimate in (\ref{L bd Hp to vectorLp}) then follows from the same proof of Theorem 3.6 with replacing the absolute value, $L^2$ norm and Calder\'on--Zygmund norm by $\|\cdot\|_{\mathcal H }, \|\cdot\|_{L^2_{\mathcal H}}$ and $\|\cdot\|_{CZ(\mathcal H)},$ respectively. This implies that (\ref{STf Lp bounded by f Hp}) holds and hence the proof of the ``if'' part
of Theorem B is concluded.

The $L^2$ boundedness of the $\mathcal H$-valued operator $\mathcal
{L}_{k_1^{'},k_2^{'}}$ follows directly from  the product Littlewood--Paley estimate (see Proposition
\ref{prop discrete Littlewood--Paley} and Theorem \ref{theorem Lp bd
of suqa func}) and the $L^2$ boundedness of the operator $T.$ Indeed,
\begin{eqnarray*}\label{vector operator L Lp bounded by f Hp}
\big\|\mathcal
{L}_{k_1^{'},k_2^{'}}(f)\big\|_{L^2_{\mathcal H}}=\|\widetilde S(Tf)\|_2\leq C\|Tf\|_2\leq C\|f\|_{2}.
\end{eqnarray*}

To show that $\mathcal
{L}_{k_1^{'},k_2^{'}}$ is a $\mathcal{H}$-valued product Carlder\'on-Zygmund singular integral operator as defined in Subsection 3.1,
we can consider, as mentioned, $\mathcal
{L}_{k_1^{'},k_2^{'}}$ as a pair $\big((\mathcal
{L}_{k_1^{'},k_2^{'}})_1,(\mathcal
{L}_{k_1^{'},k_2^{'}})_2\big)$ of operators on $M_2$ and $M_1$, respectively.
It suffices to show that $(\mathcal
{L}_{k_1^{'},k_2^{'}})_i(x_i,y_i), i=1,2,$ satisfies the properties $(i)$ - $(ii)$ in Subsection 3.1.
We need only to verify $(\mathcal
{L}_{k_1^{'},k_2^{'}})_1(x_1,y_1)$ since the proof for $(\mathcal
{L}_{k_1^{'},k_2^{'}})_2(x_2,y_2)$ is similar.

Note that
\begin{eqnarray*}
&&\big\| (\mathcal
{L}_{k_1^{'},k_2^{'}})_1(x_1,y_1)\big\|_{CZ(\mathcal H)}\\[4pt]
&&\hskip.6cm=\big\| (\mathcal
{L}_{k_1^{'},k_2^{'}})_1(x_1,y_1)\big\|_{L^2_{\mathcal H}(M_2)\rightarrow L^2_{\mathcal H}(M_2)}+ \big| (\mathcal
{L}_{k_1^{'},k_2^{'}})_1(x_1,y_1)\big|_{CZ(\mathcal H)(M_2)},
\end{eqnarray*}
where $\big| (\mathcal
{L}_{k_1^{'},k_2^{'}})_1(x_1,y_1)\big|_{CZ(\mathcal H)(M_2)}$ is the smallest constant that the inequalities $(a)$, $(b)$ and $(c)$ in Subsection 3.1 holds in the sense that the absolute value is replaced by $\mathcal{H}$-value for the kernel $(\mathcal
{L}_{k_1^{'},k_2^{'}})_1(x_1,y_1)(x_2,y_2)$ whenever $x_1,y_1$ are fixed and $x_2,y_2\in M_2.$ Therefore, to verify that
$(\mathcal
{L}_{k_1^{'},k_2^{'}})_1(x_1,y_1)$ satisfies the properties $(i)$ - $(iii)$ in Subsection 3.1, all we need to do is to show that for
$0<\varepsilon^{'}<\varepsilon$ there exists a positive constant $C=C(\varepsilon^{'})>0$ such that:
\begin{itemize}
\item[(I)] $\displaystyle\big\|(\mathcal {L}_{k_1^{'},k_2^{'}})_1(x_1,y_1)\big\|_{L^2_{\mathcal H}(M_2)\rightarrow L^2_{\mathcal H}(M_2)}\leq C {1\over V(x_1,y_1)}$;
\item[(II)] $\big\| (\mathcal
{L}_{k_1^{'},k_2^{'}})_1(x_1,y_1)-(\mathcal
{L}_{k_1^{'},k_2^{'}})_1(x_1,y_1^{'})\big\|_{L^2_{\mathcal H}(M_2)\rightarrow L^2_{\mathcal H}(M_2)}$
\item[]\qquad $\displaystyle \leq C\Big({d_1(y_1,y_1^{'})\over d_1(x_1,y_1)}\Big)^{\varepsilon^{'}} {1\over V(x_1,y_1)}$
    \qquad if\ \ $d_1(y_1,y_1')\le d_1(x_1,y_1)/2A$.
\item[] Similarly for interchanging $x_1$ and $y_1$;
\item[(III)] $\displaystyle |(\mathcal {L}_{k_1^{'},k_2^{'}})_1(x_1,y_1)(x_2,y_2)|_{\mathcal{H}}\leq C{1\over V(x_1,y_1)}{1\over V(x_2,y_2)}$;
\item[(IV)] $|(\mathcal
{L}_{k_1^{'},k_2^{'}})_1(x_1,y_1)(x_2,y_2)-(\mathcal
{L}_{k_1^{'},k_2^{'}})_1(x_1,y_1)(x_2,y_2^{'})|_{\mathcal{H}}$
\item[]\qquad $\displaystyle \leq C{1\over V(x_1,y_1)} \Big(\frac{d_2(y_2,y_2^{'})}{d_2(x_2,y_2)}\Big)^{\varepsilon^{'}}  {1\over V(x_2,y_2)}$
    \qquad if\ \ $d_2(y_2,y_2')\le d_2(x_2,y_2)/2A$.
\item[] Similarly for interchanging $x_2$ and $y_2$;
\item[(V)] $|(\mathcal {L}_{k_1^{'},k_2^{'}})_1(x_1,y_1)(x_2,y_2)-(\mathcal
{L}_{k_1^{'},k_2^{'}})_1(x_1,y_1^{'})(x_2,y_2)|_{\mathcal{H}}$
\item[]\qquad $\displaystyle \leq C\Big(\frac{d_1(y_1,y_1^{'})}{d_1(x_1,y_1)}\Big)^{\varepsilon^{'}}
{1\over V(x_1,y_1)}{1\over V(x_2,y_2)}$\qquad if\ \ $d_1(y_1,y_1')\le d_1(x_1,y_1)/2A$.
\item[] Similarly for interchanging $x_1$ and $y_1$;
\item[(VI)] $\Big|\big[(\mathcal {L}_{k_1^{'},k_2^{'}})_1(x_1,y_1)(x_2,y_2)-(\mathcal
{L}_{k_1^{'},k_2^{'}})_1(x_1,y_1^{'})(x_2,y_2)\big]$
\item[] $\hskip 3cm \big[(\mathcal {L}_{k_1^{'},k_2^{'}})_1(x_1,y_1)(x_2,y_2^{'})-(\mathcal
{L}_{k_1^{'},k_2^{'}})_1(x_1,y_1^{'})(x_2,y_2^{'})\big]\Big|_{\mathcal{H}}$
\item[]\qquad $\displaystyle \leq C\Big(\frac{d_1(y_1,y_1^{'})}{d_1(x_1,y_1)}\Big)^{\varepsilon^{'}}
{1\over V(x_1,y_1)} \Big(\frac{d_2(y_2,y_2^{'})}{d_2(x_2,y_2)}\Big)^{\varepsilon^{'}}{1\over V(x_2,y_2)}$
\item[]\hskip 3cm if $d_1(y_1,y_1^{'})\le d_1(x_1,y_1)/2A$ and $d_2(y_2,y_2^{'})\le d_2(x_2,y_2)/2A$.
\item[] Similarly for interchanging $x_1, y_1$ and $x_2, y_2$, respectively.
\end{itemize}
Note that for any fixed $x_1,y_1$ on $M_1$, $(\mathcal
{L}_{k_1^{'},k_2^{'}})_1(x_1,y_1)$ is an operator on $M_2$ associated with the kernel $(\mathcal
{L}_{k_1^{'},k_2^{'}})_1(x_1,y_1)(x_2,y_2) $ which is equal to
${\mathcal L}_{k_1^{'},k_2^{'}}(x_1,x_2,y_1,y_2),$ the kernel of the vector-valued operator $\mathcal
{L}_{k_1^{'},k_2^{'}},$ given by
\begin{eqnarray}
&&\mathcal {L}_{k_1^{'},k_2^{'}}(x_1,x_2,y_1,y_2)\label{H valued operator L}\\[4pt]
&=&  \sum_{k_1=-\infty}^{\infty}\sum_{k_2=-\infty}^{\infty}
\sum_{I_1}\sum_{I_2
}\mu_1(I_1)\mu_2(I_2)
D_{k'_1}D_{k'_2}TD_{k_1}D_{k_2}(x_1,x_2,x_{I_1},x_{I_2})
{\widetilde{\widetilde{D}}}_{k_1}(x_{I_1},y_1){\widetilde{\widetilde{D}}}_{k_2}(x_{I_2},y_2).\nonumber
\end{eqnarray}
We now first prove $(II).$ The proof for $(I)$ then follows similarly.
Note that
\begin{eqnarray*}
&&\big\| (\mathcal
{L}_{k_1^{'},k_2^{'}})_1(x_1,y_1)-(\mathcal
{L}_{k_1^{'},k_2^{'}})_1(x_1,y_1^{'})\big\|_{L^2_{\mathcal H}(M_2)\rightarrow L^2_{\mathcal H}(M_2)}\\[5pt]
&&=\sup_{f:\ \|f\|_{L^2(M_2)}\leq1}\Big(\int_{M_2} \Big\|\int_{M_2}[\mathcal
{L}_{k_1^{'},k_2^{'}}(x_1,x_2,y_1,y_2)-\mathcal
{L}_{k_1^{'},k_2^{'}}(x_1,x_2,y_{1}^{'},y_2)]f(y_2)dy_2\Big\|_{\mathcal{H}}^2dx_2\Big)^{1/2}.
\end{eqnarray*}
By the definition of the operator $\mathcal
{L}_{k_1^{'},k_2^{'}}$ as in (\ref{H valued operator L}),  we write
\begin{eqnarray*}
&&\int_{M_2}[\mathcal
{L}_{k_1^{'},k_2^{'}}(x_1,x_2,y_1,y_2)-\mathcal
{L}_{k_1^{'},k_2^{'}}(x_1,x_2,y_{1}^{'},y_2)]f(y_2)dy_2\\[4pt]
&&=\int \sum_{k_1=-\infty}^{\infty}
\sum_{I_1}\mu_1(I_{1})  D_{k'_1}(x_1,u_1)D_{k'_2}(x_2,u_2)K(u_1,u_2,v_1,v_2)D_{k_1}(v_1,x_{I_1}) \\[5pt]
&&\hskip.5cm\times \big[{\widetilde{\widetilde{D}}}_{k_1}(x_{I_1},y_1)-{\widetilde{\widetilde{D}}}_{k_1}(x_{I_1},y_{1}^{'})\big]f(v_2)
du_1du_2dv_1dv_2,
\end{eqnarray*}
where we use the discrete Calder\'on's identity on $M_2$ for the function $f$ in the above equality.
Applying the Littlewood--Paley estimate on $M_2$ yields
\begin{eqnarray}
&&\Big(\int_{M_2} \Big\|\int_{M_2}[\mathcal
{L}_{k_1^{'},k_2^{'}}(x_1,x_2,y_1,y_2)-\mathcal
{L}_{k_1^{'},k_2^{'}}(x_1,x_2,y_{1}^{'},y_2)]f(y_2)dy_2\Big\|_{\mathcal{H}}^2dx_2\Big)^{1/2}\nonumber\\[4pt]
&&=\Big(\int_{M_2} \sum_{k_1^{'}=-\infty}^{\infty}\sum_{k_2^{'}=-\infty}^{\infty} \bigg|D_{k'_2}\bigg(\int \sum_{k_1=-\infty}^{\infty}
\sum_{I_1}\mu_1(I_{1})  D_{k'_1}(x_1,u_1)K(u_1,\cdot,v_1,v_2) D_{k_1}(v_1,x_{I_1})\nonumber\\[4pt]
&&\hskip.5cm\times \big[{\widetilde{\widetilde{D}}}_{k_1}(x_{I_1},y_1)-{\widetilde{\widetilde{D}}}_{k_1}(x_{I_1},y_{1}^{'})\big]f(v_2)
du_1dv_1dv_2\bigg)(x_2)\bigg|^2   \ dx_2\Big)^{1/2}\nonumber\\[4pt]
&&\leq C \Big(\sum_{k_1^{'}=-\infty}^{\infty}\int_{M_2} \bigg|\int \sum_{k_1=-\infty}^{\infty}
\sum_{I_1}\mu_1(I_{1})  D_{k'_1}(x_1,u_1)K(u_1,x_2,v_1,v_2) D_{k_1}(v_1,x_{I_1}) \nonumber\\[4pt]
&&\hskip.5cm\times\big[{\widetilde{\widetilde{D}}}_{k_1}(x_{I_1},y_1)-{\widetilde{\widetilde{D}}}_{k_1}(x_{I_1},y_{1}^{'})\big]f(v_2)
du_1dv_1dv_2\bigg|^2dx_2\Big)^{1/2}.\label{Hp bd e5}
\end{eqnarray}

Now we claim that for any fixed $k_1^{'}$ and $\varepsilon'$ with $\varepsilon'<\varepsilon$ there exists positive constant $C$
such that for $d_1(y_1,y_1')\le d_1(x_1,y_1)/2A$ and $\|f\|_2\leq 1,$
\begin{eqnarray}
&&\Big(\int_{M_2}\bigg|\int \sum_{k_1=-\infty}^{\infty}
\sum_{I_1}\mu_1(I_{1})  D_{k'_1}(x_1,u_1)K(u_1,x_2,v_1,v_2) D_{k_1}(v_1,x_{I_1}) \nonumber\\[4pt]
&&\hskip1cm\times \big[{\widetilde{\widetilde{D}}}_{k_1}(x_{I_1},y_1)-{\widetilde{\widetilde{D}}}_{k_1}(x_{I_1},y_{1}^{'})\big]f(v_2)
du_1dv_1dv_2\bigg|^2dx_2\Big)^{1/2}\nonumber\\[4pt]
&&\leq\ C \Big({d_1(y_1,y_{1}^{'})\over 2^{-k_1^{'}}}\Big)^{\varepsilon'} {1\over V_{2^{-k_1^{'}}}(x_1)+V(x_1,y_{1})} \Big({2^{-k_1^{'}}\over 2^{-k_1^{'}}+
d_1(x_1,y_{1})}\Big)^{\varepsilon'}.\ \ \ \ \ \ \ \label{Hp bd e7}\ \ \ \
\end{eqnarray}

Assume that (\ref{Hp bd e7}) holds. Inserting (\ref{Hp bd e7}) into (\ref{Hp bd e5}) together with the following standard estimate
\begin{eqnarray*}
&&\sum_{k_1^{'}}\Big({d_1(y_1,y_{1}^{'})\over 2^{-k_1^{'}}}\Big)^{2\varepsilon'}\Big( {1\over V_{2^{-k_1^{'}}}(x_1)+V(x_1,y_{1})}\Big)^2 \Big({2^{-k_1^{'}}\over 2^{-k_1^{'}}+
d_1(x_1,y_{1})}\Big)^{2\varepsilon'}\\[4pt]
&&\leq C\Big({d_1(y_1,y_{1}^{'})\over
d_1(x_1,y_{1})}\Big)^{2\varepsilon'}{1\over V^2(x_1,y_{1})}
\end{eqnarray*}
yields that for $d_1(y_1,y_1')\le d_1(x_1,y_1)/2A$ and $\|f\|_2\leq 1,$
\begin{eqnarray*}
&&\Big(\int_{M_2} \Big\|\int_{M_2}[\mathcal
{L}_{k_1^{'},k_2^{'}}(x_1,x_2,y_1,y_2)-\mathcal
{L}_{k_1^{'},k_2^{'}}(x_1,x_2,y_{1}^{'},y_2)]f(y_2)dy_2\Big\|_{\mathcal{H}}^2dx_2\Big)^{1/2}\nonumber\\[4pt]
&&\leq C \Big({d_1(y_1,y_{1}^{'})\over
d_1(x_1,y_{1})}\Big)^{\varepsilon'}{1\over V(x_1,y_{1})},
\end{eqnarray*}
which implies $(II).$

In order to show the estimate in (\ref{Hp bd e7}), we will apply the almost orthogonal argument. For this purpose, we first write
\begin{eqnarray*}
&&\Big(\int_{M_2}\Big| \int \sum_{k_1=-\infty}^{\infty}
\sum_{I_1}\mu_1(I_{1})  D_{k'_1}(x_1,u_1)K(u_1,x_2,v_1,v_2) D_{k_1}(v_1,x_{I_1})\nonumber\\[4pt]
&&\hskip.5cm\times \big[\widetilde{\widetilde{D}}_{k_1}(x_{I_1},y_1)-\widetilde{\widetilde{D}}_{k_1}(x_{I_1},y_{1}^{'})\big]f(v_2)
du_1dv_1dv_2\Big|^2dx_2\Big)^{1/2}\\[4pt]
&&=\sup_{h:\ \|h\|_{L^2(M_2)}\leq 1}\Big| \int \sum_{k_1=-\infty}^{\infty}
\sum_{I_1}\mu_1(I_{1})  D_{k'_1}(x_1,u_1)\langle h,K_1(u_1,v_1)f\rangle D_{k_1}(v_1,x_{I_1})\nonumber\\[4pt]
&&\hskip.5cm\times \big[\widetilde{\widetilde{D}}_{k_1}(x_{I_1},y_1)-\widetilde{\widetilde{D}}_{k_1}(x_{I_1},y_{1}^{'})\big]
du_1dv_1 \Big|.
\end{eqnarray*}
Note that, as in Subsection 3.3.1, the condition that $(T)^*_1(1)=0$ implies that for $k_1>k_1^{'},$ we have the following almost orthogonal argument
that for $\|f\|_2\leq 1$ and $\|g\|_2\leq 1,$
\begin{eqnarray*}
&&\Big|\int D_{k'_1}(x_1,u_1)\langle h,K_1(u_1,v_1)f\rangle D_{k_1}(v_1,x_{I_1})du_1dv_1\Big|\\[4pt]
&&\leq C 2^{-(k_1-k'_1)\varepsilon^{'}}\frac{1}{V_{2^{-k'_1}}(x_1)+V_{2^{-k'_1}}(x_{I_1})+V(x_1,x_{I_1})}\frac{2^{-k'_1\varepsilon^{'}}}
{(2^{-k_1^{'}}+d_1(x_1,x_{I_1}))^{\varepsilon^{'}}}.
\end{eqnarray*}
This, as in Subsection 3.3.1, leads to the following decomposition
\begin{eqnarray}
&&\int \sum_{k_1=-\infty}^{\infty}
\sum_{I_1}\mu_1(I_{1})  D_{k'_1}(x_1,u_1)\langle h,K_1(u_1,v_1)f\rangle D_{k_1}(v_1,x_{I_1})\nonumber\\[4pt]
&&\hskip1cm\times \big[\widetilde{\widetilde{D}}_{k_1}(x_{I_1},y_1)-\widetilde{\widetilde{D}}_{k_1}(x_{I_1},y_{1}^{'})\big]
du_1dv_1\nonumber\\
&&=:E+F,\label{Hp bd E and F}
\end{eqnarray}
where for fixed $k_1^{'},$ $E$ corresponds to the summation over $k_1>k_1^{'}$ and $F$ for $k_1\leq k_1^{'}.$

It suffices to show that $|E|$ and $|F|$ both are bounded by the right-hand sides of (\ref{Hp bd e7}). To do this,
for $2d_1(y_1,y_1^{'})\geq 2^{-k_1^{'}}$ we write
\begin{eqnarray*}
|E|&=&\Big|\int \sum_{k_1>k'_1}
\sum_{I_1}\mu_1(I_{1})  D_{k'_1}(x_1,u_1)\langle h,K_1(u_1,v_1)f\rangle D_{k_1}(v_1,x_{I_1})du_1dv_1\Big |\nonumber\\
&&\times \Big |\big[\widetilde{\widetilde{D}}_{k_1}(x_{I_1},y_1)-\widetilde{\widetilde{D}}_{k_1}(x_{I_1},y_{1}^{'})\big]\Big|.
\end{eqnarray*}
Applying the almost orthogonal estimate as mentioned above and the size properties of\break $\widetilde{\widetilde{D}}_{k_1}(x_{I_1},y_1)$ and $\widetilde{\widetilde{D}}_{k_1}(x_{I_1},y_1^{'})$, we obtain that for $\|f\|_2\leq 1$ and $\|g\|_2\leq 1$ the last term above is bounded by
\begin{eqnarray*}
&& C\sum_{k_1>k_1^{'}}
\sum_{I_1}\mu_1(I_{1})2^{-(k_1-k_1^{'})\varepsilon^{'}}{1\over V_{2^{-k_1^{'}}}(x_{1})+V(x_{I_1},x_{1})}
\Big({2^{-k_1^{'}}\over 2^{-k_1^{'}}+
d_1(x_1,x_{I_1})}\Big)^{\varepsilon^{'}} \nonumber\\
&&\hskip 2cm\times \Big [{1\over V_{2^{-k_1}}(y_1)+V(x_{I_1},y_{1})}\Big({2^{-k_1}\over 2^{-k_1}+
d_1(x_{I_1},y_{1})}\Big)^{\varepsilon'}\\
&&\hskip 2.5cm+ {1\over V_{2^{-k_1}}(y_1^{'})+V(x_{I_1},y_{1}^{'})}\Big({2^{-k_1}\over 2^{-k_1}+
d_1(x_{I_1},y_{1}^{'})}\Big)^{\varepsilon'}\Big ].
\end{eqnarray*}
Note that
\begin{eqnarray*}
&&\sum_{k_1>k_1^{'}}
\sum_{I_1}\mu_1(I_{1})2^{-(k_1-k_1^{'})\varepsilon^{'}}{1\over V_{2^{-k_1^{'}}}(x_{1})+V(x_{I_1},x_{1})}
\Big({2^{-k_1^{'}}\over 2^{-k_1^{'}}+
d_1(x_1,x_{I_1})}\Big)^{\varepsilon^{'}} \nonumber\\
&&\hskip1cm\times {1\over V_{2^{-k_1}}(y_1)+V(x_{I_1},y_{1})}\Big({2^{-k_1}\over 2^{-k_1}+
d_1(x_{I_1},y_{1})}\Big)^{\varepsilon'}\\
&&\leq C\sum_{k_1>k_1^{'}}2^{-(k_1-k_1^{'})\varepsilon^{'}}\int_{M_1}{1\over V_{2^{-k_1^{'}}}(x_{1})+V(z_{1},x_{1})}
\Big({2^{-k_1^{'}}\over 2^{-k_1^{'}}+
d_1(x_1,z_{1})}\Big)^{\varepsilon^{'}} \nonumber\\
&&\hskip1cm\times {1\over V_{2^{-k_1}}(y_1)+V(z_{1},y_{1})}\Big({2^{-k_1}\over 2^{-k_1}+
d_1(z_{1},y_{1})}\Big)^{\varepsilon'}dz_1\\
&&\leq C {1\over V_{2^{-k_1^{'}}}(x_1)+V(x_1,y_{1})} \Big({2^{-k_1^{'}}\over 2^{-k_1^{'}}+
d_1(x_1,y_{1})}\Big)^{\varepsilon'}.
\end{eqnarray*}
Thus, for $2d_1(y_1,y_1^{'})\geq 2^{-k_1^{'}},$
\begin{eqnarray*}
|E|\leq C \Big({d_1(y_1,y_{1}^{'})\over 2^{-k_1^{'}}}\Big)^{\varepsilon'}{1\over V_{2^{-k_1^{'}}}(x_1)+V(x_1,y_{1})} \Big({2^{-k_1^{'}}\over 2^{-k_1^{'}}+
d_1(x_1,y_{1})}\Big)^{\varepsilon'},
\end{eqnarray*}
where we use the facts that $2d_1(y_1,y_1^{'})\geq 2^{-k_1^{'}}$ and if $d(y_1,y_1^{'})\leq d_1(x_1,y_1^{'})/2A$ then there exists a positive constant $C$ such that $C^{-1}d_1(x_1,y_1)\leq d_1(x_1,y_1^{'})\leq Cd_1(x_1,y_1)$.

Whenever $2d_1(y_1,y_1^{'})<2^{-k_1^{'}}$ and if $d_1(y_1,y_1^{'})\leq {1\over 2A}(2^{-k_1}+d_1(x_{I_1},y_1))$ or $d_1(y_1,y_1^{'})\leq {1\over 2A}(2^{-k_1}+d_1(x_{I_1},y_1^{'})),$ then applying the almost orthogonal estimate as mentioned above and the smoothness condition for $\big[\widetilde{\widetilde{D}}_{k_1}(x_{I_1},y_1)-\widetilde{\widetilde{D}}_{k_1}(x_{I_1},y_{1}^{'})\big]$ yields that for $\|f\|_2\leq 1$ and $\|g\|_2\leq 1,$
\begin{eqnarray*}
|E|&\leq& C\sum_{k_1>k_1^{'}}
\sum_{I_1}\mu_1(I_{1})2^{-(k_1-k_1^{'})\varepsilon^{'}}{1\over V_{2^{-k_1^{'}}}(x_{1})+V_{2^{-k_1^{'}}}(x_{I_1})+V(x_{I_1},x_{1})}
\Big({2^{-k_1^{'}}\over 2^{-k_1^{'}}+
d_1(x_1,x_{I_1})}\Big)^{\varepsilon^{'}} \nonumber\\[4pt]
&&\hskip.5cm\times \Big [\Big({d_1(y_1,y_1^{'}) \over 2^{-k_1}+d_1(x_{I_1},y_1)}\Big)^{\varepsilon'}{1\over V_{2^{-k_1}}(x_{I_1})+V(x_{I_1},y_{1})}\Big({2^{-k_1}\over 2^{-k_1}+
d_1(x_{I_1},y_{1})}\Big)^{\varepsilon'}\nonumber\\[4pt]
&&\hskip1.2cm+\Big({d_1(y_1,y_1^{'}) \over 2^{-k_1}+d_1(x_{I_1},y_1^{'})}\Big)^{\varepsilon'}{1\over V_{2^{-k_1}}(x_{I_1})+V(x_{I_1},y_{1}^{'})}\Big({2^{-k_1}\over 2^{-k_1}+ d_1(x_{I_1},y_{1}^{'})}\Big)^{\varepsilon'}\Big ]\nonumber\\[4pt]
&\leq&C\Big({d_1(y_1,y_{1}^{'})\over 2^{-k_1^{'}}}\Big)^{\varepsilon'} {1\over V_{2^{-k_1^{'}}}(x_1)+V(x_1,y_{1})} \Big({2^{-k_1^{'}}\over 2^{-k_1^{'}}+
d_1(x_1,y_{1})}\Big)^{\varepsilon'},
\end{eqnarray*}
where the fact that $C^{-1}d_1(x_1,y_1)\leq d_1(x_1,y_1^{'})\leq Cd_1(x_1,y_1)$ is also used.

The proof for $2d_1(y_1,y_1^{'})< 2^{-k_1^{'}},$  $d_1(y_1,y_1^{'})>{1\over 2A}(2^{-k_1}+d_1(x_{I_1},y_1))$ and $d_1(y_1,y_1^{'})>{1\over 2A}(2^{-k_1}+d_1(x_{I_1},y_1^{'}))$ is same as for $2d_1(y_1,y_1^{'})\geq 2^{-k_1^{'}}.$
This implies that $|E|$ is bounded by the right-hand side of (\ref{Hp bd e7}).

We now show that $|F|$ satisfies the same estimates as $|E|$ does.
To this end, again as in Subsection 3.3.1, we decompose $F$ as
\begin{eqnarray*}
F&=&  \int \sum_{k_1\leq k_1^{'}}
\sum_{I_1}\mu_1(I_{1})  D_{k'_1}(x_1,u_1)\langle h,K_1(u_1,v_1)f\rangle [D_{k_1}(v_1,x_{I_1})-D_{k_1}(x_1,x_{I_1})]\nonumber\\[4pt]
&&\hskip1cm\times \big[\widetilde{\widetilde{D}}_{k_1}(x_{I_1},y_1)-\widetilde{\widetilde{D}}_{k_1}(x_{I_1},y_{1}^{'})\big]
du_1dv_1\\[4pt]
&&+  \int \sum_{k_1\leq k_1^{'}}
\sum_{I_1}\mu_1(I_{1})  D_{k'_1}(x_1,u_1)\langle h,K_1(u_1,v_1)f\rangle D_{k_1}(x_1,x_{I_1})\nonumber\\[4pt]
&&\hskip1cm\times \big[\widetilde{\widetilde{D}}_{k_1}(x_{I_1},y_1)-\widetilde{\widetilde{D}}_{k_1}(x_{I_1},y_{1}^{'})\big]
du_1dv_1\\[4pt]
&=&F_1+F_2.
\end{eqnarray*}
Note that when $k_1\leq k_1^{'}$ we have the following almost orthogonal estimate that for $\|f\|_2\leq 1$ and $\|g\|_2\leq 1,$
\begin{eqnarray*}
&&\Big |\int D_{k'_1}(x_1,u_1)\langle h,K_1(u_1,v_1)f\rangle [D_{k_1}(v_1,x_{I_1})-D_{k_1}(x_1,x_{I_1})]du_1dv_1\Big |\\[4pt]
&&\leq C 2^{-(k'_1-k_1)\varepsilon^{'}}\frac{1}{V_{2^{-k_1}}(x_1)+V_{2^{-k_1}}(x_{I_1})+V(x_1,x_{I_1})}
\frac{2^{-k_1\varepsilon^{'}}}{(2^{-k_1}+d_1(x_1,x_{I_1}))^{\varepsilon^{'}}}.
\end{eqnarray*}
Therefore, $F_1$ satisfies the same estimate as $E.$ To estimate $F_2,$ we rewrite it as
\begin{eqnarray*}
F_2&=& \Big|  \sum_{k_1\leq k_1^{'}}
\sum_{I_1}\mu_1(I_{1})   D_{k_1}(x_1,x_{I_1})\big[\widetilde{\widetilde{D}}_{k_1}(x_{I_1},y_1)-\widetilde{\widetilde{D}}_{k_1}(x_{I_1},y_{1}^{'})\big]\nonumber\\[4pt]
&&\hskip1cm\times
\int D_{k'_1}(x_1,u_1)\langle h,K_1(u_1,\cdot)f\rangle(1) du_1\Big|
\\[4pt]
&=&  \big|S_{k_1^{'}}(x_1,y_1)-S_{k_1^{'}}(x_1,y_1^{'})\big|\Big|\int D_{k'_1}(x_1,u_1)\langle h,K_1(u_1,\cdot)f\rangle(1) du_1\Big|,
\end{eqnarray*}
where for $x_1,y_1\in M_1$, $S_{k'_1}(x_1,y_1)=\sum_{k_1\leq k_1^{'}}
\sum_{I_1}\mu_1(I_{1})   D_{k_1}(x_1,x_{I_1})\widetilde{\widetilde{D}}_{k_1}(x_{I_1},y_1)$ and similarly for $S_{k'_1}(x_1,y_1^{'}).$ Note that $S_{k'_1}(x_1,y_1)$ and $S_{k'_1}(x_1,y_1^{'})$ satisfy the size and smoothness properties as proved in Subsections 3.3.2 and 3.3.4, respectively. Similar to the argument in Subsection 3.3.3,
$\langle h,K_1(u_1,\cdot)f\rangle(1),$ as a function of $u_1$, lies in $ BMO(M_1)$ with $\|\langle h,K_1(u_1,\cdot)f\rangle(1)\|_{BMO(M_1)}\break \leq C\|f\|_{L^2(M_2)}\|h\|_{L^2(M_2)}$. Hence $\Big|\int D_{k'_1}(x_1,u_1)\langle h,K_1(u_1,\cdot)f\rangle(1) du_1\Big|\leq C\|f\|_{L^2(M_2)}\|h\|_{L^2(M_2)}$, where the constant $C$ is independent of $k_1^{'}$ and $x_1$ since for any $k_1^{'}$ and $x_1$, $D_{k'_1}(x_1,u_1)$ is in $H^1(M_1)$ with $\|D_{k'_1}(x_1,\cdot)\|_{H^1(M_1)}$  uniformly bounded.
As a consequence, we have
\begin{eqnarray*}
|F_2|
&\leq&C \big|S_{k'_1}(x_1,y_1)-S_{k'_1}(x_1,y_1^{'})\big| \|f\|_{L^2(M_2)}\|h\|_{L^2(M_2)} .
\end{eqnarray*}
Thus, applying the size properties of $S_{k'_1}(x_1,y_1)$ and $S_{k'_1}(x_1,y_1^{'})$ for the case
$k_1^{'}: 2^{-k_1^{'}}\leq 2Ad_1(y_1,y_1^{'})$ and the smoothness properties of $S_{k'_1}(x_1,y_1)$ for the case
$k_1^{'}: 2^{-k_1^{'}}> 2Ad_1(y_1,y_1^{'})$, we obtain that $F_2$ satisfies the same estimate as $F_1$ and then $F$ satisfies the same estimate as $E$ and hence, the proof for $(II)$ is concluded. Applying the same proof implies that $(II)$ holds with interchanging $x_1$ and $y_1.$

As mentioned, the proof for $(I)$ is similar and easier. Indeed, following the same steps in the proof of $(II)$, we have
\begin{eqnarray}
&&\big\| (\mathcal
{L}_{k_1^{'},k_2^{'}})_1(x_1,y_1)\big\|_{L^2_{\mathcal H}(M_2)\rightarrow L^2_{\mathcal H}(M_2)}\nonumber\\[4pt]
&&=\sup_{f:\ \|f\|_{L^2(M_2)}\leq1}\Big(\int_{M_2} \Big\|\int_{M_2}\mathcal
{L}_{k_1^{'},k_2^{'}}(x_1,x_2,y_1,y_2)f(y_2)dy_2\Big\|_{\mathcal{H}}^2dx_2\Big)^{1/2}\nonumber\\[4pt]
&&\leq C \Big(\sum_{k_1^{'}=-\infty}^{\infty}\int_{M_2} \bigg|\int \sum_{k_1=-\infty}^{\infty}
\sum_{I_1}\mu_1(I_{1})  D_{k'_1}(x_1,u_1)K(u_1,x_2,v_1,v_2) D_{k_1}(v_1,x_{I_1}) \nonumber\\[4pt]
&&\hskip.5cm\times {\widetilde{\widetilde{D}}}_{k_1}(x_{I_1},y_1)f(v_2)
du_1dv_1dv_2\bigg|^2dx_2\Big)^{1/2}.\label{Hp bd e8}
\end{eqnarray}
Then, define $E$ and $F$ similarly as in (\ref{Hp bd E and F}) with
$\widetilde{\widetilde{D}}_{k_1}(x_{I_1},y_1)-\widetilde{\widetilde{D}}_{k_1}(x_{I_1},y_{1}^{'})$ replaced by\break $\widetilde{\widetilde{D}}_{k_1}(x_{I_1},y_1)$. Then applying the same proof, we obtain that $E$ and $F$ satisfy the following estimate
\begin{eqnarray*}
|E|+|F|\leq\ C {1\over V_{2^{-k_1^{'}}}(x_1)+V(x_1,y_{1})} \Big({2^{-k_1^{'}}\over 2^{-k_1^{'}}+
d_1(x_1,y_{1})}\Big)^{\varepsilon'}.
\end{eqnarray*}
Inserting the above estimate into (\ref{Hp bd e8}) implies $(I).$ We leave the details to the reader.

We now turn to the proofs of $(III)$ - $(VI)$.

To verify $(III)$--$(VI)$, it suffices to show that there exist
positive constants $C$, $\varepsilon$ and $\varepsilon'$ with
$\varepsilon'<\varepsilon$, such that $\mathcal
{L}_{k_1^{'},k_2^{'}}(x_1,x_2,y_1,y_2),$ the kernel of $\mathcal {L}_{k_1^{'},k_2^{'}},$ satisfies the following estimates   $(D_1)$--$(D_4)$:
\begin{eqnarray*}
(D_1)\ \ |\mathcal {L}_{k_1^{'},k_2^{'}}(x_1,x_2,y_1,y_2)|&\leq&
C\frac{1}{V_{2^{-k_1^{'}}}(x_1)+V_{2^{-k_1^{'}}}(y_1)+V(x_1,y_1)}
\frac{2^{-k_1^{'}\varepsilon^{'}}}{(2^{-k_1^{'}}+d_1(x_1,y_1))^{\varepsilon^{'}}}\hskip1.4cm\\[4pt]
&&\times\frac{1}{V_{2^{-k_2^{'}}}(x_2)+V_{2^{-k_2^{'}}}(y_2)+V(x_2,y_2)}
\frac{2^{-k_2^{'}\varepsilon^{'}}}{(2^{-k_2^{'}}+d_2(x_2,y_2))^{\varepsilon^{'}}};\nonumber
\end{eqnarray*}
\begin{eqnarray*}
(D_2)\ \ \lefteqn{|\mathcal
{L}_{k_1^{'},k_2^{'}}(x_1,x_2,y_1,y_2)-\mathcal
{L}_{k_1^{'},k_2^{'}}(x_1,x_2,y_1,y_2^{'})|}\hskip1.30cm\\[4pt]
 &\leq& C \Big(\frac{d_2(y_2,y_2^{'})}{2^{-k_1^{'}}+d_2(x_2,y_2)}\Big)^{\varepsilon'} \frac{1}{V_{2^{-k_1^{'}}}(x_1)+V_{2^{-k_1^{'}}}(y_1)+V(x_1,y_1)}\frac{2^{-k_1^{'}\varepsilon^{'}}}{(2^{-k_1^{'}}+d_1(x_1,y_1))^{\varepsilon^{'}}}
\nonumber\\[4pt]
&&\times
\frac{1}{V_{2^{-k_2^{'}}}(x_2)+V_{2^{-k_2^{'}}}(y_2)+V(x_2,y_2)}
\frac{2^{-k_2^{'}\varepsilon^{'}}}{(2^{-k_2^{'}}+d_2(x_2,y_2))^{\varepsilon^{'}}}\hskip.2cm\nonumber
\end{eqnarray*}
\hskip1.5cm for  $d_2(y_2,y_2^{'})\leq {1\over{2A}}( 2^{-k_1^{'}}+
d_2(x_2,y_2) ) $;
\begin{eqnarray*}
(D_3)\ \ \lefteqn{|\mathcal
{L}_{k_1^{'},k_2^{'}}(x_1,x_2,y_1,y_2)-\mathcal
{L}_{k_1^{'},k_2^{'}}(x_1,x_2,y_1^{'},y_2)|}\hskip1.30cm\\[4pt]
 &\leq& C \Big(\frac{d_1(y_1,y_1^{'})}{2^{-k_1^{'}}+d_1(x_1,y_1)}\Big)^{\varepsilon'} \frac{1}{V_{2^{-k_1^{'}}}(x_1)+V_{2^{-k_1^{'}}}(y_1)+V(x_1,y_1)}\frac{2^{-k_1^{'}\varepsilon^{'}}}{(2^{-k_1^{'}}+d_1(x_1,y_1))^{\varepsilon^{'}}}
\nonumber\\[4pt]
&&\times
\frac{1}{V_{2^{-k_2^{'}}}(x_2)+V_{2^{-k_2^{'}}}(y_2)+V(x_2,y_2)}
\frac{2^{-k_2^{'}\varepsilon^{'}}}{(2^{-k_2^{'}}+d_2(x_2,y_2))^{\varepsilon^{'}}}\hskip.2cm\nonumber
\end{eqnarray*}
\hskip1.5cm for  $d_1(y_1,y_1^{'})\leq {1\over{2A}}( 2^{-k_1^{'}}+
d_1(x_1,y_1) ) $;
\begin{eqnarray*}
(D_4)\ \ &&|\mathcal {L}_{k_1^{'},k_2^{'}}(x_1,x_2,y_1,y_2)-\mathcal
{L}_{k_1^{'},k_2^{'}}(x_1,x_2,y_1^{'},y_2) -\mathcal
{L}_{k_1^{'},k_2^{'}}(x_1,x_2,y_1,y_2^{'})+\mathcal
{L}_{k_1^{'},k_2^{'}}(x_1,x_2,y_1^{'},y_2^{'})|\\[4pt]
 && \leq C \Big(\frac{d_1(y_1,y_1^{'})}{2^{-k_1^{'}}+d_1(x_1,y_1)}\Big)^{\varepsilon'} \frac{1}{V_{2^{-k_1^{'}}}(x_1)+V_{2^{-k_1^{'}}}(y_1)+V(x_1,y_1)}
 \frac{2^{-k_1^{'}\varepsilon^{'}}}{(2^{-k_1^{'}}+d_1(x_1,y_1))^{\varepsilon^{'}}}
\\[4pt]
&&\hskip.5cm\times
\Big(\frac{d_2(y_2,y_2^{'})}{2^{-k_2^{'}}+d_2(x_2,y_2)}\Big)^{\varepsilon'}
\frac{1}{V_{2^{-k_2^{'}}}(x_2)+V_{2^{-k_2^{'}}}(y_2)+V(x_2,y_2)}\frac{2^{-k_2^{'}\varepsilon^{'}}}{(2^{-k_2^{'}}+d_2(x_2,y_2))^{\varepsilon^{'}}}
\end{eqnarray*}
\hskip1cm for $d_1(y_1,y_1^{'})\leq {1\over2A}( 2^{-k_1^{'}}+
d_1(x_1,y_1) ) $ and $d_2(y_2,y_2^{'})\leq {1\over{2A}}(
2^{-k_2^{'}}+ d_2(x_2,y_2) ).$

\medskip
To show $(D_1)$, as in Subsection 3.3, we will decompose $\mathcal
{L}_{k_1^{'},k_2^{'}}(x_1,x_2,y_1,y_2).$ To be precise, for any fixed integers $k_1^{'}$ and $k_2^{'}$ we consider the following cases.

\smallskip
Case 1. $k_1^{'}\geq k_1 $ and $k_2^{'}\geq k_2 $;

\smallskip
Case 2. $k_1^{'}\geq k_1 $ and $k_2^{'}< k_2 $;

\smallskip
Case 3. $k_1^{'}< k_1 $ and $k_2^{'}\geq k_2 $;

\smallskip
Case 4. $k_1^{'}< k_1 $ and $k_2^{'}< k_2 $.
\smallskip

We write
\begin{eqnarray*}
&&\mathcal {L}_{k_1^{'},k_2^{'}}(x_1,x_2,y_1,y_2)\\[4pt]
&&=\mathcal {L}_{k_1^{'},k_2^{'}}^1(x_1,x_2,y_1,y_2)+\mathcal
{L}_{k_1^{'},k_2^{'}}^2(x_1,x_2,y_1,y_2)+\mathcal {L}_{k_1^{'},k_2^{'}}^3(x_1,x_2,y_1,y_2)+\mathcal
{L}_{k_1^{'},k_2^{'}}^4(x_1,x_2,y_1,y_2),
\end{eqnarray*}
where
\begin{eqnarray*}
&&\mathcal {L}_{k_1^{'},k_2^{'}}^1(x_1,x_2,y_1,y_2)\\[4pt]
&=&  \sum_{k_1\leq k_1^{'}}\sum_{k_2\leq k_2^{'}}
\sum_{I_1}\sum_{I_2
}\mu_1(I_1)\mu_2(I_2)
D_{k'_1}D_{k'_2}TD_{k_1}D_{k_2}(x_1,x_2,x_{I_1},x_{I_2})
\widetilde{\widetilde{D}}_{k_1}(x_{I_1},y_1)\widetilde{\widetilde{D}}_{k_2}(x_{I_2},y_2)
\end{eqnarray*}
and similarly for the other three terms.

We first consider $\mathcal {L}_{k_1^{'},k_2^{'}}^1(x_1,x_2,y_1,y_2)$. Following the Case 1 in Subsection 3.3, we decompose
\begin{eqnarray*}
&&D_{k'_1}D_{k'_2}TD_{k_1}D_{k_2}(x_1,x_2,x_{I_1},x_{I_2})\\[5pt]
&&=: I(x_{1},x_{2},x_{I_1},x_{I_2})+II(x_{1},x_{2},x_{I_1},x_{I_2})+III(x_{1},x_{2},x_{I_1},x_{I_2})
+IV(x_{1},x_{2},x_{I_1},x_{I_2})
\end{eqnarray*}
and then write
\begin{eqnarray*}
&&\mathcal {L}_{k_1^{'},k_2^{'}}^1(x_1,x_2,y_1,y_2)\\[4pt]
&&= \mathcal {L}_{k_1^{'},k_2^{'}}^{1.1}(x_1,x_2,y_1,y_2)+\mathcal {L}_{k_1^{'},k_2^{'}}^{1.2}(x_1,x_2,y_1,y_2)+\mathcal {L}_{k_1^{'},k_2^{'}}^{1.3}(x_1,x_2,y_1,y_2)+\mathcal {L}_{k_1^{'},k_2^{'}}^{1.4}(x_1,x_2,y_1,y_2),
\end{eqnarray*}
where
\begin{eqnarray*}
&&\mathcal {L}_{k_1^{'},k_2^{'}}^{1.1}(x_1,x_2,y_1,y_2)=  \sum_{k_1\leq k_1^{'}}\sum_{k_2\leq k_2^{'}}
\sum_{I_1}\sum_{I_2
}\mu_1(I_1)\mu_2(I_2)
I(x_1,x_2,x_{I_1},x_{I_2})
\widetilde{\widetilde{D}}_{k_1}(x_{I_1},y_1)\widetilde{\widetilde{D}}_{k_2}(x_{I_2},y_2)
\end{eqnarray*}
and similar for the other three cases.

As in Subsection 3.3.1, applying the almost orthogonality estimate and the size properties of $\widetilde{\widetilde{D}}_{k_1}(x_{I_1},y_1)$ and $\widetilde{\widetilde{D}}_{k_2}(x_{I_2},y_2)$ and following the same proof as in Case 1.1 in Subsection 3.3.1, yield
\begin{eqnarray}
&&|\mathcal {L}_{k_1^{'},k_2^{'}}^{1.1}(x_1,x_2,y_1,y_2)|\label{Hp bd e0}\\[4pt]
&\leq&  \sum_{k_1\leq k_1^{'}}\sum_{k_2\leq k_2^{'}}
\sum_{I_1}\sum_{I_2
}\mu_1(I_1)\mu_2(I_2)
|I(x_1,x_2,x_{I_1},x_{I_2})|
|\widetilde{\widetilde{D}}_{k_1}(x_{I_1},y_1)||\widetilde{\widetilde{D}}_{k_2}(x_{I_2},y_2)|\nonumber\\[4pt]
&\leq&C\frac{1}{V_{2^{-k_1^{'}}}(x_1)+V_{2^{-k_1^{'}}}(y_1)+V(x_1,y_1)}
\frac{2^{-k_1^{'}\varepsilon}}{(2^{-k_1^{'}}+d_1(x_1,y_1))^{\varepsilon}}\nonumber\\[4pt]
&&\hskip1cm\times\frac{1}{V_{2^{-k_2^{'}}}(x_2)+V_{2^{-k_2^{'}}}(y_2)+V(x_2,y_2)}
\frac{2^{-k_2^{'}\varepsilon}}{(2^{-k_2^{'}}+d_2(x_2,y_2))^{\varepsilon}},\nonumber
\end{eqnarray}
which implies that $\mathcal {L}_{k_1^{'},k_2^{'}}^{1.1}(x_1,x_2,y_1,y_2)$ satisfies
$(D_1)$.

To deal with $\mathcal {L}_{k_1^{'},k_2^{'}}^{1.4}(x_1,x_2,y_1,y_2),$ as in Subsection 3.3.2, we write
$$ IV(x_1,x_2,x_{I_1},x_{I_2})=D_{k_1^{'}}D_{k_2^{'}}(T1)(x_1,x_2)D_{k_1}(x_1,x_{I_1})D_{k_2}(x_2,x_{I_2}). $$
And then we rewrite
\begin{eqnarray*}
&&\mathcal {L}_{k_1^{'},k_2^{'}}^{1.4}(x_1,x_2,y_1,y_2)\\[4pt]
&=&  \sum_{k_1\leq k_1^{'}}\sum_{k_2\leq k_2^{'}}
\sum_{I_1}\sum_{I_2
}\mu_1(I_1)\mu_2(I_2)
D_{k_1^{'}}D_{k_2^{'}}(T1)(x_1,x_2)D_{k_1}(x_1,x_{I_1})D_{k_2}(x_2,x_{I_2})\\[4pt]
&&\hskip1cm\times\widetilde{\widetilde{D}}_{k_1}(x_{I_1},y_1)\widetilde{\widetilde{D}}_{k_2}(x_{I_2},y_2)\\[4pt]
&=& S_{k_1^{'}}(x_1,y_1)S_{k_2^{'}}(x_2,y_2)D_{k_1^{'}}D_{k_2^{'}}(T1)(x_1,x_2),
\end{eqnarray*}
where for $x_1, y_1\in M_1,$
$${S}_{k_1^{'}}(x_1,y_1)=\sum_{k_1\leq k_1^{'} }\sum_{I_1}\mu_1(I_{1})
D_{k_1}(x_1,x_{I_{1}})\widetilde{\widetilde{D}}_{k_1}(x_{I_{1}},y_1)$$
and similarly for ${S}_{k_2^{'}}(x_2,y_2)$ on $M_2.$   Moreover, as in Subsection 3.3.2, ${S}_{k_1^{'}}(x_1,y_1)$ and ${S}_{k_2^{'}}(x_2,y_2)$ satisfy similar size properties as ${D}_{k_1^{'}}(x_1,y_1)$ and ${D}_{k_2^{'}}(x_2,y_2)$ do, which implies
$$\mathcal {L}_{k_1^{'},k_2^{'}}^{1.4}(x_1,x_2,y_1,y_2)\leq |S_{k_1^{'}}(x_1,y_1)S_{k_2^{'}}(x_2,y_2)|$$
since $(T1)(x_1,x_2)\in BMO(\widetilde M)$ and hence $|D_{k_1^{'}}D_{k_2^{'}}(T1)(x_1,x_2)|$ is bounded uniformly for $k_1^{'}, k_2^{'},x_1$ and $x_2.$ This implies that $\mathcal {L}_{k_1^{'},k_2^{'}}^{1.4}(x_1,x_2,y_1,y_2)$ satisfies
$(D_1)$.

Similarly, we write, as in the Case 1.2 in Subsection 3.3.3,
\begin{eqnarray*}
&&II(x_{1},x_{2},x_{I_1},x_{I_2})\\[4pt]
&&=\int  D_{k'_1}(x_{1},u_1)D_{k'_2}(x_{2},u_2)K(u_1,u_2,v_1,v_2)\\[4pt]
&&\hskip1cm\times[D_{k_2}(v_2,x_{I_2})-D_{k_2}(x_{I_2^{'}},x_{I_2})] du_1du_2dv_1dv_2\ D_{k_1}(x_{1},x_{I_1}) \ +IV(x_{1},x_{2},x_{I_1},x_{I_2})\\[4pt]
&&=\langle D_{k'_2}(x_{2},u_2),\langle D_{k'_1}(x_{1},\cdot), K_2(u_2,v_2)(1)\rangle[D_{k_2}(v_2,x_{I_2})-D_{k_2}(x_{2},x_{I_2})]\rangle D_{k_1}(x_{1},x_{I_1})\\[4pt]
&&\hskip.5cm+IV(x_{1},x_{2},x_{I_1},x_{I_2}).
\end{eqnarray*}
Then, we have
\begin{eqnarray*}
\mathcal {L}_{k_1^{'},k_2^{'}}^{1.2}(x_1,x_2,y_1,y_2)&= & \sum_{k_1\leq k_1^{'}}\sum_{k_2\leq k_2^{'}}
\sum_{I_1}\sum_{I_2
}\mu_1(I_1)\mu_2(I_2)
\langle D_{k'_2}(x_{2},u_2),\langle D_{k'_1}(x_{1},\cdot), K_2(u_2,v_2)(1)\rangle\\[4pt]
&&\hskip.5cm\times[D_{k_2}(v_2,x_{I_2})-D_{k_2}(x_{2},x_{I_2})]\rangle D_{k_1}(x_{1},x_{I_1})
\widetilde{\widetilde{D}}_{k_1}(x_{I_1},y_1)\widetilde{\widetilde{D}}_{k_2}(x_{I_2},y_2) \\[4pt]
&&+\mathcal {L}_{k_1^{'},k_2^{'}}^{1.4}(x_1,x_2,y_1,y_2).
\end{eqnarray*}
Thus, it suffices to verify that the series above satisfies $(D_1).$ To do this, we write the series above as
\begin{eqnarray}
&& \sum_{k_1\leq k_1^{'}}\sum_{k_2\leq k_2^{'}}
\sum_{I_1}\sum_{I_2
}\mu_1(I_1)\mu_2(I_2)
\langle D_{k'_2}(x_{2},u_2),\langle D_{k'_1}(x_{1},\cdot), K_2(u_2,v_2)(1)\rangle\nonumber\\[4pt]
&&\hskip.5cm\times[D_{k_2}(v_2,x_{I_2})-D_{k_2}(x_{2},x_{I_2})]\rangle D_{k_1}(x_{1},x_{I_1})
\widetilde{\widetilde{D}}_{k_1}(x_{I_1},y_1)\widetilde{\widetilde{D}}_{k_2}(x_{I_2},y_2) \nonumber\\[4pt]
&&=\sum_{k_2\leq k_2^{'}}
\sum_{I_2
}\mu_2(I_2)
\langle D_{k'_2}(x_{2},u_2),\langle D_{k'_1}(x_{1},\cdot), K_2(u_2,v_2)(1)\rangle[D_{k_2}(v_2,x_{I_2})-D_{k_2}(x_{2},x_{I_2})]\rangle\nonumber\\[4pt]
&&\hskip.5cm\times \widetilde{\widetilde{D}}_{k_2}(x_{I_2},y_2)S_{k_1^{'}}(x_{1},y_1)\label{Hp bd e3}
\end{eqnarray}
Note that $K_2(u_2,v_2)(1)$ as a function of $u_1$ is in $BMO(M_1)$ with
$\|K_2(u_2,v_2)(1)\|_{BMO(M_1)}$ bounded by
$CV(u_2,v_2)^{-1}$, and that $D_{k'_1}(x_1,u_1)$ as a function of
$u_1$ lies in $H^1(M_1)$. Moreover, $K_2(u_2,v_2)$ is a
Calder\'on--Zygmund kernel on $M_2$ with $|K_2|_{CZ}\leq C$ and, by the fact that $(T^*)_2(1)=0,$
$\int K_2(u_2,v_2) du_2=0.$ As a consequence, we have the following almost orthogonality estimate that for $k_2^{'}\geq k_2$
\begin{eqnarray*}
&&\big|\langle D_{k'_2}(x_{2},u_2),\langle D_{k'_1}(x_{1},\cdot), K_2(u_2,v_2)(1)\rangle[D_{k_2}(v_2,x_{I_2})-D_{k_2}(x_{2},x_{I_2})]\rangle\big|\\[4pt]
&&\leq C |K_2|_{CZ} 2^{-(k_2^{'}-k_2)\varepsilon'}{1\over
V_{2^{-k_2}}(x_2)+
V_{2^{-k_2}}(x_{I_2})+V(x_2,x_{I_2})}\Big(
{ 2^{-k_2}\over 2^{-k_2}+d_2(x_2,x_{I_2})
}\Big)^{\varepsilon},
\end{eqnarray*}
which together with the side condition of $\widetilde{\widetilde{D}}_{k_2}(x_{I_2},y_2)$ implies that
the right-hand side of the equality (\ref{Hp bd e3}) is bounded by
$$C |K_2|_{CZ} {1\over
V_{2^{-k_2^{'}}}(x_2)+ V_{2^{-k_2^{'}}}(y_2)+V(x_2,y_2)}\Big( {
2^{-k_2^{'}}\over 2^{-k_2^{'}}+d_2(x_2,y_2) }\Big)^{\varepsilon}\ \ |S_{k_1^{'}}(x_{1},y_1)|.  $$
This together with the side condition of $S_{k_1^{'}}(x_{1},y_1)$ implies that the right-hand side of the equality (\ref{Hp bd e3}) is bounded by the right-hand side in $(D_1)$ and hence $\mathcal {L}_{k_1^{'},k_2^{'}}^{1.2}(x_1,x_2,y_1,y_2)$ satisfies $(D_1)$. Similarly, $\mathcal {L}_{k_1^{'},k_2^{'}}^{1.3}(x_1,x_2,y_1,y_2)$ satisfies  $(D_1)$. We conclude that $\mathcal {L}_{k_1^{'},k_2^{'}}^{1}(x_1,x_2,y_1,y_2)$ satisfies $(D_1)$.

Now we turn to $\mathcal{L}_{k_1^{'},k_2^{'}}^2(x_1,x_2,y_1,y_2)$. Note that $(T^*)_2(1)=0.$ Similar to the Case 2 in Subsection 3.3, we write
\begin{eqnarray*}
&& D_{k_1^{'}}D_{k_2^{'}} TD_{k_1}D_{k_2}(x_{1},x_{2},x_{I_1},x_{I_2})\\[4pt]
&&=\int  D_{k'_1}(x_{1},u_1)D_{k'_2}(x_{2},u_2)K(u_1,u_2,v_1,v_2)[D_{k_1}(v_1,x_{I_1})-D_{k_1}(x_{1},x_{I_1})]\\[4pt]
&&\hskip1cm\times D_{k_2}(v_2,x_{I_2})du_1du_2dv_1dv_2\\[4pt]
&&\hskip.5cm+ \int  D_{k'_1}(x_{1},u_1)D_{k'_2}(x_{2},u_2)K(u_1,u_2,v_1,v_2)D_{k_1}(x_{1},x_{I_1}) D_{k_2}(v_2,x_{I_2}) du_1du_2dv_1dv_2\\[4pt]
&&=: V(x_{1},x_{2},x_{I_1},x_{I_2})+VI(x_{1},x_{2},x_{I_1},x_{I_2}).
\end{eqnarray*}
Then we rewrite
$$
\mathcal {L}_{k_1^{'},k_2^{'}}^2(x_1,x_2,y_1,y_2)= \mathcal {L}_{k_1^{'},k_2^{'}}^{2.1}(x_1,x_2,y_1,y_2)+\mathcal {L}_{k_1^{'},k_2^{'}}^{2.2}(x_1,x_2,y_1,y_2)
$$
where
\begin{eqnarray*}
&&\mathcal {L}_{k_1^{'},k_2^{'}}^{2.1}(x_1,x_2,y_1,y_2)=  \sum_{k_1\leq k_1^{'}}\sum_{k_2> k_2^{'}}
\sum_{I_1}\sum_{\tau_2
}\mu_1(I_1)\mu_2(I_2)
V(x_1,x_2,x_{I_1},x_{I_2})
\widetilde{\widetilde{D}}_{k_1}(x_{I_1},y_1)\widetilde{\widetilde{D}}_{k_2}(x_{I_2},y_2)
\end{eqnarray*}
and similarly for $\mathcal {L}_{k_1^{'},k_2^{'}}^{2.2}(x_1,x_2,y_1,y_2)$.

By the fact that $(T^*)_2(1)=0,$ $V(x_1,x_2,x_{I_1},x_{I_2})$
satisfies the almost orthogonality estimate in (\ref{almost orth case 1.1}) as for $I(x_1,x_2,x_{I_1},x_{I_2})$ with $k_2$ and $k_2^{'}$ interchanged. Hence, applying the almost orthogonality estimate and
the size properties of $\widetilde{\widetilde{D}}_{k_1}(x_{I_1},y_1)$ and $\widetilde{\widetilde{D}}_{k_2}(x_{I_2},y_2)$ gives
\begin{eqnarray*}
|\mathcal {L}_{k_1^{'},k_2^{'}}^{2.1}(x_1,x_2,y_1,y_2)|\label{Hp bd e4}
&\leq& C\frac{1}{V_{2^{-k_1^{'}}}(x_1)+V_{2^{-k_1^{'}}}(y_1)+V(x_1,y_1)}
\frac{2^{-k_1^{'}\varepsilon}}{(2^{-k_1^{'}}+d_1(x_1,y_1))^{\varepsilon}}\\
&&\quad\times\frac{1}{V_{2^{-k_2^{'}}}(x_2)+V_{2^{-k_2^{'}}}(y_2)+V(x_2,y_2)}
\frac{2^{-k_2^{'}\varepsilon}}{(2^{-k_2^{'}}+d_2(x_2,y_2))^{\varepsilon}},
\end{eqnarray*}
which implies that $\mathcal {L}_{k_1^{'},k_2^{'}}^{2.1}(x_1,x_2,y_1,y_2)$ satisfies
$(D_1)$.

The proof of term $\mathcal {L}_{k_1^{'},k_2^{'}}^{2.2}(x_1,x_2,y_1,y_2)$ is similar to that of $\mathcal {L}_{k_1^{'},k_2^{'}}^{1.4}(x_1,x_2,y_1,y_2)$. Thus $\mathcal {L}_{k_1^{'},k_2^{'}}^{2.2}(x_1,x_2,y_1,y_2)$ satisfies
$(D_1)$. As a result, $\mathcal {L}_{k_1^{'},k_2^{'}}^{2}(x_1,x_2,y_1,y_2)$ satisfies
$(D_1)$. Following the same proof of $\mathcal {L}_{k_1^{'},k_2^{'}}^{2}(x_1,x_2,y_1,y_2)$, $\mathcal {L}_{k_1^{'},k_2^{'}}^{3}(x_1,x_2,y_1,y_2)$ satisfies $(D_1)$.

Finally, note that $(T^*)_1(1)=(T^*)_2(1)=0,$ So $D_{k'_1}D_{k'_2}TD_{k_1}D_{k_2}(x_1,x_2,x_{I_1},x_{I_2})$ satisfies the almost orthogonality estimate in (\ref{almost orth case 1.1}) with $k_1$ and $k_1^{'}$, $k_2$ and $k_2^{'}$ interchanged, respectively, and from this together with the size properties of $\widetilde{\widetilde{D}}_{k_1}(x_{I_1},y_1)$ and $\widetilde{\widetilde{D}}_{k_2}(x_{I_2},y_2)$ yields
that $\mathcal {L}_{k_1^{'},k_2^{'}}^{4}(x_1,x_2,y_1,y_2)$ satisfies
$(D_1)$.

Combing all the estimates of $\mathcal {L}_{k_1^{'},k_2^{'}}^{1}(x_1,x_2,y_1,y_2)$--$\mathcal {L}_{k_1^{'},k_2^{'}}^{4}(x_1,x_2,y_1,y_2)$ we can obtain  that  $\mathcal {L}_{k'_1,k'_2}(x_1,x_2,y_1,y_2)$ satisfies
$(D_1)$.

Replacing $\widetilde{\widetilde{D}}_{k_2}(x_{I_2},y_2)$, $\widetilde{\widetilde{D}}_{k_1}(x_{I_1},y_1)$ and $\widetilde{\widetilde{D}}_{k_1}(x_{I_1},y_1)\widetilde{\widetilde{D}}_{k_2}(x_{I_2},y_2)$ by  $\widetilde{\widetilde{D}}_{k_2}(x_{I_2},y_2)- \widetilde{\widetilde{D}}_{k_2}(x_{I_2},\break y_2^{'})$, $\widetilde{\widetilde{D}}_{k_1}(x_{I_1},y_1)-\widetilde{\widetilde{D}}_{k_1}(x_{I_1},y_1^{'})$ and $[\widetilde{\widetilde{D}}_{k_1}(x_{I_1},y_1)-\widetilde{\widetilde{D}}_{k_1}(x_{I_1},y_1^{'})]
[\widetilde{\widetilde{D}}_{k_2}(x_{I_2},y_2)-\widetilde{\widetilde{D}}_{k_2}(x_{I_2},y_2^{'})],$ respectively, and then applying the same proof as for $(D_1)$ will give the proofs of $(D_2)$ -- $(D_4).$ We leave these details to the reader.

We conclude that $\mathcal {L}_{k_1^{'},k_2^{'}}(x_1,x_2,y_1,y_2)$ satisfies
$(III)$--$(VI)$.

\subsection{``If'' part of $T1$  theorem on $CMO^p$ }

Note that if $f\in {\rm CMO}^p(\widetilde{M}),$ in general, $T(f)$ may not be well defined because $f$ is a distribution in $\big(\GGp(\beta_1,\beta_2;\gamma_1,\gamma_2)\big)'$. The same problem appears in the proof of Theorem 3.6.
The key fact used in the proof of Theorem 3.6 is that
$L^2(\widetilde{M})\cap H^p(\widetilde{M})$ is dense in
$H^p(\widetilde{M})$. It turns out that to establish the boundedness
of $T$ on $H^p(\widetilde{M})$, it suffices to show the $H^p$ boundedness of $T$ for $f\in L^2(\widetilde{M})\cap H^p(\widetilde{M})$. This
method does not work for the present proof of the ``If'' part of Theorem C because
$L^2(\widetilde{M})\cap {\rm CMO}^p(\widetilde{M})$ is not dense in
${\rm CMO}^p(\widetilde{M})$. However, as a
substitution, we have the following

\begin{lemma}\label{main lemma}
For $\max\big(\frac{ 2Q_1}{ 2Q_1+\vartheta_1},\frac{ 2Q_2}{
2Q_2+\vartheta_2} \big) <p\leq1$, $L^2(\widetilde{M})\cap {\rm CMO}^p(\widetilde{M})$ is dense in ${\rm CMO}^p(\widetilde{M})$ in the weak topology $(H^p(\widetilde{M}),{\rm CMO}^p(\widetilde{M}))$. More precisely, for each $f\in {\rm CMO}^p(\widetilde{M})$, there exists a sequence $\{f_n\}\subset L^2(\widetilde{M})\cap {\rm CMO}^p(\widetilde{M})$ such that $\|f_n\|_{{\rm CMO}^p(\widetilde{M})}\leq C\|f\|_{{\rm CMO}^p(\widetilde{M})}$, where $C$ is a positive constant independent of $n$ and $f$, and moreover, for each $g\in H^p(\widetilde{M})$, $\langle f_n,g\rangle\rightarrow \langle f,g\rangle$ as $n\rightarrow \infty.$
\end{lemma}

\begin{proof}[Proof of Lemma \ref{main lemma}]


We first recall the discrete Calder\'on identity, namely,
\begin{eqnarray}\label{special-repro-identity}
f(x_1,x_2)&=&\sum_{k_1,k_2}\sum_{I_1,I_2}|I_1||I_2|
D_{k_1}(x_1,x_{I_1})D_{k_2}(x_2,x_{I_2})
{\widetilde{\widetilde D}}_{k_1}{\widetilde{\widetilde D}}_{k_2}(f)(x_{I_1},x_{I_2}),
\end{eqnarray}
where, for the simplicity, we denote $|I_1|$ for $\mu_1(I_1)$ and similarly $|I_2|$ for $\mu_2(I_2)$, and for each $k_1$ and $k_2$, $I_1,I_2$ range over all the dyadic cubes in $M_1$ and $M_2$ with the diameter $\ell(I_1)=2^{-k_1-N_1}$ and $\ell(I_2)=2^{-k_2-N_2}$. Moreover, the series converges in the both norms in $\GGp(\beta'_1,\beta'_2,\gamma'_1,\gamma'_2)$ with $0<\beta'_i<\beta_i<\vartheta_i, 0<\gamma'_i<\gamma_i<\vartheta_i,\ i=1,2,$ and $L^p(M_1\times M_2),\ 1<p<\infty$. Note that
$D_{k_1}(x_1,x_{I_1})$ and $D_{k_2}(x_2,x_{I_2})$ as functions of $x_1$ and $x_2,$ respectively, have compact supports.

Suppose that $f\in {\rm CMO}^p(\widetilde{M}).$  Set
\begin{eqnarray*}
f_n(x_1,x_2)&:=&\sum_{|k_1|\leq n,|k_2|\leq n}\ \ \sum_{I_1,I_2:I_1\times I_2\subset B_n}|I_1||I_2|
D_{k_1}(x_1,x_{I_1})D_{k_2}(x_2,x_{I_2})
{\widetilde{\widetilde D}}_{k_1}{\widetilde{\widetilde D}}_{k_2}(f)(x_{I_1},x_{I_2}),
\end{eqnarray*}
where $B_n=\lbrace (x_1,x_2): d(x_1,x_1^0)\leq n, d(x_2,x_2^0)\leq n\rbrace.$

It is easy to see that $f_n\in L^2(\widetilde{M}).$ We will show that $f_n\in {\rm CMO}^p(\widetilde{M})$ and moreover, there exists a constant $C$ independent of $n$ and $f$ such that for any open set $\Omega\subset \widetilde{M}$ with finite measure,
\begin{eqnarray}\label{v1}
\frac{\displaystyle 1}{\displaystyle |\Omega|^{{2\over
p}-1}}\int_{\Omega} \sum_{k_1^{'},k_2^{'}} \sum_{I'\times J'\subseteq \Omega}
\big|D_{k_1^{'}}D_{k_2^{'}}(f_n)(x,y)\big|^2 \chi_{I_1^{'}}(x_1)\chi_{I_2^{'}}(x_2)dx_1dx_2\leq C\|f\|_{{\rm CMO}^p(\widetilde{M})}^2.
\end{eqnarray}
To show the above estimate, we need the following almost orthogonal estimate of Lemma 2.11 in \cite{HLL2}. Here and in the rest of the paper, for $a,b\in\mathbb{R}$ we use $a\wedge b$, $a\vee b$ to denote $\min(a,b)$, $\max(a,b)$, respectively.
\begin{lemma}[Lemma 2.11,\cite{HLL2}]\label{lemma-orth product}
Let $\{S_{k_i}\}_{k_i\in\mathbb{Z}}$ and $\{P_{k_i}\}_{k_i\in\mathbb{Z}}$ be two approximations to the identity with regularity exponent $\vartheta_i$ and $D_{k_i}=S_{k_i}-S_{k_i-1},\ E_{k_i}=P_{k_i}-P_{k_i-1},\ i=1,2.$ Then for each $\varepsilon\in(0,\vartheta_1\wedge\vartheta_2)$, there exist positive constants $C$ depending only on $\varepsilon$ such that $D_{l_1}D_{l_2}E_{k_1}E_{k_2}(x_1,x_2,y_1,y_2),$ the kernel of $D_{l_1}D_{l_2}E_{k_1}E_{k_2},$ satisfies the following estimate:
\begin{eqnarray}
&&|D_{l_1}D_{l_2}E_{k_1}E_{k_2}(x_1,x_2,y_1,y_2)|\leq C2^{-|k_1-l_1|\varepsilon}2^{-|k_2-l_2|\varepsilon}\\[4pt]
&&\hskip.7cm\times\frac{1}{V_{2^{-(k_1\wedge l_1)}}(x_1)+V_{2^{-(k_1\wedge l_1)}}(y_1)+V(x_1,y_1)}\frac{2^{-(k_1\wedge l_1)\varepsilon}}{(2^{-(k_1\wedge l_1)}+d(x_1,y_1))^{\varepsilon}}\nonumber\\[4pt]
&&\hskip1cm\times \frac{1}{V_{2^{-(k_2\wedge l_2)}}(x_2)+V_{2^{-(k_2\wedge l_2)}}(y_2)+V(x_2,y_2)}\frac{2^{-(k_2\wedge l_2)\varepsilon}}{(2^{-(k_2\wedge l_2)}+d(x_2,y_2))^{\varepsilon}}.\nonumber
\end{eqnarray}

\end{lemma}

\medskip
We turn to the proof of Lemma \ref{main lemma}. Note that from the definition of $f_n$, we have
\begin{eqnarray*}
&&
D_{k_1^{'}}D_{k_2^{'}}(f_n)(x_1,x_2) \\[5pt]
&=& \sum_{|k_1|\leq n,|k_2|\leq n} \sum_{I_1,I_2:I_1\times I_2\subset B_n}|I_1||I_2|
D_{k_1^{'}}D_{k_1}D_{k_2^{'}}D_{k_2}(x_1, x_2,x_{I_1},x_{I_2})
{\widetilde{\widetilde D}}_{k_1}{\widetilde{\widetilde D}}_{k_2}(f)(x_{I_1},x_{I_2}).
\end{eqnarray*}

Applying Lemma \ref{lemma-orth product}  for the term $D_{k_1^{'}}D_{k_1}D_{k_2^{'}}D_{k_2}(x_1, x_2,x_{I_1},x_{I_2})$ first and then using the H\"older's inequality, we obtain
\begin{eqnarray*}
&&\sup_{x_1\in I_1^{'},x_2\in I_2^{'}}
\big|D_{k_1^{'}}D_{k_2^{'}}(f_n)(x_1,x_2)\big|^2 \\[5pt]
&&\lesssim \sum_{k_1,k_2}
2^{-|k_1-k_1^{'}|\varepsilon_1}2^{-|k_2-k_2^{'}|\varepsilon_2}
\sum_{I_1,I_2}|I_1||I_2|  \frac{\displaystyle 1 }{\displaystyle
V(x_I,x_{I'}) + V_{2^{-(k_1\wedge k_1^{'})}}(x_{I})+ V_{2^{-(k_1\wedge k_1^{'})}}(x_{I'})}\\[5pt]
&&\qquad\times \bigg( \frac{\displaystyle 2^{-(k_1\wedge k_1^{'})}
}{\displaystyle 2^{-(k_1\wedge k_1^{'})}+ d(x_I,x_{I'})  }
\bigg)^{\varepsilon_1} \frac{\displaystyle 1 }{\displaystyle
V(x_{I_2},x_{I_2^{'}}) + V_{2^{-(k_2\wedge
k_2^{'})}}(x_{I_2})+V_{2^{-(k_2\wedge k_2^{'})}}(x_{I_2^{'}})   }\\[5pt]
&&\qquad\times \bigg( \frac{\displaystyle 2^{-(k_2\wedge k_2^{'})}
}{\displaystyle 2^{-(k_2\wedge k_2^{'})}+ d(x_{I_2},x_{I_2^{'}})  }
\bigg)^{\varepsilon_2} \big| \widetilde{\widetilde{D}}_{k_1}\widetilde{\widetilde{D}}_{k_2}[f](x_{I_1},x_{I_2})
\big|^2.
\end{eqnarray*}
As a consequence, we have
\begin{eqnarray}\label{estimate-of-sumR SR 1}
\lefteqn{ {1\over |\Omega|^{{2\over
p}-1}}\sum_{k_1^{'},k_2^{'}}\sum_{I_1^{'}\times I_2^{'}\subset \Omega} |I_1^{'}||I_2^{'}| \sup_{x_1\in
I_1^{'},x_2\in I_2^{'}}\big|D_{k_1^{'}}D_{k_2^{'}}[f_n](x_1,x_2)\big|^2             }
\\[5pt]
&\lesssim& {1\over |\Omega|^{{2\over
p}-1}}\sum_{k_1^{'},k_2^{'}}\sum_{I_1^{'}\times I_2^{'}\subset
\Omega}\sum_{k_1,k_2} \sum_{I_1,I_2}2^{-|k_1-k_1^{'}|\varepsilon_1}2^{-|k_2-k_2^{'}|\varepsilon_2}|I_1||I_2| |I_1^{'}||I_2^{'}|\nonumber\\[5pt]
&&\times \bigg( \frac{\displaystyle 2^{-(k_1\wedge k_1^{'})}
}{\displaystyle 2^{-(k_1\wedge k_1^{'})}+ d(x_{I_1},x_{I_1^{'}})  }
\bigg)^{\varepsilon_1}\frac{\displaystyle 1 }{\displaystyle
V(x_{I_1},x_{I_1^{'}}) + V_{2^{-(k_1\wedge k_1^{'})}}(x_{I_1})+ V_{2^{-(k_1\wedge k_1^{'})}}(x_{I_1^{'}})}\nonumber\\[5pt]
&&\times \bigg( \frac{\displaystyle 2^{-(k_2\wedge k_2^{'})}
}{\displaystyle 2^{-(k_2\wedge k_2^{'})}+ d(x_{I_2},x_{I_2^{'}})  }
\bigg)^{\varepsilon_2}\frac{\displaystyle 1 }{\displaystyle
V(x_{I_2},x_{I_2^{'}}) + V_{2^{-(k_2\wedge
k_2^{'})}}(x_{I_2})+V_{2^{-(k_2\wedge k_2^{'})}}(x_{I_2^{'}})   }\nonumber\\[5pt]
&&\times\big| \widetilde{\widetilde{D}}_{k_1}\widetilde{\widetilde{D}}_{k_2}[f](x_{I_1},x_{I_2})
\big|^2.\nonumber
\end{eqnarray}
Note that  $2^{-|k_1-k_1^{'}|}\approx \frac{\displaystyle {\rm
diam}(I_1) }{\displaystyle {\rm diam}(I_1^{'})} \wedge \frac{\displaystyle
{\rm diam}(I_1^{'}) }{\displaystyle {\rm diam}(I_1)}$, $2^{-(k_1\wedge
k_1^{'})}\approx {\rm diam}(I_1) \vee {\rm diam}(I_1^{'})$ and
$d(x_{I_1},x_{I'})\geq {\rm dist}(I_1,I_1^{'})$. Similar results hold for
$k_2,k_2^{'}$ and $I_2,I_2^{'}$. Applying the above estimate with any arbitrary points $x_{I_1^{'}}$ and $x_{I_2^{'}}$
in $I_1^{'}$ and $I_2^{'},$ respectively, and the fact that $ab=(a\vee b)^2\big({a\over b}\wedge{b\over a}\big)$
for all $a,b>0$, we obtain that the right-hand in (4.43) is dominated by a constant times
\begin{eqnarray}\label{estimate-of-sumR SR}
&&{1\over |\Omega|^{{2\over
p}-1}}\sum_{k_1^{'},k_2^{'}}\sum_{I_1^{'}\times I_2^{'}\subset
\Omega}\sum_{k_1,k_2} \sum_{I_1,I_2}\bigg[{{|I_1|}\over
{|I_1^{'}|}}\wedge{{|I_1^{'}|}\over {|I_1|}}\bigg]\bigg[{{|I_2|}\over
{|I_2^{'}|}}\wedge{{|I_2^{'}|}\over {|J|}}\bigg] \bigg[\frac{\displaystyle
{\rm diam}(I_1) }{\displaystyle {\rm diam}(I_1^{'})} \wedge
\frac{\displaystyle {\rm diam}(I_1^{'}) }{\displaystyle {\rm
diam}(I_1)}\bigg]^{\varepsilon_1}\nonumber\\[5pt]
&&\times\bigg[\frac{\displaystyle {\rm diam}(I_2) }{\displaystyle {\rm
diam}(I_2^{'})} \wedge \frac{\displaystyle {\rm diam}(I_2^{'}) }{\displaystyle
{\rm diam}(I_2)}\bigg]^{\varepsilon_2}\cdot
\big(|I_1|\vee|I_1^{'}|\big)\big(|I_2|\vee|I_2^{'}|\big)
\nonumber\\[5pt]
&&\times \frac{\displaystyle |I_1|\vee|I_1^{'}| }{\displaystyle V_{{\rm
dist}(I_1,I_1^{'})}(x_{I_1}) + |I_1|\vee|I_1^{'}| } \bigg( \frac{\displaystyle {\rm
diam}(I_1) \vee {\rm diam}(I_1^{'}) }{\displaystyle
{\rm diam}(I_1) \vee {\rm diam}(I_1^{'})+ {\rm dist}(I_1,I_1^{'})  } \bigg)^{\varepsilon_1}\nonumber\\[5pt]
&&\times \frac{\displaystyle |I_2|\vee|I_2^{'}| }{\displaystyle V_{{\rm
dist}(I_2,I_2^{'})}(x_{I_2}) + |I_2|\vee|I_2^{'}| } \bigg( \frac{\displaystyle {\rm
diam}(I_2) \vee {\rm diam}(I_2^{'}) }{\displaystyle {\rm
diam}(I_2) \vee {\rm diam}(I_2^{'})+ {\rm dist}(I_2,I_2^{'})  } \bigg)^{\varepsilon_2}\nonumber\\[5pt]
&&\times \inf_{x_1\in I_1^{'},x_2\in I_2^{'}} \big| \widetilde{\widetilde D}_{k_1}\widetilde{\widetilde D}_{k_2}[f](x_1,x_2)
\big|^2.
\end{eqnarray}
Following the same steps as in the proof of Theorem 3.2 in \cite{HLL2} gives
\begin{eqnarray}\label{v1}
&&\frac{\displaystyle 1}{\displaystyle |\Omega|^{{2\over
p}-1}}\int_{\Omega} \sum_{k_1^{'},k_2^{'}} \sum_{I_1^{'}\times I_2^{'}\subseteq \Omega}
\big|D_{k_1^{'}}D_{k_2^{'}}(f_n)(x_1,x_2)\big|^2 \chi_{I_1^{'}}(x_1)\chi_{I_2^{'}}(x_2)dx_1dx_2\\[5pt]
&&\hskip.5cm\leq C\frac{\displaystyle 1}{\displaystyle |\Omega|^{{2\over
p}-1}}\int_{\Omega} \sum_{k_1,k_2} \sum_{I_1\times I_2\subseteq \Omega}
\big|\widetilde{\widetilde D}_{k_1}\widetilde{\widetilde D}_{k_2}(f)(x_1,x_2)\big|^2 \chi_{I_1}(x_1)\chi_{I_2}(x_2)dx_1dx_2.\nonumber
\end{eqnarray}
Taking supremum over all open sets $\Omega$ with finite measures, we obtain
\begin{eqnarray*}
\|f_n\|^2_{CMO^p}\leq C\sup_{\Omega}\frac{\displaystyle 1}{\displaystyle |\Omega|^{{2\over
p}-1}}\int_{\Omega} \sum_{k_1,k_2} \sum_{I_1\times I_2\subseteq \Omega}
\big|\widetilde{\widetilde D}_{k_1}\widetilde{\widetilde D}_{k_2}(f)(x_1,x_2)\big|^2 \chi_{I_1}(x_1)\chi_{I_2}(x_2)dx_1dx_2.
\end{eqnarray*}
The last term above, however, by the Plancherel--P\^olya inequality
for the space $CMO^p(\widetilde M)$ in [HLL2], is dominated by
\begin{eqnarray*}
C\sup_{\Omega}\frac{\displaystyle 1}{\displaystyle |\Omega|^{{2\over
p}-1}}\int_{\Omega} \sum_{k_1,k_2} \sum_{I_1\times I_2\subseteq \Omega}
\big| D_{k_1} D_{k_2}(f)(x_1,x_2)\big|^2 \chi_{I_1}(x_1)\chi_{I_2}(x_2)dx_1dx_2.
\end{eqnarray*}
This implies that $\|f_n\|^2_{CMO^p}\leq C \|f\|^2_{CMO^p}.$

We verify that $f_n$ converges to $f$ in the week topology $(H^p, CMO^p)$.
To do this, for any $h\in \GGp(\beta_1,\beta_2;\gamma_1,\gamma_2)$, by the discrete Calder\'on's identity,
\begin{eqnarray*}
\langle f-f_n, h\rangle &= & \langle \sum \limits_{|k_1|> n,\ \text{or}\ |k_2|>n,\ \text{or}\ I_1\times I_2\nsubseteq B_n }  |I_1||I_2|
D_{k_1}(\cdot,x_{I_1})D_{k_2}(\cdot,x_{I_2})
{\widetilde{\widetilde D}}_{k_1}{\widetilde{\widetilde D}}_{k_2}(f)(x_{I_1},x_{I_2}), h\rangle \\[5pt]
&= & \sum \limits_{|k_1|> n,\ \text{or}\ |k_2|>n,\ \text{or}\ I\times
J\nsubseteq B_n}  |I_1||I_2|
D_{k_1}D_{k_2}(h)(x_{I_1},x_{I_2})
{\widetilde{\widetilde D}}_{k_1}{\widetilde{\widetilde D}}_{k_2}(f)(x_{I_1},x_{I_2}).
\end{eqnarray*}
To see that the last term above tends to zero as $n$ tends to infinity, we write
\begin{eqnarray*}
&&\sum \limits_{|k_1|> n,\ \text{or}\ |k_2|>n,\ \text{or}\ I_1\times
I_2\nsubseteq B_n}  |I_1||I_2|D_{k_1}D_{k_2}(h)(x_{I_1},x_{I_2})
{\widetilde{\widetilde D}}_{k_1}{\widetilde{\widetilde D}}_{k_2}(f)(x_{I_1},x_{I_2})\\[5pt]
&&=\big\langle \sum \limits_{|k_1|> n,\ \text{or}\ |k_2|>n,\ \text{or}\ I_1\times
I_2\nsubseteq B_n}  |I_1||I_2|D_{k_1}(\cdot, x_{I_1})D_{k_2}(\cdot, x_{I_2})
{\widetilde{\widetilde D}}_{k_1}{\widetilde{\widetilde D}}_{k_2}(f)(x_{I_1},x_{I_2}), h\big\rangle.
\end{eqnarray*}
Following the proof of Proposition 2.14, the Plancherel-P\^olya inequality,
$$\sum \limits_{|k_1|> n,\ \text{or}\ |k_2|>n,\ \text{or}\ I_1\times
I_2\nsubseteq B_n}  |I_1||I_2|D_{k_1}(x_1, x_{I_1})D_{k_2}(x_2, x_{I_2})
{\widetilde{\widetilde D}}_{k_1}{\widetilde{\widetilde D}}_{k_2}(f)(x_{I_1},x_{I_2})$$
tends to zero in the $H^p$ norm as $n$ tends to infinity and
hence, by the duality argument, $\langle f-f_n, h\rangle $ tends to 0 as $n$ tends to infinity. Note that $ \GGp(\beta_1,\beta_2;\gamma_1,\gamma_2)$ is dense in $H^p(\widetilde{M}).$ Then for any $g\in H^p(\widetilde M),$ $\langle f-f_n, g\rangle $ still tends to 0 as $n$ tends to infinity. Indeed, if $g\in H^p(\widetilde{M})$ and for any $\varepsilon>0$, there exists a function $h\in  \GGp(\beta_1,\beta_2;\gamma_1,\gamma_2)$ such that $\|g-h\|_{H^p(\widetilde{M})}<\varepsilon$. Now by the duality and the fact that $\|f_n\|_{{\rm CMO}^p(\widetilde{M})}\leq C\|f\|_{{\rm CMO}^p(\widetilde{M})}$, we have
\begin{eqnarray*}
\big|\langle f-f_n, g\rangle\big| &\leq & \big|\langle f-f_n, g-h\rangle\big|+ \big|\langle f-f_n, h\rangle\big|  \\[5pt]
&\leq&  \|f-f_n\|_{{\rm CMO}^p(\widetilde{M})}\|g-h\|_{H^p(\widetilde{M})}+ \big|\langle f-f_n, h\rangle\big|  \\[5pt]
&\leq&  C\varepsilon\|f\|_{{\rm CMO}^p(\widetilde{M})}+ \big|\langle f-f_n, h\rangle\big|,
\end{eqnarray*}
which implies that $\lim\limits_{n\rightarrow\infty}\langle f-f_n, g\rangle=0$. The proof of Lemma \ref{main lemma} is completed.

\end{proof}

We are ready to show ``if'' part of Theorem C.

We first define $T$ on ${\rm CMO}^p(\widetilde{M})$ as follows.  Given $f\in {\rm CMO}^p(\widetilde{M}),$ by Lemma \ref{main lemma}, there is a sequence $\lbrace f_n\rbrace\subset L^2(\widetilde{M})\cap {\rm CMO}^p(\widetilde{M})$ such that $ \Vert f_n\Vert_{{\rm CMO}^p(\widetilde{M})}\leq C\Vert f\Vert_{{\rm CMO}^p(\widetilde{M})},$ and  for each $g\in L^2(\widetilde{M})\cap H^p(\widetilde{M}), \langle f_n, g\rangle\rightarrow \langle f, g\rangle$ as $n\rightarrow \infty.$ Thus, for  $f\in {\rm CMO}^p(\widetilde{M})$, we  define
$$ \langle T(f), g\rangle:=\lim \limits_{n \rightarrow \infty} \langle T(f_n), g\rangle$$
for each $g\in L^2(\widetilde{M})\cap H^p(\widetilde{M}).$

To see that this limit exists, we note that $\langle T(f_j-f_k), g\rangle=\langle f_j-f_k, T^*(g)\rangle$
 since both $f_j-f_k$ and $g $ belong to $L^2$ and $T$ is bounded on $L^2$. $T^*$ is bounded on $L^2$ and the kernel of $T^*$ satisfies the conditions in Theorem B. Moreover, $((T^*)_1)^*(1)=T_1(1)=0$ and $((T^*)_2)^*(1)=T_2(1)=0.$ Therefore, by the ``if" part of Theorem B which has been proved in Subsection 4.1, $T^*(g) \in L^2(\widetilde{M})\cap  H^p(\widetilde{M}).$ Thus, by Lemma \ref{main lemma}, $\langle f_j-f_k, T^*(g)\rangle$ tends to zero as $j,k\rightarrow \infty.$ It is also easy to see that this limit is independent of the choice of the sequence $f_n$ that satisfies the conditions in Lemma~\ref{main lemma}.

To finish the proof of ``if'' part of Theorem C, we claim that for each $f\in L^2(\widetilde{M})\cap {\rm CMO}^p(\widetilde{M})$,
\begin{eqnarray}\label{T bd on L2 cap CMO}
\|T(f)\|_{{\rm CMO}^p(\widetilde{M})}\leq C\|f\|_{{\rm CMO}^p(\widetilde{M})},
\end{eqnarray}
where the constant $C$ is independent of $f$.

To see the above claim implies the ``if'' part of Theorem C, by the definition of $T$ on ${\rm CMO}^p(\widetilde{M})$, for each $g\in L^2(\widetilde{M})\cap H^p(\widetilde{M}),
\langle T(f),g\rangle= \lim_{n\rightarrow\infty} \langle T(f_n), g\rangle,$ where $f_n$ satisfies the conditions in Lemma~\ref{main lemma}.
Particularly, taking $g(x,y)=D_{k_2}D_{k_1}(x,y)\in  \GGp(\beta_1,\beta_2;\gamma_1,\gamma_2)$ and applying the claim yield
\begin{eqnarray*}
\|T(f)\|_{{CMO}^p(\widetilde{M})}&=&\| \lim_{n\rightarrow\infty}T(f_n)\|_{{CMO}^p(\widetilde{M})}\\[5pt]
&\leq&  \liminf_{n\rightarrow\infty} \|T(f_n)\|_{{\rm CMO}^p(\widetilde{M})}
\leq C \|f_n\|_{{CMO}^p(\widetilde{M})}\\[5pt]
&\leq&  C \|f\|_{{CMO}^p(\widetilde{M})}.
\end{eqnarray*}
Thus, it remains to show the claim.
The proof of the claim follows from Theorem 2.18, the duality between $H^p(\widetilde M)$ and $CMO^p(\widetilde M)$, and the ``if" part of Theorem B. To be more precisely, let $f \in L^2\cap CMO^p(\widetilde M)$ and $g \in L^2\cap H^p(\widetilde M)$. By the duality first and then the ``if" part of Theorem B, we have
$$
|\langle T(f),g\rangle |= |\langle f, {T}^\ast(g)\rangle | \le \|f\|_{CMO^p(\widetilde{M})} \|{T}^\ast(g)\|_{H^p(\widetilde{M})}\le C \|f\|_{CMO^p(\widetilde{M})}\|g\|_{H^p(\widetilde{M})}.
$$
This implies that for each $f \in L^2(\widetilde{M})\cap CMO^p(\widetilde M), \ell_f(g)=\langle T(f),g\rangle$ defines a continuous linear functional on $L^2(\widetilde{M})
\cap H^p(\widetilde M)$. Note that $L^2(\widetilde{M})\cap H^p(\widetilde M)$ is dense in $H^p(\widetilde M).$ Thus, $\ell_f(g)=\langle T(f), g\rangle$ belongs to the dual of $H^p(\widetilde M)$ and the norm of this linear functional is dominated by $C\|f\|_{CMO^p}.$ By the duality, that is Theorem 2.18, again, there exists $h \in CMO^p(\widetilde M)$ such that
$\langle T(f),g\rangle =\langle h, g\rangle$ for each $g\in \GGp(\beta_1,\beta_2;\gamma_1,\gamma_2)$ and $\|h\|_{CMO^p}\le C\|\ell_f\|\le C \|f\|_{CMO^p(\widetilde{M})}.$ The crucial fact we will use is that, taking $g(x,y)=D_{k_2}D_{k_1}(x,y),$ we obtain that $\langle T(f), D_{k_2}D_{k_1}\rangle =\langle h, D_{k_2}D_{k_1}\rangle$. Therefore, by the definition of space $CMO^p(\widetilde{M}),$ we have
\begin{eqnarray*}
\|T(f)\|_{CMO^p(\widetilde{M})}&=& \sup_{\Omega}\bigg\{ \frac {1}{|\Omega|^{\frac{2}{p}-1}}
  \sum_{k_1, k_2\in \Bbb Z} \sum_{I_1,I_2: I_1 \times I_2\subset \Omega}
  |D_{k_2}D_{k_1}(T(f))(x_{I_1}, x_{I_2}) | ^2 |I_1| |I_2| \bigg\}^{1/2} \\[5pt]
  &=& \sup_{\Omega}\bigg\{ \frac {1}{|\Omega|^{\frac{2}{p}-1}}
  \sum_{k_1, k_2\in \Bbb Z} \sum_{I_1,I_2:  I_1 \times I_2\subset \Omega}
  |D_{k_2}D_{k_1}(h)(x_{I_1}, x_{I_2}) | ^2 |I_1| |I_2| \bigg\}^{1/2}\\[5pt]
  &=& \|h\|_{CMO^p(\widetilde{M})}\\[5pt]
  & \leq& C \|f\|_{CMO^p(\widetilde{M})}.
\end{eqnarray*}

The proof of the claim is concluded and hence the proof of ``if '' part of Theorem C is complete.


\subsection{``Only if'' part of $T1$  theorems on $H^p$ and $CMO^p$}

We first show the ``only if'' part of Theorem C. Suppose that $T$ is a
Calder\'on--Zygmund operator defined in Subsection 3.1 and bounded on
$CMO^p(\widetilde M).$  For each $f_2(x_2)\in C^\eta _0(M_2)$, we
define the function $f(x_1,x_2)$ on $\widetilde M$ by $f(x_1,x_2):=
\chi_1(x_1) f_2(x_2)$, where $\chi_1(x_1)=1$ on $M_1$. It is clear
that $f$ is in $CMO^p(\widetilde M)$ with $\|f\|_{CMO^p(\widetilde{M})}=0.$
Consequently, we have $Tf \in CMO^p(\widetilde M)$ and
$\|Tf\|_{CMO^p(\widetilde{M})}=0.$ Therefore,
$$\int_{M_2} \int_{M_1}\int_{M_2}\int_{M_1} g_1(x_1)g_2(x_2) K(x_1,y_1,x_2,y_2)f_2(y_2)dx_1dx_2dy_1dy_2=0$$
for all $g_1\in C^\eta_0(M_1)$ with $\int g_1(x_1)dx_1=0, g_2\in C^\eta _0(M_2)$ with $\int g_2(x_2)dx_2=0$ and all $f_2\in C^\eta_0(M_2).$
Note that the above equality is equivalent to
$$\int_{M_2} \int_{M_1} T^*(g_1\otimes g_2)(y_1, y_2) f_2(y_2)dy_1dy_2=0.$$
Since $T$ is bounded on $L^2(\widetilde{M})$, so $T^*$ is also bounded on $L^2(\widetilde{M}).$ Therefore, $T^*(g_1\otimes g_2)\in L^1(\widetilde{M})\cap L^2(\widetilde{M})$ since $(g_1\otimes g_2)\in H^1(\widetilde M).$ Note that $C^\eta_0(M_2)$ is dense in $L^2(M_2).$
This implies
$$\int_{M_1} T^*(g_1\otimes g_2)(y_1, y_2) dy_1=0=\int_{M_1} \int_{M_2}\int_{M_1}g_1(x_1)g_2(x_2) K(x_1,y_1,x_2,y_2)dx_1dx_2dy_1$$
for all $g_1\in C^\eta_0(M_1)$ with $\int g_1(x_1)dx_1=0, g_2\in C^\eta_0(M_2)$ with $\int g_2(x_2)dx_2=0$ and for $y_2\in M_2$ almost everywhere. Thus, $T_1(1)=0.$ Similarly we can prove that $T_2(1)=0$.

We now prove the ``only if'' part of Theorem B. We claim that if $T$ is bounded on $L^2$ and $H^p(\widetilde M),$ then the adjoint operator $T^*$ extends to a bounded operator from $CMO^p(\widetilde M)$ to itself, where $T^*$ is defined originally by
$$\langle Tf, g\rangle=\langle f, T^*g\rangle $$
for all $f, g\in L^2(\widetilde{M}).$

To see this, let $f\in L^2(\widetilde{M})\cap H^p(\widetilde M)$ and $g\in L^2(\widetilde{M})\cap CMO^p(\widetilde M),$ then, by the duality between $H^p(\widetilde{M})-CMO^p(\widetilde{M}),$
$$|\langle T^*g, f\rangle |=|\langle g,Tf\rangle |\leq C\|f\|_{H^p(\widetilde{M})}\|g\|_{CMO^p(\widetilde{M})}.$$
This implies that $\langle T^*g, f\rangle $ defines a continuous linear functional on $H^p(\widetilde M)$ because $L^2(\widetilde{M})\cap H^p(\widetilde M)$ is dense in
$H^p(\widetilde M).$ Moreover, applying the same proof given in Subsection 4.2 yields
$$\|T^*g\|_{CMO^p(\widetilde{M})}\leq C\|g\|_{CMO^p(\widetilde{M})}.$$
Then, applying the ``only if'' part of Theorem C for the operator $T^*$ implies that $(T^*)_1(1)=(T^*)_2(1)=0.$

\section{The $T1$  theorem of $n$ factors}

In this section we consider the $T1$  theorem on $\widetilde{M}=M_1\times \cdots\times M_n$. To do this, we first consider the case $n=3$, i.e., $\widetilde{M}=M_1\times M_2\times M_3.$ The general case with $n$ factors will follow by induction.

We first recall the definition of the Littlewood--Paley square function on $\widetilde M.$
\begin{definition}\label{def-of-squa-func-n-factor}
Let $\{S_{k_i}\}_{k_i\in\mathbb{Z}}$ be approximations to the
identity on $M_i$ and $D_{k_i}=S_{k_i}-S_{k_i-1}, i=1,2,3.$ For $f\in
\big(\GGp(\beta_1,\beta_2,\beta_3;\gamma_1,\gamma_2,\gamma_3)\big)'$ with
$0<\beta_i,\gamma_i<\vartheta_i, i=1,2,3$, $\widetilde S_d(f),$ the
discrete Littlewood--Paley square function of $f,$ is defined by
\begin{eqnarray*}
&&\widetilde S_d(f)(x_1,x_2,x_3)\\[5pt]
&&=\Big\{\sum_{k_1=-\infty}^\infty\sum_{k_2=-\infty}^\infty \sum_{k_3=-\infty}^\infty\sum_{I_1}\sum_{I_2} \sum_{I_3}
|D_{k_1}D_{k_2} D_{k_3}(f)(x_1,x_2,x_3)|^2\chi_{I_{1}}(x_1)\chi_{I_{2}}(x_2)\chi_{I_{3}}(x_3) \Big\}^{1/2},
\end{eqnarray*}
where for each $k_i$, $I_{i}$ ranges over all
the dyadic cubes in $M_i$ with side-length
$\ell(I_i)=2^{-k_i-N_i}$, and $N_i$ is a large fixed positive integers, for $i=1,2,3$.
\end{definition}
We recall the Hardy spaces $H^p$ and generalized Carleson measure spaces $CMO^p$ on $\widetilde M$ as follows.
\begin{definition}[\cite{HLL2}]\label{def-of-product-Hardy-space}
Let $\max\big(\frac{ Q_1}{ Q_1+\vartheta_1},\frac{ Q_2}{ Q_2+\vartheta_2},\frac{ Q_3}{
Q_3+\vartheta_3} \big) <p\leq1$ and
$0<\beta_i,\gamma_i<\vartheta_i$ for $i=1,2,3$.
$$H^p(\widetilde M):=\big\lbrace f \in
\big(\GGp(\beta_1,\beta_2,\beta_3;\gamma_1,\gamma_2,\gamma_3)\big)':\
\widetilde S_d(f)\in L^p(\widetilde M)\big\rbrace$$
and if $f\in H^p(\widetilde M),$ the norm of $f$ is defined by $\Vert f\Vert_{H^p(\widetilde M)}=\Vert \widetilde S_d(f)\Vert_p.$
\end{definition}

\begin{definition}[\cite{HLL2}]\label{def-of-product-CMO-space}
Let $\max\big(\frac{ 2Q_1}{ 2Q_1+\vartheta_1},\frac{ 2Q_2}{ 2Q_2+\vartheta_2},\frac{ 2Q_3}{
2Q_3+\vartheta_3} \big) <p\leq1$ and
$0<\beta_i,\gamma_i<\vartheta_i$ for $i=1,2,3$. Let
$\{S_{k_i}\}_{k_i\in\mathbb{Z}}$ be approximations to the identity
on $M_i$ and for $k_i\in\mathbb{Z}$, set
$D_{k_i}=S_{k_i}-S_{k_i-1}$, $i=1,2,3$. The generalized Carleson
measure space $CMO^p(\widetilde{M})$ is defined, for
$f\in\big(\GGp(\beta_1,\beta_2,\beta_3;\gamma_1,\gamma_2,\gamma_3)\big)',$ by
\begin{eqnarray}\label{product-CMOp-norm}
\|f\|_{CMO^p(\widetilde{M})}
&=&\sup_{\Omega} \bigg\{
\frac{\displaystyle 1}{\displaystyle \mu(\Omega)^{{2\over
p}-1}}\int_{\Omega} \sum_{k_1,k_2,k_3} \sum_{I_1\times I_2 \times I_n\subseteq \Omega}
\big|D_{k_1}D_{k_2} D_{k_3}(f)(x_1,x_2,x_3)\big|^2\\
\hskip1cm&&\times\chi_{I_1}(x_1)\chi_{I_2}(x_2)\chi_{I_3}(x_3)dx_1dx_2 dx_3
\bigg\}^{1\over 2}<\infty,\nonumber
\end{eqnarray}
where $\Omega$ are taken over all open sets in $\widetilde{M}$ with
finite measures and for each $k_i$, $I_i$ ranges over all
the dyadic cubes in $M_i$ with length
$\ell(I_i)=2^{-k_i-N_i}$, $i=1,2,3$.
\end{definition}

To consider singular integral operators on $\widetilde{M},$ we first introduce the space $C_0^\eta(\widetilde{M})$ by induction. Note that we have introduced $C_0^\eta(M_1\times M_2)$ in Subsection 3.1. A function $f(x_1,x_2,x_3)$ is said to be in $C_0^\eta(\widetilde{M})$ if $f$ has compact support and
$$ \|f(x_1,x_2,\cdot)\|_{C_0^\eta(M_1\times M_2)} \in C_0^\eta(M_3). $$

Now we introduce a class of {\it product Calder\'on--Zygmund singular integral
operators} on $\widetilde{M}$.

Let $T: C_0^\eta(\widetilde{M})\rightarrow \big(C_0^\eta(\widetilde{M})\big)'$ be a linear operator with an associated distribution kernel $K(x_1,y_1,x_2,y_2,x_3,y_3)$,
which is a continuous function on
$\widetilde{M}\backslash \{(x_1,y_1,x_2,y_2,x_3,y_3):\ x_i=y_i,\ {\rm for\ some\
} i,\ 1\leq i\leq 3\}$. Moreover,
\begin{itemize}
\item[\rm (i)] $\langle T(\varphi_1\otimes\varphi_2\otimes\varphi_3),
\psi_1\otimes\psi_2\otimes\psi_3\rangle$
\item[] $=\int K(x_1,y_1,x_2,y_2,x_3,y_3)
\prod_1^3\varphi_i(x_i)\psi_i(y_i)dx_1dy_1dx_2dy_2 dx_3dy_3$
\item[] whenever $\varphi_i$ and $\psi_i$ are in $C_0^\eta(M_i)$ with disjoint supports,
for $1\leq i\leq 3$.\\
\item[\rm (ii)] There exists a Calder\'on--Zygmund valued operator $K_3(x_3,y_3)$ on $M_1\times M_2$ such that
\begin{eqnarray*}
&&\langle T(\varphi_1\otimes\varphi_2\otimes\varphi_3),
\psi_1\otimes\psi_2\otimes\psi_3\rangle\\[5pt]
&&\hskip.6cm =\int \langle K_3(x_3,y_3)(\varphi_1\otimes\varphi_{2}), \psi_1\otimes\psi_{2}\rangle
\varphi_3(x_3)\psi_3(y_3) dx_3dy_3
\end{eqnarray*}
whenever $\varphi_i$ and $\psi_i$ are in $C_0^\eta(M_i)$
for $1\leq i\leq 3$ and supp$\varphi_3\cap$supp$\psi_3=\emptyset$. Moreover, $\|K_3(x_3,y_3)\|_{CZ(M_1\times M_2)}$ as a function of $x_3, y_3\in M_3,$ satisfies the following conditions:
\begin{itemize}
\item[(ii-a)] $\|K_3(x_3,y_3)\|_{CZ,1,2}\leq C V(x_3,y_3)^{-1}$;
\item[(ii-b)] $\displaystyle \|K_3(x_3,y_3)-K_3(x_3,y_3^{'})\|_{CZ,1,2}$
\item[] $\displaystyle \leq C \Big(\frac{d_3(y_3,y_3^{'})}{d_3(x_3,y_3)}\Big)^{\varepsilon}V(x_3,y_3)^{-1}$\qquad if\ \ $\displaystyle d_3(y_3,y_3^{'})\le {d_3(x_3,y_3)\over 2A}$;
\item[(ii-c)] $\displaystyle \|K_3(x_3,y_3)-K_3(x_3^{'},y_3)\|_{CZ,1,2}$
\item[] $\displaystyle \leq C \Big(\frac{d_3(x_3,x_3^{'})}{d_3(x_3,y_3)}\Big)^{\varepsilon}V(x_3,y_3)^{-1}$\qquad if\ \ $\displaystyle d_3(x_3,x_3^{'})\le {d_3(x_3,y_3)\over 2A}$.
\end{itemize}
Here we use $\|\cdot\|_{CZ(M_1\times M_2)}$ to denote the Calder\'on--Zygmund norm of the product Calder\'on--Zygmund operators on $M_1\times M_2.$ More precisely, $\|T\|_{CZ(M_1\times M_2)}=\|T\|_{L^2\rightarrow L^2}+|K|_{CZ(M_1\times M_2)}$, where $|K|_{CZ,1,2}=\min(|K_1|_{CZ},|K_2|_{CZ})$ by considering $K$ as a pair $(K_1,K_2)$ as in Subsection 3.1\\
\item[\rm (iii)] There exists a Calder\'on--Zygmund valued operator $K_{1,2}(x_1,y_1,x_{2},y_{2})$ on $M_3$ such that
\begin{eqnarray*}
&&\langle T(\varphi_1\otimes\varphi_2\otimes\varphi_3),
\psi_1\otimes\psi_2\otimes\psi_3\rangle\\[5pt]
&&\hskip.6cm =\int \langle K_{1,2}(x_1,y_1,x_{2},y_{2})(\varphi_{3}), \psi_{3}\rangle
\prod_{i=1}^{2}\varphi_i(x_i)\psi_i(y_i) dx_1dy_1 dx_{2}dy_{2}
\end{eqnarray*}
whenever $\varphi_i$ and $\psi_i$ are in $C_0^\eta(M_i)$
for $1\leq i\leq 3$, and $\varphi_i$ and $\psi_i$ have disjoint supports for $i=1,2.$ Moreover, as a function of $(x_1,y_1,x_{2},y_{2})$, $K_{1,2}(x_1,y_1,x_{2},y_{2})$ satisfies the following conditions:
\begin{itemize}
\item[(iii-a)] $\|K_{1,2}(x_1,y_1,x_{2},y_{2})\|_{CZ}\leq C V(x_1,y_1)^{-1}V(x_2,y_2)^{-1}$;
\item[(iii-b)] $\|K_{1,2}(x_1,y_1,x_{2},y_{2})-K_{1,2}(x_1^{'},y_1,x_{2},y_{2})\|_{CZ}$
\item[] $\displaystyle\leq C \Big(\frac{d_1(x_1,x_1^{'})}{d_1(x_1,y_1)}\Big)^{\varepsilon}V(x_1,y_1)^{-1}V(x_2,y_2)^{-1}$\qquad  if\ \ $\displaystyle d_1(x_1,x_1^{'})\le {d_1(x_1,y_1)\over 2A}$;
\item[(iii-c)] above (iii-b) holds for interchanging $x_1, x_2$ with $y_1, y_2$;
\item[(iii-d)] $\|K_{1,2}(x_1,y_1,x_{2},y_{2})-K_{1,2}(x_1^{'},y_1,x_{2},y_{2})$
\item[]\hskip 3cm $-K_{1,2}(x_1,y_1,x_{2}^{'},y_{2})+K_{1,2}(x_1^{'},y_1,x_{2}^{'},y_{2})\|_{CZ}$
\item[] $\displaystyle\leq C \Big(\frac{d_1(x_1,x_1^{'})}{d_1(x_1,y_1)}\Big)^{\varepsilon}V(x_1,y_1)^{-1}
    \Big(\frac{d_2(x_2,x_2^{'})}{d_2(x_2,y_2)}\Big)^{\varepsilon}V(x_2,y_2)^{-1}$
\item[]\qquad if\ \ $\displaystyle d_1(x_1,x_1^{'})\le {d_1(x_1,y_1)\over 2A}$\ \ and\ \ $\displaystyle d_2(x_2,x_2^{'})\le {d_2(x_2,y_2)\over 2A}$
\item[(iii-e)] above (iii-d) holds for interchanging $x_1, x_2$ with $y_1, y_2$.
\end{itemize}
\item[\rm (iv)] The same conditions (ii) and (iii)
hold for any permutation of the indices $1,2,3$. That is, we can consider $T$ as a pair of $(K_{1,3},K_2)$, as well as a pair of $(K_{1},K_{2,3})$.
Both $K_1$ and $K_2$ satisfy (ii). Similarly, both $K_{1,3}$ and $K_{2,3}$ satisfy (iii).\\
\end{itemize}

To state the $T1$ theorem on $\widetilde M,$ we need to deal with the partial adjoint operators $\widetilde{T}$.
We have the following two classes of partial adjoint operators. For the first class,
$\widetilde{T}_{1},$ the partial adjoint operator of $T,$ is defined as
\begin{eqnarray*}
\langle \widetilde{T}_{1}(\varphi_1\otimes\varphi_2\otimes\varphi_3),
\psi_1\otimes\psi_2\otimes\psi_3\rangle =\langle T(\psi_1\otimes\varphi_2\otimes\varphi_3),
\varphi_1\otimes\psi_2\otimes\psi_3\rangle,
\end{eqnarray*}
and similarly for ${\widetilde T}_{2}$ and ${\widetilde T}_{3}$. For the second class, $\widetilde{T}_{1,2},$ the partial adjoint operator of $T,$ is defined as
\begin{eqnarray*}
\langle \widetilde{T}_{1,2}(\varphi_1\otimes\varphi_2\otimes\varphi_3),
\psi_1\otimes\psi_2\otimes\psi_3\rangle =\langle T(\psi_1\otimes\psi_2\otimes\varphi_3),
\varphi_1\otimes\varphi_2\otimes\psi_3\rangle,
\end{eqnarray*}
and similarly $\widetilde{T}_{1,2}$ and $\widetilde{T}_{2,3}$. Thus, there are totally $C_3^1+C_3^2=6$ partial adjoint operators.

We also define the weak boundedness property. Let $T$ be a product Calder\'on--Zygmund singular integral operator
on $\widetilde M.$ We say that  $T$ has the WBP if
\begin{eqnarray*}
&&\|\langle K_{1}(\varphi_2\otimes\varphi_3), \psi_2\otimes\psi_3\rangle \|_{CZ(M_1)}\leq C V_{r_2}(x_2^0)V_{r_3}(x_3^0)\\[5pt]
&&\hskip2cm {\rm
for\ all\ } \varphi_2, \psi_2\in A_{M_2}(\delta, x_2^0, r_2),\ \varphi_3, \psi_3\in A_{M_3}(\delta, x_3^0, r_3) \ {\rm and},\\[5pt]
&&\|\langle K_{1,2}(\varphi_3), \psi_3\rangle \|_{CZ(M_1\times M_2)}\leq C V_{r_3}(x_3^0)\hskip1cm {\rm
for\ all\ } \varphi_3, \psi_3\in A_{M_3}(\delta, x_3^0, r_3),
\end{eqnarray*}
and the same conditions
hold for $K_{1}$, $K_{2}$ and $K_{1,3}$, $K_{2,3}$, respectively.

Now we can state the $T1$ theorem on $\widetilde M.$

\vskip.3cm \noindent{\bf Theorem A$^{'}$}\ \ Let $T$ be a product
Calder\'on--Zygmund singular integral operator on $\widetilde M$.
Then $T$ is bounded on $L^2(\widetilde M)$
if and only if $T1$, $T^*1$, $\widetilde{T}_11$, $\widetilde{T}_21$, $\widetilde{T}_31$, $\widetilde{T}_{1,2}1$, $\widetilde{T}_{1,3}1$ and $\widetilde{T}_{2,3}1$.
lie on $BMO(\widetilde M)$ and $T$ has the weak boundedness
property.

\vskip.3cm \noindent{\bf Theorem B$^{'}$}\ \ Let $T$ be the $L^2$ bounded product
Calder\'on--Zygmund singular integral operator on $\widetilde M$. Then $T$ extends to a bounded operator from
$H^p(\widetilde{M}), \max\big(\frac{ Q_1}{ Q_1+ \vartheta_1 },\frac{ Q_2}{ Q_2+ \vartheta_2 },\frac{
Q_3}{ Q_3+ \vartheta_3 }\big)<p\leq1,$ to itself if and only if
$(T^*)_1(1)=(T^*)_2(1)=(T^*)_3(1)=0.$

\vskip.3cm \noindent{\bf Theorem C$^{'}$}\ \ Let $T$ be the $L^2$ bounded product
Calder\'on--Zygmund operator on $\widetilde M$. Then $T$ extends to a bounded operator from
$CMO^p(\widetilde{M}), \max\big(\frac{ 2Q_1}{ 2Q_1+ \vartheta_1
},\frac{ 2Q_2}{ 2Q_2+ \vartheta_2
},\frac{ 2Q_3}{ 2Q_3+ \vartheta_3 }\big)<p\leq1,$ to itself,
particularly from $BMO(\widetilde{M})$ to itself, if and only if
$T_1(1)=T_2(1)=T_3(1)=0.$

\smallskip

\begin{flushleft}
Department of Mathematics, Auburn University, Auburn, AL 36849-5310, U.S.A.\\
E-mail address: hanyong@auburn.edu
\end{flushleft}

\smallskip

\begin{flushleft}
Department of Mathematics, Sun Yat-Sen University, Guangzhou 510275, China\\
E-mail address: liji6@mail.sysu.edu.cn
\end{flushleft}

\smallskip

\begin{flushleft}
Department of Mathematics, National Central University, Chung-Li 320, Taiwan\\
E-mail address: clin@math.ncu.edu.tw
\end{flushleft}

\end{document}